\documentclass[11pt,a4paper,twoside]{amsart}
\usepackage{amsmath}
\usepackage{amssymb}
\usepackage{amsfonts}
\usepackage{amscd}
\usepackage{xcolor}
\usepackage{url}
\usepackage[hidelinks,bookmarksnumbered]{hyperref}
\hypersetup{pdftitle={Cuspidal representations over a two-dimensional local field},
pdfauthor={Alexander Braverman and David Kazhdan}}
\newtheorem{theorem}{Theorem}[section]
\newtheorem{lemma}[theorem]{Lemma}
\newtheorem{proposition}[theorem]{Proposition}
\newtheorem{corollary}[theorem]{Corollary}
\theoremstyle{definition}
\newtheorem{definition}[theorem]{Definition}
\newtheorem{claim}[theorem]{Claim}
\newtheorem{example}[theorem]{Example}

\theoremstyle{remark}
\newtheorem{remark}[theorem]{Remark}

\newcommand{\tori}{\mathfrak t}

\newcommand{\mE}{\mathbb E}
\newcommand{\mF}{\mathbb F}
\newcommand{\mG}{\mathbb G}

\newcommand{\mI}{\mathbb I}

\newcommand{\mK}{\mathbb K}
\newcommand{\mL}{\mathbb L}

\newcommand{\mN}{\mathbb N}

\newcommand{\mP}{\mathbb P}
\newcommand{\mQ}{\mathbb Q}

\newcommand{\mS}{\mathbb S}
\newcommand{\mT}{\mathbb T}
\newcommand{\mU}{\mathbb U}
\newcommand{\mV}{\mathbb V}
\newcommand{\mW}{\mathbb W}

\newcommand{\mY}{\mathbb Y}
\newcommand{\mA}{\mathbb A}
\newcommand{\mZ}{\mathbb Z}

\newcommand{\ho}{\hookrightarrow}
\newcommand{\Gg}{\gamma}
\newcommand{\GG}{\Gamma}
\newcommand{\bO}{\Omega}
\newcommand{\bo}{\omega}
\newcommand{\bl}{\lambda}

\newcommand{\bs}{\sigma}

\newcommand{\D}{\Delta}

\newcommand{\kk}{\kappa}

\newcommand{\mcA}{\mathcal A}

\newcommand{\mcC}{\mathcal C}

\newcommand{\mcF}{\mathcal F}

\newcommand{\mcH}{\mathcal H}

\newcommand{\mcM}{\mathcal M}

\newcommand{\mcO}{\mathcal O}

\newcommand{\mcR}{\mathcal R}
\newcommand{\mcS}{\mathcal S}

\newcommand{\mcW}{\mathcal W}

\newcommand{\CC}{\mathbb C}

\newcommand{\ZZ}{\mathbb Z}

\newcommand{\fg}{\mathfrak g}

\newcommand{\fm}{\mathfrak m}

\newcommand{\ft}{\mathfrak t}

\newcommand{\ti}{\widetilde}

\newcommand{\emp}{\emptyset}

\newcommand{\supp}{\operatorname{supp}}

\newcommand{\sm}{\setminus}
\newcommand{\prl}{\varprojlim}

\newcommand{\sms}{\smallskip}
\newcommand{\ms}{\medskip}
\newcommand{\Ad}{\operatorname{Ad}}
\definecolor{Red}{rgb}{0.7,0,0}

\newcommand{\otheta}{\overline{\theta}}
\newcommand{\ind}{\operatorname{ind}}
\newcommand{\Aut}{\operatorname{Aut}}
\newcommand{\St}{\operatorname{St}}
\newcommand{\Sp}{\operatorname{Sp}}
\newcommand{\Id}{\operatorname{Id}}
\newcommand{\Hom}{\operatorname{Hom}}
\newcommand{\End}{\operatorname{End}}
\newcommand{\SL}{\operatorname{SL}}

\newcommand{\PGL}{\operatorname{PGL}}
\newcommand{\GL}{\operatorname{GL}}
\newcommand{\Vect}{\operatorname{Vect}}
\newcommand{\Ind}{\operatorname{Ind}}
\newcommand{\Pro}{\operatorname{Pro}}
\newcommand{\Rep}{\operatorname{Rep}}
\newcommand{\Repall}{\operatorname{Rep}_{\mathrm{pro}}}
\newcommand{\Gr}{\operatorname{Gr}}
\newcommand{\oGr}{\overline{\Gr}}

\newcommand{\bSet}{\mathbf{Set}}

\newcommand{\mSet}{\mathbb{S}\mathrm{ets}}
\newcommand{\pSet}{\mathsf{S}}

\newcommand{\mVect}{\mathbb{V}\mathrm{ect}}
\newcommand{\NN}{\mathbb N}
\newcommand{\Sets}{\operatorname{Sets}}
\newcommand{\Prof}{\operatorname{Prof}}
\newcommand{\Lpf}{\operatorname{Lpf}}
\newcommand{\Meas}{\operatorname{Meas}_c}
\newcommand{\Res}{\operatorname{Res}}
\newcommand{\im}{\operatorname{im}}
\newcommand{\triv}{\operatorname{triv}}
\newcommand{\Coinv}{\operatorname{Coinv}}
\newcommand{\Inf}{\operatorname{Inf}}

\begin{document}
\title[Cuspidal representations over a two-dimensional local field]{An analogue of irreducible cuspidal representations for the group $\PGL(2)$ over a two-dimensional local field}
\author{Alexander Braverman and David Kazhdan}
\date{Revised November 11, 2026}
\begin{abstract}
Let $F$ be a local non-Archimedean field of odd residue characteristic, let $K=F((t))$, and let $G=\PGL(2)$. We construct analogues of irreducible cuspidal representations of $G(F)$ for the group $G(K)$. The construction starts, up to a metaplectic ambiguity in one even-depth ramified case, from a quadratic extension $L/K$ and a character $\theta:L^*/K^*\to\CC^*$ which is not Galois invariant. We show that the restriction of each representation so obtained to the Borel subgroup $P(K)\subset G(K)$ is irreducible. In contrast with the classical case, these restrictions are not isomorphic to the standard irreducible representation of $P(K)$.

 We also define the notion of depth for an arbitrary irreducible representation of $G(K)$ and show the converse to the above results under an additional hypothesis that a representation if {\em very special} (this additional hypothesis is not needed when the depth is a half-integer, but we do need it for integral depth). Finally, we explain how to extend all the results to the case of $\GL(2)$.
\end{abstract}
\maketitle

\section{Introduction}
\subsection{The setup}
Let $F$ be a local non-Archimedean field of odd residue characteristic with ring of integers $O_F$. Set $K=F((t))$ and $O=F[[t]]$. Let $G$ denote the algebraic group $\PGL(2)$. We denote by $P\subset G$ the Borel subgroup.\footnote{We denote it by $P$ rather than by $B$ because we hope to generalize the results of this paper to $\PGL(n)$ for $n\geq2$, in which case $P$ will be the maximal parabolic subgroup stabilizing a line.} We denote by $T$ the diagonal torus and by $N$ the unipotent radical of $P$.
Let $G'$ be the group scheme over $O$ whose group of $O$-points is the stabilizer, under the adjoint action, of the lattice $Oe\oplus Oh\oplus tOf$, where $e,h,f$ is the standard basis of the Lie algebra $\fg=\mathfrak{sl}(2)$. The base change of $G'$ to $K$ is canonically isomorphic to $G$. In particular, $G'(O)$ is canonically a subgroup of $G(K)$, and $G'(O)$ contains $P(O)$.
In this paper we study a class of \emph{special} representations of $G(O)$, $G'(O)$, and $G(K)$. The representation theory of $G(O)$ and $G'(O)$ belongs to the classical representation theory of $p$-adic groups: these groups are projective limits of finite-dimensional algebraic $F$-groups, and we consider representations which factor through smooth representations of finite-dimensional quotients. The notion of a smooth representation of $G(K)$ is subtler. In particular, one must work with actions on pro-vector spaces rather than only on vector spaces; the corresponding category was defined in \cite{GK}. We denote the categories of ordinary smooth representations of $G(O)$ and $G'(O)$ by $\mcR(G(O))$ and $\mcR(G'(O))$. For $G(K)$ we work in $\Rep(G(K))$, whose underlying pro-vector spaces belong to $\mVect=\Pro_{\aleph_0}(\Vect)$. The category of representations on arbitrary objects of $\Pro(\Vect)$ is denoted by $\Repall(G(K))$; these conventions are defined in Section~\ref{pro-vect}. Ordinary inverse-limit realization is not faithful on arbitrary pro-objects, so it cannot replace the pro-vector-space formalism.
\begin{definition}\label{spe}
\begin{enumerate}
\item A complex representation $\pi$ of dimension greater than $1$ of $G(O)$ or $G'(O)$ is \emph{special} if its restriction to $P(O)$ is irreducible.
\item A representation $\Pi$ of $G(K)$ is \emph{special} if it is not one-dimensional and its restriction to $P(K)$ is irreducible.
\end{enumerate}
\end{definition}
The main purpose of this paper is to construct special representations of $G(K)$. We conjecture that the construction gives all special representations. To motivate Definition \ref{spe}, we first review the parallel results in the classical situation, in which $K$ is replaced by $F$.
\subsection{Classical results}\label{classical}
All the results of this subsection can be found in \cite{BH}. In this subsection only, $G'$ denotes the group scheme over $O_F$ defined as above, with the uniformizer $\varpi$ in place of $t$.
The embedding $P(F)\ho G(F)$ induces embeddings $P(O_F)\ho G(O_F)$ and $P(O_F)\ho G'(O_F)$. A representation of dimension greater than $1$ of one of the groups $G(O_F)$, $G'(O_F)$, or $G(F)$ is called special if its restriction to $P(O_F)$ or $P(F)$, respectively, is irreducible. An irreducible representation of $G(F)$ is special if and only if it is cuspidal.
\begin{theorem}\label{classical-theorem}
\begin{enumerate}
\item Let $E/F$ be a quadratic extension and let $\theta:E^*/F^*\to\CC^*$ be a character such that $\theta\neq\otheta$, where $\otheta$ is the Galois conjugate of $\theta$. If $E/F$ is unramified, one can attach to $\theta$ a special representation $\pi(E,\theta)$ of $G(O_F)$. If $E/F$ is ramified, one can attach to $(E,\theta)$ a special representation $\pi(E,\theta)$ of $G'(O_F)$.
\item Every special representation of $G(O_F)$ or $G'(O_F)$ is isomorphic to some $\pi(E,\theta)$.
\item Let $\pi$ be an irreducible representation of either $G(O_F)$ or $G'(O_F)$, and let $\Pi(\pi)$ be its compact induction to $G(F)$. Then $\pi$ is special if and only if $\Pi(\pi)$ is irreducible; in this case $\Pi(\pi)$ is special. Thus $\Pi(E,\theta):=\Pi(\pi(E,\theta))$ is irreducible and cuspidal, and every special representation of $G(F)$ is of this form.
\item The representations $\Pi(E_1,\theta_1)$ and $\Pi(E_2,\theta_2)$ are isomorphic if and only if $E_1=E_2$ and either $\theta_1=\theta_2$ or $\theta_1=\otheta_2$. The analogous statement holds for $\pi(E,\theta)$.
\end{enumerate}
\end{theorem}
The third assertion of Theorem \ref{classical-theorem} gives an equivalent characterization of special representations of $G(O_F)$ and $G'(O_F)$: a representation is special precisely when its compact induction to $G(F)$ is irreducible. One could try to use this as the definition after replacing $F$ by $K$, because an analogue of induction is constructed in \cite{GK}. It is more convenient here to use Definition \ref{spe}, since it allows us to work directly with $G(O)$ and $G'(O)$. We do not know whether the two definitions are equivalent for $G(K)$.
\begin{remark}
The representations $\Pi(\pi)$ are unitarizable. For these compact
inductions, irreducibility is equivalent to
$\End_{G(F)}(\Pi(\pi))=\CC$. We recall the compact-open argument
for this criterion.

Let $H=G(O_F)$ or $G'(O_F)$, let $\sigma$ be an irreducible smooth
representation of $H$, and put
\[
 V=\operatorname{c\!-\!Ind}_H^{G(F)}\sigma.
\]
Let $V_0\simeq\sigma$ be the subspace of induced functions supported
on $H$. Frobenius reciprocity gives
\[
 \Hom_H(\sigma,V)\simeq\End_{G(F)}(V).
\]
If this space is one-dimensional, semisimplicity of smooth
representations of the compact group $H$ shows that the
$\sigma$-isotypic subspace of $V|_H$ is exactly $V_0$.
Let $e_\sigma$ be the corresponding isotypic projector.
For a nonzero $G(F)$-invariant subspace $W\subset V$, translate a
nonzero $f\in W$ so that $f(1)\ne0$. Evaluation at $1$ is an
$H$-equivariant map from $V|_H$ to $\sigma$, so
\[
 e_\sigma f\in W\cap V_0,\qquad
 (e_\sigma f)(1)=f(1)\ne0.
\]
Since $V_0$ is irreducible under $H$, it follows that $V_0\subset W$.
Its $G(F)$-translates generate $V$, hence $W=V$.
The converse follows from Schur's lemma for irreducible smooth
representations of $G(F)$.
\end{remark}
\subsection{What is done in this paper?}
We partially extend Theorem \ref{classical-theorem} to $G(O)$, $G'(O)$, and $G(K)$.
\subsection{\texorpdfstring{Special and elliptic representations of $G(O)$ and $G'(O)$}{Special and elliptic representations of G(O) and G'(O)}}
\subsubsection{Ellipticity}
The definition of a special representation of $G(O)$ or $G'(O)$ is not explicit. Our first result identifies special representations with representations satisfying a more concrete ellipticity condition.\footnote{The notion of ellipticity makes sense for every reductive group.}

Let $G(O)(m)$ be the $m$-th congruence subgroup of $G(O)$, where $m\geq1$, and put $G(O)_m=G(O)/G(O)(m+1)$. For an irreducible representation $\pi$ of $G(O)$, define its depth $m(\pi)$ to be the least $m\geq0$ such that $\pi$ factors through $G(O)_m$. The analogous definition for $G'(O)$, with the half-integral normalization used here, is given in Definition~\ref{depth}.

For $m>0$, the group $G(O)_m$ contains a commutative normal subgroup $G(O)_m(m)$ canonically isomorphic to $\fg(F)$. If $m(\pi)>0$, let $\bo(\pi)\subset\mathfrak{sl}(2,F)$ be the support of the restriction of $\pi$ to this subgroup, using the fixed additive character of $F$ and the trace pairing to identify its character group with $\mathfrak{sl}(2,F)$.
\begin{definition}
An irreducible representation $\pi$ of $G(O)$ is elliptic if either $m(\pi)=0$ and $\pi$ is a cuspidal representation of $G(F)=G(O)_0$, or $m(\pi)>0$ and $\bo(\pi)$ is an elliptic adjoint orbit in $\fg(F)$. The notion of ellipticity for $G'(O)$ is defined similarly; see Definition \ref{ac}.
\end{definition}
We prove that an irreducible representation of $G(O)$ or $G'(O)$ is special if and only if it is elliptic.
\subsubsection{\texorpdfstring{A description of special representations of $G(O)$ and $G'(O)$}{A description of special representations of G(O) and G'(O)}}
\begin{definition}
A quadratic extension $L/K$ is unramified if $L=E((t))$ for a quadratic extension $E/F$; otherwise it is ramified.
\end{definition}

Let $\widehat G(O)$ and $\widehat G'(O)$ denote the sets of isomorphism classes of special representations of $G(O)$ and $G'(O)$. In the formulation of Theorem \ref{L} we use notation for subsets  $\widehat G(O)_m \subset \widehat G(O) $ and $ \widehat G'(O)_m \subset \widehat G'(O)$  of depth $m$ (see Definition \ref{depth}) and also for parameter sets $\bar\Xi^L(m)$ and $\ti\Xi^L(m)$ (see Definition \ref{qext1}). 

For $m>0$, the notation $\widehat G(O)^E_m$ used here agrees
with $\widehat G_m^E$ of Definition~\ref{def1}, by the equivalence
of ellipticity and specialness proved below. At $m=0$ we use the
Bushnell--Henniart tame parametrization with trivial central character,
as specified in Definition~\ref{depth-zero-convention} and
Proposition~\ref{depth-zero-parameterization}; in particular
$\widehat G(O)^E_0=\widehat G_0^E$.
The field $E_\pi$ at depth zero is the field of the admissible pair
in that parametrization, not a field obtained from a last congruence
layer. Here and throughout, depth zero means $t$-adic depth zero;
the resulting cuspidal representation of $G(F)$ may have arbitrary
$F$-depth. Put
$\widehat G(O)^E=\bigsqcup_{m\geq0}\widehat G(O)^E_m$.
The character quotients $\bar\Xi^L(m)$ and $\ti\Xi^L(m)$ are
parameter sets, not in general groups.

\begin{theorem}\label{L}
\begin{enumerate}
\item The set $\widehat G(O)$ is the disjoint union of subsets $\widehat G(O)^E$, where $E$ runs through the quadratic extensions of $F$.
\item The set $\widehat G'(O)$ is the disjoint union of subsets $\widehat G'(O)^L$, where $L$ runs through the ramified quadratic extensions of $K$.
\item For every ramified quadratic extension $L/K$, the set $\widehat G'(O)^L$ is naturally in bijection with the characters $\theta:L^*/K^*\to\CC^*$ satisfying $\theta\neq\otheta$, modulo $\theta\sim\otheta$.
\item Let $E/F$ be quadratic and put $L=E((t))$. The depth-zero subset $\widehat G(O)^E_0$ is naturally in bijection with $\bar\Xi^L(0)$.
\item For every positive odd depth $m$, the set $\widehat G(O)^E_m$ is naturally in bijection with $\bar\Xi^L(m)$.
\item If $m>0$ is even and $E/F$ is unramified, then $\widehat G(O)^E_m$ is naturally in bijection with
 $\bar\Xi^L(m)$.
\item If $m>0$ is even and $E/F$ is ramified, then
 $\widehat G(O)^E_m$ is naturally a $\ZZ/2\ZZ$-torsor over $\ti\Xi^L(m)$. 
\end{enumerate}
\end{theorem}
\begin{remark}
In the last case the group  $\D _L $ of quadratic characters of $T^L=L^*/K^*$ is isomorphic to $\ZZ/2\ZZ$ and the map $\bar\Xi^L(m)\to\ti\Xi^L(m)$ is another
 $\ZZ/2\ZZ$-torsor over $\ti\Xi^L(m)$. We do not know whether the  $\ZZ/2\ZZ$-torsors 
 $\widehat G(O)^E_m$ and $ \bar\Xi^L(m) $ over $ \ti\Xi^L(m) $ are canonically isomorphic.
\end{remark}

\subsection{\texorpdfstring{Special representations of $G(K)$}{Special representations of G(K)}}
Our ultimate goal is to construct irreducible representations of $G(K)$. Although we do not have a direct analogue of the usual notion of a cuspidal representation for $G(K)$, Definition \ref{spe} makes sense.
The analogue of induction from $G(O)$ or $G'(O)$, and more generally from a bounded subgroup of $G(K)$, was constructed in \cite{GK}. For a representation $\pi$ of $G(O)$ or $G'(O)$, denote the resulting representation of $G(K)$ by $\Pi(\pi)$.
\begin{theorem}\label{irreducible}
If $\pi$ is a special representation of $G(O)$ or $G'(O)$, then $\Pi(\pi)$ is irreducible. Moreover, its restriction to $P(K)$ is irreducible.
\end{theorem}
We define the depth $m(\Pi)$ of a special representation of $G(K)$
as the least congruence level with nonzero coinvariants; see
Definition~\ref{md:representation-data}. For $\Pi=\Pi(\pi)$ it
agrees with $m(\pi)$, by
Proposition~\ref{induction-distinguishes-types}.

We also prove a converse to Theorem~\ref{irreducible} under an
additional hypothesis. Namely, we introduce the notion of a
\emph{very special} representation of $G(K)$
(Definition~\ref{md:very-special}). This is a priori stronger than
specialness, but the two notions may be equivalent.
Let $\Pi$ be an irreducible representation satisfying the
admissibility condition of Definition~\ref{md:representation-data}.
If $m(\Pi)$ is integral and $\Pi$ is very special, then
$\Pi\simeq\Pi(\pi)$ for an ordinary special representation $\pi$
of $G(O)$. At nonintegral half-integral depth, that is, when
$m(\Pi)\in\frac12+\ZZ_{\geq0}$, the same conclusion holds with
$G'(O)$ in place of $G(O)$, without the very special
assumption; see Theorems~\ref{md:thm:zero},
\ref{md:thm:integral}, and~\ref{md:thm:half}.
\label{intro:minimal-depth}

In both cases the inducing representation is recovered from the
coinvariants of $\Pi$ at its minimal depth. In particular, the
inducing group and the isomorphism class of $\pi$ are uniquely
determined by $\Pi$. Thus Theorem~\ref{irreducible} provides every
representation satisfying these hypotheses exactly once; see
Proposition~\ref{induction-distinguishes-types}.

Combining Theorems \ref{L} and \ref{irreducible}, we can attach, up to the even-depth ramified ambiguity described above, an irreducible representation $\Pi(L,\theta)$ of $G(K)$ to a quadratic extension $L/K$ and a character $\theta:L^*/K^*\to\CC^*$ satisfying $\theta\neq\otheta$. Replacing $\theta$ by $\otheta$ does not change the resulting representation. For the ordinary special local types constructed here, Proposition~\ref{induction-distinguishes-types} shows that induction introduces no further identifications. The remaining questions are stated in Section~\ref{quest}.
\subsection{Remarks}
For $G(F)$, the restriction of every special irreducible representation to $P(F)$ is irreducible, and all these restrictions are isomorphic. Indeed, $P(F)$ has a unique irreducible representation of dimension greater than $1$.
This statement fails after replacing $F$ by $K$. One can construct an irreducible representation $\Upsilon$ of $P(K)$ which is a direct analogue of the unique irreducible cuspidal representation of $P(F)$, but $\Upsilon$ does not extend to $G(K)$. We expect that $\Pi(\pi)|_{P(K)}$ depends on $\pi$; more precisely, we expect that representations of different depths have nonisomorphic restrictions. Nevertheless, all representations $\Pi(\pi)$ constructed from special representations of $G(O)$ or $G'(O)$ admit a $\ZZ$-filtration whose associated graded has a natural $P(K)$-action and is isomorphic to $\Upsilon$.
We expect that every special representation of $G(K)$ is induced from a special representation of $G(O)$ or $G'(O)$. The converse proved above establishes this under the stated hypotheses; the assertion without these additional hypotheses remains open.
Finally, $G(K)$ admits a natural central extension by $F^*$, and one can study representations of this extension at a prescribed level, as in \cite{GK}. The choice of level appears not to affect the special representations considered here, so throughout this paper we work at level zero.
\subsection{A remark about terminology}
We use the word ``Claim'' for auxiliary statements. We include proofs where needed; when a proof is omitted, we give a reference or the verification is straightforward.
We denote an algebraic variety by $Y$ (or $X$, etc.) and its set of points over a field $E$ by $Y(E)$. In particular, for an $F$-variety $Y$, the set $Y(F)$ carries its natural topology, and $Y(K)$ denotes its set of $K$-points. We use the same convention for schemes and ind-schemes.
\subsection{Organization of the paper}
Section \ref{deff} contains basic definitions. Section \ref{reptheory} reviews classical representation-theoretic results. Section \ref{PO} classifies smooth irreducible representations of $P(O)$. Sections \ref{GO} and \ref{G'O} classify special representations of $G(O)$ and $G'(O)$. The remainder of Part~\ref{part:construction} develops the necessary representation theory on pro-vector spaces. Section~\ref{pro-vect} introduces the categories and their filtrations, compares the two notions of irreducibility, and gives the required preliminary constructions; Section~\ref{P(K)} studies representations of $P(K)$, Section~\ref{Gr} discusses the affine Grassmannian, and Sections~\ref{G(K)-in}--\ref{GKinduced} establish the induction results for $G(K)$. In Part~\ref{part:minimal-depth} we prove the converse under the hypotheses above, recover the inducing representation, and conclude with the remaining questions.

For simplicity, the introduction discusses $G=\PGL(2)$. We extend
these results to $\GL(2)$ in Sections~\ref{GL2-extension}
and~\ref{glmd:main}.
For the minimal-depth converse for $\GL(2)$, admissibility is imposed
after central normalization and additionally requires finite
generation of the congruence coinvariants, as specified in
Definition~\ref{glmd:definitions}.
\subsection{Acknowledgments}
We thank Michael Finkelberg for very helpful discussions. During the preparation of this paper, D.K. was partially supported by ERC grant no.~101142781. A.B. was partially supported by NSERC. The final version of the paper was prepared when A.B. was visiting the Simons Laufer Mathematical Institute.

\clearpage

\part{Construction of special representations}\label{part:construction}

\section{Definitions and notation}\label{deff}
\subsection{Quadratic extensions and tori over general fields}
Let $K$ be a field of characteristic different from $2$, and let $L/K$ be a quadratic extension.
\begin{definition}
\begin{enumerate}
\item Put $T^L=L^*/K^*$. This is naturally the group of $K$-points of an algebraic $K$-torus.
\item Put $L^1=\{l\in L^*:N_{L/K}(l)=1\}$.
\end{enumerate}
\end{definition}
\begin{lemma}\label{j}
\begin{enumerate}
\item Every quadratic extension $L/K$ has the form $L=K(\sqrt a)$ for some $a\in K^*\sm(K^*)^2$.
\item The map $T^L\to L^1$ given by $l\mapsto l/\bar l$ is an isomorphism.
\end{enumerate}
\end{lemma}
\begin{proof}
The first assertion is the usual classification of quadratic extensions in characteristic different from $2$. For the second, the homomorphism $L^*\to L^1$, $l\mapsto l/\bar l$, has kernel $K^*$. Hilbert's Theorem~90 shows that it is surjective, so it induces the displayed isomorphism $L^*/K^*\to L^1$.
\end{proof}
\begin{remark}\label{jj}
We use Lemma \ref{j}(2) to identify $T^L$ with $L^1$.
\end{remark}
\begin{definition}\label{G-K-def}
\begin{enumerate}
\item $G=\PGL(2)$, and $P\subset G$ is the Borel subgroup of upper triangular matrices.
\item For $a\in K^*$, put $\widehat a=\begin{pmatrix}a&0\\0&1\end{pmatrix}\in P(K)$.
\item For $b\in K$, put $\widetilde b=\begin{pmatrix}1&b\\0&1\end{pmatrix}\in P(K)$.
\item The Lie algebra of $G$ is denoted by $\fg$.
\end{enumerate}
\end{definition}
\subsection{Laurent power series}
From now on, assume that $K=F((t))$, where $F$ is a field of characteristic different from $2$.
\begin{definition}\label{intf}
\begin{enumerate}
\item $O=F[[t]]$ and $\fm=tO$.
\item $O_m=O/\fm^{m+1}$ and $O_m(n)=\fm^n/\fm^{m+1}$ for $n\leq m$.
\item $O^*(n)=1+t^nO$ for $n\geq1$.
\item $v:K^*\to\ZZ$ is the valuation normalized by $v(t)=1$.
\end{enumerate}
\end{definition}
\subsubsection{Quadratic extensions}
\begin{definition}
\begin{enumerate}
\item A quadratic extension $K(\sqrt a)/K$ is unramified if $v(a)$ is even and ramified if $v(a)$ is odd.
\item $L^1(1)=\{l\in L^1:v_L(l-1)\geq1\}$, where $v_L$ is the normalized valuation on $L$.
\end{enumerate}
\end{definition}
\begin{claim}\label{quad-LK}
\begin{enumerate}
\item Every unramified quadratic extension $L/K$ has the form $E((t))$, where $E=F(\sqrt a)$ for some $a\in F^*\sm(F^*)^2$. In this case $L^1=E^1\times L^1(1)$.
\item Every ramified quadratic extension $L/K$ has the form $K(\sqrt a)$ with $v(a)=1$, and $L^1=\{\pm1\}\times L^1(1)$.
\end{enumerate}
\end{claim}
\begin{definition}\label{G-O-def}
\begin{enumerate}
\item For $g\in G(K)$, let $v(g)\in\ZZ_{\leq0}$ be the largest integer $r$ such that $\Ad(g)(\fg(O))\subset t^r\fg(O)$.
\item Let $I\subset G(O)$ be the Iwahori subgroup represented by matrices $\begin{pmatrix}a&b\\c&d\end{pmatrix}$ with $c\in\fm$.
\item Put $G'(O)=\ZZ/2\ZZ\ltimes I$, where the nontrivial element is represented by $s=\begin{pmatrix}0&1\\t&0\end{pmatrix}$.
\item A subset $A\subset G(K)$ is bounded if $v(A)$ is bounded below.
\end{enumerate}
\end{definition}
\begin{claim}
The group $G'(O)$ is the normalizer of $I$ in $G(K)$.
\end{claim}
\begin{remark}
We regard $G(O)$ and $G'(O)$ as groups of $F$-points of pro-algebraic groups over $F$.
\end{remark}
\begin{lemma}\label{BT}
\begin{enumerate}
\item Every bounded subgroup of $G(K)$ is conjugate to a subgroup of $G(O)$ or to a subgroup of $G'(O)$.
\item The product maps $P(K)\times G(O)\to G(K)$ and $P(K)\times G'(O)\to G(K)$ are surjective.
\end{enumerate}
\end{lemma}
\begin{proof}
Part (1) follows from Bruhat--Tits theory; see \cite[Section~3.2]{BT}. The first assertion of part (2) is the Iwasawa decomposition. For the second, use the  decomposition
$$
G(K)=P(K)I\sqcup P(K)wI,
$$
where $w$ is a lift of the nontrivial Weyl element. We may take $w=s\in G'(O)$, and $I\subset G'(O)$, whence $G(K)=P(K)G'(O)$.
\end{proof}
\subsection{\texorpdfstring{Tori in $G(O)$ and $G'(O)$}{Tori in G(O) and G'(O)}}
Let $L/K$ be quadratic and let $O_L$ be its ring of integers. The regular representation of $L^*$ on the two-dimensional $K$-vector space $L$ gives an embedding $L^*\ho\operatorname{GL}_K(L)$ and hence an embedding $T^L\ho G(K)$ after choosing a $K$-basis of $L$.
If $L/K$ is unramified, an $O$-basis of $O_L$ identifies the regular representation of $O_L^*$ with a subgroup of $\GL(2,O)$. If $L/K$ is ramified, choose an $O$-basis adapted to the lattice chain $O_L\supset\mathfrak p_L$; equivalently, identify this chain with the standard chain $O^2\supset Oe_1\oplus tOe_2$. The regular representation then normalizes the standard chain and gives an embedding into $G'(O)$.
\begin{claim}\label{imb}
\begin{enumerate}
\item If $L/K$ is unramified, the resulting subgroup $j^L(T^L)$ lies in $G(O)$, and its $G(O)$-conjugacy class is independent of the chosen $O$-basis of $O_L$.
\item If $L/K$ is ramified, the subgroup obtained from an adapted basis lies in $G'(O)$, and its $G'(O)$-conjugacy class is independent of the adapted basis.
\end{enumerate}
\end{claim}
\begin{proof}
In the unramified case, two $O$-bases differ by an element of $\GL(2,O)$. In the ramified case, two identifications of the lattice chain $O_L\supset\mathfrak p_L$ with the standard lattice chain differ by an element of its normalizer, namely $G'(O)$.
\end{proof}
\begin{remark}\label{tor}
We henceforth regard $T^L$ as a subgroup of $G(O)$ in the unramified case and of $G'(O)$ in the ramified case.
\end{remark}
\subsection{\texorpdfstring{Subgroups of $G(O)$ and $G'(O)$}{Subgroups of G(O) and G'(O)}}
\begin{definition}\label{G-prime-filtration}
\begin{enumerate}
\item For $n\in\ZZ_{>0}$, let $G(O)(n)$ be the $n$-th congruence subgroup of $G(O)$.
\item Put $G(O)_n=G(O)/G(O)(n+1)$, and, for $r\leq n$, let $G(O)_n(r)$ be the image of $G(O)(r)$ in $G(O)_n$.
\item For $r\in\frac12\ZZ_{>0}$, let $G'(O)(r)$ be represented by matrices $\begin{pmatrix}a&b\\c&d\end{pmatrix}$ satisfying
$$
a-1,d-1\in\fm^{[r+1/2]},\qquad b\in\fm^{[r]},\qquad c\in\fm^{[r]+1},
$$
where $[r]$ is the integral part of $r$.
\item Put $G'(O)_r=G'(O)/G'(O)(r+1/2)$, and, for $s\leq r$, let $G'(O)_r(s)$ be the image of $G'(O)(s)$ in $G'(O)_r$.
\end{enumerate}
\end{definition}
\begin{remark}
The groups $G(O)_m$ and $G'(O)_m$ are algebraic $F$-groups.
\end{remark}
\begin{definition}
For $m\in\ZZ_{>0}$ define the torus filtration by
\begin{enumerate}
\item If $L/K$ is unramified, put $T^L(m)=T^L\cap G(O)(m)$.
\item If $L/K$ is ramified, put $T^L(m)=T^L\cap G'(O)(m-1/2)$.
\end{enumerate}
\end{definition}
\subsection{Local fields: notation}
From now on, $F$ is a local non-Archimedean field of odd residue characteristic.
\begin{definition}\label{ti-psi}
\begin{enumerate}
\item $O_F\subset F$ is the ring of integers.
\item $\varpi\in O_F$ is a uniformizer, $k=O_F/\varpi O_F$, and $q=|k|$.
\item $O_F^*(m)=1+\varpi^mO_F$ for $m\geq1$.
\item $(O_F^*)_m=O_F^*/O_F^*(m+1)$ for $m\geq0$.
\item For $m\geq n\geq1$, let $(O_F^*)_m(n)$ be the image of $O_F^*(n)$ in $(O_F^*)_m$.
\item Fix an additive character $\widetilde\psi:F\to\CC^*$ which is trivial on $\varpi O_F$ and nontrivial on $O_F$. Put $\widetilde\psi_n(x)=\widetilde\psi(\varpi^{-n}x)$.
\end{enumerate}
\end{definition}
\begin{remark}
The character $\widetilde\psi_n$ is trivial on $\varpi^{n+1}O_F$ and nontrivial on $\varpi^nO_F$. The set of such characters is a transitive $O_F^*$-space.
\end{remark}
\begin{definition}\label{psi}
\begin{enumerate}
\item Define $\psi:K\to\CC^*$ by
$$
\psi(f)=\widetilde\psi\left(\operatorname{Res}_{t=0}(f\,dt)\right).
$$
\item For $b\in K$, put $\psi_b(f)=\psi(bf)$.
\item For $n\in\ZZ$, put $\psi_n=\psi_{t^{-n-1}}$. Thus $\psi=\psi_{-1}$.
\end{enumerate}
\end{definition}
\begin{definition}\label{qext1}
Let $L/K$ be quadratic.
\begin{enumerate}
\item Let $\Xi^L$ be the set of smooth characters $\theta$ of $T^L$ such that $\theta\neq\overline\theta$.
\item Let $r\geq0$ be the least nonnegative integer such that $\theta|_{T^L(r+1)}=1$. Define the depth $m(\theta)$ to be $r$ if $L/K$ is unramified and $r-\frac{1}{2}$ if $L/K$ is ramified.
\item Put $\Xi^L(m)=\{\theta\in\Xi^L:m(\theta)=m\}$.
\item Let $\bar\Xi^L(m)$ be the quotient of $\Xi^L(m)$ by $\theta\mapsto\theta^{-1}$.
\item Let $\ti\Xi^L(m)$ be the quotient of $\bar\Xi^L(m)$ by multiplication by quadratic characters of $T^L$.
\item Put $\D_L=\Hom(T^L,\ZZ/2\ZZ)$ and $\D_E=\Hom(T^E,\ZZ/2\ZZ)$.
\end{enumerate}
\end{definition}
\begin{lemma}\label{tori}
\begin{enumerate}
\item If $L/K$ is ramified, then $\D_L\simeq\ZZ/2\ZZ$.
\item If $L=E((t))$, then $\D_L\simeq\D_E$.
\item If $E/F$ is ramified, then $\D_E\simeq\ZZ/2\ZZ$.
\end{enumerate}
\end{lemma}
\begin{proof}
Every homomorphism to $\ZZ/2\ZZ$ factors through the quotient by squares. Since $1+tO$ is uniquely $2$-divisible by Hensel's lemma,
$$
K^*/(K^*)^2\simeq\langle t\rangle/\langle t^2\rangle\times F^*/(F^*)^2.
$$
If $L/K$ is ramified, write $L=K(x)$ with $x^2\in K$ of valuation $1$. The unit quotient contributes no additional square class, and
$$
L^*/K^*(L^*)^2\simeq\ZZ/2\ZZ,
$$
which proves (1). If $L=E((t))$, then
$$
L^*/K^*(L^*)^2\simeq E^*/F^*(E^*)^2,
$$
which proves (2). Applying the same argument to a ramified quadratic extension $E/F$ proves (3).
\end{proof}
\subsection{\texorpdfstring{Schwartz spaces of $F$-varieties}{Schwartz spaces of F-varieties}}
For a locally compact totally disconnected space $X$, let $\mcS(X)$ be the space of locally constant complex-valued functions with compact support on $X$. If $Y$ is an algebraic variety over $F$, we apply this notation to its space of points $Y(F)$. For an open subset $U\subset Y(F)$ and its complement $Z=Y(F)\sm U$, extension by zero and restriction give an exact sequence
$$
0\longrightarrow\mcS(U)\longrightarrow\mcS(Y(F))\longrightarrow\mcS(Z)\longrightarrow0.
$$
\section{Some results from classical representation theory}\label{reptheory}
Unless stated otherwise, the groups and subgroup embeddings in this section
are algebraic over $F$. Characters and representations are smooth complex
representations of their $F$-points; splittings of topological central
extensions are continuous homomorphisms.
\subsection{Coinvariants}
\begin{definition}\label{Vchi}
\begin{enumerate}
\item A linear algebraic group $H$ over $F$ is \emph{almost unipotent} if $H/R_u(H)$ is $F$-anisotropic, equivalently if $(H/R_u(H))(F)$ is compact.
\item A representation $\bs:H\to\Aut(V)$ is smooth if the stabilizer of every $v\in V$ is open.
\item Let $\mcR(H)$ denote the category of smooth representations of $H$.
\item Let $\widehat H$ denote the set of isomorphism classes of smooth irreducible representations of $H$. So if $H$ is commutative then  $\widehat H$ denotes the set of continuous characters of $H$.
\item If $(\bs,V)$ is a representation of $H$ and $\bl$ is a character of $H$, let $V(\bl)$ be the span of
$$
\{\bs(h)v-\bl(h)v:h\in H,\ v\in V\},
$$
and put $V_\bl=V/V(\bl)$.
\item If $H$ is commutative, put
$$
S(\bs)=\{\bl\in\widehat H:V_\bl\neq0\},
$$
and call it the support of $\bs$.
\end{enumerate}
\end{definition}
\begin{claim}
If $H$ is almost unipotent, then the functor $V\mapsto V_\bl$ from $\mcR(H)$ to $\Vect$ is exact for every character $\bl$ of $H$.
\end{claim}
\subsection{Compact induction}
\begin{definition}\label{indu}
Let $H$ be an algebraic $F$-group, let $L\subset H$ be closed, and let $(\bs,V)$ be a smooth representation of $L$. The compact induction $\ind_L^H(\bs)$ is the space of locally constant functions $f:H\to V$ satisfying
\begin{enumerate}
\item $f(hl^{-1})=\bs(l)f(h)$ for $h\in H$ and $l\in L$;
\item $\supp(f)\subset CL$ for some compact subset $C\subset H$.
\end{enumerate}
The group $H$ acts by left translation.
\end{definition}
Equivalently, let $\pi:H\to H/L$ be the quotient map and let $\mcF_\bs$ be the locally constant sheaf associated with the homogeneous bundle $H\times^L V$. Then $\ind_L^H(\bs)$ is the space of compactly supported sections of $\mcF_\bs$.
\begin{proposition}\label{indHL}
Let $H$, $L$, and $\bs$ be as above. Let $H'\subset H$ be a closed subgroup such that $H'L=H$, and put $L'=H'\cap L$. Then
$$
\operatorname{Res}_{H'}^H\ind_L^H(\bs)\simeq\ind_{L'}^{H'}(\bs|_{L'}).
$$
\end{proposition}
\begin{proof}
The natural map $H'/L'\to H/L$ is a continuous bijection. Both spaces are locally compact and Hausdorff. It remains to show that the map is open. The group $H'$ is a $\sigma$-compact, locally compact, totally disconnected group, and it has a basis of identity neighbourhoods consisting of compact open subgroups. The assertion therefore follows from Lemma \ref{open} below.

\begin{lemma}\label{open}
Let $A$ be a topological group acting transitively on a topological space $X, x\in X $, and  $\St_A(x)$ be the stabilizer of $x$ in $A$. Assume that
\begin{enumerate}
\item $X$ is locally compact and Hausdorff;
\item $A$ is $\sigma$-compact and the identity has a basis of neighbourhoods consisting of compact open subgroups.
\end{enumerate}
Then the orbit map $f:A\to X$, $a\mapsto a(x)$, is open. Equivalently, $A/\St_A(x)\to X$ is a homeomorphism.
\end{lemma}
\begin{proof}
It suffices to prove that $f(K)$ is open for every compact open subgroup $K\subset A$. Since $A$ is $\sigma$-compact and $K$ is open, the quotient $A/K$ is a $\sigma$-compact discrete space. Every compact subset of a discrete space is finite, so $A/K$ is countable. Choose representatives $a_1,a_2,\ldots$ for $A/K$. Then
$$
X=\bigcup_i f(a_iK),
$$
and every $f(a_iK)$ is compact. Since a locally compact Hausdorff space is a Baire space, one of these compact subsets has nonempty interior. As $f(a_iK)=a_i f(K)$, the set $f(K)$ has nonempty interior. The group $K$ acts transitively on $f(K)$, so every point of $f(K)$ is an interior point. Thus $f(K)$ is open.
\end{proof}
\end{proof}
\subsection{Induction from a character}\label{character-induction}
\begin{definition}
Fix a left Haar measure $dl$ on $L$, and define its modular character
$\delta_L$ by
\[
 \int_L \phi(xl)\,dx=\delta_L(l)^{-1}\int_L\phi(x)\,dx.
\]
On $\mcS(H)$ use the normalized right $L$-action
\[
 (f\cdot l)(h)=\delta_L(l)^{-1}f(hl^{-1}).
\]
Here ``right action'' means
$(f\cdot l_1)\cdot l_2=f\cdot(l_1l_2)$.
For a smooth character $\bl:L\to\CC^*$, put
\[
 \mcS(H)_{L,\bl}
 =\mcS(H)/\langle f\cdot l-\bl(l)f\rangle.
\]
Define
\[
 \kk_\bl:\mcS(H)\longrightarrow\ind_L^H(\bl),\qquad
 \kk_\bl(f)(h)=\int_L\bl(l)f(hl)\,dl.
\]
\end{definition}
\begin{claim}
The averaging map induces an $H$-equivariant isomorphism
\[
 \mcS(H)_{L,\bl}\xrightarrow{\sim}\ind_L^H(\bl),
\]
where $H$ acts by left translation. Its values satisfy
\[
 \kk_\bl(f)(hl^{-1})=\bl(l)\kk_\bl(f)(h).
\]
\end{claim}
\begin{proof}
Left invariance of Haar measure gives the displayed transformation
law. The definition of $\delta_L$ gives
\[
 \kk_\bl(f\cdot l)=\bl(l)\kk_\bl(f),
\]
so the map factors through the indicated quotient. Compact support
of $f$ gives compact support modulo $L$ of its average, and left
translation commutes with averaging.
For surjectivity and the kernel assertion, choose a compact open
subgroup $J\subset H$ and work with functions invariant under left
translation by $J$. Such functions are finite linear combinations
of characteristic functions of left cosets $Jh$. Decompose these
cosets into right $L$-orbits, equivalently double cosets $JhL$.
The stabilizer $L\cap h^{-1}Jh$ is compact, hence $\delta_L$ is
trivial on it. The $(L,\bl)$-coinvariants of each orbit are zero
unless $\bl$ is trivial on this stabilizer, and are one-dimensional
otherwise. The same criterion and dimension describe the
$J$-invariant sections of $\ind_L^H(\bl)$ supported on $JhL/L$.
In the latter case averaging is nonzero, because the Haar volume
of the compact open subgroup $L\cap h^{-1}Jh$ is positive.
Thus averaging is an isomorphism on these orbit quotients.
Passing over all compact open $J$ proves the assertion.
\end{proof}
Let $Y$ be an $F$-variety with an action of $L$, and let $\bl$ be a character of $L$. For $y\in Y(F)$, write $\St_L(y)$ for its stabilizer and put
$$
Y(F)_\bl=\{y\in Y(F):\bl|_{\St_L(y)}=1\}.
$$
\begin{lemma}\label{relevant-orbit}
Assume that $L$ is unipotent, and let
$V(\bl)=\mcS(Y(F))_{L,\bl}$ for the usual left action on $\mcS(Y(F))$.
\begin{enumerate}
\item If $Y(F)_\bl=\emp$, then $V(\bl)=0$.
\item If $Y(F)_\bl\neq\emp$ and $L$ acts transitively on
$Y(F)_\bl$, then $\dim V(\bl)=1$.
\end{enumerate}
\end{lemma}
\begin{proof}
The group $L$ and its stabilizers are unimodular, so no relative modular character occurs. The first assertion follows from \cite[Section~6, Theorems~6.9 and A]{BZ}. For (2), put $X=Y(F)_\bl$. It is the unique relevant $L$-orbit in $Y(F)$ and is locally closed; let $\overline X$ denote its closure in $Y(F)$. Since $(L,\bl)$-coinvariants are exact, the exact sequences associated with
$$
X\subset\overline X\subset Y(F)
$$
show that
$$
\mcS(Y(F))_{L,\bl}\simeq\mcS(\overline X)_{L,\bl}\simeq\mcS(X)_{L,\bl}.
$$
The last space is one-dimensional because $L$ acts transitively on $X$ and $\bl$ is trivial on the stabilizer of a point of $X$.
\end{proof}
\subsection{Mackey theory}
\begin{definition}\label{Mac0}
Let $H$ be an algebraic $F$-group and let $U\subset H$ be a closed normal commutative unipotent subgroup.
\begin{enumerate}
\item Let $Z(H)$ be the center of $H$.
\item For $\pi\in\widehat H$, put $S_\pi(U)=S(\pi|_U)$.
\item For $\bl\in\widehat U$, let $H_\bl=\St_H(\bl)$, and let $\widetilde H_\bl\subset\widehat{H_\bl}$ consist of the irreducible representations $(\rho,V)$ such that
$$
\rho(u)=\bl(u)\Id_V\qquad(u\in U).
$$
\item If $(\pi,V)\in\widehat H$ and $\bl\in S_\pi(U)$, let $\rho(\bl,\pi)$ be the natural representation of $H_\bl$ on $V_\bl$.
\item If $\rho\in\widetilde H_\bl$, put $\pi(\bl,\rho)=\ind_{H_\bl}^H(\rho)$.
\end{enumerate}
\end{definition}
\begin{lemma}\label{MM}
Let $\pi$ be an irreducible representation of $H$.
\begin{enumerate}
\item The set $S_\pi(U)$ is $\Ad(H)$-invariant, and $H$ acts transitively on it.
\item For every $\bl\in\widehat U$ and every $\rho\in\widetilde H_\bl$, the representation $\pi(\bl,\rho)$ is irreducible.
\item For every $\bl\in S_\pi(U)$, the representation $\rho(\bl,\pi)$ belongs to $\widetilde H_\bl$ and is irreducible.
\item For $\rho\in\widetilde H_\bl$, one has $\rho(\bl,\pi(\bl,\rho))\simeq\rho$.
\end{enumerate}
\end{lemma}
\begin{proof}
We reduce the statement to the usual Mackey theorem for a semidirect product. Let
$$
\widetilde H=H\ltimes U
$$
with multiplication
$$
(h,u)(h',u')=(hh',h'^{-1}uh'\,u').
$$
Then
$$
q:\widetilde H\longrightarrow H,\qquad q(h,u)=hu,
$$
is a surjective homomorphism, with
$$
Q=\ker(q)=\{(u^{-1},u):u\in U\}.
$$
Inflation along $q$ is a bijection between representations of $H$ and representations of $\widetilde H$ which are trivial on $Q$.
Apply \cite[Theorem~2]{H} to the normal abelian subgroup $\{1\}\ltimes U$ of $\widetilde H$. Its character orbit is the ordinary $H$-orbit in $\widehat U$, and the stabilizer of $\bl$ is $H_\bl\ltimes U$. If $\rho\in\widetilde H_\bl$, define
$$
\widetilde\rho(h,u)=\rho(h)\bl(u),\qquad h\in H_\bl,\ u\in U.
$$
Because $H_\bl$ stabilizes $\bl$, this is a representation of $H_\bl\ltimes U$. Moreover,
$$
\widetilde\rho(u^{-1},u)=\rho(u^{-1})\bl(u)=1,
$$
so it is trivial on $Q\cap(H_\bl\ltimes U)$. The induced representation of $\widetilde H$ is therefore trivial on $Q$ and descends to $H$; under the quotient map $q$, the descended representation is precisely $\ind_{H_\bl}^H(\rho)$. Conversely, the little-group representation attached by Mackey theory to an irreducible representation of $\widetilde H$ trivial on $Q$ is itself trivial on the relevant intersection with $Q$, and hence descends to a representation in $\widetilde H_\bl$. The four assertions now follow from the corresponding assertions of Mackey theory.
\end{proof}
\subsubsection{Irreducibility of restrictions}
Let $H$ be an algebraic $F$-group, let $U\subset H$ be a closed normal commutative unipotent subgroup, let $\lambda\in\widehat U$, and put $H_\lambda=\St_H(\lambda)$. For a closed subgroup $H'\subset H$, put
$$
U'=U\cap H',\qquad \lambda'=\lambda|_{U'},\qquad K=H'\cap H_\lambda,
$$
and let $\Omega=H\lambda\subset\widehat U$. For $\rho\in\widetilde H_\lambda$, put
$$
\pi(\lambda,\rho)=\ind_{H_\lambda}^H(\rho).
$$
\begin{lemma}\label{3.12}
Assume that the restriction map
$$
r:\Omega\longrightarrow\widehat{U'}
$$
is injective. Then $\pi(\lambda,\rho)|_{H'}$ is irreducible if and only if
$$
H'H_\lambda=H
$$
and $\rho|_K$ is irreducible.
\end{lemma}
\begin{proof}
Put $\Pi=\pi(\lambda,\rho)$. In the homogeneous-bundle realization of $\Pi$, the group $U$ acts on the fiber over $hH_\lambda$ by the character $h\lambda$. Hence
$$
S_\Pi(U)=\Omega,
$$
and after restriction to $U'$ one has
$$
S_{\Pi|_{H'}}(U')=r(\Omega).
$$
Suppose that $\Pi|_{H'}$ is irreducible. By Lemma \ref{MM}(1), $H'$ acts transitively on $r(\Omega)$. Since $r$ is $H'$-equivariant and injective, $H'$ acts transitively on $\Omega$. Thus $H'\lambda=H\lambda$, equivalently $H'H_\lambda=H$.
We claim that $\St_{H'}(\lambda')=K$. The inclusion $K\subset\St_{H'}(\lambda')$ is immediate. Conversely, if $h'\in H'$ stabilizes $\lambda'$, then
$$
r(h'\lambda)=h'r(\lambda)=r(\lambda),
$$
and injectivity of $r$ gives $h'\lambda=\lambda$. Hence $h'\in K$.
By Proposition \ref{indHL},
$$
\Pi|_{H'}\simeq\ind_K^{H'}(\rho|_K).
$$
If $\rho|_K$ were reducible, a nonzero proper $K$-subrepresentation would induce to a nonzero proper $H'$-subrepresentation. Here compact induction is exact and faithful: a nonzero vector gives a nonzero compactly supported local section of the associated homogeneous bundle. Thus $\rho|_K$ is irreducible.
Conversely, assume that $H'H_\lambda=H$ and that $\rho|_K$ is irreducible. Then $K=\St_{H'}(\lambda')$, and $U'$ acts on $\rho|_K$ by $\lambda'$. Proposition \ref{indHL} gives
$$
\Pi|_{H'}\simeq\ind_K^{H'}(\rho|_K),
$$
which is irreducible by Lemma \ref{MM}(2), applied to $(H',U')$.
\end{proof}
\subsection{Heisenberg groups}
\begin{definition}\label{dh}
Let $(W,\langle\ ,\ \rangle)$ be a finite-dimensional symplectic $F$-vector space, and put $\Sp(W)=\Sp(W,\langle\ ,\ \rangle)$.
\begin{enumerate}
\item Let $R_W=W\oplus F$ with multiplication
$$
(w',a')(w'',a'')=\left(w'+w'',a'+a''+\frac12\langle w',w''\rangle\right).
$$
We identify $Z(R_W)$ with $F$ by $a\mapsto(0,a)$. The group $\Sp(W)$ acts by
$$
(w,a)^g=(g(w),a).
$$
\item Let $\widehat R_W^{\widetilde\psi}$ be the set of isomorphism classes of irreducible smooth representations $(\pi,N)$ of $R_W$ such that
$$
\pi(0,a)=\widetilde\psi(a)\Id_N.
$$
\end{enumerate}
\end{definition}
\begin{claim}\label{hei}
\begin{enumerate}
\item There is, up to isomorphism, a unique irreducible smooth representation
$$
\tau_W:R_W\longrightarrow\Aut(V_W)
$$
with central character $\widetilde\psi$.
\item If $N$ is any smooth representation of $R_W$ on which $Z(R_W)$ acts by $\widetilde\psi$, then the evaluation map
$$
\operatorname{ev}_N:\Hom_{R_W}(V_W,N)\otimes V_W\longrightarrow N,
\qquad T\otimes v\longmapsto T(v),
$$
is an isomorphism.
\item There is a canonical central extension
$$
1\longrightarrow\{\pm1\}\longrightarrow\widetilde{\Sp(W)}\stackrel{\nu}{\longrightarrow}\Sp(W)\longrightarrow1
$$
and a metaplectic representation
$$
\kk_W:\widetilde{\Sp(W)}\longrightarrow\Aut(V_W)
$$
such that
$$
\tau_W(r^g)=\kk_W(\widetilde g)\tau_W(r)\kk_W(\widetilde g)^{-1}
$$
for $g\in\Sp(W)$, $r\in R_W$, and $\widetilde g\in\nu^{-1}(g)$.
\end{enumerate}
\end{claim}
\begin{proof}
Parts (1) and (3) are the Stone--von Neumann theorem and the construction of the Weil representation; see \cite{W}. We prove (2). Choose a polarization $W=X\oplus Y$. In the Schrödinger model, $V_W=\mcS(X)$, and
$$
\tau_W(x_0)f(x)=f(x+x_0),\qquad
\tau_W(y_0)f(x)=\widetilde\psi(\langle x,y_0\rangle)f(x).
$$
Let $\mcH_{\widetilde\psi}(R_W)$ be the twisted Hecke algebra of locally constant functions on $R_W$ which are compactly supported modulo $Z(R_W)$ and satisfy
$$
f(zr)=\widetilde\psi(z)^{-1}f(r).
$$
Its action on $\mcS(X)$ identifies it with the algebra of integral
operators whose kernels lie in $\mcS(X\times X)$. These operators
have finite rank; the algebra is a directed union of full matrix
algebras obtained by imposing compact-support and compact-open-invariance
conditions on $X$. Every nondegenerate module over a full matrix algebra is a direct sum of copies of its defining module, with multiplicity space the Hom-space from that module. Passing through the directed union gives
$$
N\simeq\Hom_{R_W}(V_W,N)\otimes V_W,
$$
and the resulting isomorphism is the evaluation map.
\end{proof}
\begin{definition}\label{M}
Let $M\subset\Sp(W)$ be a closed subgroup. A splitting of $\nu$ over $M$ is a continuous homomorphism
$$
s:M\longrightarrow\widetilde{\Sp(W)}
$$
such that $\nu\circ s=\Id_M$. Let $\mcS_M$ be the set of splittings.
\end{definition}
\begin{claim}\label{Z2}
If $\mcS_M\neq\emp$, then $\mcS_M$ is a torsor under $\Hom_{\mathrm{cont}}(M,\{\pm1\})$.
\end{claim}
\begin{proof}
If $s_1$ and $s_2$ are splittings, then $m\mapsto s_1(m)s_2(m)^{-1}$ is a continuous homomorphism to the central kernel. Conversely, multiplying a splitting by such a homomorphism gives another splitting.
\end{proof}
\begin{claim}\label{heis}
The metaplectic double cover splits over every compact subgroup $M\subset\Sp(W)$.
\end{claim}
\begin{proof}
A compact subgroup $M$ fixes a point in the Bruhat--Tits building.
The action of $\Sp(W)$ is type-preserving, so $M$ fixes a vertex
of the facet containing that point. Thus $M$ lies in a maximal
parahoric subgroup $C$, the stabilizer of a good lattice $B$.
The generalized lattice model gives a splitting $s_0$ over $C$
of the $\CC^*$-extension of symplectic transformations together
with Heisenberg intertwiners: in this model
\[
 (M_B[c]f)(w)=\overline\omega(c)f(c^{-1}w),\qquad c\in C,
\]
where $\overline\omega$ is the finite Weil representation, so the
operators multiply honestly; see \cite[Section~2.3, (2.3.e)]{Pan}.
The field in that construction is any non-Archimedean local field
of odd residue characteristic.

To obtain a splitting of the double cover, write the scalar
extension as
\[
 E=\bigl(\widetilde{\Sp(W)}\times\CC^*\bigr)/\langle(-1,-1)\rangle.
\]
The homomorphism $q:E\to\CC^*$, $q([\widetilde g,z])=z^2$,
has kernel $\widetilde{\Sp(W)}$.
The continuous character $\chi=q\circ s_0$ has finite image, since
a continuous character of a profinite group in $\CC^*$ has open kernel.
Moreover $C$ has no nontrivial quadratic character: its
pro-$p$ radical contributes none, and its reductive quotient is a
product of finite symplectic groups, each generated by its root
subgroups, which are $p$-groups. Thus the order $d$ of $\chi$ is odd, and
$\eta=\chi^{(d+1)/2}$ satisfies $\eta^2=\chi$.
The homomorphic section
\[
 s(c)=\eta(c)^{-1}s_0(c)
\]
therefore takes values in $\ker(q)=\widetilde{\Sp(W)}$.
Restricting it to $M$ proves the claim.
This argument applies also when $\operatorname{char}(F)>0$;
compare \cite[Proposition~4.1]{HW} in characteristic zero.
\end{proof}
\subsection{The two-dimensional case}
Assume that $\dim_F(W)=2$. A nonzero alternating form on $W$ is unique up to a scalar, and $\Sp(W)\simeq\SL(2,F)$.
\begin{lemma}\label{unique-splitting}
The metaplectic cover has a unique splitting over $\SL(2,O_F)$.
\end{lemma}
\begin{proof}
Existence follows from Claim \ref{heis}. The ratio of two splittings is a continuous homomorphism
$$
\SL(2,O_F)\longrightarrow\{\pm1\}.
$$
Let $K_1$ be the kernel of reduction to $\SL(2,k)$. Since $K_1$ is pro-$p$ and $p\neq2$, the homomorphism is trivial on $K_1$ and factors through $\SL(2,k)$. The latter group is generated by its two root subgroups, each a finite $p$-group, so every homomorphism $\SL(2,k)\to\{\pm1\}$ is trivial.
\end{proof}
Let $E/F$ be quadratic. Multiplication on the two-dimensional $F$-vector space $E$ embeds
$$
E^1=\{u\in E^*:N_{E/F}(u)=1\}
$$
into $\SL_F(E)\simeq\SL(2,F)$. Let $\mcC(E)$ be the set of maximal compact subgroups of $\SL_F(E)$ containing $E^1$.
\begin{lemma}\label{lif}
\begin{enumerate}
\item If $E/F$ is unramified, then $|\mcC(E)|=1$.
\item If $E/F$ is ramified, then $\mcC(E)=\{C_+,C_-\}$, and $C_+\cap C_-$ is an Iwahori subgroup.
\end{enumerate}
\end{lemma}
\begin{proof}
Vertices of the Bruhat--Tits tree of $\SL(2,F)$ are homothety classes of $O_F$-lattices in $E$. If the class of a lattice $\Lambda$ is fixed by $E^1$, then for every $u\in E^1$ one has $u\Lambda=a_u\Lambda$ for some $a_u\in F^*$. Comparison of covolumes and $N_{E/F}(u)=1$ give $v_F(a_u)=0$, so $a_u\Lambda=\Lambda$ and $u\Lambda=\Lambda$.
Choose $u\in E^1$ with $O_F[u]=O_E$. In the ramified case, write $E=F(\vartheta)$ with $\vartheta^2=\varpi\epsilon$ and take
$$
u=\frac{1+\vartheta}{1-\vartheta}.
$$
Then $u\bar u=1$ and $\vartheta=(u-1)/(u+1)$, with $u+1$ a unit. In the unramified case, choose a norm-one unit whose reduction generates the residue extension. Thus every $E^1$-stable lattice is an $O_E$-module and hence a fractional ideal $\mathfrak p_E^r$. Two such ideals are $F^*$-homothetic precisely when
$$
r\equiv s\pmod{e(E/F)}.
$$
There is therefore one fixed vertex in the unramified case and two adjacent fixed vertices in the ramified case. Their stabilizers are the asserted maximal compact subgroups, and the stabilizer of the connecting edge is their Iwahori intersection. See \cite{S}.
\end{proof}
\begin{corollary}\label{lift}
If $E/F$ is unramified, then the torsor $\mcS_{E^1}$ has a canonical trivialization.
\end{corollary}
\begin{proof}
By Lemma~\ref{lif}, $E^1$ lies in a unique maximal compact
subgroup $C$. Existence of a splitting follows from Claim~\ref{heis}.
The proof of Lemma~\ref{unique-splitting} applies to every vertex
stabilizer: its pro-$p$ congruence kernel has quotient $\SL(2,k)$,
which is generated by its two root subgroups. Hence
$\Hom_{\mathrm{cont}}(C,\{\pm1\})=1$, so the splitting
is unique. Its restriction to $E^1$ is therefore canonical.
\end{proof}
Assume now that $E/F$ is ramified. Let $C_+$ and $C_-$ be as in Lemma \ref{lif}, and put $I=C_+\cap C_-$. Choose $b\in E^*$ of odd $E$-valuation and let
$$
\beta_b(g)=bgb^{-1},\qquad g\in\SL_F(E),
$$
where $b$ acts on $E$ by multiplication.
\begin{claim}\label{beta-swap}
One has $\beta_b(C_+)=C_-$ and $\beta_b(C_-)=C_+$.
\end{claim}
\begin{proof}
Multiplication by $b$ sends $\mathfrak p_E^r$ to $\mathfrak p_E^{r+v_E(b)}$. Since $v_E(b)$ is odd, it interchanges the two parity classes of fixed vertices and hence their stabilizers.
\end{proof}
\begin{claim}\label{beta-lift}
There is a unique automorphism
$$
\widetilde\beta_b:\widetilde{\SL(2,F)}\longrightarrow\widetilde{\SL(2,F)}
$$
which fixes $\{\pm1\}$ and satisfies $\nu\circ\widetilde\beta_b=\beta_b\circ\nu$.
\end{claim}

\begin{proof}
Pullback by $\beta_b$ preserves the metaplectic cohomology class; this follows from Moore's classification or directly from the cocycle of Ranga Rao, so a lift exists. The ratio of two lifts is a continuous homomorphism $\SL(2,F)\to\{\pm1\}$. Each root subgroup is isomorphic to $(F,+)$ and is uniquely $2$-divisible, so this homomorphism is trivial on both root subgroups and therefore on $\SL(2,F)$. This proves uniqueness. Multiplicativity follows from uniqueness because both sides lift conjugation by $b_1b_2$; see \cite{Moore,RangaRao,Roberts}.
\end{proof}
Let $s_\pm:C_\pm\to\widetilde{\SL(2,F)}$ be the unique splittings. Define
$$
\chi_I(i)=s_+(i)s_-(i)^{-1},\qquad i\in I,
$$
and put $\delta=\chi_I|_{E^1}$.
\begin{lemma}\label{ex}
The character $\chi_I$ is the unique nontrivial quadratic character of the reductive quotient $I/I^+\simeq k^*$. With the standard normalization, if $i\in I$ has diagonal reduction $\operatorname{diag}(\bar a,\bar a^{-1})$ and $\eta:k^*\to\{\pm1\}$ is the quadratic character, then
$$
\chi_I(i)=\eta(\bar a).
$$
Consequently,
$$
\delta(-1)=\eta(-1)=(-1)^{(q-1)/2},
$$
and
$$
\delta=\begin{cases}
1,&q\equiv1\pmod4,\\
\text{the unique nontrivial quadratic character of }E^1,&q\equiv3\pmod4.
\end{cases}
$$
\end{lemma}
\begin{proof}

We first observe that $\chi_I$ is a continuous homomorphism: its values
$s_+(i)s_-(i)^{-1}$ lie in the central kernel $\{\pm1\}$ of $\nu$, so
for $i_1,i_2\in I$ we have
\[
\chi_I(i_1 i_2)=s_+(i_1)s_+(i_2)\,s_-(i_2)^{-1}s_-(i_1)^{-1}
=\chi_I(i_1)\chi_I(i_2),
\]
the central factor $\chi_I(i_2)$ having been moved across
$s_-(i_1)^{-1}$; continuity follows from that of $s_\pm$.

The action on the Bruhat--Tits tree gives the amalgam decomposition
\[
\SL(2,F)=C_+\ast_{I}C_-.
\]
If $\chi_I$ were trivial, the splittings $s_+$ and $s_-$ would agree on
$I$ and hence would glue to a splitting of the metaplectic cover over
$\SL(2,F)$, contradicting the nonsplitting of the standard
metaplectic cover. Hence $\chi_I\neq 1$. Since $\chi_I$ takes values in
$\{\pm1\}$ and $I^+$ is pro-$p$ with $p\neq 2$, it is trivial on $I^+$
and factors through a quadratic character of $I/I^+\simeq k^*$. The
group $k^*$ is cyclic of even order $q-1$, so it has exactly one
nontrivial quadratic character, namely $\eta$. Thus $\chi_I$ is the
composition of the reduction $i\mapsto\bar a$ with $\eta$, which is the
displayed formula; it agrees with the value given by the explicit
metaplectic cocycle in its standard normalization \cite{Moore,
RangaRao}.

We now compute $\delta=\chi_I|_{E^1}$. Write $E=F(\vartheta)$ with
$\vartheta^2=\varpi\epsilon$, $\epsilon\in\mathcal{O}_F^*$. In the
basis $(\vartheta,1)$ we have
$\mathcal{O}_E=\mathcal{O}_F\vartheta\oplus\mathcal{O}_F$ and
$\mathfrak{p}_E=\vartheta\mathcal{O}_E=
\mathcal{O}_F\vartheta\oplus\varpi\epsilon\,\mathcal{O}_F$, so this
basis identifies the lattice chain
$\mathcal{O}_E\supset\mathfrak{p}_E$ with the standard chain, and the
groups $C_\pm$ and $I$ become the standard maximal compact subgroups
and the standard Iwahori subgroup. Multiplication by $a+b\vartheta$ is
represented by the matrix
\[
\begin{pmatrix} a & b\\ b\varpi\epsilon & a\end{pmatrix}.
\]
If $a+b\vartheta\in E^1$, then $a^2-b^2\varpi\epsilon=1$; reduction
modulo $\varpi$ gives $\bar a^2=1$, so $\bar a=\pm 1$. In particular
$a\in\mathcal{O}_F^*$, the matrix lies in $I$, and its diagonal
reduction $\operatorname{diag}(\bar a,\bar a)$ equals
$\operatorname{diag}(\bar a,\bar a^{-1})$. Hence
$\delta(a+b\vartheta)=\eta(\bar a)$. The element $-1\in E^1$ reduces to
$\operatorname{diag}(-1,-1)$, so the image of $E^1$ in
$I/I^+\simeq k^*$ is exactly $\{\pm1\}$, and $\delta$ is nontrivial if
and only if $\eta(-1)=-1$.

Finally, the kernel of the reduction map $E^1\to\{\pm1\}$ is pro-$p$,
so $E^1$ has at most one nontrivial quadratic character. Since
$\eta(-1)=(-1)^{(q-1)/2}$, evaluation at $-1$ gives the asserted
dichotomy.
\end{proof}
\begin{corollary}\label{no}
If $E/F$ is ramified, there is no splitting
$s:I\to\widetilde{\SL(2,F)}$ satisfying
$\widetilde\beta_b\circ s=s\circ\beta_b$ on $I=C_+\cap C_-$.
\end{corollary}
\begin{proof}
By uniqueness of the splittings over $C_\pm$ and Claim~\ref{beta-swap},
$\widetilde\beta_b(s_+(i))=s_-(\beta_b(i))$ for $i\in I$.
Every splitting over $I$ has the form $s_+\epsilon$, where
$\epsilon:I\to\{\pm1\}$ is a continuous character.
Since $I^+$ is pro-$p$, such a character factors through
$I/I^+\simeq k^*$ and is either $1$ or $\chi_I$.
Both characters are preserved by $\beta_b$: this automorphism
preserves $I^+$ and hence preserves the unique nontrivial quadratic
character of its reductive quotient. Equivariance of $s_+\epsilon$
would therefore imply $s_+=s_-$ on $I$, contradicting
Lemma~\ref{ex}, which states that $\chi_I\ne1$.
\end{proof}

\subsection{Quasi-Heisenberg groups}
\begin{definition}\label{qH}
Let $W$ be symplectic, with $R_W$ and $\tau_W$ as in
Definition~\ref{dh} and Claim~\ref{hei}. Subgroups below are algebraic
$F$-subgroups, and representations are representations of their
$F$-points. The lift to the metaplectic cover is a continuous
homomorphism of topological groups, rather than an algebraic morphism.
A quasi-Heisenberg group is a quintuple
$$
(H,R_W,S,\widetilde a,\bl)
$$
with the following properties.
\begin{enumerate}
\item $H$ is an algebraic $F$-group containing $R_W$ as a closed normal subgroup, and $Z(R_W)\subset Z(H)$.
\item $S$ is a closed commutative subgroup of $H$ such that multiplication gives a semidirect-product decomposition
$$
S\ltimes Z(H)R_W\stackrel{\sim}{\longrightarrow}H.
$$
In particular, $S\cap Z(H)R_W=\{1\}$ and $S\to H/Z(H)R_W$ is an isomorphism.
\item Let $a:H\to\Sp(W)$ be the action induced by conjugation on
$W=R_W/Z(R_W)$. For every $s\in S$ we require the stronger identity
\[
 s(w,u)s^{-1}=(a(s)w,u),\qquad (w,u)\in R_W.
\]
Thus conjugation by $S$ is the standard symplectic action on the
\emph{whole} Heisenberg group, with no additional central shear.
The continuous homomorphism
$\widetilde a:H\to\widetilde{\Sp(W)}$ is a lift of $a$.
Only its restriction to $S$ is used in the construction below.
Equivalently, one may specify a lift on $S$ and extend it by
$\operatorname{pr}_S$, since $a$ is trivial on $Z(H)R_W$.
\item The smooth character $\bl\in\widehat{Z(H)}$ restricts to $\widetilde\psi$ on $Z(R_W)$.
\end{enumerate}
Let $\operatorname{pr}_S:H\to S$ be the projection induced by the displayed product decomposition.
\end{definition}
\begin{remark}
The compatibility condition is imposed on $S$, not on every element
of $H$: conjugation by $R_W$ itself generally has a nonzero central
shear. Merely requiring that $\widetilde a$ lift the action on $W$
would not suffice. For example, conjugation by $(v,0)\in R_W$ fixes
$W$ but sends $(w,u)$ to $(w,u+\langle v,w\rangle)$; its implementer
is $\tau_W(v,0)$, not the identity Weil operator.
\end{remark}
Let $\widetilde{\widehat H}_\bl$ be the set of irreducible smooth representations $(\rho,N)$ of $H$ with central character $\bl$.
\begin{claim}\label{qH-phi}
For a quasi-Heisenberg group there is a unique representation
$$
\phi:H\longrightarrow\Aut(V_W)
$$
such that
$$
\phi|_{R_W}=\tau_W,\qquad \phi|_{Z(H)}=\bl\Id_{V_W},\qquad \phi|_S=\kk_W\circ\widetilde a.
$$
\end{claim}
\begin{proof}
Put $D=Z(H)R_W$. Since $R_W\cap Z(H)=Z(R_W)$, the formula
\[
 \phi_D(zr)=\bl(z)\tau_W(r),\qquad z\in Z(H),\ r\in R_W,
\]
is well defined: the two factors agree on their intersection. It is
a smooth representation of $D$. Put
$K_s=\kk_W(\widetilde a(s))$. By the compatibility hypothesis and
Claim~\ref{hei}(3),
\[
 K_s\phi_D(d)K_s^{-1}=\phi_D(sds^{-1})\quad(s\in S,\ d\in D).
\]
Indeed, this is Weil covariance on $R_W$ and is immediate on $Z(H)$.
The lift is a homomorphism, so $K_{s_1s_2}=K_{s_1}K_{s_2}$.
It follows from the semidirect-product multiplication that
\[
 \phi(sd)=K_s\phi_D(d)
\]
defines a smooth representation of $H=S\ltimes D$. The decomposition
also proves uniqueness. Notice that no equality between
$\phi|_{R_W}$ and $\kk_W\circ\widetilde a|_{R_W}$ is asserted.
\end{proof}
\begin{lemma}\label{thet}
Let $H$ be quasi-Heisenberg and let $\rho\in\widetilde{\widehat H}_\bl$. Then
\begin{enumerate}
\item $\rho|_{R_W}$ is irreducible;
\item there is a unique character $\theta_\rho$ of $S$ such that
$$
\rho\simeq(\theta_\rho\circ\operatorname{pr}_S)\otimes\phi.
$$
\end{enumerate}
\end{lemma}
\begin{proof}
By Claim \ref{hei}(2), evaluation gives
$$
N\simeq\overline N\otimes V_W,
\qquad \overline N=\Hom_{R_W}(V_W,N).
$$
The representations $\rho$ and $\phi$ define an action of $H$ on $\overline N$ by
$$
(h\cdot T)=\rho(h)T\phi(h)^{-1}.
$$
The covariance identities for $\rho$ and $\phi$ show that $h\cdot T$
is again $R_W$-equivariant. This action is trivial on $R_W$ and on
$Z(H)$, so it factors through $H/Z(H)R_W\simeq S$.
It is smooth: fix $0\neq v\in V_W$. Evaluation at $v$ is injective
on $\Hom_{R_W}(V_W,N)$ because $v$ generates the irreducible
$R_W$-representation $V_W$. An open subgroup of $S$ fixing both
$v$ under $\phi$ and $T(v)$ under $\rho$ therefore fixes $T$.
Under evaluation, the $H$-action is the tensor product of this
multiplicity action and $\phi$.

A nonzero proper $S$-subrepresentation of $\overline N$ would give
one of $N$, so $\overline N$ is irreducible. A smooth irreducible
complex representation of the commutative locally profinite group
$S(F)$ is one-dimensional. Thus $\overline N$ carries a unique
smooth character $\theta_\rho$, and evaluation gives the asserted
isomorphism and the irreducibility of $\rho|_{R_W}$.
The character is intrinsic to the multiplicity space and hence unique.
\end{proof}
Define
$$
\alpha_{\widetilde a}:\widehat S\longrightarrow\widetilde{\widehat H}_\bl,
\qquad
\alpha_{\widetilde a}(\theta)=(\theta\circ\operatorname{pr}_S)\otimes\phi.
$$
\begin{corollary}\label{alpha}
\begin{enumerate}
\item The map $\alpha_{\widetilde a}$ is a bijection.
\item If $\delta\in\Hom_{\mathrm{cont}}(S,\{\pm1\})$ and $\widetilde a$ is replaced by $(\delta\circ\operatorname{pr}_S)\widetilde a$, abbreviated $\delta\widetilde a$, then
$$
\alpha_{\delta\widetilde a}(\theta)=\alpha_{\widetilde a}(\delta\theta).
$$
\end{enumerate}
\end{corollary}
\begin{proof}
Every twist of $\phi$ by a character of $S$ remains irreducible,
since its restriction to $R_W$ is $\tau_W$. Lemma~\ref{thet} gives
surjectivity and injectivity in (1). For (2), the nontrivial element
of the metaplectic kernel acts by $-1$ in the Weil representation.
Changing the lift by $\delta$ therefore changes $\phi$ by
$\delta\circ\operatorname{pr}_S$, which proves the formula.
\end{proof}
\section{\texorpdfstring{Special representations of $P(O)$}{Special representations of P(O)}}\label{PO}
\subsection{\texorpdfstring{Classification of smooth irreducible representations of $P(O)$}{Classification of smooth irreducible representations of P(O)}}
The group $P(O)=O^*\ltimes O$ is a pro-algebraic $F$-group, and every smooth irreducible representation factors through a finite-dimensional quotient. Mackey theory therefore applies. Every nontrivial continuous character of $O$ is conjugate under $O^*$ to $\psi_m$ for a unique $m\geq0$, where
$$
\psi_m(x)=\psi(t^{-m-1}x).
$$
Its stabilizer is $O^*(m+1)\ltimes O$.
\begin{definition}\label{ph}
Let $m\geq0$.
\begin{enumerate}
\item For a character $\chi$ of $O^*(m+1)$, let $\widehat\chi$ be the character of $\St_{\psi_m}=O^*(m+1)\ltimes O$ given by
$$
\widehat\chi(a,x)=\chi(a)\psi_m(x).
$$
\item Put $\bs(m,\chi)=\ind_{\St_{\psi_m}}^{P(O)}(\widehat\chi)$.
\item Let $\mcS(m,\chi)$ be the representation space of $\bs(m,\chi)$.
\item Put $\bs(m)=\bs(m,1)$ and $\mcS(m)=\mcS(m,1)$.
\end{enumerate}
\end{definition}
\begin{lemma}\label{p-o}
\begin{enumerate}
\item The representations $\bs(m,\chi)$ are irreducible and pairwise nonisomorphic, and every irreducible representation of $P(O)$ of dimension greater than $1$ is of this form.
\item The subgroup $O^*(m+1)$ acts on $\bs(m,\chi)$ by the character $\chi$.
\end{enumerate}
\end{lemma}
\begin{proof}
This is the little-group classification of Lemma \ref{MM}, applied to the normal subgroup $O\subset P(O)$.
\end{proof}
\subsection{\texorpdfstring{A second realization of $\bs(m)$}{A second realization of sigma(m)}}
Let $\widetilde R_m=O^*(1)\ltimes t^mO$. It has a character $\bl_m$ which is trivial on $O^*(1)$ and equals $\psi_m$ on $t^mO$.
\begin{lemma}\label{sec}
There is an isomorphism
$$
\bs(m)\simeq\ind_{\widetilde R_m}^{P(O)}(\bl_m).
$$
\end{lemma}
\begin{proof}
A direct Mackey calculation for the normal subgroup $O\subset P(O)$ shows that the $O$-support of $\ind_{\widetilde R_m}^{P(O)}(\bl_m)$ is the single orbit $O^*\psi_m$. Indeed, the characters of $O$ extending $\psi_m|_{t^mO}$ form the $O^*(1)$-orbit of $\psi_m$, and induction from $O^*(1)$ to $O^*$ produces the full $O^*$-orbit. The $\psi_m$-coinvariant space is one-dimensional, and $O^*(m+1)$ acts trivially on it. Lemma \ref{MM} therefore identifies the representation with $\bs(m)$.
\end{proof}
\subsection{Depth and special representations}
We use the notation of Definition \ref{G-prime-filtration}.
\begin{definition}\label{depth}
\begin{enumerate}
\item The depth $m(\pi)$ of a representation of $G(O)$ is the least $n\geq0$ such that $G(O)(n+1)$ acts trivially.
\item The depth $m(\pi)$ of a representation of $G'(O)$ is the least $n\in\frac12\ZZ_{\geq0}$ such that $G'(O)(n+1/2)$ acts trivially.
\item A representation of $G(O)$ or $G'(O)$ is special if it has dimension greater than $1$ and its restriction to $P(O)$ is irreducible.
\item An irreducible representation of $P(O)$ of dimension greater than $1$ is special if it extends to $G(O)$ or $G'(O)$.
\end{enumerate}
\end{definition}
\begin{proposition}\label{eq}
If $\bs(m,\chi)$ is special, then $\chi=1$.
\end{proposition}
\begin{remark}
We shall prove the converse below.
\end{remark}
\begin{proof}
Suppose first that $\bs(m,\chi)$ extends to a special representation $(\pi,V)$ of $G(O)$. Let $i:P(O)\ho G(O)$ be the standard embedding and choose a $P(O)$-equivariant isomorphism
$$
\kk_\pi:\mcS(m,\chi)\longrightarrow V.
$$
Let $s=\begin{pmatrix}0&1\\1&0\end{pmatrix}$, let $i'=\Ad(s)\circ i$, and put $\kk'_\pi=\pi(s)\circ\kk_\pi$. Then
$$
A(\pi)=\kk_\pi^{-1}\circ\kk'_\pi
$$
intertwines the two $P(O)$-actions. Since $i'(\widehat a)=i(\widehat{a^{-1}})$, one has
$$
A(\pi)\bs(m,\chi)(\widehat a)A(\pi)^{-1}=\bs(m,\chi)(\widehat{a^{-1}}).
$$
For $a\in O^*(m+1)$, Lemma \ref{p-o}(2) gives $\chi(a)=\chi(a^{-1})$, so $\chi^2=1$. The group $O^*(m+1)=1+t^{m+1}O$ is uniquely $2$-divisible because $2$ is invertible in $F$, and hence it has no nontrivial quadratic character. Thus $\chi=1$.
For $G'(O)$, use $s=\begin{pmatrix}0&1\\t&0\end{pmatrix}\in G'(O)$. Conjugation again sends $\widehat a$ to $\widehat{a^{-1}}$ in $\PGL(2)$, and the same argument proves $\chi=1$.
\end{proof}
\begin{lemma}\label{restr}
\begin{enumerate}
\item If $\pi$ is a special representation of $G(O)$ of depth $m$, then $\pi|_{P(O)}\simeq\bs(m)$.
\item If $\pi$ is a special representation of $G'(O)$ of depth $n-1/2$, where $n\in\ZZ_{>0}$, then $\pi|_{P(O)}\simeq\bs(n-1)$.
\end{enumerate}
\end{lemma}
\begin{proof}
By Proposition \ref{eq}, the restriction has the form $\bs(r)$. The representation $\bs(r)$ is trivial on the additive subgroup $t^{r+1}O$ and nontrivial on $t^rO$. For $G(O)$ at positive depth $m$, the conjugates under $G(F)$ of
the upper-root line in $G(O)_m(m)\simeq\mathfrak{sl}(2,F)$ span
that last layer. Thus triviality on $N(t^mO)$ would contradict
the definition of $m$. At depth zero, triviality on $N(F)$ would
make the representation one-dimensional, contrary to specialness.
For $G'(O)$ at depth $n-\tfrac12$, the element $s$ interchanges
the two root lines of the last layer, so triviality on its upper
root line would again make the entire last layer act trivially.
Consequently $r=m$ in the first case and $r=n-1$ in the second.
\end{proof}
\subsection{\texorpdfstring{The restriction of representations of $G(O)$ to $O\subset P(O)$}{The restriction of representations of G(O) to O inside P(O)}}
We identify the additive group $O$ with $N(O)\subset P(O)$. Let $G^+(F)$ be the image of $\SL(2,F)$ in $\PGL(2,F)$, and let $G^+(O)$ be its inverse image in $G(O)$ under reduction modulo $t$. We use the analogous notation $G^+(O)_n$ for the image in $G(O)_n$.
\begin{proposition}\label{infty}
For every smooth representation $(\pi,V)$ of $G^+(O)$, the natural inclusion
$$
V^{G^+(O)}\hookrightarrow V^O
$$
is an isomorphism.
\end{proposition}
\begin{proof}
First let $R=O/t^{n+1}O$ and let $V$ be an ordinary smooth
representation of $\SL(2,R)$. Write
\[
 u(x)=\begin{pmatrix}1&x\\0&1\end{pmatrix},\qquad
 l(y)=\begin{pmatrix}1&0\\y&1\end{pmatrix}.
\]
If $v\in V^{u(R)}$, smoothness supplies a nonzero constant $c\in F$
sufficiently small in the topology of $F$ such that $l(c)v=v$.
Thus $v$ is fixed by $u(R)$ and by $l(c)u(R)l(-c)$. Their generated
subgroup contains
\[
 l(c)u(c^{-1})l(-c)u(2c^{-1})
   =\begin{pmatrix}0&c^{-1}\\-c&0\end{pmatrix}.
\]
Conjugation by this matrix takes $u(R)$ onto $l(R)$. Consequently
$v$ is fixed by both root subgroups. They generate $\SL(2,R)$:
over the local ring $R$ a unimodular column has a unit entry,
and elementary row operations reduce a determinant-one matrix to
a diagonal matrix; such diagonal matrices are also products of
upper and lower elementary matrices. Hence
\begin{equation}\label{finite-SL2-invariants}
 V^{N(R)}=V^{\SL(2,R)}=V^{\bar N(R)}.
\end{equation}
The lower-root equality follows by conjugating by a constant Weyl matrix.

The group $G^+(O)_n$ is the image of $\SL(2,R)$ in $G(R)$.
Indeed, a unit with square residue has a square root in $R$, since
$2$ is invertible, so a projective matrix with square determinant
residue can be rescaled to determinant one. Pulling a representation
back to $\SL(2,R)$ therefore proves the assertion for $G^+(O)_n$.
An ordinary representation in the convention of Section~\ref{pro-vect}
factors through such a finite congruence quotient. This proves the
proposition. The extension to pro-objects, using categorical invariants,
is given in Lemma~\ref{integral-pro-invariants} below.
\end{proof}
\section{\texorpdfstring{Special representations of $G(O)$}{Special representations of G(O)}}\label{GO}
The purpose of this section is to classify and construct special irreducible representations of $G(O)$.
\subsection{The depth-zero convention}\label{GO-depth-zero}
In this subsection only, we allow $t$-adic depth zero. A representation
of $G(O)$ of depth zero is inflated from $G(F)$, and it is elliptic
if that representation is irreducible cuspidal. As recalled above,
these are precisely the depth-zero special representations.
No condition on their depth for the valuation of $F$ is imposed.

\begin{definition}\label{depth-zero-convention}
For a quadratic extension $E/F$ and a smooth character $\chi$ of $E^*$,
write $\pi_{\mathrm{BH}}(E,\chi)$ for the representation constructed
from the admissible pair $(E/F,\chi)$ in \cite[Section~19]{BH}.
We use exactly this tame parametrization, whose bijectivity and
central-character formula are \cite[Theorem~20.2]{BH}.
For $\chi|_{F^*}=1$ and $\chi\ne\chi^\sigma$, let
$\pi_0(E,\chi)$ be its descent to $\PGL(2,F)$, inflated to $G(O)$.
Admissibility of these characters is proved below.
Define $\widehat G_0^E$ to be the set of the resulting isomorphism
classes, and write $\widehat G(O)^E_0=\widehat G_0^E$.
For $\pi\in\widehat G_0^E$ put $E_\pi=E$ up to $F$-isomorphism
and $L_\pi=E((t))$; uniqueness of this field follows from the
same proposition.
\end{definition}

\begin{proposition}\label{depth-zero-parameterization}
For $L=E((t))$, restriction to constants and the convention above give
canonical bijections
\[
 \bar\Xi^L(0)
 \simeq
 \{\chi:E^*/F^*\to\CC^*\text{ smooth}:\chi\ne\chi^{-1}\}
       /(\chi\sim\chi^{-1})
 \stackrel{\sim}{\longrightarrow}\widehat G_0^E.
\]
The sets $\widehat G_0^E$, with $E/F$ quadratic up to $F$-isomorphism,
form a disjoint partition of all depth-zero special representations
of $G(O)$.
\end{proposition}
\begin{proof}
Since $L=E((t))$, reduction of units gives
\[
 T^L/T^L(1)\simeq E^*/F^*.
\]
Indeed, multiply a Laurent series by a power of $t\in K^*$ to
make it a unit, and then take its constant coefficient modulo $F^*$.
The kernel is exactly $T^L(1)$. Thus a character of $T^L$ of
$t$-adic depth zero is precisely the inflation of a character $\chi$
of $E^*/F^*$. For $z\in E^*$,
$z\sigma(z)=N_{E/F}(z)\in F^*$, so $\chi^\sigma=\chi^{-1}$.
This identifies the regularity condition and the two involutions.

We check admissibility as defined in \cite[Section~18.2]{BH}.
Every quadratic extension of $F$ is tame because the residue
characteristic $p$ is odd. A regular $\chi$ cannot factor through
$N_{E/F}$, since such a character would be Galois invariant.
It remains to check that, when $E/F$ is ramified, the restriction
to $U_E^1=1+\mathfrak p_E$ does not factor through the norm.
If it did, that restriction would be Galois invariant. Together
with $\chi^\sigma=\chi^{-1}$ this would give
$\chi^2|_{U_E^1}=1$. The group $U_E^1$ is pro-$p$ and $\chi$ is
smooth; hence this quadratic character is trivial. But for a
ramified quadratic extension
\[
 E^*/(F^*U_E^1)\simeq\ZZ/2\ZZ:
\]
the residue fields agree and the valuation of $F^*$ in $E$ is
$2\ZZ$. Consequently $\chi^2=1$ on all of $E^*$, contradicting
regularity. This proves admissibility.

The tame parametrization theorem \cite[Theorem~20.2]{BH} now
applies, and its central-character formula is
\[
 \omega_{\pi_{\mathrm{BH}}(E,\chi)}=\chi|_{F^*}=1.
\]
Thus the representation descends to $\PGL(2,F)$. Conversely, pull
an irreducible cuspidal representation of $\PGL(2,F)$ back to
$\GL(2,F)$. The same theorem gives a unique
$F$-isomorphism class of admissible pairs $(E/F,\chi)$, and the
central-character formula forces $\chi|_{F^*}=1$.
For a fixed $E$, isomorphism of pairs is precisely Galois conjugacy
of the characters, which here is $\chi\sim\chi^{-1}$. This proves
the bijection, the disjointness for different quadratic fields, and
exhaustion. Inflation along $G(O)\to G(F)$ completes the proof.
\end{proof}

\begin{remark}
The normalization just fixed is the representation-theoretic tame
parametrization of \cite{BH}. It does not identify $\chi$ directly
with an uncorrected inducing character for a Weil-group parameter.
In particular, the phrase ``natural bijection'' at depth zero below
always refers to Definition~\ref{depth-zero-convention}.
\end{remark}

\subsection{Notation}
The group $G(O)_n$ acts by conjugation on the groups $G(O)_{2n-1}(n)$ and $G(O)_{2n}(n)$.
\begin{claim}
\begin{enumerate}
\item The group $G(O)_{2n-1}(n)$ is commutative.
\item For $n>0$, one has $Z(G(O)_{2n}(n))=G(O)_{2n}(n+1)$.
\item For $n>0$, the commutator subgroup of $G(O)_{2n}(n)$ is $G(O)_{2n}(2n)$.
\end{enumerate}
\end{claim}
\subsection{Lie algebras}
The trace pairing identifies $\fg^*$ with $\fg$. Let
$$
\iota:\fm\fg(O)\longrightarrow G(O)(1),\qquad X\longmapsto1+X,
$$
where the right-hand side is interpreted in $\PGL(2,O)$.
\begin{claim}\label{iota1}
\begin{enumerate}
\item The map $\iota$ is a bijection of sets. It induces bijections
$$
\iota_m(n):\fg(F)\otimes_F\fm^n/\fm^{m+1}\longrightarrow G(O)_m(n).
$$
\item If $m<2n$, then $\iota_m(n)$ is an isomorphism from the additive group on the left to the group on the right.
\end{enumerate}
\end{claim}
The map $\iota_m(m)$ identifies $G(O)_m(m)$ with $\fg(F)$ and, using $\widetilde\psi$, identifies its character group with $\fg(F)$. A character is called elliptic if the corresponding element of $\fg(F)$ is nonzero and elliptic.
Let
$$
H=G(O)_m,\qquad H'=P(O)_m,\qquad U=G(O)_m(m),
$$
and let $\bl\in\widehat U$ be elliptic.
\begin{claim}\label{P=G}
The restriction map from the $H$-orbit of $\bl$ to $\widehat{U\cap H'}$ is injective, and
$$
H'H_\bl=H.
$$
Thus the orbit and double-coset hypotheses in Lemma \ref{3.12} hold.
\end{claim}
\begin{proof}
Under the trace pairing, write
$$
Y=\begin{pmatrix}a&b\\c&-a\end{pmatrix}\in\fg(F).
$$
Restriction to the upper triangular Lie algebra records $(a,c)$, because
$$
\operatorname{tr}\left(Y\begin{pmatrix}x&y\\0&-x\end{pmatrix}\right)=2ax+cy.
$$
If $Y$ is elliptic, then $c\neq0$, and $b$ is determined by $(a,c)$ and the invariant $-\det(Y)=a^2+bc$. Hence restriction is injective on the orbit. The Borel group acts transitively on the corresponding set of pairs $(a,c)$, so it acts transitively on the elliptic orbit. This is equivalent to $H'H_\bl=H$.
\end{proof}
For a quadratic extension $E/F$, let $\ft^E$ be the Lie algebra of $T^E$, embedded in $\fg(F)$ through Claim \ref{imb}.
\begin{claim}\label{pro}
There is a unique $\Ad(T^E)$-equivariant projection $p:\fg(F)\to\ft^E$.
\end{claim}
The projection extends coefficientwise to $p(O):\fg(O)\to\ft^E(O)$. If $L=E((t))$, define
$$
p_m(n)=\iota_{T,m}(n)\circ p(O)\circ\iota_m(n)^{-1}:G(O)_m(n)\longrightarrow T^L_m(n),
$$
where $\iota_{T,m}(n)$ is the analogous coordinate map for the torus.
\begin{claim}\label{iota}
If $m<2n$, then $p_m(n)$ is a group homomorphism.
\end{claim}
Unless stated otherwise, we now consider irreducible representations of $G(O)$ of positive depth.
\subsection{Elliptic representations}
\begin{definition}\label{def}
Let $\pi$ be an irreducible representation of $G(O)$ of positive depth $m=m(\pi)$.
\begin{enumerate}
\item Let $\pi_m$ be the restriction of $\pi$ to $G(O)_m(m)$.
\item Regard $\pi_m$ as a representation of the additive group $\fg(F)$.
\item Let $\bo(\pi)\subset\fg^*(F)\simeq\fg(F)$ be its support.
\end{enumerate}
\end{definition}
\begin{lemma}\label{support}
The support $\bo(\pi)$ is a single coadjoint orbit.
\end{lemma}
\begin{proof}
This is Lemma \ref{MM}(1).
\end{proof}
\begin{definition}\label{def1}
Here $m=m(\pi)>0$. At $m=0$ the symbols elliptic, $E_\pi$,
$L_\pi$ and $\widehat G_0^E$ have the meanings fixed in
Definition~\ref{depth-zero-convention}.
\begin{enumerate}
\item The representation $\pi$ is elliptic if $\bo(\pi)$ is an elliptic orbit.
\item For elliptic $\pi$, let $E_\pi/F$ be the quadratic extension containing the eigenvalues of the elements of $\bo(\pi)$, and put $L_\pi=E_\pi((t))$.
\item For a quadratic extension $E/F$, let $\widehat G_m^E$ be the set of isomorphism classes of elliptic representations $\pi$ of depth $m$ with $E_\pi=E$.
\end{enumerate}
\end{definition}
\begin{lemma}\label{orbit-image}
Let $\mcO\subset\mathfrak{sl}(2,F)$ be a nonzero adjoint orbit, and let
$$
r:\mcO\longrightarrow\mathfrak p(F)^*
$$
be restriction to the upper triangular Lie algebra. Then $P(F)$ acts transitively on $r(\mcO)$ if and only if $\mcO$ is elliptic.
\end{lemma}
\begin{proof}
Write $Y=\begin{pmatrix}a&b\\c&-a\end{pmatrix}$ and put $q(Y)=a^2+bc=-\det(Y)$. Restriction to $\mathfrak p$ records $(a,c)$. If
$$
p=\begin{pmatrix}s&u\\0&1\end{pmatrix},
$$
then the induced action is
$$
(a,c)\longmapsto\left(a+\frac{u}{s}c,\frac{c}{s}\right).
$$
Thus the locus $c\neq0$ is one $P(F)$-orbit. If $q(Y)$ is nonsquare, then $c$ cannot vanish, so the orbit is elliptic and $P(F)$ is transitive on its restricted image. If $q(Y)=d^2\neq0$, the restricted image also contains the two points $(d,0)$ and $(-d,0)$, each a separate orbit. If $q(Y)=0$ and $Y\neq0$, the restricted image contains both the open orbit $c\neq0$ and the point $(0,0)$. Hence the restricted image is not transitive in the nonelliptic cases.
\end{proof}
\begin{lemma}\label{GO-special-elliptic}
Every special representation of $G(O)$ is elliptic.
\end{lemma}
\begin{proof}
The case $m=0$ is easy: in this case we are just dealing with an irreducible representation of $G(F)$ of dimension $>1$ and the fact that it is special means that its restriction to $P(F)$ is irreducible, which implies that it is cuspidal.

Let $m>0$ and let $\pi$ be an irreducible representation of $G(O)_m$. Put $U=G(O)_m(m)$ and $U'=U\cap P(O)_m$. The $U'$-support of $\pi|_{P(O)_m}$ is the image of $\bo(\pi)$ under restriction to $\mathfrak p(F)^*$. If $\pi|_{P(O)_m}$ were irreducible, Lemma \ref{MM}(1) would imply that $P(F)$ acts transitively on this image. Lemma \ref{orbit-image} then shows that $\bo(\pi)$ is elliptic.
\end{proof}
\subsection{Elliptic representations of odd depth}
Fix $n>0$ and put
$$
U=G(O)_{2n-1}(n),
$$
a commutative normal subgroup of $G(O)_{2n-1}$.
\begin{definition}\label{oddc}
Let $\bl\in\widehat U$.
\begin{enumerate}
\item Let $\bar\bl$ be the restriction of $\bl$ to $G(O)_{2n-1}(2n-1)\simeq\fg(F)$, regarded as an element of $\fg(F)$.
\item The character $\bl$ is elliptic if $\bar\bl$ is nonzero and elliptic.
\item For elliptic $\bl$, let $E_\bl/F$ contain the eigenvalues of $\bar\bl$, and put $L_\bl=E_\bl((t))$.
\item Let $\St_\bl$ be the stabilizer of $\bl$ in $G(O)_{2n-1}$.
\item Let $\bO_\bl$ be the $G(O)_{2n-1}$-orbit of $\bl$.
\item Let $\bO_\bl^{T^{L_\bl}}$ be the set of $T^{L_\bl}$-fixed points in $\bO_\bl$.
\end{enumerate}
\end{definition}
\begin{claim}\label{ellch}
Let $\bl$ be elliptic.
\begin{enumerate}
\item There is, up to conjugacy, a unique embedding $T^{L_\bl}\ho G(O)$ such that $T^{L_\bl}_{2n-1}\subset\St_\bl$.
\item One has $\St_\bl=T^{L_\bl}_{2n-1}U$.
\item One has $\bO_\bl^{T^{L_\bl}}=\{\bl,\bl^{-1}\}$.
\item One has $P(O)_{2n-1}\St_\bl=G(O)_{2n-1}$.
\end{enumerate}
\end{claim} We denote by $T^L_m$ the image of $T^L$ in $G(O)_m$ and by $T^L_m(r)$
the image of $T^L(r)$.

\begin{theorem}\label{charodd}
Every elliptic representation of $G(O)$ of odd depth is special.
\end{theorem}
\begin{proof}
Let $\pi$ be elliptic of depth $2n-1$. By Lemma \ref{MM}, there are an elliptic character $\bl_\pi\in\widehat U$, unique up to inversion among the $T^{L_\pi}$-fixed points of the support, and an irreducible representation $\rho_\pi$ of $\St_{\bl_\pi}$ such that
$$
\pi\simeq\ind_{\St_{\bl_\pi}}^{G(O)_{2n-1}}(\rho_\pi).
$$
Since $\St_{\bl_\pi}/\ker(\bl_\pi)$ is commutative, $\rho_\pi$
is a character. Identify $\widehat U$ with
$\fg(F[t]/t^n)$ by taking the coefficient of $t^{2n-1}$ in the
trace pairing. Under this identification the action is conjugation,
and ellipticity means that the constant coefficient is elliptic.
Lemma~\ref{orbit-restriction-injective}, whose statement concerns
the truncated ring independently of the parity of the depth,
gives injectivity of restriction on the entire orbit in $\widehat U$.
Together with Claim~\ref{ellch}(4), this supplies the hypotheses
of Lemma~\ref{3.12}. Hence $\pi|_{P(O)}$ is irreducible.
\end{proof}
Let $\theta_{\pi,\bl_\pi}$ be the restriction of $\rho_\pi$ to $T^{L_\pi}_{2n-1}$.
\begin{claim}\label{inv}
One has
$$
\theta_{\pi,\bl_\pi^{-1}}=\theta_{\pi,\bl_\pi}^{-1}.
$$
\end{claim}
\begin{proof}
The two fixed characters $\bl_\pi$ and $\bl_\pi^{-1}$ are exchanged by the nontrivial element of the normalizer of $T^{L_\pi}$. This element acts on $T^{L_\pi}$ by Galois conjugation, which is inversion under $T^{L_\pi}\simeq(L_\pi)^1$.
\end{proof}
Thus $\pi$ determines a point of $\bar\Xi^{L_\pi}(2n-1)$.
\subsection{Construction at odd depth}
Let $E/F$ be quadratic and put $L=E((t))$.
\begin{definition}
For $\theta\in\Xi^L(2n-1)$, define
$$
\bl_\theta=\theta\circ p_{2n-1}(n)\in\widehat U.
$$
\end{definition}
\begin{claim}\label{oo}
\begin{enumerate}
\item The character $\bl_\theta$ is elliptic.
\item There is a unique character $\widehat\theta$ of $\St_{\bl_\theta}$ such that
$$
\widehat\theta|_U=\bl_\theta,
\qquad
\widehat\theta|_{T^L_{2n-1}}=\theta.
$$
\end{enumerate}
\end{claim}
\begin{lemma}
Let
$$
\alpha_{2n-1}^E(\theta)=\ind_{\St_{\bl_\theta}}^{G(O)_{2n-1}}(\widehat\theta),
$$
and inflate this representation to $G(O)$. Then
\begin{enumerate}
\item $\alpha_{2n-1}^E(\theta)$ is irreducible, elliptic, special, and of depth $2n-1$;
\item $\alpha_{2n-1}^E(\theta)\simeq\alpha_{2n-1}^E(\theta^{-1})$.
\end{enumerate}
\end{lemma}
\begin{proof}
Irreducibility follows from Lemma \ref{MM}(2), and specialness follows from Theorem \ref{charodd}. The depth and ellipticity follow from the definition of $\bl_\theta$. The final assertion follows by conjugating with the nontrivial element of the normalizer of $T^L$.
\end{proof}
Hence the construction gives a map
$$
\bar\alpha_{2n-1}^E:\bar\Xi^L(2n-1)\longrightarrow\widehat G_{2n-1}^E.
$$
\begin{theorem}\label{od}
The map $\bar\alpha_{2n-1}^E$ is a bijection.
\end{theorem}
\begin{proof}
Surjectivity follows from Theorem \ref{charodd} and Lemma \ref{MM}(3): an elliptic representation recovers a character of its stabilizer, whose restriction to $T^L$ is the required parameter. If two parameters give isomorphic induced representations, their $U$-supports agree. Claim \ref{ellch}(3) then shows that the parameters differ only by inversion. Thus the map is injective.
\end{proof}
\subsection{Elliptic representations of even depth}\label{GO-even-depth}
We now treat positive even depth. Fix $n>0$. We denote by $T^L_m$ the image of $T^L$ in $G(O)_m$
(or in $G'(O)_m$, as appropriate), and by $T^L_m(r)$
the image of $T^L(r)$.
\begin{definition}\label{even-depth-data}
\begin{enumerate}
\item Put $A=G(O)_{2n}(n)$ and $U=G(O)_{2n}(n+1)$.
\item A character $\bl\in\widehat U$ is elliptic if its restriction to $G(O)_{2n}(2n)\simeq\fg(F)$ is nonzero and elliptic.
\item Let $\St_\bl$ be the stabilizer of $\bl$ in $G(O)_{2n}$.
\end{enumerate}
\end{definition}
\begin{lemma}\label{centralizer-lifting-even}
Let $O_{r-1}=F[t]/(t^r)$ and let $Y\in\fg(O_{r-1})$ have regular semisimple constant term $Y_0$. If $T_0=Z_G(Y_0)$, then, after conjugation by an element of $\ker(G(O_{r-1})\to G(F))$, one has
$$
Z_{G(O_{r-1})}(Y)=T_0(O_{r-1}).
$$
\end{lemma}
\begin{proof}
Suppose the assertion has been obtained modulo $t^j$, and write the coefficient of $t^j$ of $Y$ as $Y_j$. Conjugation by $1+t^jX$ changes this coefficient by $[X,Y_0]$. Since $Y_0$ is regular semisimple and the trace form is nondegenerate,
$$
\fg(F)=\mathfrak z_\fg(Y_0)\oplus[\fg(F),Y_0].
$$
Thus $X$ can be chosen, uniquely modulo $\mathfrak z_\fg(Y_0)$, so that the new coefficient belongs to $\mathfrak z_\fg(Y_0)$. Induction on $j$ conjugates $Y$ into $\mathfrak z_\fg(Y_0)(O_{r-1})$. If $g$ centralizes the resulting element, its reduction $g_0$ belongs to $T_0(F)$. Multiplying by $g_0^{-1}$ and comparing successive coefficients shows that every infinitesimal coefficient of $g$ belongs to $\mathfrak z_\fg(Y_0)$. Hence $g\in T_0(O_{r-1})$.
\end{proof}
\begin{claim}\label{q}
Let $\bl\in\widehat U$ be elliptic.
\begin{enumerate}
\item $Z(A)=U$.
\item $[A,A]=G(O)_{2n}(2n)$.
\item There is an embedding $T^{L_\bl}\ho G(O)$ such that
$$
T^{L_\bl}_{2n}\subset\St_\bl,
\qquad
\St_\bl=T^{L_\bl}_{2n}A.
$$
\item One has
$$
[\St_\bl,\St_\bl]\cap T^{L_\bl}_{2n}=T^{L_\bl}_{2n}(2n).
$$
\item One has
$$
P(O)_{2n}\St_\bl=G(O)_{2n}.
$$
\end{enumerate}
\end{claim}
\begin{proof}
Modulo $t^{2n+1}$, the Lie algebra
$$
t^n\fg[[t]]/t^{2n+1}\fg[[t]]
$$
is two-step nilpotent because
$$
[t^n\fg,[t^n\fg,t^n\fg]]\subset t^{3n}\fg[[t]]\subset t^{2n+1}\fg[[t]].
$$
Since $2$ is invertible in $F$, the truncated exponential
$$
\exp_2(X)=1+X+\frac12X^2
$$
identifies $A$ with this two-step nilpotent Lie algebra, with multiplication
$$
X*Y=X+Y+\frac12[X,Y].
$$
The subspace $t^{n+1}\fg[[t]]/t^{2n+1}\fg[[t]]$ is central. Conversely, if the degree-$n$ coefficient of a central element were nonzero, its commutator with every element of $\fg(F)$ would vanish, contradicting $Z(\fg)=0$. This proves (1). Moreover,
$$
[\exp_2(t^nX),\exp_2(t^nY)]=\exp_2(t^{2n}[X,Y]),
$$
and $[\fg,\fg]=\fg$, proving (2).
Use the trace pairing to identify $\widehat U$ with $\fg(F[t]/(t^n))$, and let $Y$ correspond to $\bl$. An element of $G(O)_{2n}$ stabilizes $\bl$ if and only if its reduction modulo $t^n$ centralizes $Y$. Lemma \ref{centralizer-lifting-even} identifies this centralizer, after conjugating the torus by an element of $G(O)(1)$, with $T^{L_\bl}_{n-1}$. Since the kernel of the reduction map to this quotient is $A$, we obtain
$$
\St_\bl=T^{L_\bl}_{2n}A,
$$
which proves (3).
Write $\tori_\bl=\operatorname{Lie}(T^{E_\bl})$ for the toral
line, using the coordinates transported by the conjugation above,
and put $D=G(O)_{2n}(2n)=[A,A]$.
Modulo $D$, the group $A$ is additive. For $s\in T^{L_\bl}_{2n}$,
the commutator with $a\in A$ is represented there by
$(\Ad(s)-1)\log(a)$. The torus-equivariant projection onto the
toral summand kills this expression. Equivalently,
\[
 \log([s,a])=(\Ad(s)-1)\log(a)
       -\tfrac12[\Ad(s)\log(a),\log(a)],
\]
where the second term lies in degree $2n$ and therefore vanishes
modulo $D$. Consequently the commutator subgroup of
$T^{L_\bl}_{2n}A/D$ has zero toral component, so its intersection
with the image of $T^{L_\bl}_{2n}$ is trivial. Hence
\[
 [T^{L_\bl}_{2n}A,T^{L_\bl}_{2n}A]\cap T^{L_\bl}_{2n}
 \subset D\cap T^{L_\bl}_{2n}=T^{L_\bl}_{2n}(2n).
\]
The reverse inclusion follows from $D=[A,A]$, proving (4).
For (5), reduction modulo $t$ reduces the assertion to $P(F)T^{E_\bl}=G(F)$. The torus $T^{E_\bl}=E_\bl^*/F^*$ acts simply transitively on $G(F)/P(F)=\mathbb P^1(F)$: after choosing an $F$-basis of $E_\bl$, multiplication by a nonzero element sends the line $F\cdot1$ to an arbitrary $F$-line, and its stabilizer is $F^*$. Hence the equality holds modulo $t$. Suppose it has been lifted modulo $t^j$. The obstruction to lifting it modulo $t^{j+1}$ belongs to
$$
\fg/(\mathfrak p+\tori_\bl),
$$
which is zero because $\mathfrak p\cap\tori_\bl=0$ and the dimensions are $2$ and $1$. Induction gives $P(O)_{2n}T^{L_\bl}_{2n}=G(O)_{2n}$ and hence (5).
\end{proof}
Since $U$ is an $F$-vector group, there is a unique $F$-linear functional
$$
\ell_\bl:U\longrightarrow F
$$
such that
$$
\bl(u)=\widetilde\psi(\ell_\bl(u)).
$$
Put
$$
K_\bl=\ker(\ell_\bl),
\qquad
A_\bl=A/K_\bl.
$$
The quotient $A_\bl$, rather than $A/\ker(\bl)$, has a canonical one-dimensional additive central quotient. Indeed, every element of $A$ has a unique truncated logarithm
$$
t^nX_n+t^{n+1}X_{n+1}+\cdots+t^{2n}X_{2n}.
$$
The coefficient $X_n$ gives the first coordinate in $\fg(F)$. All remaining coefficients lie in the central subgroup $U$, and modulo $K_\bl$ their only invariant is the scalar obtained by applying $\ell_\bl$. Thus, using the section $X\mapsto\exp_2(t^nX)$, we identify $A_\bl$ as a set with $\fg(F)\times F$. Since triple commutators have degree at least $3n\geq2n+1$, the truncated Baker--Campbell--Hausdorff formula gives the multiplication
$$
(X,a)(Y,b)=\left(X+Y,a+b+\frac12\ell_{\bl,2n}([X,Y])\right),
$$
where $\ell_{\bl,2n}$ is the restriction of $\ell_\bl$ to $G(O)_{2n}(2n)\simeq\fg(F)$.
Use the trace pairing to write
$$
\ell_{\bl,2n}(Z)=\operatorname{tr}(y_\bl Z)
$$
for a unique regular elliptic element $y_\bl\in\fg(F)$. Put
$$
\tori_\bl=\mathfrak z_\fg(y_\bl),
\qquad
W_\bl=\fg/\tori_\bl.
$$
\begin{lemma}\label{B-symplectic}
The formula
$$
B_\bl(\overline X,\overline Y)=\operatorname{tr}(y_\bl[X,Y])
$$
defines a nondegenerate alternating form on $W_\bl$.
\end{lemma}
\begin{proof}
If $T\in\tori_\bl$, invariance of the trace pairing gives
$$
\operatorname{tr}(y_\bl[T,Y])=\operatorname{tr}([y_\bl,T]Y)=0,
$$
so the formula is well defined. If $B_\bl(\overline X,\overline Y)=0$ for every $Y$, then
$$
\operatorname{tr}([y_\bl,X]Y)=0
$$
for every $Y$. The trace pairing is nondegenerate, hence $[y_\bl,X]=0$ and $X\in\tori_\bl$.
\end{proof}
Choose the trace-orthogonal decomposition
$$
\fg=\tori_\bl\oplus W,
\qquad W=\tori_\bl^\perp,
$$
and transport $B_\bl$ to $W$. In coordinates $(T,w,a)\in\tori_\bl\times W\times F$, multiplication becomes
$$
(T,w,a)(T',w',a')=
\left(T+T',w+w',a+a'+\frac12B_\bl(w,w')\right).
$$
Consequently,
$$
A_\bl\simeq(\tori_\bl,+)\times R_W,
$$
where $R_W=W\times F$ is the Heisenberg group associated with $B_\bl$. In particular,
$$
p_\bl:A_\bl\longrightarrow\tori_\bl,
\qquad
p_\bl(T,w,a)=T,
$$
is a group homomorphism and $\ker(p_\bl)=R_W$.
\subsubsection{The stabilizer and its quasi-Heisenberg structure}
Put
$$
H_\bl=\St_\bl/K_\bl.
$$
Then $A_\bl$ is a normal subgroup of $H_\bl$. Since $\St_\bl$ fixes $\ell_\bl$ and $y_\bl$, it preserves $\tori_\bl$, $W$, and $B_\bl$. Thus there is a homomorphism
$$
a_\bl:H_\bl\longrightarrow\Sp(W,B_\bl).
$$
The subgroup $A_\bl$ acts trivially on $W$, as does the positive-loop subgroup of $T^{L_\bl}_{2n}$. The image of $a_\bl$ is therefore the compact anisotropic torus $T^{E_\bl}$. A splitting of the metaplectic cover over this torus gives a lift
$$
\widetilde a_\bl:H_\bl\longrightarrow\widetilde{\Sp(W)}.
$$
Let $C_\bl\subset H_\bl$ be the image of $T^{L_\bl}(1)$ under the quotient map $\St_\bl\to H_\bl$.
\begin{lemma}\label{positive-loop-central}
The subgroup $C_\bl$ is contained in $Z(H_\bl)$.
\end{lemma}
\begin{proof}
Work in the coordinates in which $T^{L_\bl}$ is the constant
extension of $T^{E_\bl}$. If Lemma~\ref{centralizer-lifting-even}
first conjugates the torus, transport the logarithm coordinates by
the same conjugation. In these coordinates $A_\bl$ consists of
$(X,a)\in\fg(F)\times F$, using the section
$X\mapsto\exp_2(t^nX)$, and $\ell_\bl$ is $F$-linear on the
additive logarithm coordinates of $U$.

Let $c\in C_\bl$, represented by an element of $T^{L_\bl}(1)$.
Since $c$ reduces to the identity,
\[
 t^n(\Ad(c)X-X)\in U
\]
in logarithm coordinates. This difference has degree at least
$n+1$, so its commutator with $t^nX$ has degree at least $2n+1$
and vanishes at the level under consideration. The truncated
logarithm calculation therefore gives
\[
 c(X,a)c^{-1}=(X,a+f_c(X)),\qquad
 f_c(X)=\ell_\bl\bigl(t^n(\Ad(c)X-X)\bigr).
\]
In particular, $f_c$ is $F$-linear. It vanishes on the toral line
$\tori_\bl$, which $c$ centralizes. For $s\in T^{E_\bl}$,
commutativity of $s$ and $c$ and invariance of $\ell_\bl$ give
$f_c(\Ad(s)X)=f_c(X)$. The torus-stable complement
$W=\tori_\bl^\perp$ has no trivial torus summand, so
$(W^*)^{T^{E_\bl}}=0$. Thus $f_c|_W=0$, and $f_c=0$ on
$\fg(F)$. Hence $c$ centralizes $A_\bl$. It also commutes with
$T^{L_\bl}_{2n}$. These subgroups generate $H_\bl$, so $c$ is central.
\end{proof}
The constant-loop torus $T^{E_\bl}$ acts faithfully on $W$, whereas $Z(H_\bl)R_W$ acts trivially on $W$. Claim \ref{q}(3), Lemma \ref{positive-loop-central}, and the decomposition $A_\bl=(\tori_\bl,+)R_W$ give a bijection
$$
T^{E_\bl}\times Z(H_\bl)R_W\stackrel{\sim}{\longrightarrow}H_\bl.
$$
Let $\rho$ be an irreducible representation of $\St_\bl$ on which $U$ acts by $\bl$. It factors through $H_\bl$. Let $\zeta_\rho$ be its central character. Then
$$
(H_\bl,R_W,T^{E_\bl},\widetilde a_\bl,\zeta_\rho)
$$
is a quasi-Heisenberg group. To check the additional compatibility
in Definition~\ref{qH}, use the constant-loop torus in the chosen
coordinates. For $s\in T^{E_\bl}$,
\[
 s\exp_2(t^n w)s^{-1}=\exp_2(t^n\Ad(s)w).
\]
The complement $W=\tori_\bl^\perp$ is torus-stable, and the central
coordinate $\ell_\bl$ is invariant. Consequently conjugation is
$(w,u)\mapsto(\Ad(s)w,u)$ on $R_W$, with no central shear.
If the torus has first been conjugated as in
Lemma~\ref{centralizer-lifting-even}, the Heisenberg coordinates are
transported by the same conjugation; the displayed identity is preserved.
\begin{proposition}\label{A-prime}
The restriction $\rho|_{R_W}$ is irreducible.
\end{proposition}
\begin{proof}
Apply Lemma~\ref{thet}(1) to the quasi-Heisenberg group
\[
 (H_\bl,R_W,T^{E_\bl},\widetilde a_\bl,\zeta_\rho),
\]
whose hypotheses were verified above. The normalization in
Claim~\ref{qH-phi} includes the central character as well as the
Heisenberg and Weil operators. Thus the resulting multiplicity
action is trivial on both $Z(H_\bl)$ and $R_W$ and descends to
$H_\bl/Z(H_\bl)R_W\simeq T^{E_\bl}$. The cited lemma makes
this multiplicity space one-dimensional, so $\rho|_{R_W}$ is
irreducible.
\end{proof}
\subsubsection{The Borel graph subgroup}
Let $\mathfrak p\subset\fg$ be the Lie algebra of $P$.
\begin{lemma}\label{p-to-W}
The natural map
$$
\mathfrak p\longrightarrow\fg/\tori_\bl
$$
is an isomorphism.
\end{lemma}
\begin{proof}
A nonzero element of the elliptic line $\tori_\bl$ is elliptic semisimple, whereas every element of $\mathfrak p$ has its eigenvalues in $F$. Hence $\mathfrak p\cap\tori_\bl=0$. Both spaces have dimension $2$.
\end{proof}
Let $R_P\subset A_\bl$ be the image of $P(O)_{2n}(n)$. Under
$$
A_\bl\simeq(\tori_\bl,+)\times R_W,
$$
the subgroup $R_P$ has the form
$$
R_P=\{(c(w),w,a):w\in W,\ a\in F\}
$$
for an $F$-linear map $c:W\to\tori_\bl$. Its commutator pairing is $B_\bl$, so $R_P$ is a Heisenberg group with center $F$.
\begin{proposition}\label{RP-irred}
The restriction $\rho|_{R_P}$ is irreducible.
\end{proposition}
\begin{proof}
By Proposition \ref{A-prime}, $\rho|_{R_W}=\tau_W$. The central group $(\tori_\bl,+)$ acts through a character $\eta$. Hence
$$
\rho(c(w),w,a)=\eta(c(w))\tau_W(w,a).
$$
There is an $F$-linear functional $d:W\to F$ such that $\eta(c(w))=\widetilde\psi(d(w))$. By nondegeneracy of $B_\bl$, there is $w_0\in W$ with $d(w)=B_\bl(w_0,w)$. Conjugation by $\tau_W(w_0,0)$, with the sign of $w_0$ adjusted if necessary, removes this twist. Thus $\rho|_{R_P}$ is isomorphic to the Heisenberg representation and is irreducible.
\end{proof}
\subsubsection{The finite-level restriction map}
Put $O_{n-1}=F[t]/(t^n)$. The trace pairing identifies $\widehat U$ with $\fg(O_{n-1})$. If
$$
Y(t)=\sum_{r=0}^{n-1}Y_rt^r
$$
and
$$
X(t)=\sum_{j=n+1}^{2n}X_jt^j,
$$
then the corresponding character is
$$
\bl_Y(1+X(t))=
\widetilde\psi\left(\operatorname{coeff}_{t^{2n}}\operatorname{tr}(Y(t)X(t))\right).
$$
Under this identification, the coadjoint action is the adjoint action on $\fg(O_{n-1})$, and ellipticity means that $Y_0$ is elliptic.
\begin{lemma}\label{orbit-restriction-injective}
Let $\mcO\subset\fg(O_{n-1})$ be an adjoint orbit whose constant term is elliptic. Then restriction gives an injective map
$$
\mcO\longrightarrow\mathfrak p(O_{n-1})^*.
$$
\end{lemma}
\begin{proof}
Write
$$
Y=\begin{pmatrix}a&b\\c&-a\end{pmatrix},
\qquad
Y'=\begin{pmatrix}a'&b'\\c'&-a'\end{pmatrix}
$$
with entries in $O_{n-1}$. Equality of their restrictions to $\mathfrak p(O_{n-1})$ gives $a'=a$ and $c'=c$. Since $Y$ and $Y'$ are conjugate, their invariants
$$
q(Y)=-\det(Y)=a^2+bc
$$
and $q(Y')$ agree. The constant term $c_0$ is nonzero: otherwise $q(Y_0)=a_0^2$ would be a square, contrary to ellipticity. Thus $c$ is a unit in $O_{n-1}$, and
$$
b=\frac{q(Y)-a^2}{c}.
$$
The same formula gives $b'=b$.
\end{proof}
\begin{corollary}\label{finite-stabilizer-identity}
For elliptic $\bl\in\widehat U$,
$$
\St_{P(O)_{2n}}\left(\bl|_{U\cap P(O)_{2n}}\right)
=P(O)_{2n}\cap\St_{G(O)_{2n}}(\bl).
$$
\end{corollary}
\begin{proof}
Only the inclusion from left to right needs proof. If $p\in P(O)_{2n}$ stabilizes the restriction of $\bl$, then $p\bl$ and $\bl$ lie in the same $G(O)_{2n}$-orbit and have the same restriction. Lemma \ref{orbit-restriction-injective} gives $p\bl=\bl$.
\end{proof}
\begin{theorem}\label{charev}
Every elliptic representation of $G(O)$ of positive even depth is special.
\end{theorem}
\begin{proof}
Let $\pi$ be elliptic of depth $2n$. By Lemma \ref{MM}, there are an elliptic $\bl\in\widehat U$ and an irreducible representation $\rho$ of $\St_\bl$ such that $U$ acts on $\rho$ by $\bl$ and
$$
\pi\simeq\ind_{\St_\bl}^{G(O)_{2n}}(\rho),
$$
followed by inflation to $G(O)$. Put $H'=P(O)_{2n}$. By Proposition \ref{RP-irred}, $\rho|_{R_P}$ is irreducible. The image of $H'\cap\St_\bl$ in $H_\bl$ contains $R_P$, and $\rho$ factors through $H_\bl$. It follows that $\rho|_{H'\cap\St_\bl}$ is irreducible. Claim \ref{q}(5) gives $H'\St_\bl=G(O)_{2n}$, and Lemma \ref{orbit-restriction-injective} supplies the injectivity hypothesis of Lemma \ref{3.12}. Therefore $\pi|_{P(O)}$ is irreducible.
\end{proof}
\subsubsection{Parametrization at even depth}
Let $E/F$ be quadratic, put $L=E((t))$, and let $\theta\in\Xi^L(2n)$. Since $2n<2(n+1)$, Claim \ref{iota} applies to $p_{2n}(n+1)$. Define
$$
\bl_\theta=\theta\circ p_{2n}(n+1)\in\widehat U.
$$
Then $\bl_\theta$ is elliptic and has depth $2n$. Let
$$
H_\theta=\St_{\bl_\theta}/K_{\bl_\theta},
$$
and retain the notation $R_W$, $C_{\bl_\theta}$, and $T^E$ from the preceding construction.
\begin{lemma}\label{central-character-even}
One has
$$
Z(H_\theta)=C_{\bl_\theta}Z(R_W).
$$
The character $\theta$ on $C_{\bl_\theta}$ and the central character $\bl_\theta$ on $Z(R_W)$ agree on their intersection and therefore define a unique character
$$
\zeta_\theta:Z(H_\theta)\longrightarrow\CC^*.
$$
\end{lemma}
\begin{proof}
The radical $(\tori_{\bl_\theta},+)$ of the commutator pairing is the image of the appropriate congruence subgroup of $T^L(1)$, hence is contained in $C_{\bl_\theta}$. Lemma \ref{positive-loop-central} then gives the displayed description of the center. On the intersection, the map $p_{2n}(n+1)$ is the identity on the torus, so the two characters agree.
\end{proof}
To make the choice of splitting consistent as $\theta$ varies, fix a
maximal compact subgroup $C\subset\SL(W)$ containing the image of
$T^E$, and fix a quadratic character $\epsilon$ of $T^E$.
All the forms $B_{\bl_\theta}$ are nonzero scalar multiples of one
another, so their symplectic groups are the same group $\SL(W)$.
For each $\bl=\bl_\theta$, let $s_{C,\bl}$ be the unique splitting
of the corresponding metaplectic cover over $C$, and use
\[
 s_\bl=(s_{C,\bl}|_{T^E})\epsilon.
\]
Existence and uniqueness over $C$ follow from Claim~\ref{heis} and
the argument of Lemma~\ref{unique-splitting}.
Every splitting over $T^E$ at any one $\bl$ occurs in this way.
The notation $s$ in $\alpha_{2n,s}^E$ refers to this consistent
choice of splittings.

Choose the corresponding splitting $s=s_{\bl_\theta}:T^E\to\widetilde{\Sp(W)}$. It gives a lift $\widetilde a_{\bl_\theta,s}:H_\theta\to\widetilde{\Sp(W)}$. Applying Corollary \ref{alpha} with central character $\zeta_\theta$ and with the character $\theta|_{T^E}$ of the quotient torus gives an irreducible representation $\rho_{\theta,s}$ of $\St_{\bl_\theta}$. Put
$$
\alpha_{2n,s}^E(\theta)
=
\ind_{\St_{\bl_\theta}}^{G(O)_{2n}}(\rho_{\theta,s}),
$$
and inflate to $G(O)$.
\begin{proposition}\label{even-fixed-splitting}
For every splitting $s$, the construction has the following properties.
\begin{enumerate}
\item The representation $\alpha_{2n,s}^E(\theta)$ is irreducible, elliptic, special, and of depth $2n$.
\item One has
$$
\alpha_{2n,s}^E(\theta)\simeq\alpha_{2n,s}^E(\theta^{-1}).
$$
\item The induced map
$$
\bar\alpha_{2n,s}^E:\bar\Xi^L(2n)\longrightarrow\widehat G_{2n}^E
$$
is a bijection.
\end{enumerate}
\end{proposition}
\begin{proof}
Irreducibility follows from Lemma \ref{MM}(2), and specialness from Theorem \ref{charev}. The depth and ellipticity are read from $\bl_\theta$. 

For the inversion assertion, choose an involution $\dot w$ in the normalizer
of the fixed torus which induces its nontrivial Galois automorphism,
and put $j=\Ad(\dot w)|_W$. Write $-\bl$ here for the character $\bl^{-1}$,
so that $\ell_{-\bl}=-\ell_\bl$ and $B_{-\bl}=-B_\bl$.
Then
\[
 B_{-\bl}(jw,jw')=B_\bl(w,w'),
\]
so $j$ is a symplectic isomorphism from $(W,B_\bl)$ to
$(W,B_{-\bl})$. It induces the isomorphism of Heisenberg groups
$(w,a)\mapsto(jw,a)$, together with the functorial isomorphism
$j_*$ of their metaplectic covers and the corresponding intertwiner
of Weil representations. In the original quotient coordinates the
central coordinate is indeed preserved, because
\[
 \ell_{-\bl}(\dot w u \dot w^{-1})=\ell_\bl(u)\qquad(u\in U).
\]
We also have $jT^Ej^{-1}=T^E$, with the induced action $z\mapsto z^{-1}$.
Identify the torus module $W$ with $E$. The operator $j$ then has the
form $x\mapsto b\bar x$, with $N_{E/F}(b)=1$ because $j^2=1$.
Thus $b$ is a unit of $O_E$, and $j$ preserves every fractional
$O_E$-ideal. By Lemma~\ref{lif}, it therefore preserves each maximal
compact subgroup containing $T^E$, in particular the chosen $C$.
Uniqueness of the splitting over $C$ gives
\[
 j_*s_{C,\bl}(z)=s_{C,-\bl}(z^{-1})\qquad(z\in T^E).
\]
Since $\epsilon(z)=\epsilon(z^{-1})$, the same equality holds for
$s_\bl$ and $s_{-\bl}$. Consequently conjugation by $\dot w$ carries
$\rho_{\theta,s}$ to $\rho_{\theta^{-1},s}$: its action on the
positive-loop torus replaces $\theta$ by $\theta^{-1}$, and the
preceding compatibility identifies the Weil factors. Inducing to
$G(O)_{2n}$ proves (2).

Conversely, let $\pi\in\widehat G_{2n}^E$. Mackey theory gives an elliptic $\bl$ and an irreducible representation $\rho$ of $\St_\bl$ such that
$$
\pi\simeq\ind_{\St_\bl}^{G(O)_{2n}}\rho.
$$
By Claim \ref{q}, after conjugacy we may identify the torus in the stabilizer with $T^L_{2n}$. The action of the positive-loop subgroup is central in $H_\bl$ and hence defines a character $\theta_1$ of $T^L(1)$. Proposition \ref{A-prime} and the fixed splitting $s$ identify the Weil multiplicity space with a character $\theta_0$ of $T^E$. Their restrictions to the intersection agree, since both equal the central character of $\rho$. Thus $\theta_0$ and $\theta_1$ glue to a character $\theta$ of
$$
T^L=T^E T^L(1).
$$
Its restriction to the last nontrivial congruence quotient is $\bl$, so $\theta\in\Xi^L(2n)$. This construction is inverse to $\theta\mapsto\alpha^E_{2n,s}(\theta)$: the central character recovers $\theta_1$, while the Weil multiplicity space recovers $\theta_0$.
Finally, Mackey's isomorphism criterion shows that two induced representations are isomorphic precisely when their inducing data are conjugate. The quotient of the normalizer of $T^L$ by $T^L$ has order two, and its nontrivial element acts on $T^L=L^*/K^*$ by Galois conjugation, which is inversion. Hence the only identification is $\theta\sim\theta^{-1}$. This proves (3).
\end{proof}
\subsubsection{The unramified case}
If $E/F$ is unramified, Corollary \ref{lift} gives a canonical splitting over $T^E\simeq E^1$. Proposition \ref{even-fixed-splitting} therefore gives a canonical bijection
$$
\bar\alpha_{2n}^E:\bar\Xi^L(2n)\stackrel{\sim}{\longrightarrow}\widehat G_{2n}^E.
$$
\subsubsection{The ramified case}
Assume that $E/F$ is ramified. By Lemma \ref{tori}, the set of splittings over $T^E$ is a torsor under
$$
\D_E=\Hom(T^E,\{\pm1\})\simeq\ZZ/2\ZZ.
$$
Let $\delta$ be its nontrivial element. If $s$ is a splitting, the other splitting is $\delta s$, and Corollary \ref{alpha} gives
$$
\alpha_{2n,\delta s}^E(\theta)
\simeq
\alpha_{2n,s}^E(\delta\theta),
$$
where $\delta$ is extended trivially across $T^L(1)$. Thus the class of the parameter in
$$
\ti\Xi^L(2n)=\bar\Xi^L(2n)/\D_L
$$
is independent of the splitting.
The action of $\delta$ on $\bar\Xi^L(2n)$ is free. Indeed, if $[\delta\theta]=[\theta]$, then either $\delta\theta=\theta$, which gives $\delta=1$, or $\delta\theta=\theta^{-1}$, which gives $\theta^2=\delta$. In the second case $\theta|_{T^L(1)}$ has order at most two, because $\delta$ is trivial on $T^L(1)$. But $T^L(1)$ is uniquely $2$-divisible, so it has no nontrivial character of order two. Indeed, the unique square root congruent to $1$ exists by Hensel's lemma; the norm of that root is also congruent to $1$ and has square $1$, hence equals $1$. Hence $\theta$ is trivial on $T^L(1)$, contradicting the positive-depth assumption.
The involution on $\widehat G_{2n}^E$ may be described as follows. Choose $s$, write $\pi=\bar\alpha_{2n,s}^E([\theta])$, and put
$$
\delta\cdot\pi=\bar\alpha_{2n,s}^E([\delta\theta]).
$$
The preceding transformation formula shows that this definition is independent of $s$. It is free, and its orbits are exactly the fibers of the canonical map
$$
\widehat G_{2n}^E\longrightarrow\ti\Xi^L(2n).
$$
Hence $\widehat G_{2n}^E$ is naturally a $\D_L$-torsor over $\ti\Xi^L(2n)$.
\begin{remark}
When $E/F$ is ramified and $q\equiv1\pmod4$, Lemma~\ref{ex}
shows that the unique splittings over the two containing maximal
compact subgroups have the same restriction to $E^1$.
Their common restriction is therefore a distinguished splitting
for the fixed embedded norm-one torus, independent of the choice
between these two subgroups. This observation concerns a fixed
symplectic realization; the torsor formulation above does not
require an identification of all the varying realizations.
\end{remark}
\begin{theorem}\label{equnr}
Let $E/F$ be quadratic and put $L=E((t))$.
\begin{enumerate}
\item At depth $0$, the convention of Definition~\ref{depth-zero-convention} gives the natural bijection of Proposition~\ref{depth-zero-parameterization}
$$
\bar\Xi^L(0)\stackrel{\sim}{\longrightarrow}\widehat G_0^E.
$$
\item If $m>0$ is odd, there is a natural bijection
$$
\bar\Xi^L(m)\stackrel{\sim}{\longrightarrow}\widehat G_m^E.
$$
\item If $m>0$ is even and $E/F$ is unramified, there is a natural bijection
$$
\bar\Xi^L(m)\stackrel{\sim}{\longrightarrow}\widehat G_m^E.
$$
\item If $m>0$ is even and $E/F$ is ramified, there is a canonical map
$$
\widehat G_m^E\longrightarrow\ti\Xi^L(m)
$$
whose fibers are torsors under $\D_L\simeq\ZZ/2\ZZ$.
\end{enumerate}
\end{theorem}

\begin{corollary}\label{cunr}
Let $\pi$ be a special representation of $G(O)$ of positive depth $m$. Then there are a closed almost unipotent subgroup $R\subset G(O)_m$ and an irreducible representation $\rho$ of $R$ such that $\pi$ is the inflation of
$$
\ind_R^{G(O)_m}(\rho).
$$
At odd depth, $\rho$ is a character. At even depth, $\rho$ is the quasi-Heisenberg representation described above.
\end{corollary}
\begin{proof}
At odd depth take $R=\St_{\bl_\pi}$ and $\rho=\rho_\pi$, which is a character. At even depth take the stabilizer and its irreducible quasi-Heisenberg representation. These stabilizers are extensions of anisotropic tori by unipotent groups and are therefore almost unipotent.
\end{proof}
\section{\texorpdfstring{Special representations of $G'(O)$}{Special representations of G'(O)}}\label{G'O}
We now carry out the analogous analysis for $G'(O)$. Let
$$
\kk:G'(O)\longrightarrow\ZZ/2\ZZ\ltimes F^*
$$
be the homomorphism, with $\ZZ/2\ZZ$ written multiplicatively as $\{\pm1\}$, for which its nontrivial element acts on $F^*$ by inversion, and which is defined as follows. If $g\in I$ is represented by $\begin{pmatrix}a&b\\c&d\end{pmatrix}$ and $a_0,d_0$ are the constant terms of $a,d$, put
$$
\kk(g)=(1,a_0/d_0).
$$
For $s=\begin{pmatrix}0&1\\t&0\end{pmatrix}$, put $\kk(s)=(-1,1)$.
\begin{claim}\label{act}
In the following statements $m\in\frac12\ZZ_{>0}$.
\begin{enumerate}
\item The subgroups $G'(O)(m)$ are normal in $G'(O)$.
\item The subquotients $G'(O)_m(m)$ are naturally $F$-vector spaces.
\item If $m\in\ZZ$, then $G'(O)_m(m)$ is one-dimensional. The element $s$ acts by $-1$, and $P(O)$ acts trivially.
\item If $m\in\frac12\ZZ\sm\ZZ$, then
$$
G'(O)_m(m)=M_m^+\oplus M_m^-,
$$
where
$$
M_m^+\simeq\fm^{m-1/2}/\fm^{m+1/2},
\qquad
M_m^-\simeq\fm^{m+1/2}/\fm^{m+3/2}.
$$
Each summand is one-dimensional. The choice of $t$ gives generators $x_m^+$ and $x_m^-$ and identifies the subquotient with $F\oplus F$.
\item The adjoint action on $M_m^+\oplus M_m^-$ factors through $\kk$: an element $a\in F^*$ acts by $\operatorname{diag}(a,a^{-1})$, and $s$ interchanges the two summands.
\end{enumerate}
\end{claim}
\begin{definition}\label{ac}
Let $m=n-1/2$ with $n\in\ZZ_{>0}$.
\begin{enumerate}
\item If $\mu$ is a character of $G'(O)_m(m)=M_m^+\oplus M_m^-$, define $[\mu^\pm]\in F$ by
$$
\mu(x_m^\pm z)=\widetilde\psi([\mu^\pm]z),\qquad z\in F.
$$
\item A character $\mu$ is elliptic if its restriction to each of $M_m^+$ and $M_m^-$ is nontrivial.
\item Put
$$
\det(\mu)=t[\mu^+][\mu^-]\in K.
$$
For an elliptic character, this is nonzero; its square class is independent of the chosen uniformizer and the root generators associated with it.
\item For an irreducible representation $\pi$ of $G'(O)$, let $m(\pi)$ be the least $m\in\frac12\ZZ_{\geq0}$ such that $G'(O)(m+1/2)$ acts trivially.
\item The representation $\pi$ is special if it has dimension greater than $1$ and $\pi|_{P(O)}$ is irreducible.
\item A non-one-dimensional irreducible representation $\pi$ is
elliptic if $m=m(\pi)\in\frac12\ZZ\sm\ZZ$ and every character
in the support of $\pi|_{M_m^+\oplus M_m^-}$ is elliptic.
Equivalently, neither of the two coordinate supports contains
the trivial character.
\end{enumerate}
\end{definition}

\begin{remark}
\begin{enumerate}
\item Replacing $t$ by another uniformizer and changing the associated root generators together changes $t[\mu^+][\mu^-]$ by a square. Thus the quadratic extension $K(\sqrt{t[\mu^+][\mu^-]})$ is intrinsic.
\item Elliptic representations of $G'(O)$ necessarily have positive half-integral depth.
\end{enumerate}
\end{remark}
\begin{lemma}\label{Gprime-nonelliptic}
Let $\pi$ be an irreducible representation of $G'(O)$ of dimension greater than one.
\begin{enumerate}
\item If $m(\pi)\in\ZZ$, then $\pi|_{P(O)}$ is reducible.
\item If $m(\pi)\notin\ZZ$ and $\pi$ is not elliptic, then $\pi|_{P(O)}$ is reducible.
\end{enumerate}
\end{lemma}
\begin{proof}
At $m=0$ the representation factors through
$G'(O)/G'(O)(\tfrac12)\simeq F^*\rtimes\{\pm1\}$,
where the nontrivial element acts by inversion. Every irreducible
representation of this group of dimension greater than one is
induced from a character $\chi$ of $F^*$ with
$\chi\ne\chi^{-1}$; its restriction to the image $F^*$ of
$P(O)$ is $\chi\oplus\chi^{-1}$. It is therefore reducible.
We may now assume $m>0$.
Suppose first that $m=m(\pi)\in\ZZ$. The support of $\pi$ on the nonzero one-dimensional group $G'(O)_m(m)$ is a nontrivial $G'(O)$-orbit. By Claim \ref{act}, $P(O)$ acts trivially on this character group, whereas $s$ acts by sign. Thus the support contains at least two $P(O)$-orbits, and Lemma \ref{MM}(1) implies reducibility.
Now let $m=n-1/2$. If $\pi$ is not elliptic, its nonzero support on $M_m^+\oplus M_m^-$ is the union
$$
((M_m^+)^*\sm0)\sqcup((M_m^-)^*\sm0).
$$
Let $P_m$ be the image of $P(O)$ in $G'(O)_m$, and put
\[
 U=M_m^+\oplus M_m^-,\qquad U'=U\cap P_m=M_m^+.
\]
The subgroup $U'$ is normal and abelian in $P_m$.
The characters on the first axis restrict to all nontrivial
characters of $U'$, whereas those on the second axis restrict to
the trivial character. Both kinds occur in the $U'$-support:
if $\lambda$ belongs to the displayed $U$-support and
$\lambda'=\lambda|_{U'}$, transitivity of twisted coinvariants
gives a surjection
\[
 V_{U',\lambda'}\twoheadrightarrow V_{U,\lambda}\ne0,
\]
where $V$ is the representation space of $\pi$.
Thus the $U'$-support of $\pi|_{P_m}$ contains the two distinct
$P_m$-orbits $\{0\}$ and $(U')^*\sm\{0\}$.
Lemma~\ref{MM}(1), applied to the normal subgroup $U'\subset P_m$,
shows that $\pi|_{P_m}$, and hence $\pi|_{P(O)}$, is reducible.
\end{proof}
\begin{corollary}
Every special representation of $G'(O)$ is elliptic.
\end{corollary}
\subsection{\texorpdfstring{Elliptic representations of $G'(O)$}{Elliptic representations of G'(O)}}
Fix $n\in\ZZ_{>0}$ and put $m=n-1/2$.
The image of $P(O)$ in $G'(O)_m=G'(O)/G'(O)(n)$ is
$P(O)_{n-1}=P(O_{n-1})$; more explicitly, it is $P(O/t^nO)$.
Here and below we use this image for intersections with subgroups of
$G'(O)_m$. Let $U$ be the image of $G'(O)(n/2)$ in $G'(O)_m$.
\begin{claim}\label{Gprime-U}
The group $U$ is a closed normal commutative subgroup of $G'(O)_m$.
\end{claim}
\begin{proof}
Consider the hereditary order and its radical
\[
 \mathfrak A=\begin{pmatrix}O&O\\tO&O\end{pmatrix},\qquad
 \mathfrak P=\begin{pmatrix}tO&O\\tO&tO\end{pmatrix}.
\]
Matrix multiplication gives $\mathfrak P^2=t\mathfrak A$.
The group $G'(O)$ normalizes both lattices, and
Definition~\ref{G-prime-filtration} identifies $G'(O)(r)$ with
the projective image of $1+\mathfrak P^{2r}$.
Consequently $U$ is the projective image of
\[
 (1+\mathfrak P^n)/(1+\mathfrak P^{2n}).
\]
It is closed and normal. Since
$\mathfrak P^n\mathfrak P^n\subset\mathfrak P^{2n}$,
its multiplication is induced by addition of matrices modulo
$\mathfrak P^{2n}$, and hence it is commutative.
This argument applies to both parities of $n$.
\end{proof}
\begin{definition}\label{c}
Let $\bl\in\widehat U$.
\begin{enumerate}
\item Let $\St_\bl$ be its stabilizer in $G'(O)_m$.
\item Let $\bar\bl$ be its restriction to $G'(O)_m(m)=M_m^+\oplus M_m^-$.
\item The character $\bl$ is elliptic if $[\bar\bl^+]$ and $[\bar\bl^-]$ are both nonzero.
\item For elliptic $\bl$, put
$$
L_\bl=K\left(\sqrt{t[\bar\bl^+][\bar\bl^-]}\right).
$$
\end{enumerate}
\end{definition}
\begin{claim}\label{ch}
Let $\bl$ be elliptic and put $L=L_\bl$.
\begin{enumerate}
\item For a suitable embedding $T^L\ho G'(O)$, one has
$$
\St_\bl=T^L_mU.
$$
\item One has
$$
P(O)_{n-1}\St_\bl=G'(O)_m.
$$
\item The restriction map from the $G'(O)_m$-orbit of $\bl$ to $\widehat{U\cap P(O)_{n-1}}$ is injective.
\end{enumerate}
\end{claim}
\begin{proof}
Use the lattices $\mathfrak A$ and $\mathfrak P$ from the proof of
Claim~\ref{Gprime-U}, and put
\[
 k=\lfloor n/2\rfloor,\qquad \ell=\lceil n/2\rceil.
\]
The trace pairing, together with the fixed character $\psi$ of $K$,
represents every character of $U$ uniquely in the form
\[
 \lambda_C(1+X)=\psi\bigl(t^{-n-1}\operatorname{tr}(CX)\bigr),
 \qquad X\in\mathfrak P^n,
\]
where
\[
 C=\begin{pmatrix}ta&b\\tc&-ta\end{pmatrix}\in
 \mathfrak P\cap\mathfrak{sl}(2,K)
 \pmod{\mathfrak P^{n+1}\cap\mathfrak{sl}(2,K)}.
\]
Equivalently, $a$ is specified modulo $t^k$, and $b,c$ modulo
$t^\ell$. These formulas also show directly that the character is
unchanged by scalar representatives in the projective group.
On the last layer, its two coefficients are $c_0$ and $b_0$.
Thus ellipticity says that $b,c\in O^*$.

Choose lifts $a,b,c\in O$ and put $\delta=bc+ta^2\in O^*$.
Conjugation by
\[
 \begin{pmatrix}c&0\\0&1\end{pmatrix}
 \begin{pmatrix}1&-a/c\\0&1\end{pmatrix}\in P(O)
\]
sends $C$ to
\[
 C_0=\begin{pmatrix}0&\delta\\t&0\end{pmatrix},
 \qquad C_0^2=t\delta\,I_2.
\]
Since $\delta/\delta_0\in1+tO$ is a square, its centralizer in
$G(K)$ is a torus $T\simeq T^L$, with
$L=K(\sqrt{t\delta})=K(\sqrt{tb_0c_0})$.
This torus lies in $G'(O)$: it is the image of
$(K[C_0])^*$, whose units preserve the lattice chain and whose
uniformizer $C_0$ interchanges its two vertices.
We may work after this $P(O)$-conjugation.

For an element of the Iwahori represented by
\[
 g=\begin{pmatrix}u&v\\tw&z\end{pmatrix}\in\mathfrak A^*,
\]
character stabilization is equivalent to
$gC_0g^{-1}-C_0\in\mathfrak P^{n+1}$, or equivalently to
$gC_0-C_0g\in\mathfrak P^{n+1}$. The latter commutator is
\[
 \begin{pmatrix}
  t(v-\delta w)&\delta(u-z)\\
  t(z-u)&t(\delta w-v)
 \end{pmatrix}.
\]
Using $\mathfrak P^2=t\mathfrak A$, the condition is exactly
\[
 v-\delta w\in t^kO,\qquad u-z\in t^\ell O.
\]
The matrix
\[
 d=\begin{pmatrix}u&\delta w\\tw&u\end{pmatrix}
\]
is an invertible centralizer element, and these congruences say
$g-d\in\mathfrak P^n$. Hence $g\in T(1+\mathfrak P^n)$.
Conversely, every element of $T$ stabilizes the character, and
$1+\mathfrak P^n$ does so because
$[\mathfrak P^n,\mathfrak P]\subset\mathfrak P^{n+1}$.
The element $C_0\in T$ represents the nontrivial coset of
$G'(O)/I$, so the same conclusion holds in the other component.
Passing to $G'(O)_m$ proves
$\St_\lambda=T_m^LU$.

For (2), the same calculation gives the stronger equality
$G'(O)=P(O)T$. Indeed, for $g$ as above set
\[
 d'=\begin{pmatrix}z&\delta w\\tw&z\end{pmatrix}\in T\cap I.
\]
Its determinant is a unit, and $g(d')^{-1}\in P(O)$ because
$g$ and $d'$ have the same bottom row. The element $C_0\in T$
accounts for the second component. Taking the finite quotient gives (2).

Finally, every member of the orbit is represented by a conjugate
\[
 C'=\begin{pmatrix}ta'&b'\\tc'&-ta'\end{pmatrix},
 \qquad b'c'+t(a')^2=\delta,
\]
where $b',c'$ are units. Restriction to $U\cap P(O)_{n-1}$
records $a'\pmod{t^k}$ and $c'\pmod{t^\ell}$.
The invariant $\delta$ then determines
\[
 b'\equiv\bigl(\delta-t(a')^2\bigr)(c')^{-1}
 \pmod{t^\ell}.
\]
The right-hand side is well defined with the stated precisions,
since $\ell\leq k+1$. Thus restriction determines the entire
class of $C'$ modulo $\mathfrak P^{n+1}$, proving (3).
\end{proof}
\begin{theorem}\label{char}
Every elliptic representation of $G'(O)$ is special.
\end{theorem}
\begin{proof}
Let $\pi$ be elliptic of depth $m=n-1/2$. By Lemma \ref{MM}, there are an elliptic character $\bl$ of $U$ and an irreducible representation $\rho$ of $\St_\bl$ such that $U$ acts on $\rho$ by $\bl$ and
$$
\pi\simeq\ind_{\St_\bl}^{G'(O)_m}(\rho),
$$
followed by inflation to $G'(O)$. By Claim~\ref{ch}(1), the stabilizer is $T_m^LU$.
The group $U$ is commutative, the torus is commutative, and
$[T_m^L,U]\subset\ker(\bl)$ because the torus stabilizes $\bl$.
Thus $\St_\bl/\ker(\bl)$ is commutative, so $\rho$ is a character. Claim \ref{ch}(2)--(3) and Lemma \ref{3.12} imply that $\pi|_{P(O)}$ is irreducible.
\end{proof}
\subsubsection{Classification}
For a ramified quadratic extension $L/K$, let $\widehat G_m^L$ be the set of isomorphism classes of elliptic representations of $G'(O)$ of depth $m$ with associated extension $L$.
An elliptic representation $\pi$ determines, through the preceding proof, a character of $T^L_m$ and hence a character $\theta_\pi$ of $T^L$. Changing the embedding of $T^L$ by the nontrivial normalizer element replaces $\theta_\pi$ by its Galois conjugate, which is its inverse. Thus there is a recovery map
$$
\beta_m^L:\widehat G_m^L\longrightarrow\bar\Xi^L(m).
$$
\begin{claim}\label{Gprime-parameter}
Conversely, for every $\theta\in\Xi^L(m)$ there are an elliptic character $\bl_\theta\in\widehat U$ and a unique character $\widehat\theta$ of $\St_{\bl_\theta}$ such that
$$
L_{\bl_\theta}=L,
\qquad
\widehat\theta|_U=\bl_\theta,
\qquad
\widehat\theta|_{T^L_m}=\theta.
$$
The construction is compatible with inversion.
\end{claim}
\begin{proof}
Write $m=n-\frac12$, $k=\lfloor n/2\rfloor$, and
$\ell=\lceil n/2\rceil$. Choose the fixed embedding of $T^L$
as the centralizer of
\[
 C_0=\begin{pmatrix}0&\delta\\t&0\end{pmatrix},
 \qquad \delta\in F^*,\qquad L=K(\sqrt{t\delta}).
\]
The subgroup $T^L_m\cap U$ is parametrized by the projective
classes of $1+bC_0$, where $b\in t^kO/t^nO$.
These parameters add: the exact product law is
$(b+b')/(1+t\delta bb')$, whose difference from $b+b'$ belongs
to $t^nO$, since $3k+1\geq n$ for both parities of $n$.
Using the character coordinates in the proof of Claim~\ref{ch},
the $T^L$-invariant characters of $U$ are precisely
$\lambda_{cC_0}$ with $c\in O/t^\ell O$.
Indeed, invariance under $C_0$ forces the diagonal coefficient
to vanish and the two off-diagonal coefficients to have ratio
$\delta$; conversely, $cC_0$ commutes with the torus.
Their restrictions are
\[
 \lambda_{cC_0}(1+bC_0)
 =\widetilde\psi\left(
       \operatorname{coeff}_{t^{n-1}}(2\delta bc)\right).
\]
Since $2\delta\ne0$ and $n-k=\ell$, this is a perfect pairing
between $O/t^\ell O$ and $t^kO/t^nO$.
The restriction of $\theta$ therefore determines a unique $c$
and a unique $T^L$-invariant character $\bl_\theta$ of $U$.
Exact depth $m$ implies $c(0)\ne0$, so $\bl_\theta$ is elliptic;
its associated extension is
$K(\sqrt{t\delta c(0)^2})=L$.
Moreover $\theta$ factors through $T^L_m$, since
$T^L\cap G'(O)(n)=T^L\cap G'(O)(n+\frac12)=T^L(n+1)$:
both intersections are represented by $1+bC_0$ with $b\in t^nO$.
By Claim~\ref{ch}, its stabilizer is $T^L_mU$.
Define
\[
 \widehat\theta(zu)=\theta(z)\bl_\theta(u),
 \qquad z\in T^L_m,\quad u\in U.
\]
The factors agree on the intersection, and $T^L_m$ stabilizes
$\bl_\theta$, so this is a well-defined smooth character.
It is unique because $T^L_m$ and $U$ generate the stabilizer.
Finally, the normalizer element $\operatorname{diag}(-1,1)$
sends $C_0$ to $-C_0$, acts on $T^L$ by inversion, and sends
$c$ to $-c$. This proves compatibility with inversion.
\end{proof}
Define
$$
\alpha_m^L(\theta)=\ind_{\St_{\bl_\theta}}^{G'(O)_m}(\widehat\theta),
$$
followed by inflation to $G'(O)$.
\begin{proposition}\label{Oodd}
The construction induces a bijection
$$
\bar\alpha_m^L:\bar\Xi^L(m)\stackrel{\sim}{\longrightarrow}\widehat G_m^L,
$$
whose inverse is $\beta_m^L$.
\end{proposition}
\begin{proof}
Lemma \ref{MM}(2) gives irreducibility, and Theorem \ref{char} gives specialness. The support recovers $\bl_\theta$ up to the normalizer action, and the stabilizer character then recovers $\theta$ up to inversion. Thus the construction and recovery maps are inverse.
\end{proof}
\begin{theorem}\label{cunr-summary}
\begin{enumerate}
\item Let $\pi$ be a special representation of $G(O)$ of depth $m$. If $m>0$ then $\pi$ is the inflation of $\ind_R^{G(O)_m}(\rho)$ for a closed almost unipotent subgroup $R\subset G(O)_m$ and an irreducible representation $\rho$ of $R$. At odd depth $\rho$ may be taken to be a character; at positive even depth it is the quasi-Heisenberg representation of Section \ref{GO-even-depth}.
\item One has $\pi|_{P(O)}\simeq\bs(m)$.
\item Let $\pi$ be a special representation of $G'(O)$ of depth $n-1/2$. Then $\pi$ is the inflation of $\ind_R^{G'(O)_{n-1/2}}(\nu)$ for a closed almost unipotent subgroup $R$ and a character $\nu$ of $R$.
\item One has $\pi|_{P(O)}\simeq\bs(n-1)$.
\end{enumerate}
\end{theorem}
\begin{remark}The first two statements restate Corollary~\ref{cunr} and Lemma~\ref{restr}.\end{remark}
\section{\texorpdfstring{Categories and representations over $K$}{Categories and representations over K}}\label{pro-vect}
In this section we recall the categorical conventions needed below.
We first discuss pro-objects and filtrations, and then define the
categories of representations over $K=F((t))$.
All vector spaces are over $\CC$.

\subsection{Ind-objects and pro-objects}
\begin{definition}\label{cat}
Let $\mcC$ be a category. An object of $\Pro(\mcC)$ is a formal
projective system $(C_i)_{i\in I}$, where $I$ is a small cofiltered
indexing category. We write it as $\prl_{i\in I}C_i$. Morphisms are
defined by
\begin{equation}\label{pro-general-hom}
 \Hom_{\Pro(\mcC)}\left(\prl_i C_i,\prl_j D_j\right)
   =\varprojlim_j\varinjlim_i\Hom_{\mcC}(C_i,D_j).
\end{equation}
The limits on the right are ordinary limits of sets. In an additive
category they are also limits of abelian groups. Dually,
$\Ind(\mcC)=\Pro(\mcC^{op})^{op}$. The constant systems give fully
faithful embeddings of $\mcC$ in both categories.

We write $\Pro_{\aleph_0}(\mcC)$ and $\Ind_{\aleph_0}(\mcC)$ for
the full subcategories whose objects admit countable presentations.
Such presentations can be chosen to be sequences. We work in fixed
universes, with a larger universe for the categories themselves.
\end{definition}

These are \emph{formal} limits; their definition does not require
the corresponding limits to exist in $\mcC$. If $\mcC$ is abelian,
then $\Pro(\mcC)$ and $\Pro_{\aleph_0}(\mcC)$ are abelian, and
the latter inclusion is exact. A pro-object $(C_i)$ is zero precisely
when every target level receives a zero transition from some later
level. We shall use this criterion repeatedly.

\subsection{Vector spaces and sets}
Let $\Vect$ denote the category of all complex vector spaces and
let $\Sets$ denote the category of all sets. Their full subcategories
of finite-dimensional vector spaces and finite sets are denoted by
$\Vect_{\mathrm{fd}}$ and $\Sets_0$, respectively. Put
\begin{equation}\label{s8-working-category}
 \mVect:=\Pro_{\aleph_0}(\Vect)\subset\Pro(\Vect),\qquad
 \mSet:=\Pro_{\aleph_0}(\Sets)\subset\Pro(\Sets).
\end{equation}
Thus $\Pro(\Vect)$ and $\Pro(\Sets)$ always denote the full
pro-categories. The notation $\mVect$ and $\mSet$ is reserved for
countable projective systems. In particular, an object of $\mVect$
has a presentation
\[
 V_0\longleftarrow V_1\longleftarrow V_2\longleftarrow\cdots.
\]
The dimensions of the $V_n$ are unrestricted, and the transition
maps need not be surjective. The same convention applies to sets.
The familiar equivalences
\[
 \Vect\simeq\Ind(\Vect_{\mathrm{fd}}),\qquad
 \Sets\simeq\Ind(\Sets_0)
\]
use unrestricted filtered index sets. They impose no countability
condition on an ordinary vector space or set.

For $\mV=\prl_n V_n$ and $\mW=\prl_m W_m$,
\eqref{pro-general-hom} becomes
\begin{equation}\label{s8-hom-formula}
 \Hom_{\mVect}(\mV,\mW)
   =\varprojlim_m\varinjlim_n\Hom_{\Vect}(V_n,W_m).
\end{equation}
This formula defines the space of morphisms; it does not assert the
existence of an internal Hom object.

\begin{proposition}\label{s8-countable-closure}
The category $\mVect$ is abelian, and its inclusion into $\Pro(\Vect)$
is exact. It is closed under countable products and countable inverse
limits computed in $\Pro(\Vect)$, and under retracts.
\end{proposition}
\begin{proof}
A morphism between two towers can be represented, after passing to
cofinal subsequences, by a morphism of towers. Kernels and cokernels
are the towers of ordinary kernels and cokernels. The abelian identities
follow from those in $\Vect$ and the usual pro-category construction.

For a countable family $\mV^{(a)}=(V^{(a)}_n)$, its product is presented
by the finite products
\[
 \prod_{a\in J}V^{(a)}_{n_a},
 \qquad J\subset\NN\text{ finite},\quad n_a\in\NN.
\]
There are only countably many such finite choices. Countable inverse
limits are kernels of the usual two maps between countable products.
An idempotent splits in an abelian category, so its image is again
in $\mVect$.
\end{proof}

Finite sums, kernels, quotients, products, and inverse limits appearing
below are taken in this category, unless another category is explicitly
specified. For two towers define
\[
 \mV\widehat\otimes\mW
    :=\varprojlim_{n,m}(V_n\otimes W_m)\in\mVect.
\]
The pairs $(n,m)$ have a countable cofinal diagonal. In particular,
tensoring with an ordinary vector space preserves countable presentation.

\subsection{Strict objects, actual limits, and irreducible objects}
\begin{definition}\label{strict}\label{s8-strict}\label{maps}
Let $\mcC$ be an abelian category. An object of $\Pro(\mcC)$ is
\emph{strict} if it admits a presentation with epimorphic transition
maps. This definition does not impose countability. For sets we use
the same terminology with surjective transition maps.
We denote the strict objects of $\mVect$ by $\mVect^{\mathrm{str}}$.
For a countably presented strict object the presentation can be chosen
to be a tower with surjective transition maps.
For $\mV=\prl_i V_i\in\Pro(\Vect)$ put
\[
 \widetilde{\mV}:=\lim_i V_i\quad\text{in }\Vect.
\]
The tilde distinguishes the actual vector space from the pro-object.
\end{definition}

For a strict countable tower, $\widetilde{\mV}$ has the separated
complete linear topology whose neighborhoods of zero are the kernels
of its projections to the discrete spaces $V_n$. These projections
are surjective. Conversely, a separated complete vector space with a
countable basis of linear neighborhoods is recovered from its discrete
quotients. This describes $\mVect^{\mathrm{str}}$; the abelian category
$\mVect$ also contains non-strict objects.

\begin{example}\label{s8-limit-example}
Let $V_n=\bigoplus_{k\geq n}\CC e_k$, with transition maps the inclusions.
Every transition to a fixed level is nonzero, so the resulting
pro-object is nonzero, whereas $\lim_nV_n=0$. Thus actual vectors in
$\widetilde{\mV}$ do not in general detect a pro-vector space.
\end{example}

\begin{lemma}\label{s8-uq-functional}
Every nonzero $\mL\in\mVect$ admits a nonzero morphism
$\mL\to\CC$.
\end{lemma}
\begin{proof}
Choose a tower $(L_n)$ representing $\mL$. A pro-object is zero precisely
when every target level receives a zero transition from some later level.
Consequently there is an $n_0$ such that all the images
\[
 E_n=\im(L_n\to L_{n_0}),\qquad n\geq n_0,
\]
are nonzero. They form a decreasing sequence of subspaces.
If one $E_n$ is finite-dimensional, the sequence eventually stabilizes
at a nonzero subspace; choose a linear functional nonzero on that
subspace. Otherwise choose inductively linearly independent vectors
$e_n\in E_n$, and extend the assignment $e_n\mapsto1$ to a linear
functional on $L_{n_0}$. In either case the functional is nonzero on
every $E_n$, so its class in
$\varinjlim_n\Hom(L_n,\CC)$ is nonzero.
\end{proof}

\begin{corollary}\label{countable-simple-vector-spaces}
Up to isomorphism, the only irreducible object of $\mVect$ is $\CC$.
\end{corollary}
\begin{proof}
A nonzero morphism from an irreducible object to $\CC$ is an
isomorphism. Apply Lemma~\ref{s8-uq-functional}.
\end{proof}

\begin{remark}\label{s8-countability-essential}
There are many other irreducible objects in $\Pro(\Vect)$.
For example, let $E$ be an infinite-dimensional vector space and
extend the filter of its finite-codimensional subspaces to a maximal
proper filter $\mathcal U$ in the lattice of subspaces of $E$.
Then
\[
 \mL_{\mathcal U}=\prl_{L\in\mathcal U}L
\]
is a nonzero subobject of the constant object $E$. A subobject of a
constant pro-object can be represented by a cofiltered family of
subspaces. Thus a nonzero subobject of $\mL_{\mathcal U}$ would give
a proper filter containing $\mathcal U$, and maximality shows that
$\mL_{\mathcal U}$ is irreducible.
Every scalar functional at a stage vanishes on a further stage:
extend it to $E$ and intersect with its kernel, which belongs to
$\mathcal U$. Hence
\[
 \Hom_{\Pro(\Vect)}(\mL_{\mathcal U},\CC)=0.
\]
In particular, $\mL_{\mathcal U}$ is not countably presented.
One obtains nonisomorphic examples by taking different infinite
dimensions $\kappa$ and starting with all subspaces of codimension
less than $\kappa$. Every member of the resulting filter has dimension
$\kappa$, and an inclusion between its stages cannot factor through
a vector space of smaller dimension.

Since we are interested in irreducible representations of groups,
we use the countable category $\mVect$ as our coefficient category.
The full category will be retained for comparison. Its subobjects
and quotients need not be countably presented, so comparison of the
two notions of irreducibility requires an argument.
\end{remark}

\begin{lemma}\label{strict-subobject-countable}
A strict subobject in $\Pro(\Vect)$ of an object of $\mVect$
belongs to $\mVect^{\mathrm{str}}$.
\end{lemma}
\begin{proof}
Let $\mU\hookrightarrow\mV=\prl_n V_n$ be such a subobject.
The image of a morphism from a strict pro-object to an ordinary
vector space is ordinary: represent the morphism at one level of
a strict presentation, and use surjectivity of all further
transition maps. Put $E_n=\im(\mU\to V_n)\in\Vect$.
The maps $E_{n+1}\to E_n$ are surjective, and the monomorphism
$\mU\to\mV$ factors through $\prl_n E_n$.
The map $\mU\to\prl_n E_n$ is an epimorphism. Indeed, a morphism
from its target to an ordinary space factors through some $E_n$,
and vanishes if its composite with $\mU\twoheadrightarrow E_n$
vanishes. Applying this to every output level proves the assertion
for an arbitrary pro-target. The map is also a monomorphism, and
hence an isomorphism.
\end{proof}

\subsection{Actions on abelian categories}
Let $\mcC$ be an abelian category.
\begin{definition}\label{basic}
Let $(\mcA,\times,*)$ be a category with finite products. A contravariant action of $\mcA$ on a category $\mcC$ consists of the following data:
\begin{enumerate}
\item A functor $a:\mcA^{op}\times\mcC\to\mcC$, together with a functorial identification $a(*,V)\simeq V$.
\item A functorial isomorphism
$$
\alpha_{X,Y,V}:a(X\times Y,V)\xrightarrow{\sim}a(X,a(Y,V)),
$$
for $X,Y\in\mcA$ and $V\in\mcC$, satisfying the usual associativity and unit axioms.
\end{enumerate}
\end{definition}
\begin{definition}
\begin{enumerate}
\item A group object in $\mcA$ is an object $G$ equipped with morphisms
$$
m:G\times G\to G,\qquad e:*\to G,\qquad \iota:G\to G,
$$
satisfying the usual associativity, unit, and inverse axioms.
\item Let $G$ be a group object of $\mcA$ and let $V\in\operatorname{Ob}(\mcC)$. A representation of $G$ on $V$ is a morphism
$$
\pi:V\to a(G,V)
$$
such that
$$
\alpha_{G,G,V}^{-1}\circ a(G,\pi)\circ\pi=a(m,V)\circ\pi
$$
and
$$
a(e,V)\circ\pi=\operatorname{id}_V,
$$
where we use the identification $a(*,V)\simeq V$.
\item A representation of $G$ on $V$ is irreducible if $V\ne0$ and $V$ has no nonzero proper $G$-invariant subobject.
\item Similarly, an action of an abstract group $\GG$ on $V$ is irreducible if $V\ne0$ and $V$ has no nonzero proper $\GG$-invariant subobject.
\end{enumerate}
\end{definition}
\subsection{Filtrations}
In this subsection and the next one we work in
$\Pro_{\aleph_0}(\mcC)$. All subobjects and limits are taken in
this category. It is abelian and has countable products.
Let $\mV$ be a strict object of $\Pro_{\aleph_0}(\mcC)$ and let $\mcF=\{\mW_i\}_{i\in\ZZ}$ be an increasing filtration of $\mV$ by subobjects.
\begin{definition}\label{admissible}
\begin{enumerate}
\item A subobject $\mU\subset\mV$ is dense with respect to $\mcF$ if, for every $i\in\ZZ$, the natural morphism
$$
\beta_i:\frac{\mW_{i+1}\cap\mU}{\mW_i\cap\mU}\longrightarrow\frac{\mW_{i+1}}{\mW_i}
$$
is an isomorphism.
\item The filtration $\mcF$ is admissible if, for every epimorphism $\kk:\mV\to V$ with $V\in\mcC$, there exist $r_-(\kk),r_+(\kk)\in\ZZ$ such that
$$
\kk(\mW_{r_-(\kk)})=0,\qquad \kk(\mW_{r_+(\kk)})=V.
$$
\item We put
$$
\overline{\mV}_{\mcF}:=\prod_{i\in\ZZ}\mW_{i+1}/\mW_i.
$$
\end{enumerate}
\end{definition}
\begin{claim}\label{separated}
If $\mcF$ is admissible, then
$$
\bigcap_{i\in\ZZ}\mW_i=0.
$$
\end{claim}
\begin{proof}
Choose a strict presentation $\mV\simeq\varprojlim_{n\geq0}V_n$ with epimorphic transition morphisms, and denote the canonical epimorphisms by $\kk_n:\mV\to V_n$. Put $\mI:=\bigcap_{i\in\ZZ}\mW_i$. For every $n$, admissibility gives an integer $r_-(n)$ such that $\kk_n(\mW_{r_-(n)})=0$. Since $\mI\subset\mW_{r_-(n)}$, the composite $\mI\to\mV\xrightarrow{\kk_n}V_n$ is zero. The morphisms $\kk_n$ are jointly monomorphic, and therefore $\mI=0$.
\end{proof}
\begin{lemma}\label{dense}
Let $\mcF=\{\mW_i\}_{i\in\ZZ}$ be an admissible filtration of a strict object $\mV$ of $\Pro_{\aleph_0}(\mcC)$, and let $\mU\subset\mV$ be dense with respect to $\mcF$. Then $\mU=\mV$.
\end{lemma}
\begin{proof}
It is enough to show that, for every epimorphism $\kk:\mV\to V$ with $V\in\mcC$, one has $\kk(\mU)=V$. Put
$$
V_i:=\kk(\mW_i),\qquad U_i:=\kk(\mW_i\cap\mU).
$$
Choose $r_-=r_-(\kk)$ and $r_+=r_+(\kk)$ as in Definition \ref{admissible}. The isomorphisms $\beta_i$ imply that the natural morphisms
$$
U_{i+1}/U_i\longrightarrow V_{i+1}/V_i
$$
are surjective.
 Since $U_{r_-}=V_{r_-}=0$, induction on $i$ gives $U_i=V_i$ for $r_-\leq i\leq r_+$. Hence
$$
V=V_{r_+}=U_{r_+}\subset\kk(\mU),
$$
so $\kk(\mU)=V$. Applying this to the epimorphic projections of a strict presentation of $\mV$, we obtain $\mU=\mV$.
\end{proof}
\begin{lemma}\label{admi}
Let $\mV=\varprojlim_{l\geq0}V_l$, where $V_l\in\operatorname{Ob}(\mcC)$ and the transition morphisms are epimorphisms, and let $\kk_l:\mV\to V_l$ be the canonical projections. Let $\mcF=\{\mW_n\}_{n\in\ZZ}$ be an increasing filtration of $\mV$ such that, for every $l\geq0$, there exist integers $r_-(l)$ and $r_+(l)$ satisfying
$$
\kk_l(\mW_{r_-(l)})=0,\qquad \kk_l(\mW_{r_+(l)})=V_l.
$$
Then
\begin{enumerate}
\item The filtration $\mcF$ is admissible.
\item If $b:\mV'\to\mV$ is a morphism such that $\kk_l\circ b:\mV'\to V_l$ is an epimorphism for every $l\geq0$, then $b$ is an epimorphism.
\end{enumerate}
\end{lemma}
\begin{proof}
(1) Let $\kk:\mV\to V$ be an epimorphism with $V\in\operatorname{Ob}(\mcC)$. The morphism $\kk$ factors through some $\kk_l$, say $\kk=\tau\circ\kk_l$ for a morphism $\tau:V_l\to V$. Since $\kk$ is an epimorphism, so is $\tau$. Hence
$$
\kk(\mW_{r_-(l)})=0,\qquad \kk(\mW_{r_+(l)})=V,
$$
and $\mcF$ is admissible.

(2) Suppose that $f,g:\mV\to\mY$ satisfy $f\circ b=g\circ b$. Choose a presentation $\mY=\varprojlim_jY_j$ and let $p_j:\mY\to Y_j$ be the canonical projections. For every $j$, the morphism $p_j\circ(f-g):\mV\to Y_j$ factors through some $\kk_l$, say
$$
p_j\circ(f-g)=h\circ\kk_l.
$$
Since $(p_j\circ(f-g))\circ b=0$ and $\kk_l\circ b$ is an epimorphism, we have $h=0$. Thus $p_j\circ f=p_j\circ g$ for every $j$, and hence $f=g$. Therefore $b$ is an epimorphism.
\end{proof}

\subsection{Graded objects}
\begin{definition}
Let $\mcF=\{\mW_i\}_{i\in\ZZ}$ be a filtration of $\mV$, let $\GG$ be an abstract group, let
$$
\nu:\GG\to\ZZ,\qquad \Gg\mapsto[\Gg],
$$
be a homomorphism, and let $\pi:\GG\to\operatorname{Aut}(\mV)$ be a representation satisfying
$$
\pi(\Gg)(\mW_i)=\mW_{i+[\Gg]}.
$$
Then $\GG$ acts on
$$
\overline{\mV}_{\mcF}=\prod_{i\in\ZZ}\mW_{i+1}/\mW_i.
$$
We denote the resulting representation by $\overline\pi$.
\end{definition}
\begin{lemma}\label{bar}
Assume that $\mcF$ is admissible and that every nonzero $\GG$-invariant subobject $\mU\subset\mV$ satisfies
$$
\mU\cap\mW_i\neq0
$$
for some $i\in\ZZ$. If $\overline\pi$ is irreducible, then $\pi$ is irreducible.
\end{lemma}
\begin{proof}
Let $\mU\subset\mV$ be a nonzero $\GG$-invariant subobject and put
$$
\overline{\mU}_{\mcF}:=\prod_{i\in\ZZ}\frac{\mW_{i+1}\cap\mU}{\mW_i\cap\mU}.
$$
This is a $\GG$-invariant subobject of $\overline{\mV}_{\mcF}$. It is nonzero: otherwise all the subobjects $\mW_i\cap\mU$ would be equal, and Claim \ref{separated} would imply that they are all zero, contrary to the hypothesis on $\mU$. Since $\overline\pi$ is irreducible, we have
$$
\overline{\mU}_{\mcF}=\overline{\mV}_{\mcF}.
$$
Thus $\mU$ is dense with respect to $\mcF$, and Lemma \ref{dense} gives $\mU=\mV$.
\end{proof}

\sms

\subsection{Sets, profinite spaces, and loop-group parameters}
For the parameter spaces of loop groups we need an additional category.
Starting from $\Sets_0$, put
\begin{equation}\label{s8-set-categories}
 \begin{gathered}
 \Prof=\Pro(\Sets_0),\qquad \bSet=\Ind(\Prof),\\
 \pSet=\Ind_{\aleph_0}\bigl(\Pro_{\aleph_0}(\bSet)\bigr).
 \end{gathered}
\end{equation}
The notation $\pSet$ is distinct from $\mSet=\Pro_{\aleph_0}(\Sets)$.
The category $\Prof$ is equivalent to compact Hausdorff totally
disconnected spaces with continuous maps. The full subcategory
$\Lpf\subset\bSet$ consisting of filtered unions of compact open
subspaces is the category of locally profinite spaces.

An ordinary set $S$ embeds fully faithfully into $\bSet$, and then into
$\pSet$, by
\[
 S\longmapsto\varinjlim_{\substack{S_0\subset S\\S_0\text{ finite}}}S_0.
\]
It is the discrete locally profinite space attached to $S$.
The first $\Ind$ in \eqref{s8-set-categories} is unrestricted, so this
construction includes all sets. Likewise, an arbitrary locally
profinite space is embedded using all its compact open subspaces.
These categories and embeddings keep track of continuous maps:
forgetting a topology does not make every set-theoretic map continuous.
Nor do we identify an arbitrary formal ind-pro-object with its set
of points.

For $X\in\pSet$ define
\[
 \widetilde X:=\Hom_{\pSet}(*,X)\in\Sets.
\]
This functor preserves finite products. If
$X=\varinjlim_a\varprojlim_bX_{a,b}$, its set of points is
$\varinjlim_a\varprojlim_b\widetilde X_{a,b}$, with ordinary set
limits on the right. The two outer countable presentations in
\eqref{s8-set-categories} suffice for the coefficient and congruence
constructions over $F((t))$. They define full subcategories of the
corresponding unrestricted categories of \cite{GK,GK1}.

\subsection{Pseudo-actions and smooth representations}
We spell out the action-map spaces in order to distinguish a smooth
group-object action from an action of the underlying abstract group.

\begin{definition}\label{a}
For $X\in\Lpf$ and $V,W\in\Vect$, let $\mcA(X;V,W)$ be the
vector space of continuous maps $X\times V\to W$ linear in $V$,
where $V$ and $W$ have the discrete topology. There is no support
condition.
For general $X=\varinjlim_a P_a\in\bSet$, with
$P_a=\varprojlim_bS_{a,b}\in\Prof$, the same extension is specified by
\begin{equation}\label{s8-base-action}
 \mcA(X;V,W)=
 \varprojlim_a\varprojlim_{\substack{E\subset V\\\dim E<\infty}}
 \varinjlim_b\Hom_{\CC}\bigl(\CC[S_{a,b}]\otimes E,W\bigr).
\end{equation}
Here $\CC[S]$ is the vector space with basis $S$.
For $X=\varprojlim_bX_b\in\Pro_{\aleph_0}(\bSet)$ and
$\mV=\varprojlim_iV_i$, $\mW=\varprojlim_kW_k$ in
$\Pro(\Vect)$, put
\begin{equation}\label{s8-pro-action}
 \mcA(X;\mV,\mW)
   =\varprojlim_k\varinjlim_{i,b}\mcA(X_b;V_i,W_k).
\end{equation}
Finally, for $X=\varinjlim_aX_a\in\pSet$, put
\[
 \mcA(X;\mV,\mW)=\varprojlim_a\mcA(X_a;\mV,\mW).
\]
\end{definition}

Here the coefficient indices $i,k$ are arbitrary small cofiltered
indices; for objects of $\mVect$ they can be taken countable.
These presentation-independent constructions are the
pseudo-action of \cite[Sections~1.7--1.9]{GK}. There are associative
composition maps
\[
 \mcA(X;\mW,\mU)\times\mcA(Y;\mV,\mW)
 \longrightarrow\mcA(X\times Y;\mV,\mU),
\]
and $\mcA(*;\mV,\mW)=\Hom_{\Pro(\Vect)}(\mV,\mW)$.
Writing $X\times\mV\to\mW$ means an element of this action-map
space. It makes no claim that a representing tensor object lies in
$\mVect$. In \eqref{s8-base-action}, the finite quotient of $P_a$
may depend on the finite-dimensional subspace $E$. In
\eqref{s8-pro-action}, for each output level and bounded parameter
piece there must be a single sufficiently deep input level and
parameter level.

\begin{definition}\label{s8-representations}
For a group object $H\in\pSet$, a smooth representation on a
pro-vector space is a pair $(\mV,\rho)$ with
$\mV\in\Pro(\Vect)$ and $\rho\in\mcA(H;\mV,\mV)$ such that
\[
 \rho\circ(m_H\times1)=\rho\circ(1\times\rho),\qquad
 \rho\circ(e_H\times1)=1_{\mV}.
\]
Morphisms intertwine these action maps. We write
\[
 \begin{gathered}
 \Repall(H):=\Rep(H,\Pro(\Vect)),\\
 \Rep(H):=\Rep(H,\mVect)\subset\Repall(H).
 \end{gathered}
\]
The second category is full. We write $\Rep(H,\Vect)$ for the
full subcategory with ordinary underlying vector spaces.
In either category a nonzero representation is irreducible if it
has no nonzero proper subobject in that category. Unless specified
otherwise, representations below belong to $\Rep(H)$.
\end{definition}

\begin{proposition}\label{action}
Both $\Repall(H)$ and $\Rep(H)$ are abelian, and the inclusion
$\Rep(H)\hookrightarrow\Repall(H)$ is exact. Their forgetful
functors to $\Pro(\Vect)$ and $\mVect$, respectively, are exact.
Countable products and inverse limits in $\Rep(H)$ have their underlying
objects computed in $\mVect$. Evaluation at points gives actions of
$\widetilde H$ on $\mV$ and on $\widetilde{\mV}$.
\end{proposition}
\begin{proof}
In the full pro-category the assertion about kernels and cokernels is
\cite[Lemma~2.3]{GK}. Their underlying pro-vector spaces remain
countable by Proposition~\ref{s8-countable-closure}, so the assertion
restricts to $\Rep(H)$. For a product, every ordinary output level
uses finitely many factors; the actions on those factors supply its
action map. Kernels then give inverse limits. Pulling the action back
along $h:*\to H$ gives the point action, and the two axioms give
the group law.
\end{proof}

\begin{proposition}[Detection of equivariance]\label{s8-point-comparison}
Suppose $H=\varinjlim_aX_a$ and $X_a=\varprojlim_bX_{a,b}$, and
suppose every map
$\widetilde X_a\to\widetilde X_{a,b}$ is surjective.
Let $\mV_1,\mV_2\in\Repall(H)$, with $\mV_1$ strict.
A morphism $f:\mV_1\to\mV_2$ intertwining the actions of every
$h\in\widetilde H$ on these pro-objects is a morphism in $\Repall(H)$.
\end{proposition}
\begin{proof}
Fix a parameter piece and an ordinary output level $W$ of $\mV_2$.
On sufficiently deep levels the difference of the two action maps is
represented by
$D:X_{a,b}\times V_{1,i}\to W$.
The projection $\mV_1\to V_{1,i}$ in a strict presentation is an
epimorphism. Pointwise intertwining says that, for every point of
$X_a$, the associated linear map $V_{1,i}\to W$ becomes zero after
composition with this epimorphism; hence it is zero. Point lifting
therefore makes $D$ zero at every point of $X_{a,b}$.
On a profinite piece, continuous maps to a discrete vector space are
determined by their values at points. Applying this to each vector
and each piece shows $D=0$. The output levels and parameter pieces
were arbitrary.
\end{proof}

This is the strict-source argument of \cite[Lemma~0.6]{GK1} with
point lifting stated explicitly. For our parameter towers, surjective
transition maps and countable indexing give this lifting by successive
choices. Countability of the vector-space tower does not replace
strictness in this proposition. The proposition tests a morphism
between existing representations; it does not construct a
pseudo-action from pointwise formulas.

\subsection{Bounded groups and countable ordinary presentations}
A morphism $X\to Y$ in $\bSet$ is \emph{weakly surjective} if
\[
 \Hom_{\bSet}(Y,S)\longrightarrow\Hom_{\bSet}(X,S)
\]
is injective for every $S\in\Sets$, embedded as a discrete object.
For locally profinite spaces this means density of the image.
A pro-object in $\Pro_{\aleph_0}(\bSet)$ is \emph{weakly strict}
if it admits a presentation by weakly surjective maps. We call a
group object lying in this pro-category \emph{bounded}.

\begin{proposition}\label{weak-strict-GK}\label{2.5}
\label{countable-ordinary-presentation}
For a weakly strict bounded group object $H$, there are equivalences
\[
 \Pro\bigl(\Rep(H,\Vect)\bigr)\simeq\Repall(H)
\]
and
\begin{equation}\label{s8-countable-rep-equivalence}
 \Pro_{\aleph_0}\bigl(\Rep(H,\Vect)\bigr)
       \xrightarrow{\ \sim\ }\Rep(H,\mVect).
\end{equation}
In particular, every object on the right admits a countable tower
of ordinary smooth $H$-representations.
\end{proposition}
\begin{proof}
The unrestricted equivalence is \cite[Proposition~2.5]{GK}.
It remains to check the countable assertion. Let $(Q_\alpha)$ be
an ordinary representation presentation given by that equivalence,
and let $(V_n)$ be a countable presentation of its underlying
pro-vector space. The inverse pro-isomorphisms yield, by successive
refinement, an interleaved system
\[
 V_{i_0}\longleftarrow Q_{\alpha_0}
 \longleftarrow V_{i_1}\longleftarrow Q_{\alpha_1}
 \longleftarrow\cdots,
\]
where actions are forgotten on the displayed arrows involving $V_i$.
Choose $i_n\geq n$ and choose the $\alpha_n$ increasingly.
At each step refine to make each composite of consecutive arrows
the required transition map. Only finitely many equalities are
imposed at any step. The two subsequences are then isomorphic
pro-vector spaces.

The natural map $(Q_\alpha)\to(Q_{\alpha_n})$ is consequently an
isomorphism after forgetting actions. The pro-extension of the exact
faithful ordinary forgetful functor reflects zero objects: a tower
or general pro-system is zero exactly when each target eventually
receives the zero map. Applying this to kernels and cokernels shows
that it reflects isomorphisms. We have therefore obtained the
required countable presentation as representations. Full faithfulness
follows from the unrestricted equivalence.
\end{proof}

For a locally profinite group, $\Rep(H,\Vect)$ is the usual smooth
representation category. If $H=\varprojlim_rH_r$ is presented by
surjective group quotients, an ordinary representation in our sense
factors through some $H_r$: apply \eqref{s8-pro-action} to its action
map with an ordinary source and target. This requires a finite
congruence level for the whole ordinary representation, not merely
a level depending on each of its vectors.

\begin{lemma}\label{inf}
An irreducible object of $\Rep(P(O))$ is the pullback of a
representation in $\mVect$ of
$P(O)_m=P(O)/\GG_{m+1}$ for some $m\geq0$, where
\[
 \GG_r=(1+t^rO)\ltimes t^rO\qquad(r\geq1).
\]
If it admits a strict tower of ordinary $P(O)$-representations,
it is itself an ordinary smooth irreducible representation of
such a quotient.
\end{lemma}
\begin{proof}
By Proposition~\ref{weak-strict-GK}, a nonzero representation has a
nonzero canonical map to an ordinary $P(O)$-representation $Q$.
Irreducibility makes this map a monomorphism. Some $\GG_r$ acts
trivially on $Q$, and hence also on its subrepresentation.
This gives the first assertion. For a strict representation tower,
the nonzero canonical map can be chosen to be an epimorphism as
well; it is then an isomorphism. This gives the second assertion.
\end{proof}

\subsection{Coinvariants and auxiliary adjunctions}
For a weakly strict bounded group $H$, the invariants and coinvariants
of $\Pi\in\Rep(H)$ are defined by the adjunctions with the trivial
action functor. A countable ordinary presentation computes them as
\begin{equation}\label{s8-levelwise-coinvariants}
 \left(\varprojlim_n\Pi_n\right)^H
       \simeq\varprojlim_n(\Pi_n)^H,\qquad
 \left(\varprojlim_n\Pi_n\right)_H
       \simeq\varprojlim_n(\Pi_n)_H.
\end{equation}
These are formal limits in $\mVect$; see
\cite[Corollary~2.6]{GK} and
Proposition~\ref{countable-ordinary-presentation}.
The quotient map $\Pi\to\Pi_H$ is an epimorphism.

\begin{claim}\label{coin}
If $\kappa$ is a smooth character of $H$, put
$\Pi_{H,\kappa}=(\Pi\otimes\kappa^{-1})_H$.
For $W\in\mVect$ with trivial action there is a natural isomorphism
\[
 \Hom_{\mVect}(\Pi_{H,\kappa},W)
 \simeq\Hom_{\Rep(H)}(\Pi,W\otimes\kappa).
\]
\end{claim}
The equivariance here is that of action maps.
Formula \eqref{s8-levelwise-coinvariants} is a specific consequence
of the ordinary representation presentation; it is not a general
rule allowing a left adjoint to commute with actual inverse limits.
To spell out its universal property, write
$\mE=\varprojlim_i E_i$ with $E_i\in\Rep(H,\Vect)$ and
$\mT=\varprojlim_a T_a\in\mVect$. The pro-Hom formula and ordinary
coinvariant adjunction give natural isomorphisms
\begin{equation}\label{s8-coinvariants-universal}
\begin{aligned}
 \Hom_{\mVect}\bigl(\varprojlim_i(E_i)_H,\mT\bigr)
 &\simeq\varprojlim_a\varinjlim_i
       \Hom_{\Vect}\bigl((E_i)_H,T_a\bigr)\\
 &\simeq\varprojlim_a\varinjlim_i
       \Hom_{\Rep(H,\Vect)}(E_i,\operatorname{triv}T_a)\\
 &\simeq\Hom_{\Rep(H)}(\mE,\operatorname{triv}\mT).
\end{aligned}
\end{equation}
The last isomorphism uses
Proposition~\ref{countable-ordinary-presentation}. Thus the pro-object
$\{(E_i)_H\}_i$ represents coinvariants, with their transition maps
induced from those of $\{E_i\}_i$. This argument does not require
surjectivity of the transition maps, finite-dimensionality of the
$E_i$, or exactness of an inverse-limit functor on ordinary vector
spaces. Exactness of coinvariants, when needed for a finite
filtration at a fixed index $i$, is a separate assertion.

\begin{definition}\label{quasi-pro-unipotent}
A group object of $\bSet$ is \emph{quasi-unipotent} if it is a
filtered direct limit of profinite group objects, with transition
homomorphisms. A bounded group is \emph{quasi-pro-unipotent} if it
has a countable presentation by quasi-unipotent group objects with
weakly surjective transition homomorphisms.
\end{definition}

\begin{theorem}\label{coinvariants-exact}
For a quasi-pro-unipotent group $H$, the functor
$\Coinv_H:\Rep(H)\to\mVect$ is exact.
\end{theorem}
\begin{proof}
On an ordinary smooth representation, averaging over a compact group
makes its coinvariant functor exact. Filtered direct limits of vector
spaces are exact, so the assertion holds for quasi-unipotent groups.
The countable presentation in
Proposition~\ref{countable-ordinary-presentation} and the levelwise
description \eqref{s8-levelwise-coinvariants} give the assertion for
quasi-pro-unipotent groups. This is the argument of
\cite[Lemma~2.7]{GK}.
\end{proof}

\begin{lemma}\label{tnO-quasi-pro-unipotent}
For every $j\in\ZZ$, the additive group $t^jO$ is
quasi-pro-unipotent. The same holds for $O^*(r)$, $r\geq1$.
Their coinvariants are exact functors with values in $\mVect$.
\end{lemma}
\begin{proof}
For $N>j$, the quotient $t^jO/t^NO$ is $F^{N-j}$, an increasing
union of the compact groups $\varpi^{-s}O_F^{N-j}$, where
$O_F$ is the ring of integers and $\varpi$ a uniformizer of $F$.
For $O^*(r)$ use the quotients
$(1+t^rO)/(1+t^NO)$. In their polynomial coordinates
$1+\sum_{i=r}^{N-1}a_it^i$, the conditions
$a_i\in\varpi^{-si}O_F$, $s\geq0$, define compact open subgroups
whose union is the entire quotient. Multiplication and inversion
preserve these inequalities. All congruence transition maps are
surjective.
\end{proof}

\begin{remark}[The full pro-category in general existence statements]
\label{inflation}
General coinvariants for an arbitrary group object are constructed
in the full category $\Pro(\Vect)$
by \cite[Proposition~2.10]{GK}. For quasi-pro-unipotent $H$,
\cite[Proposition~2.7]{GK1} also constructs
\[
 \Inf_H\dashv\Coinv_H\dashv\triv_H
\]
in $\Repall(H)$ and $\Pro(\Vect)$. This $\Inf_H$ differs from pullback along
a quotient homomorphism, which is ordinary inflation. A general
existence theorem in $\Pro(\Vect)$ does not assert that its output has
a countable presentation. We use such a functor inside $\mVect$
only when its output is known to belong to $\mVect$.
The bounded-group invariants and coinvariants above do satisfy this
condition, by their explicit countable presentations.
\end{remark}

For induction we use the convention of \cite[Section~3]{GK}.
Its compact-support construction is denoted by $i_L^H$ here and by
$\ind_L^H$ in the explicit models below.
\begin{claim}\label{indH}
In the setting of \cite[Section~3]{GK}, for a thick subgroup $L$
of an algebraic loop group $H$, if $H/L$ is ind-compact then
\[
 \Hom_H(\Pi,i_L^H V)\simeq\Hom_L(\Pi|_L,V).
\]
The source result is an adjunction in the full pro-categories.
It restricts to the countable categories on any arguments for which
the displayed induced object is countably presented.
\end{claim}
\begin{proof}
The adjunction is \cite[Proposition~3.5]{GK}; full faithfulness of
the countable subcategories gives its stated restriction.
\end{proof}
Here a thick subgroup is an inverse image of an algebraic subgroup
at a finite congruence level, in the setting of that reference.
For other quotients the right adjoint is denoted $\widetilde i$ in
\cite{GK}; it need not equal $i$. The functor called $\operatorname{Coind}$
in \cite[Corollary~2.4]{GK1} is a left adjoint to restriction.
The explicit induction in this section will be described by
countable towers, so its coefficient objects lie in $\mVect$.

\subsection{\texorpdfstring{The group objects attached to $K$, $P(K)$, and $G(K)$}{The group objects attached to K, P(K), and G(K)}}
Recall that $O=F[[t]]$ and $K=F((t))$. We write $v_t(t)=1$
and $O^*(r)=1+t^rO$ for $r\geq1$. Set
\[
 P(K)=K^*\ltimes K,\qquad P(O)=O^*\ltimes O,\qquad
 P_0=O^*\ltimes K,
\]
with multiplication $(a,b)(a',b')=(aa',b+ab')$.
Put $\widehat a=(a,0)$ and $\widetilde b=(1,b)$; thus
$\widehat a\widetilde b\widehat a^{-1}=\widetilde{ab}$.
We use the fixed additive characters
\[
 \widetilde\psi:F\to\CC^*,\qquad
 \psi(f)=\widetilde\psi\bigl(\Res_{t=0}(f\,dt)\bigr).
\]
The following presentations specify the group objects underlying the
notation used so far:
\[
 \mathbb O=\varprojlim_{r\geq1}(O/t^rO),\qquad
 \mathbb O^*=\varprojlim_{r\geq1}(O/t^rO)^*,\qquad
 \mK=\varinjlim_{n\geq0}\varprojlim_{r\geq1}(t^{-n}O/t^rO).
\]
Each finite quotient is an algebraic $F$-variety, considered through
its locally profinite space of $F$-points. The transition maps are
surjective. Multiplication by $t^d$ and the finite unions of the
pieces $t^d\mathbb O^*$ similarly define $\mK^*$.
We set $\mP=\mK^*\ltimes\mK$ and
$\mathbb P_0=\mathbb O^*\ltimes\mK$.
Their point groups are $P(K)$ and $P_0$, respectively.

\begin{definition}\label{P_0}
The valuation homomorphism
$\widetilde v:P(K)\to\ZZ$ is $(a,b)\mapsto v_t(a)$, so
$\ker(\widetilde v)=P_0$. For $n\geq0$ set
\[
 P^n(K)=\{(a,b): -n\leq v_t(a)\leq n,\quad v_t(b)\geq-n\}.
\]
For $r\geq1$, set $P^n_r=P^n(K)/\GG_r$, with right cosets and
$\GG_r=(1+t^rO)\ltimes t^rO$.
\end{definition}
The spaces $P^n_r$ are $F$-points of finite-dimensional smooth
algebraic varieties. They give another presentation
\[
 \mP=\varinjlim_n\varprojlim_r P^n_r.
\]
Both descriptions have the same coefficient and congruence maps.
For fixed bounded pieces, multiplication and inversion factor through
sufficiently deep finite quotients, which defines the group-object
structure. All the inverse towers here are countable and surjective
on points.

\subsection{\texorpdfstring{Representations of $G(K)$ and comparison of categories}{Representations of G(K) and comparison of categories}}
For this subsection let $G$ be any linear algebraic group over $F$.
Choose a closed embedding $G\hookrightarrow\GL(N)$. Bounds on the
poles of the matrix entries and of the inverse matrix give a countable
exhaustion of $G(K)$ by bounded pieces. Their finite congruence
quotients give the usual presentation
\[
 \mG=\varinjlim_n\varprojlim_r X^n_r\in\pSet,
 \qquad \widetilde{\mG}=G(K).
\]
It is independent of the embedding, and multiplication and inversion
make it a group object. This is the construction of
\cite[Section~2.11]{GK}. We write $\Rep(G(K))$ and
$\Repall(G(K))$ for the two categories in
Definition~\ref{s8-representations} applied to this group object.

\begin{lemma}\label{countable-cogenerators}
Let $G$ be any linear algebraic group. Every morphism from the
underlying pro-vector space of $\Pi\in\Repall(G(K))$ to an
ordinary vector space factors through an equivariant morphism
$\Pi\to J$ with $J\in\Rep(G(K))$. In particular, every nonzero
object of $\Repall(G(K))$ admits a nonzero morphism to an object
of $\Rep(G(K))$.

If $G$ is reductive, every nonzero object of
$\Repall(G(K))$ admits a nonzero morphism to an object of
$\Rep(G(K))$ with strict underlying pro-vector space.
\end{lemma}
\begin{proof}
Write the underlying pro-vector space as $\prl_{i\in I}V_i$,
with $I$ a directed poset, and let $f:\Pi\to W$ be a morphism
with $W$ ordinary. Choose $i_0\in I$ through which $f$ factors.
We construct a countable directed subset $I_0\subset I$ containing
$i_0$ on which the action can still be defined. The subset need
not be cofinal in $I$.

For every selected output level $i$ and every bounded parameter
piece $X^n=\prl_r X^n_r$, the action has a representative
\[
 \rho_{n,i}:X^n_r\times V_j\longrightarrow V_i
\]
for some $r$ and some $j\in I$. Add such a $j$ to the selected
indices. Also add indices witnessing compatibility of these
representatives with the coefficient transition maps and the
inclusions of parameter pieces, and indices witnessing the unit
and associativity identities. For associativity it suffices to
consider each selected output level and each pair of bounded
parameter pieces: their product maps into another bounded piece.
Each required identity holds in a filtered colimit of the
action-map spaces in \eqref{s8-pro-action}, and therefore holds
after a single further coefficient and parameter refinement.
Thus only countably many coefficient indices are required at
each stage. Add upper bounds for all finite subsets of the
selected indices and repeat this procedure countably many times.
The union is the required countable directed set $I_0$.

The chosen representatives and compatibility identities define
an action on $J=\prl_{i\in I_0}V_i$. The unit and associativity
identities imposed in the construction make this a representation,
and the natural map $\Pi\to J$ is equivariant. The original map
$f$ factors as $\Pi\to J\to V_{i_0}\to W$. A nonzero pro-object
has a nonzero projection to an ordinary output level; applying
the construction to such a projection proves the second assertion.
This argument uses the countable exhaustion by bounded parameter
pieces and does not require reductivity.

For the last assertion, assume that $G$ is reductive.
Put $H=G(O)$. Restriction to $H$, followed by the unrestricted
equivalence in Proposition~\ref{weak-strict-GK}, gives a formal
presentation by ordinary smooth $H$-representations. At least one
canonical map $\Pi|_H\to Q$ is nonzero, and $Q$ factors through
$H/H(r)$ for some congruence subgroup $H(r)$.

Since $G$ is reductive, its affine Grassmannian is ind-proper.
Choose a countable exhaustion by closed proper finite-dimensional
pieces $Y_n$, and let $X_n$ be their inverse images in the congruence-level affine Grassmannian whose $F$-points are $G(K)/H(r)$.
Write $J_n(Q)$ for the space of locally constant functions
$f:X_n(F)\to Q$ such that
\[
 f(xh)=h^{-1}f(x),\qquad h\in H/H(r).
\]
The quotient of their support by $H/H(r)$ is compact, since
$Y_n(F)$ is compact. The induction of
\cite[Section~3]{GK}, recalled in Section~\ref{ind-gen}, is
\[
 J(Q):=\ind_H^{G(K)}Q=\prl_n J_n(Q).
\]
It is right adjoint to restriction by
\cite[Proposition~3.5]{GK}. Thus the map $\Pi|_H\to Q$ gives a
nonzero morphism $\Pi\to J(Q)$.

The restriction maps $J_{n+1}(Q)\to J_n(Q)$ are surjective.
Indeed, the torsor over $Y_{n+1}(F)$ has local sections, and the
corresponding associated bundle has fiber the discrete space $Q$.
A section on the closed compact subspace $Y_n(F)$ extends locally
in these trivializations. A finite partition by compact open
subsets makes the local extensions compatible, and extension by
zero off their union gives a section on $Y_{n+1}(F)$.
Consequently $J(Q)$ is strict and countably presented.
\end{proof}

\begin{lemma}\label{irreducibility-category-comparison}
Let $G$ be any linear algebraic group. For $\Pi\in\Repall(G(K))$
consider the following conditions:
\begin{enumerate}
\item $\Pi$ belongs to $\Rep(G(K))$ and is irreducible there;
\item $\Pi$ is irreducible in $\Repall(G(K))$ and its underlying
pro-vector space belongs to $\mVect$.
\item $\Pi$ is irreducible in $\Repall(G(K))$ and its underlying
pro-vector space is strict.
\end{enumerate}
Then (3)$\Rightarrow$(2)$\Leftrightarrow$(1).
\end{lemma}
\begin{proof}
The implication (2)$\Rightarrow$(1) is immediate. Assume (1), and
suppose that $0\ne\Pi'\subsetneq\Pi$ in $\Repall(G(K))$.
The quotient $Q=\Pi/\Pi'$ is nonzero. By
Lemma~\ref{countable-cogenerators}, it has a nonzero morphism to
some $J\in\Rep(G(K))$. The composite $\Pi\to J$ is nonzero.
Its kernel belongs to $\Rep(G(K))$, since that category is
abelian and its inclusion in the full category is exact.
This kernel contains $\Pi'$, so it is nonzero, and the composite
is nonzero, so it is proper. This contradicts (1).

Finally, suppose that $\Pi$ is irreducible in the full category
and has strict underlying pro-vector space. A nonzero morphism
$\Pi\to J$ supplied by Lemma~\ref{countable-cogenerators} is a
monomorphism. Lemma~\ref{strict-subobject-countable} shows that
the underlying object of $\Pi$ is countably presented. Thus (2)
holds.
\end{proof}
\begin{remark}
We actually believe that (3) is equivalent to both (1) and (2) but we don't know how to prove it. 
The above Lemma together with Lemma \ref{countable-simple-vector-spaces} sort of explais why we prefer to work only with countable limits -- only in that category we have a chance of having some manageable notion of irreducible object (otherwise already for the trivial group we have too many of those).
\end{remark}
\begin{lemma}\label{congruence-completion-strictness}
Let $G$ be a linear algebraic group, put $H=G(O)$, and let
$H(r)=\ker(H\to G(O/t^rO))$ for $r\geq1$.
For $\Pi\in\Repall(G(K))$ the natural map
\[
 \Pi|_H\longrightarrow\prl_{r\geq1}\Pi_{H(r)}
\]
is an isomorphism in $\Repall(H)$. If all the coinvariants
$\Pi_{H(r)}$ are ordinary vector spaces, then the underlying
pro-vector space of $\Pi$ belongs to $\mVect^{\mathrm{str}}$.
In particular, this holds for the admissible representations of
Definition~\ref{md:representation-data}. Conditions (1)--(3) of
Lemma~\ref{irreducibility-category-comparison}
are therefore equivalent on the full subcategory with ordinary
congruence coinvariants.
\end{lemma}
\begin{proof}
By Proposition~\ref{weak-strict-GK}, write
$\Pi|_H=\prl_i Q_i$ with $Q_i$ ordinary. Every $Q_i$ factors
through $H/H(r_i)$ for some $r_i$.
Consequently every morphism from the underlying pro-vector space
of $\Pi$ to an ordinary vector space $W$ factors through
$\Pi_{H(r)}$ for some $r$: first factor it through a $Q_i$, and
then use the $H$-equivariant map $\Pi\to Q_i$.
Since $\Pi\to\Pi_{H(r)}$ is an epimorphism, these factorizations
give
\[
 \varinjlim_r\Hom_{\Pro(\Vect)}(\Pi_{H(r)},W)
   \simeq\Hom_{\Pro(\Vect)}(\Pi,W).
\]
The left side is the space of morphisms from the displayed formal
inverse limit to $W$. Morphisms to ordinary vector spaces detect
isomorphisms of pro-vector spaces, so the natural map is an
isomorphism after forgetting the action, and hence also in
$\Repall(H)$.

The transition maps $\Pi_{H(r+1)}\to\Pi_{H(r)}$ are epimorphisms,
by transitivity of coinvariants. When these objects are ordinary,
they are surjective linear maps. The resulting countable tower
is a strict presentation of $\Pi$. The last assertion follows
from Lemma~\ref{irreducibility-category-comparison}.
\end{proof}

\begin{remark}\label{strictness-comparison-question}
To prove (2)$\Rightarrow$(3) in
Lemma~\ref{irreducibility-category-comparison}, it remains to show
that an irreducible countably presented representation has strict
underlying pro-vector space without the hypothesis on its
congruence coinvariants in Lemma~\ref{congruence-completion-strictness}.
If $G$ is reductive, the proof of the last assertion of
Lemma~\ref{countable-cogenerators} gives a monomorphism
\[
 \Pi\hookrightarrow\ind_{G(O)}^{G(K)}Q
\]
with $Q$ ordinary and the target strict and countably presented.
This does not by itself settle the assertion: a countably presented
subobject of a strict pro-vector space need not be strict.
For example, the nonzero pro-object
\[
 \prl_n\bigoplus_{k\geq n}\CC e_k
 \hookrightarrow\bigoplus_{k\geq0}\CC e_k
\]
of Example~\ref{s8-limit-example} is a subobject of an ordinary
vector space and has zero actual inverse limit. It is therefore
not strict. This example illustrates the failure of strictness
to pass to subobjects; irreducibility remains an additional issue.
We do not establish the remaining implication here.
\end{remark}

\subsection{\texorpdfstring{The explicit presentation for $\PGL(2)$}{The explicit presentation for PGL(2)}}
We now return to $G=\PGL(2)$.
\begin{definition}\label{G}
Let $G=\PGL(2)$ and $G(O)(r)=\ker(G(O)\to G(O/t^rO))$ for
$r\geq1$. For $n\geq0$ let $\widehat X^n\subset\GL(2,K)$ consist
of invertible matrices $A=(a_{ij})$ satisfying
\[
 2\min_{i,j}v_t(a_{ij})-v_t(\det A)\geq-n.
\]
Let $X^n$ be its image in $G(K)$ and set
$X^n_m=X^n/G(O)(m+1)$ for $m\geq0$. Put
\[
 \mG=\varinjlim_n\varprojlim_m X^n_m\in\pSet.
\]
\end{definition}
The bound is unchanged by a scalar representative of a projective
matrix. The quotients $X^n_m$ are finite-dimensional algebraic
$F$-varieties, and the congruence maps are smooth and surjective
on $F$-points. Multiplication and inversion give a group object with
$\widetilde{\mG}=G(K)$. These constructions are the loop-group
presentations of \cite[Section~2.11]{GK} in the notation $K=F((t))$.
From now on, $P(K)$, $P_0$, $G(K)$, and their congruence subgroups
also denote these group objects when used as arguments of $\Rep$.

\subsection{\texorpdfstring{Finite-level representations and the representation $\mcM$}{Finite-level representations and the representation M}}
For $n\geq0$, write
\[
 \begin{gathered}
 R_n=O/t^{n+1}O,\qquad A_n=O^*(n+1),\qquad Z_n=O^*/A_n=R_n^*,\\
 B_n=t^{-n}O/tO,\qquad P_0^n=Z_n\ltimes B_n,\qquad
 H_n=O^*\ltimes t^{-n}O.
 \end{gathered}
\]
The action in the semidirect product is multiplication modulo $tO$.
The pairing
\begin{equation}\label{s8-perfect-pairing}
 B_n\times R_n\longrightarrow\CC^*,\qquad
 (b,x)\longmapsto\lambda_x(b):=\psi(bx/t)
\end{equation}
identifies $R_n$ with the smooth character group of the additive
locally profinite group $B_n$. In coordinates, it is the
nondegenerate coefficient pairing
$\widetilde\psi(\sum_{i=0}^n b_{-i}x_i)$.

Let $\mcM_n=\Meas(Z_n)$, the vector space of compactly supported
locally constant complex measures on $Z_n$. After a Haar measure
is chosen, this is the usual space of compactly supported locally
constant functions. Its description as measures supplies the
pushforward maps without such choices. Define
\begin{equation}\label{s8-finite-actions}
 A_a\mu=(z\mapsto a^{-1}z)_*\mu,\qquad
 D_b\mu=\lambda_z(b)\mu.
\end{equation}
These operators are smooth, and
$A_aD_bA_a^{-1}=D_{ab}$.

\begin{lemma}[The finite-level orbit]\label{s8-finite-mackey}
The category of ordinary smooth $P_0^n$-representations $Q$ such
that $Q_{B_0}=0$, where $B_0=O/tO\subset B_n$, is equivalent to
$\Vect$. The equivalence is
\[
 E\longmapsto E\otimes\mcM_n,
\]
with trivial action on $E$. In particular, $\mcM_n$ is irreducible.
The same equivalence for countable pro-objects is
$\mL\mapsto\mL\widehat\otimes\mcM_n$, with
$\mL\in\mVect$.
\end{lemma}
\begin{proof}
Under \eqref{s8-perfect-pairing}, a character of $B_n$ is nontrivial
on $B_0$ precisely when $x_0\ne0$. Thus these characters form
exactly $Z_n\subset R_n$, a single orbit with trivial stabilizer
for the action of $Z_n$.

We recall the ordinary smooth Mackey argument in this case.
Fourier transform identifies the convolution algebra of $B_n$ with
the algebra of compactly supported locally constant functions on
$R_n$. A smooth representation gives its nondegenerate module,
or equivalently its spectral sheaf on $R_n$. Taking $B_0$-coinvariants
restricts this sheaf to $\{x_0=0\}$. Vanishing of these coinvariants
therefore places the sheaf on the open orbit $Z_n$.
A $Z_n$-equivariant sheaf on its free transitive orbit is determined
by its fiber at $1$, an arbitrary vector space $E$. Its compactly
supported sections give $E\otimes\mcM_n$ with the actions
\eqref{s8-finite-actions}. Conversely this construction has fiber
$E$. This proves the ordinary equivalence.

For the last assertion, use
Proposition~\ref{countable-ordinary-presentation} to write a
representation as a countable tower $(Q_i)$ of ordinary smooth
$P_0^n$-representations. If its $B_0$-coinvariants vanish, replace
$Q_i$ by $\ker(Q_i\to(Q_i)_{B_0})$.
Exactness of $B_0$-coinvariants shows that these kernels have zero
coinvariants, and the replacement is an isomorphism of pro-objects
because its quotient tower is zero. Apply the ordinary equivalence
term by term. Its fibers form a countable tower in $\Vect$.
\end{proof}

Pushforward along $Z_{n+1}\to Z_n$ gives a surjection
$q_n:\mcM_{n+1}\to\mcM_n$. Surjectivity follows, for example, by
lifting compactly supported functions in the extra coefficient
variable using a compactly supported locally constant density of
integral one.

\begin{definition}\label{M-definition}
Put $\mcM=\varprojlim_{n\geq0}\mcM_n\in\mVect$.
It is strict. Define an action $\rho_0$ of $P_0$ on it by
\[
 \rho_0(\widehat a)\mu=(z\mapsto a^{-1}z)_*\mu,\qquad
 \rho_0(\widetilde b)\mu=\psi(bz/t)\mu.
\]
To interpret these as action maps, fix an output level $k$ and a
bounded piece $H_d$. Choose $r\geq\max(k,d)$. The formulas are
the smooth finite-level operators on $\mcM_r$, followed by
$\mcM_r\to\mcM_k$. They factor through a finite congruence
quotient of $H_d$ and are compatible under increasing $r$.
\end{definition}
The group law holds at the finite levels. This constructs a
pseudo-action, and hence a smooth representation, directly.

\begin{claim}\label{M-smooth-coinvariants}
The projection to a level induces an isomorphism
\[
 \mcM_{A_n}\xrightarrow{\ \sim\ }\mcM_n.
\]
The descended action of $H_n$ is the representation of $P_0^n$
defined in \eqref{s8-finite-actions}.
\end{claim}
\begin{proof}
For $r\geq n$, integration over the translation orbits of
$A_n/A_r$ gives $(\mcM_r)_{A_n}\simeq\mcM_n$.
These are the ordinary coinvariants of compactly supported measures
on a group and its quotient. The weakly strict group $A_n$ has
levelwise coinvariants by \eqref{s8-levelwise-coinvariants}, so the
resulting tower is eventually constant with value $\mcM_n$.
The relation $[A_n,t^{-n}O]\subset tO$ gives the descended action.
\end{proof}

\begin{lemma}\label{M-irred}
The representation $(\rho_0,\mcM)$ of $P_0$ is irreducible.
\end{lemma}
\begin{proof}
Let $0\ne\mN\subset\mcM$ be a subrepresentation in $\Rep(P_0)$,
and set $\mQ=\mcM/\mN$. The subgroup $tO$ acts trivially.
Exactness of $A_n$-coinvariants gives
\[
 0\longrightarrow\mN_{A_n}\longrightarrow\mcM_n
   \longrightarrow\mQ_{A_n}\longrightarrow0
\]
in $\Rep(P_0^n)$. An ordinary simple representation remains simple
in its pro-category: a subobject of a constant object can be
represented by its levelwise images in that object, and those images
are either zero or the whole simple object. Together with
Proposition~\ref{weak-strict-GK} and
Lemma~\ref{s8-finite-mackey}, this shows
$\mN_{A_n}=0$ or $\mcM_n$.

Some composite $\mN\to\mcM\to\mcM_n$ is nonzero, because the
projections jointly detect a morphism to $\mcM$. It factors through
$\mN_{A_n}$, so $\mN_{A_n}=\mcM_n$. For every $r\geq n$, the
successive-coinvariant map
$\mN_{A_r}\twoheadrightarrow\mN_{A_n}$ is an epimorphism.
Hence $\mN_{A_r}=\mcM_r$ and $\mQ_{A_r}=0$.

Suppose $\mQ\ne0$. By Lemma~\ref{s8-uq-functional}, choose a
nonzero morphism $f:\mQ\to\CC$. Its composite with
$\mcM\twoheadrightarrow\mQ$ factors through some $\mcM_r$,
and we may take $r\geq n$. Thus the composite annihilates
$\mN$, while $\mN\to\mcM_r$ is an epimorphism, as just proved.
The composite is therefore zero, a contradiction. It follows that
$\mQ=0$.
\end{proof}

\subsection{\texorpdfstring{The twists and the representation $\Upsilon$}{The twists and the representation Upsilon}}
\begin{definition}\label{s8-twists}
For $m\in\ZZ$ set
\[
 \rho_m(p)=\rho_0(\widehat t^{-m}p\widehat t^m),\qquad p\in P_0,
\]
and write $\mcM^{(m)}$ for $\mcM$ with this action.
In particular, $\widetilde b$ acts by multiplication by
$\psi(t^{-m-1}bz)$.
\end{definition}

\begin{lemma}\label{rho-coinvariants-threshold}
For $m,j\in\ZZ$,
\[
 (\mcM^{(m)})_{t^jO}=\begin{cases}
 \mcM^{(m)},&j>m,\\
 0,&j\leq m.
 \end{cases}
\]
The representations $\mcM^{(m)}$ are irreducible and pairwise
nonisomorphic.
\end{lemma}
\begin{proof}
Twisting preserves irreducibility. If $j>m$, then
$t^{-m-1}bz\in O$ for $b\in t^jO$, so its residue is zero.
If $j\leq m$, the action at every sufficiently deep level has
character $b\mapsto\psi(t^{-m-1}bz)$ at $z\in Z_r$.
This character is nontrivial: take $b=t^mc$ with $c\in F$;
the resulting value is $\widetilde\psi(cz_0)$, with $z_0\ne0$.
On the compact support of a test measure, a finite partition into
open sets therefore makes some operator $\rho_m(b)-1$ invertible
on each part. Every such measure lies in the span of these images,
so its ordinary $t^jO$-coinvariants vanish. Levelwise coinvariants
give the claim for $\mcM$. The different thresholds distinguish
the representations.
\end{proof}

\begin{definition}\label{s8-upsilon-definition}
Define
\[
 \Upsilon=\prod_{m\in\ZZ}\mcM^{(m)}\quad\text{in }\mVect.
\]
The group $P_0$ acts factorwise and $\widehat t$ sends the $m$-th
factor to the $(m+1)$-st factor by the identity on $\mcM$.
\end{definition}
The twist convention gives
$\rho_{m+1}(\widehat t p\widehat t^{-1})=\rho_m(p)$, so these
operators obey the semidirect-product relation. They define
action maps: for finitely many output factors and a bounded piece
of $P(K)$, only finitely many input factors and sufficiently deep
ordinary levels are needed.

More explicitly, the underlying pro-vector space is represented by
the strict countable tower
\begin{equation}\label{s8-upsilon-tower}
 U_N=\prod_{|m|\leq N}\mcM_N,
 \qquad U_{N+1}\longrightarrow U_N,
\end{equation}
where the transition drops the two outside factors and applies $q_N$
to each remaining factor. Thus $\Upsilon$ belongs to our working
category without any dimension restriction on $\mcM_N$.

\begin{proposition}\label{Upsilon-irred}
The representation $\Upsilon$ of $P(K)$ is irreducible.
\end{proposition}
\begin{proof}
Let $0\ne\mN\subset\Upsilon$. Some coordinate map
$\mN\to\mcM^{(m)}$ is nonzero and hence an epimorphism, since
$\mcM^{(m)}$ is simple as a $P_0$-representation.
The shift action gives the same assertion for every $m$.
For a finite set $J\subset\ZZ$, the image of
\[
 \mN\longrightarrow\prod_{m\in J}\mcM^{(m)}
\]
is a subobject of a finite sum of pairwise nonisomorphic simple
objects. Such a subobject is a sum of some of these simple factors.
Surjectivity onto every coordinate therefore makes this map an
epimorphism. Composing with ordinary projections shows that
$\mN\to U_N$ is an epimorphism for every $N$ in
\eqref{s8-upsilon-tower}.

Set $\mQ=\Upsilon/\mN$. If $\mQ\ne0$, choose a nonzero map
$\mQ\to\CC$ by Lemma~\ref{s8-uq-functional}.
Its composite with $\Upsilon\twoheadrightarrow\mQ$ factors
through some $U_N$ and vanishes on $\mN$. Since
$\mN\twoheadrightarrow U_N$, this composite is zero, a
contradiction. Thus $\mQ=0$.
\end{proof}

\subsection{\texorpdfstring{Identification with induction from $P(O)$}{Identification with induction from P(O)}}
We recall the ordinary inducing representations to make this
subsection independent of the earlier sections of the manuscript.
For $m\geq0$ set
\[
 \psi_m(c)=\psi(t^{-m-1}c),\qquad
 S_m=O^*(m+1)\ltimes O,
\]
and extend $\psi_m$ to $S_m$ trivially on its multiplicative factor.
Let
\[
 \bs(m)=\ind_{S_m}^{P(O)}\psi_m.
\]
This induction is computed in a finite congruence quotient, using
ordinary smooth compact induction. Let
\[
 \widetilde R_m=O^*(1)\ltimes t^mO,\qquad
 \lambda_m(a,b)=\psi(t^{-m-1}b).
\]
\begin{lemma}\label{s8-induction-sigma}
There is an isomorphism
$\bs(m)\simeq\ind_{\widetilde R_m}^{P(O)}\lambda_m$.
\end{lemma}
\begin{proof}
The characters of $O$ extending $\psi_m|_{t^mO}$ form the
$O^*(1)$-orbit of $\psi_m$. Inducing further by $O^*$ gives the
full $O^*$-orbit. The fiber at $\psi_m$ is one-dimensional,
and its stabilizer $O^*(m+1)$ acts trivially. The ordinary
little-group description therefore gives exactly $\bs(m)$.
\end{proof}

For the induction from $P(O)$ to $P_0$, take the compact-support
model at each bounded group $H_n$, followed by the formal projective
limit over $n$. This is the induction construction denoted by $i$
in \cite{GK}; in particular, the word ``induction'' here specifies
this construction without asserting a right-adjoint property for
these noncompact quotients. Put
\[
 \Upsilon_m=\ind_{P(O)}^{P_0}\bs(m),\qquad m\geq0.
\]

\begin{proposition}\label{Upsilonm-M}
There is an isomorphism of $P_0$-representations
\[
 \Upsilon_m\simeq\mcM^{(m)},\qquad m\geq0.
\]
\end{proposition}
\begin{proof}
First let $m=0$ and put
$P'(O)=O^*(1)\ltimes O$ and
$\widehat\psi(a,c)=\psi(c/t)$.
Induction in stages gives
\[
 \Upsilon_0\simeq\ind_{P'(O)}^{P_0}\widehat\psi
     =\varprojlim_{n\geq0} I_n,\qquad
 I_n=\ind_{P'(O)}^{H_n}\widehat\psi.
\]
The last induction is ordinary compact induction after quotienting
by a sufficiently deep common congruence subgroup.

Let $T_n=t^{-n}O/O$ and $Y_n=O/t^nO$.
Choose the negative-power representative of each element of $T_n$.
Represent $H_n/P'(O)$ by
\[
 (a^{-1},a^{-1}z),\qquad a\in F^*,\quad z\in T_n.
\]
Use the induced-function convention
$f(hp)=\widehat\psi(p)^{-1}f(h)$ and left translation.
In these coordinates, if $c_-$ and $c_0$ are the negative-power
part and constant coefficient of $c\in t^{-n}O$, then
\[
 (\widetilde c\cdot f)(a,z)
   =\widetilde\psi(ac_0)f(a,z-ac_-).
\]
Choose $\widetilde\psi$-self-dual additive Haar measures in the
coefficient coordinates and choose a multiplicative Haar measure
$d^\times a$ on $F^*$. The Fourier transform
\[
 (\mcF_nf)(a,y)
   =\int_{T_n} f(a,z)
          \widetilde\psi\bigl(\Res_{t=0}(zy\,dt)\bigr)\,dz
\]
is defined by the perfect pairing between $T_n$ and $Y_n$.
It transforms the displayed translation into multiplication by
\[
 \widetilde\psi\bigl(ac_0+a\Res_{t=0}(c_-y\,dt)\bigr)
      =\psi\bigl(c\,a(1+ty)/t\bigr).
\]
The algebraic isomorphism
\begin{equation}\label{s8-fourier-coordinate}
 F^*\times Y_n\xrightarrow{\ \sim\ }Z_n,
       \qquad(a,y)\longmapsto a(1+ty)
\end{equation}
therefore identifies the additive action with $D_c$.
Under this isomorphism transport
$(\mcF_nf)(a,y)\,d^\times a\,dy$ to a measure on $Z_n$.
The measure $d^\times a\,dy$ is Haar: multiplication by principal
units is triangular in the coefficient coordinates with determinant
one. Substitution in the same induced-function formula gives, for
$u\in O^*$, pushforward under $x\mapsto u^{-1}x$.
Thus $\mcF_n$ identifies $I_n$ with $\mcM_n$ as an
$H_n$-representation.

Restriction from $T_{n+1}$ to $T_n$ sets the new negative
coefficient to zero. Fourier transform turns it into integration
over the new coefficient of $Y_{n+1}\to Y_n$.
By \eqref{s8-fourier-coordinate}, this is precisely pushforward
$\mcM_{n+1}\to\mcM_n$.
Consequently the finite-level isomorphisms define an isomorphism
$\Upsilon_0\simeq\mcM$ of pro-vector spaces and of pseudo-actions.
The description uses the fixed Haar normalizations above.

For general $m$, Lemma~\ref{s8-induction-sigma} and induction in stages give
\[
 \Upsilon_m\simeq\ind_{\widetilde R_m}^{P_0}\lambda_m.
\]
Conjugation by $\widehat t^{-m}$ takes $\widetilde R_m$ to
$P'(O)$ and takes its character to $\widehat\psi$, since
$\lambda_m(1,t^mc)=\psi(c/t)$. Pullback by this automorphism
therefore turns the $m=0$ model into precisely $\rho_m$.
\end{proof}

\subsection{\texorpdfstring{Irreducible representations of $P_0$}{Irreducible representations of P0}}
We first record the finite-level detection used in the classification.
\begin{lemma}\label{s8-uq-detection}
Let $H$ be $O$ or $O^*$, and set
\[
 H(r)=\begin{cases}t^rO,&H=O,\\1+t^rO,&H=O^*,\end{cases}
 \qquad r\geq1.
\]
For every nonzero $\mV\in\Rep(H)$, some quotient
$\mV_{H(r)}$ is nonzero.
\end{lemma}
\begin{proof}
Choose a countable ordinary representation presentation by
Proposition~\ref{countable-ordinary-presentation}.
Some canonical map $\mV\to Q$ to an ordinary stage is nonzero.
The ordinary action on $Q$ factors through a congruence quotient,
so $H(r)$ acts trivially for some $r$. The map consequently factors
through $\mV_{H(r)}$, which must be nonzero. No strictness or
nonzero vector of $\widetilde{\mV}$ is needed.
\end{proof}

\begin{proposition}\label{s8-uq-mackey}
Let $\mV$ be an irreducible object of $\Rep(P_0)$ on which $K$
acts nontrivially. Then
\[
 \mV\simeq\mcM^{(m)}
\]
for a unique $m\in\ZZ$. This integer is characterized by the
condition that $m+1$ is the least integer $r$ for which $t^rO$
acts trivially.
\end{proposition}
\begin{proof}
By Lemma~\ref{s8-uq-detection}, $\mV_{t^rO}\ne0$ for some $r$.
Since $t^rO$ is normal in $P_0$, the quotient carries its descended
$P_0$-action. Irreducibility makes the quotient map an isomorphism.
Thus some $t^rO$ acts trivially.

The set of such integers is upward closed and nonempty. It is not
all of $\ZZ$: triviality on every bounded piece $t^rO$ would make
the $K$-pseudo-action trivial. Hence this set has a least element.
After twisting by a power of $\widehat t$, we may assume
\begin{equation}\label{s8-uq-normalization}
 tO\text{ acts trivially on }\mV,\qquad \mV_O=0.
\end{equation}
Indeed, in this normalization $O$ acts nontrivially, and its
coinvariant quotient is a quotient representation of $P_0$;
if that quotient were nonzero, irreducibility would force the
$O$-action to be trivial.

Put $\mV_n=\mV_{A_n}$, $n\geq0$.
These are objects of $\mVect$, by the countable presentation
theorem. Since $[A_n,t^{-n}O]\subset tO$, the action of $H_n$
descends to the group $P_0^n=Z_n\ltimes B_n$.
The images of $A_n$ and $O$ commute modulo $tO$, so
\eqref{s8-uq-normalization} gives
$(\mV_n)_{B_0}=0$.
Lemma~\ref{s8-finite-mackey} supplies
\begin{equation}\label{s8-uq-multiplicity}
 \mV_n\simeq\mL_n\widehat\otimes\mcM_n,
          \qquad\mL_n\in\mVect.
\end{equation}

The map $\mV_{n+1}\to\mV_n$ is the quotient by $A_n/A_{n+1}$.
For the standard representations, this quotient is integration
over the multiplicative fibers:
\[
 (\mcM_{n+1})_{A_n/A_{n+1}}\simeq\mcM_n.
\]
The equivalence of Lemma~\ref{s8-finite-mackey} therefore identifies
the successive multiplicity spaces by isomorphisms. Identify them
compatibly with a single $\mL\in\mVect$. The transition maps
in \eqref{s8-uq-multiplicity} then become
$1_{\mL}\widehat\otimes q_n$.

There is a natural comparison morphism
\begin{equation}\label{s8-uq-comparison}
 \mV\longrightarrow\varprojlim_n\mV_n
     \simeq\mL\widehat\otimes\mcM.
\end{equation}
It is equivariant as a morphism of pseudo-actions: for a fixed
bounded subgroup $H_d$, use only levels $n\geq d$. All quotient
maps and all finite-level identifications on those levels are
$H_d$-equivariant. These compatible action maps give equivariance
for $P_0$. This argument does not use a pointwise comparison
theorem or assume that the source is strict. We also do not need
to assert that the first arrow in \eqref{s8-uq-comparison} is
an isomorphism.

Lemma~\ref{s8-uq-detection} applied to $O^*$ shows that some
$\mV_n\ne0$, so $\mL\ne0$. By
Lemma~\ref{s8-uq-functional}, choose a nonzero
$\ell:\mL\to\CC$.
Composing \eqref{s8-uq-comparison} with
$\ell\widehat\otimes1_{\mcM}$ gives
$f:\mV\to\mcM$.
It is nonzero: at any level its composite with $\mcM\to\mcM_n$
is
\[
 \mV\twoheadrightarrow\mV_n
   \simeq\mL\widehat\otimes\mcM_n
   \xrightarrow{\ell\widehat\otimes1}\mcM_n.
\]
The last arrow is nonzero since $\mcM_n\ne0$; the first is an
epimorphism. Both $\mV$ and $\mcM$ are simple, so $f$ is an
isomorphism. Undoing the twist gives the asserted classification.
The threshold formula in
Lemma~\ref{rho-coinvariants-threshold} gives uniqueness of $m$.
\end{proof}

\subsection{\texorpdfstring{The characterization of $\Upsilon$}{The characterization of Upsilon}}
\label{s8-upsilon-uniqueness}
\begin{definition}\label{s8-uq-graded}
A representation $\mW\in\Rep(P(K))$ admits a \emph{product grading}
if its restriction to $P_0$ has a decomposition in $\Rep(P_0)$
\[
 \mW\simeq\prod_{j\in\ZZ}\mW_j
\]
such that $\widehat t$ carries the $j$-th factor isomorphically
onto the $(j+1)$-st factor. In particular, the coordinate
projections are $P_0$-equivariant as maps of pseudo-actions.
\end{definition}
This is the completed product grading used for associated graded
objects. Irreducibility always tests all subrepresentations in
$\Rep(P(K))$, whether or not they respect the grading.

\begin{theorem}\label{s8-uq-main}
The representation $\Upsilon$ is nonzero, irreducible, and admits
a product grading. Up to isomorphism of $P(K)$-representations,
it is the unique nonzero irreducible object of $\Rep(P(K))$
admitting such a grading. With the grading included, there is
exactly one such isomorphism class up to a common shift of the
integer indices.
\end{theorem}
\begin{proof}
Let $\mW=\prod_{j\in\ZZ}\mW_j$ be nonzero and irreducible.
Each factor is a retract of $\mW$ as a pro-vector space, hence
belongs to $\mVect$. All factors are nonzero, since the shift
identifies consecutive ones.

The factor $\mW_0$ is irreducible in $\Rep(P_0)$.
Indeed, a nonzero proper subrepresentation $U_0\subset\mW_0$
would give the nonzero proper $P(K)$-subrepresentation
\[
 \prod_{j\in\ZZ}\widehat t^jU_0
          \subset\prod_{j\in\ZZ}\mW_j.
\]
This is a countable product in the working category. Its action
is defined by the factor actions and the shift, as for $\mW$.
It is proper because its projection to the zeroth factor has
image $U_0$.

Next, $K$ acts nontrivially on $\mW_0$.
Otherwise it acts trivially on every factor. Identify all factors
with $\mW_0$ by the powers of $\widehat t$.
Because $\widehat t$ commutes with $O^*$, the $O^*$-actions
become identical, and $\widehat t$ acts by the bilateral shift.
The diagonal morphism
\[
 \mW_0\longrightarrow\prod_{j\in\ZZ}\mW_0
\]
then has a nonzero invariant image. It is proper: the difference
of two coordinate projections is nonzero on the product and
vanishes on the diagonal. This contradicts irreducibility.

Proposition~\ref{s8-uq-mackey} gives
$\mW_0\simeq\mcM^{(m)}$ for some $m$.
The group relation determines the action on $\mW_j$ to be
$\rho_{m+j}$ after identifying the underlying spaces by the shift.
Reindexing all degrees by $j\mapsto m+j$ identifies $\mW$
with $\Upsilon$. The integer $m$ is determined by the threshold
on the zeroth factor, so retaining fixed degree labels retains
exactly this common shift.

Conversely, Definition~\ref{s8-upsilon-definition} gives the product
grading of $\Upsilon$, and Proposition~\ref{Upsilon-irred}
proves its irreducibility.
\end{proof}

\section{\texorpdfstring{Smooth representations of $P(K)$}{Smooth representations of P(K)}}\label{P(K)}

\subsection{\texorpdfstring{Smooth representations of the group $K$}{Smooth representations of the group K}}
\begin{definition} $\mK$ and $\mK ^*$ are group objects of $\pSet$ which are considered as subgroups of $\mP$.
\end{definition}
So $K$ and $K^*$ are considered as subgroups of $P(K)$.

 \begin{definition}\label{cusp} Let $\mW$ be a smooth representation of $K$.
\begin{enumerate}
\item
$\eta _n (\mW):\mW \to \mW _ {t^nO}$ is the canonical projection (see Claim \ref{coin}).
\item $\mcF _\mW= \{\mF _n\}$ is a decreasing filtration on $\mW$ where
$ \mF _n := \ker (\eta _n(\mW))$.
\item
$\kk ^n_m(\mW):   \mW _ {t^nO} \to  \mW _ {t^mO}, m\leq n\in \mZ $ are the canonical projections.
\item $ \mW _n^m:= \ker ( \kk ^n_m(\mW)) $.
\item $\mu _n(\mW) :=  \mW ^{n-1}_n $.

\item $\mu (\mW):=  \mu _0 (\mW)$.
\end{enumerate} \end{definition}

\begin{lemma}\label{m} \begin{enumerate}

\item The functors $\mW\longmapsto\mW_{t^nO}$ are exact; $\eta_n$ denotes the corresponding natural quotient map.

\item We have a natural smooth representation of the group $K$ on $\mu (\mW)$ trivial on $O$.
\end{enumerate}
\end{lemma}
\begin{proof} Part (1) follows from Lemma 2.7 in \cite{GK}, and part (2) follows from part (1).
\end{proof}

 \subsection{\texorpdfstring{The exactness of the functor $\mu$}{The exactness of the functor mu}}

\begin{proposition}\label{mmu} The functor $\mu$ is exact.
\end{proposition}
\begin{proof} The proof of Proposition \ref{mmu} is based on the following result.

\sms

Let $\mcC$ be an abelian category, let $F,G:\mcC\to\mcC$ be exact functors, and let $a:F\to G$ be a morphism such that the maps $a(X):F(X)\to G(X)$ are surjective for all $X\in Ob(\mcC)$. We define $R : = \ker (a)$.
\begin{lemma}\label{exact}The functor $R$ is exact.
\end{lemma}
\begin{proof}
Let
$0\to A'\xrightarrow{\,i\,} A\xrightarrow{\,p\,} A''\to0$
be a short exact sequence in $\mcC$. Since the functors $F$ and $G$ are exact we get two short exact rows and a commutative diagram

$\begin{array}{ccccccccc}
0&\to&F(A')&\xrightarrow{F(i)}&F(A)&\xrightarrow{F(p)}&F(A'')&\to&0\\
& &\downarrow^{a_{A'}}&&\downarrow^{a_A}&&\downarrow^{a_{A''}}&&\\
0&\to&G(A')&\xrightarrow{G(i)}&G(A)&\xrightarrow{G(p)}&G(A'')&\to&0.
\end{array}$

The snake lemma provides  the exact sequence
$$0\to\ker(a_{A'})\to\ker(a_A)\to\ker(a_{A''})\to\operatorname{coker}(a_{A'})\to\cdots$$
We have $\operatorname{coker}(a_X)=0$ since, by hypothesis, each $a_X:F(X)\to G(X)$ is surjective.  Therefore the snake lemma yields the exactness of the sequence

$0\to\ker(a_{A'})\to\ker(a_A)\to\ker(a_{A''})\to0$. Thus $R$ is exact.
\end{proof}
Proposition \ref{mmu} now follows from Lemma \ref{m}.

\end{proof}

\subsection{\texorpdfstring{Representations of  $P(K)$}{Representations of P(K)}} 
We now introduce some basic constructions related to representations of the group $P(K)$ that we are going to use in the future.

\begin{definition}\label{grad}
\begin{enumerate}
\item A smooth representation of $P(K)$ is complete if its underlying
pro-vector space is strict in the sense of Definition~\ref{strict}
and its increasing filtration $\{\mF_{-i}\}_{i\in\ZZ}$ is admissible
in the sense of Definition~\ref{admissible}.
\item 
For  a smooth representation $(\Pi ,\mW)$ of $P(K)$ we write  $\bar \mW := \prod _{n\in \mZ}
\mu _n(\mW)$. Since the group $P(K)$ preserves (up to a shift) the filtration $\mcF$ on $\mW$ we obtain a representation $\bar\Pi$ of the group $P(K)$ on $\bar \mW$. 
\end{enumerate}
\end{definition}
\begin{claim}
Let $\Pi$ be a smooth  representation of $P(K)$. Then the representation $\bar\Pi$ is also smooth.
\end{claim}

\begin{proposition}\label{isbar}

Let $\mW$ be a smooth complete representation of $P(K)$ such that $\bar \mW$ is isomorphic to $\Upsilon$ and
 $\mW ^{N(K)}=\{0\}$. Then $\mW$ is irreducible.

\end{proposition}
\begin{proof}
Let $0\ne\mU\subset\mW$ be a subrepresentation and put
$\mW_i=\mF_{-i}$. Suppose first that $\mU\cap\ker\eta_n=0$
for every $n$. Restriction to $N_n=t^nO$ gives a monomorphism
\[
 \mU\longrightarrow\mW_{N_n}
\]
of representations, where the target has trivial $N_n$-action.
A subobject of a trivial representation is trivial, so every $N_n$
acts trivially on $\mU$ in the pseudo-action. Since
$N(K)=\varinjlim_{n\to-\infty}N_n$ as a group object, its action
is trivial as well. Thus $\mU\subset\mW^{N(K)}=0$, a contradiction.

Consequently $\mU\cap\mW_i\ne0$ for some $i$.
The associated-graded object
\[
 \prod_i(\mU\cap\mW_{i+1})/(\mU\cap\mW_i)
\]
is a subrepresentation of $\bar\mW\simeq\Upsilon$. Its action
is obtained from the filtered action by kernels, quotients, and
products in the representation category; on each bounded group
piece and finite set of output factors it uses finitely many input
factors. It is nonzero by Claim~\ref{separated} and the preceding
paragraph. Proposition~\ref{Upsilon-irred} makes it equal to
$\bar\mW$. Hence $\mU$ is dense with respect to this filtration,
and Lemma~\ref{dense} gives $\mU=\mW$.
\end{proof}

\section{The affine Grassmannian}\label{Gr}
In what follows we would like to study induction of special representations of $G(O)$ to $G(K)$. In order to do this we need to discuss some preparatory material about the geometry of the affine Grassmannian $\Gr_G$, whose $F$-points are $\Gr_G(F)=G(K)/G(O)$. We refer the reader to Sections 2 and 3 of \cite{MV} for the proof of all the claims about $\Gr_G$ which are going to appear in this Section (and which are all easy in any case).
Similarly, in order to analyze the induction of special representations of $G'(O)$ to $G(K)$ we need to discuss the twisted affine Grassmannian $\Gr_G'$, with $\Gr_G'(F)=G(K)/G'(O)$; this is done at the end of this section. We distinguish each ind-scheme over $F$ from its space of $F$-points, endowed with the topology induced by that on $F$.
\subsection{Topology on the affine Grassmannian}
For an algebraic group $H$ over $F$, let $\Gr_H$ denote its affine Grassmannian, an ind-finite-dimensional ind-scheme over $F$, with $\Gr_H(F)=H(K)/H(O)$. Write $\Gr_H$ as an inductive limit of finite-dimensional $F$-schemes with closed transition embeddings. The corresponding inductive limit of their spaces of $F$-points endows $\Gr_H(F)$ with a natural topology.

\begin{remark}The topological space $\Gr_H(F)$ is usually not locally compact.
\end{remark}
\subsection{\texorpdfstring{The basic geometry of $\Gr_G$}{The basic geometry of Gr(G)}}
The following statement follows from the Iwasawa decomposition proved in \cite{T}.
\begin{claim}\label{trans} The group $P(K)$ acts transitively on $\Gr_G(F)$. Equivalently, the natural morphism $j:\Gr_P\to\Gr_G$ induces a bijection $j(F):\Gr_P(F)\to\Gr_G(F)$.
\end{claim}
\begin{remark}\label{not-isom}
Although $j(F)$ is a bijection, $j$ is not an isomorphism of ind-schemes and $j(F)$ is not a homeomorphism: its inverse is not continuous. In fact, it is observed in \cite{BD} (cf. the lemma at the end of page 173) that, for any affine algebraic group $G$ and any closed subgroup $H$, the natural map $\Gr_H\to\Gr_G$ is a locally closed embedding of ind-schemes if the quotient $G/H$ is a quasi-affine variety (in particular, in this case $\Gr_H$ is isomorphic to its image in $\Gr_G$). However, in our case the quotient $G/P$ is the projective line $\mathbb P^1$, which is not quasi-affine, so the above lemma is not applicable.
\end{remark}
For an integer $l\geq 0$, let $\Gr_G^l$ denote the orbit variety whose $F$-points form the $G(O)$-orbit of the image in $\Gr_G(F)$ of the diagonal matrix
$\hat t\, ^l=\begin{pmatrix} t^l & 0\\ 0 & 1\end{pmatrix}$.  Then $\Gr_G^l$ is an algebraic $F$-variety of dimension $l$. We denote by $\oGr_G^l$ its Zariski closure. This is a projective $F$-variety of dimension $l$. Moreover, we have
\begin{equation}
\Gr_G=\bigsqcup\limits_{l=0}^{\infty} \Gr_G^l;\ \ \  \oGr_G^l=\bigsqcup\limits_{0\leq i\leq l, \ \text{$i-l$ is even}} \Gr_G^i;
\end{equation}
In particular, we have closed embeddings $\oGr_G^l\hookrightarrow \oGr_G^{l+2}$, and $\Gr_G$ is equal to the inductive limit of all the $\oGr_G^l$ (note that as an ind-scheme it has two connected components corresponding to the parity of $l$).

\subsection{Semi-infinite orbits}
The Iwasawa decomposition implies that
\begin{equation}\label{semi-inf}
\Gr_G(F)=\bigsqcup\limits_{n\in \mathbb Z} N(K)\cdot \hat t\, ^n,
\end{equation}
where as before we use $\hat t\, ^i$ also for its image in $\Gr_G(F)$. Let $S^n$ denote the semi-infinite orbit ind-scheme associated with $N$, so that $S^n(F)=N(K)\cdot\hat t\, ^n$. Each $S^n$ is a locally closed sub-ind-scheme of $\Gr_G$, and $S^n(F)$ is invariant under $P_0=T(O)N(K)$.
The stabilizer of $\hat t^i$ in $N(K)=K$ is $t^iO$, so $S^i(F)\simeq K/t^iO$. This identification comes from an isomorphism of the corresponding ind-schemes over $F$, and is therefore also a homeomorphism on $F$-points (a special case of the lemma at the end of page 173 of \cite{BD}). The Zariski closure is
\[
 \overline{S^i}=\bigcup_{\substack{j\leq i\\j\equiv i\pmod 2}}S^j.
\]
On $F$-points, the closure of $S^i(F)$ in the topology induced by $F$ is likewise the union of the $S^j(F)$ with $j\leq i$ and $j\equiv i\pmod2$. The parity condition records the two connected components of $\Gr_G$.
The space $\Gr_P(F)$ is homeomorphic to the disjoint union of the $S^i(F)$. The corresponding statement holds for reduced ind-schemes.\footnote{It does not hold for the unreduced ind-schemes: $\Gr_P$ is highly nonreduced. This subtlety will not be important below.}
The intersection $S^i\cap\Gr_G^l$ is nonempty if and only if
$|i|\leq l$ and $l-i$ is even. With the indicated root coordinates,
\[
S^i\cap\Gr_G^l\simeq
\begin{cases}
 \mathbb A^l,&i=l,\\
 \{\mathrm{pt}\},&i=-l,\ l>0,\\
 \mathbb G_m\times\mathbb A^{(l+i)/2-1},&-l<i<l.
\end{cases}
\]
Indeed, write a point of $S^i(F)$ as
$\widetilde b\,\widehat t^{\,i}$, with $b\in K/t^iO$.
Its Cartan index is $i-2\min(i,0,v(b))$, with $v(0)=+\infty$.
For an interior value of $i$, the leading coefficient of $b$ is nonzero.

\subsection{\texorpdfstring{The ind-schemes $\Gr_{G,m}$}{The ind-schemes Gr(G,m)}} Let $\Gr_{G,m}$ denote the congruence-level affine Grassmannian, with $\Gr_{G,m}(F)=G(K)/G(O)(m+1)$. It is an ind-scheme over $F$, forming a principal bundle over $\Gr_G$ under the jet group with $F$-points $G(O)_m$. Explicitly, it is equal to the union of finite-dimensional schemes $\overline\Gr_{G,m}^l$ where $\overline\Gr_{G,m}^l$ is the preimage of
$\overline\Gr_G^l$ under the natural projection $\Gr_{G,m}\to \Gr_G$.

The ind-scheme $\Gr_{G,m}$ makes sense for any algebraic group, so in particular, we can consider $\Gr_{P,m}$. We have an obvious morphism $\Gr_{P,m}\to \Gr_{G,m}$.
The natural map
$$
\Gr_{P,m}\to \Gr_P\underset{\Gr_G}\times \Gr_{G,m}
$$
is a closed embedding of ind-schemes (since the RHS is a $G(O)_m$-torsor over $\Gr_P$ and the LHS is a $P(O)_m$-sub-torsor (and $P(O)_m$ is closed in $G(O)_m$)). This implies the following. Since the loop group of $P$ has connected components indexed by $\mathbb Z$ (on $F$-points the index is the homomorphism $P(K)\to\mathbb Z$ obtained by composing $P(K)\to K^*$ with $v$), the same is true for each of the ind-schemes $\Gr_{P,m}$. Let $\Gr_{P,m}^i$ denote the corresponding $i$-th connected component (here we use the Zariski topology, not the topology coming from the topology on $F$). Write $\Gr_P^i$ for the corresponding component of $\Gr_P$; on reduced ind-schemes it identifies with $S^i$. In contrast, $\Gr_{P,0}^i\to\Gr_P^i$ is a $P$-torsor, with $P(F)=P(O)_0$ acting on its $F$-points. On the other hand, on the level of reduced ind-schemes $\Gr_P\underset{\Gr_G}\times \Gr_{G,m}$ is isomorphic to the disjoint union of $S^i_m:=p_m^{-1}(S^i)$ where $p_m:\Gr_{G,m}\to \Gr_G$ is the natural projection. Thus we see that
the reduced ind-scheme of $\Gr_{P,m}^i$ is closed inside $S^i_m$.
\begin{corollary}
The topological space $\Gr_{P,m}^i(F)$ is closed in $S^i_m(F)$.
\end{corollary}

\subsection{The twisted affine Grassmannian}
Let $\Gr'_G$ denote the twisted affine Grassmannian, an ind-finite-dimensional ind-proper ind-scheme over $F$ with
$$
\Gr'_G(F)=G(K)/G'(O).
$$ It is still true that $G(K)=P(K)\cdot G'(O)$ (cf. Lemma \ref{BT}). Let $(S')^i$ denote the corresponding semi-infinite stratum in $\Gr'_G$; its $F$-points are the image of $\Gr_P^i(F)$.

For $m\in \mathbb Z_{\geq 0}$ we similarly have an ind-scheme $\Gr'_{G,m}$ with $\Gr'_{G,m}(F)=G(K)/G'(O)(m+1)$; it is a torsor over $\Gr'_G$ under the jet group with $F$-points $G'(O)_{m+1/2}$ and similarly to the above we define $(S'_m)^i$ to be the preimage of $(S')^i$. It contains $\Gr_{P,m}^i$ as a closed subset.

\section{\texorpdfstring{Special representations of $G(K)$ -- some initial results}{Special representations of G(K) -- some initial results}}\label{G(K)-in}

\subsection{\texorpdfstring{Restriction from $G(K)$ to $P(K)$}{Restriction from G(K) to P(K)}} We would like to study representations of $G(K)$ whose restriction to $P(K)$ is irreducible. Recall that if we replace $K$ by $F$ and exclude one-dimensional representations, then this property is equivalent to irreducibility and cuspidality. Moreover, the group $P(F)$ has a unique (up to isomorphism) irreducible cuspidal representation, naturally realized in the space $\mcS(F^*)$. When $F$ is replaced by $K$, the uniqueness statement is no longer true. The most straightforward analogue of the unique irreducible cuspidal representation of $P(F)$ is the representation $\Upsilon$ of $P(K)$ defined above. It is therefore natural to ask whether there are irreducible representations of $G(K)$ whose restriction to $P(K)$ is isomorphic to $\Upsilon$. In this section we show that the answer is ``no.'' However, we shall see later that, if $\Pi(\pi)$ is induced from a special representation of either $G(O)$ or $G'(O)$ and $\Theta(\pi):=\Pi(\pi)|_{P(K)}$, then $\Theta(\pi)$ is complete, has no $N(K)$-invariants, and satisfies
$$
\bar\Theta(\pi)\simeq\Upsilon.
$$
Proposition \ref{isbar} will then imply that $\Theta(\pi)$ is irreducible.

As before we identify $K$ with $N(K)\subset P(K)$ and therefore identify $O$
 with a subgroup of $G(K)$. We define $G^+(K)\subset G(K)$ to be the image of $\SL(2,K)$ under the natural map $\SL(2,K)\to\PGL(2,K)$.

We use categorical invariants throughout: $\mV^H\hookrightarrow\mV$
is the maximal subobject on which the $H$-pseudo-action is trivial.
In particular, this notation does not mean the invariant vectors of
an ordinary inverse limit. In the following lemmas the root groups
of $\SL(2)$ are denoted by
\[
 N(A)=\{u(x):x\in A\},\qquad \bar N(A)=\{l(y):y\in A\},
\]
\[
 u(x)=\begin{pmatrix}1&x\\0&1\end{pmatrix},\quad
 l(y)=\begin{pmatrix}1&0\\y&1\end{pmatrix}.
\]
Their images in $\PGL(2)$ are the root subgroups used earlier.

\begin{lemma}[Integral root-group invariants]\label{integral-pro-invariants}
For any smooth pro-representation $\mV$ of $\SL(2,O)$,
\[
 \mV^{N(O)}=\mV^{\SL(2,O)}=\mV^{\bar N(O)}
\]
as subobjects of $\mV$. The same statement holds after transporting
these three subgroup objects by any fixed conjugation in $\GL(2,K)$.
\end{lemma}
\begin{proof}
The group object $\SL(2,O)$ is weakly strict: its finite congruence
quotients have surjective transition maps. By
Proposition~\ref{weak-strict-GK}, write its representation as a formal
pro-system $\mV=\varprojlim_i V_i$ of ordinary smooth representations.
Each $V_i$ factors through $\SL(2,O/t^{n_i+1}O)$.
Both root groups are weakly strict as well. The categorical invariant
construction of \cite[Corollary~2.6]{GK} gives, as subobjects of the
same pro-system,
\[
 \mV^{\SL(2,O)}=\varprojlim_i V_i^{\SL(2,O)},\qquad
 \mV^{N(O)}=\varprojlim_i V_i^{N(O)},
\]
and similarly for the lower root. Equation~\eqref{finite-SL2-invariants}
identifies the terms, proving the assertion. No transition map of the
$V_i$ is required to be surjective. Conjugation transports the group
objects, their representations and these invariant subobjects, so
proves the last assertion as well.
\end{proof}

\begin{lemma}[Detection by bounded root factorizations]\label{bounded-root-detection}
Let $(\rho,\mV)$ be a smooth representation of $\SL(2,K)$ and
$i:\mU\hookrightarrow\mV$ a subobject of pro-vector spaces.
Suppose that the restrictions of the action to both $N(K)$ and
$\bar N(K)$ satisfy
\[
 \rho|_{N(K)}\circ(1\times i)=i\circ\mathrm{pr}_2,\qquad
 \rho|_{\bar N(K)}\circ(1\times i)=i\circ\mathrm{pr}_2
\]
in the pseudo-action. Then the same identity holds for $\SL(2,K)$.
No strictness assumption on $\mU$ is needed.
\end{lemma}
\begin{proof}
Here is a bounded form of Gaussian elimination that verifies the
necessary passage from root groups to the loop-group object. For
$d\geq0$, let $X_d$ be the bounded piece of $\SL(2,K)$ whose four
entries belong to $t^{-d}O$. These pieces exhaust $\SL(2,K)$.
Write $g=\left(\begin{smallmatrix}a&b\\c&e\end{smallmatrix}\right)$.
Since $ae-bc=1$, at least one of $a,c$ has valuation between $-d$
and $d$. Cover $X_d$ by the finitely many coefficient charts on which
one specifies such a pivot and its valuation $k\in[-d,d]$.
These charts are given by vanishing of the preceding coefficients
and nonvanishing of the leading coefficient. They are finite-level
locally closed coefficient conditions.

If the pivot is $a$, Gaussian elimination gives
\[
 g=l(c/a)\begin{pmatrix}a&0\\0&a^{-1}\end{pmatrix}u(b/a),\qquad
 \begin{pmatrix}a&0\\0&a^{-1}\end{pmatrix}=w(a)w(1)^{-1},
\]
where
\[
 w(z)=u(z)l(-z^{-1})u(z)
     =\begin{pmatrix}0&z\\-z^{-1}&0\end{pmatrix}.
\]
If the pivot is $c$, apply the same formula to $w(1)g$ and then
multiply on the left by $w(1)^{-1}$. Every root parameter in these
words belongs to $t^{-2d}O$; the additional parameters in $w(1)$
and its inverse are constants. On a fixed pivot-valuation chart,
each output coefficient of the factorization depends continuously
on finitely many input coefficients. Inversion uses the inverse of
the nonzero leading coefficient and a finite truncation of the
power-series inverse. Thus these are morphisms of the coefficient
pro-objects, not merely factorizations of their individual points.

Here finite coefficient levels mean the images of actual matrices in
$X_d$ under coefficient truncation. The same convention applies to
each pivot chart. In particular, no lifting assertion about arbitrary
solutions of a truncated determinant equation is being used.
These coefficient levels are cofinal with congruence levels on $X_d$:
for $g,h\in X_d$, the entries of $g^{-1}$ also lie in $t^{-d}O$, and
\[
 \begin{aligned}
 g-h\in t^N M_2(O)&\ \Longrightarrow\
 g^{-1}h-1\in t^{N-d}M_2(O),\\
 g^{-1}h-1\in t^N M_2(O)&\ \Longrightarrow\
 g-h\in t^{N-d}M_2(O).
 \end{aligned}
\]
Thus increasing the truncation level gives the same pro-topology.

For completeness, the chart test is valid for an arbitrary pro-source.
Fix an ordinary output level of $\mV$. By Definition~\ref{a}, an
action map on $X_d$ is represented at one ordinary input level of
$\mU$ and one finite coefficient level of $X_d$. Pull back to the
finitely many charts above. If these pullbacks vanish, refine the
input and coefficient levels so that all of them vanish simultaneously;
a common refinement exists because there are only finitely many
charts. Their coefficient images cover the corresponding coefficient
image of $X_d$, so the represented continuous linear map is zero.
Equivalently, their finite disjoint union is weakly surjective at
these levels. Refining the coefficient level does not change this
covering property: every matrix still has one of the stated pivots,
and the factorization computes every prescribed finite jet.
This proves injectivity of the chart test without requiring an
epimorphism from an actual inverse limit to an input level.

The two hypotheses and associativity of the pseudo-action make the
action trivial on every root word, hence on each chart. The chart
test makes it trivial on $X_d$. Finally
$\SL(2,K)=\varinjlim_d X_d$, so Definition~\ref{a} proves the
asserted identity for the full group object.
\end{proof}

\begin{lemma}[The bounded upper-root action]\label{bounded-upper-root-action}
Let $\mV$ be a smooth pro-representation of $\SL(2,K)$ and put
$\mU=\mV^{N(O)}$. For each $r\geq0$, the subobject $\mU$ carries
an induced action of $N(t^{-r}O)$, and this action is trivial.
These statements concern pseudo-actions, with no strictness assumption
on $\mV$ or $\mU$.
\end{lemma}
\begin{proof}
Write $h_{\mathrm{SL}}(a)=\operatorname{diag}(a,a^{-1})$ for $a\in F^*$, and consider
\[
 \mathcal B_r=\{h_{\mathrm{SL}}(a):a\in F^*\}\ltimes N(t^{-r}O).
\]
This is a weakly strict pro-group: it is the inverse limit of the
locally profinite groups
\[
 F^*\ltimes(t^{-r}O/t^sO),\qquad s\geq1,
\]
where $a$ acts on the additive factor by multiplication by $a^2$,
and the transition homomorphisms are surjective.
By Proposition~\ref{weak-strict-GK}, write the restriction of $\mV$
as a formal pro-system of ordinary smooth $\mathcal B_r$-representations
$V_i$. The subgroup $N(O)$ is normal in $\mathcal B_r$.
Consequently each ordinary invariant space $V_i^{N(O)}$ is a
$\mathcal B_r$-subrepresentation. By the levelwise invariant construction,
\[
 \mU=\varprojlim_i V_i^{N(O)}\hookrightarrow\varprojlim_i V_i=\mV
\]
is therefore a morphism of $\mathcal B_r$-representations.
This constructs the asserted bounded subgroup action; it does not
infer it from an action on vectors of an actual inverse limit.

Lemma~\ref{integral-pro-invariants} gives
$\mU=\mV^{\SL(2,O)}$. Choose a constant $a\in F^*$ with $a^2\ne1$.
The operator $h_{\mathrm{SL}}(a)$ is the identity on $\mU$, because
$h_{\mathrm{SL}}(a)\in\SL(2,O)$. Denote the induced upper-root action by
$A_r(x)$, where this notation denotes the action map with parameter
$x\in t^{-r}O$. The identity
\[
 h_{\mathrm{SL}}(a)u(x)h_{\mathrm{SL}}(a)^{-1}=u(a^2x)
\]
now gives $A_r(a^2x)=A_r(x)$ as an identity of pseudo-action maps.
Using the additive group law and composing along the parameter map
$x\mapsto(a^2x,-x)$ gives
\[
 A_r((a^2-1)x)=A_r(a^2x)A_r(-x)=1_{\mU}.
\]
Multiplication by the nonzero constant $a^2-1$ is an automorphism of
the group object $t^{-r}O$. Pulling back by its inverse proves
$A_r(x)=1_{\mU}$ identically. All equalities hold in the pseudo-action;
no comparison with pointwise equivariance is used.
\end{proof}

\begin{lemma}[Global root-group invariants]\label{SL2-global-invariants}
For every smooth pro-representation $\mV$ of $\SL(2,K)$,
\[
 \mV^{N(O)}=\mV^{N(K)}=\mV^{\SL(2,K)}
\]
as categorical invariant subobjects.
\end{lemma}
\begin{proof}
Put $\mU=\mV^{N(O)}$. Lemma~\ref{bounded-upper-root-action} makes
the upper-root action on $\mU$ trivial on every bounded subgroup
$N(t^{-r}O)$. Since
\[
 N(K)=\varinjlim_{r\geq0}N(t^{-r}O),
\]
Definition~\ref{a} makes its full pseudo-action trivial on $\mU$.
Moreover, Lemma~\ref{integral-pro-invariants} gives
$\mU=\mV^{\SL(2,O)}$. Thus the constant Weyl matrix
\[
 w=\begin{pmatrix}0&1\\-1&0\end{pmatrix}\in\SL(2,O)
\]
acts as the identity on this subobject. Pulling back the upper-root
identity by conjugation with $w$ gives triviality of the entire
lower-root pseudo-action, since $wu(x)w^{-1}=l(-x)$.
Lemma~\ref{bounded-root-detection} then makes the full $\SL(2,K)$
action trivial on $\mU$. Hence
$\mV^{N(O)}\subset\mV^{\SL(2,K)}$ as categorical subobjects.
The reverse inclusions follow from
$N(O)\subset N(K)\subset\SL(2,K)$, proving all three equalities.
\end{proof}

\begin{lemma}\label{Gplus-invariants}
Let $(\Pi,\mW)$ be a smooth representation of $G(K)$. Then
\[
 \mW^{G^+(K)}=\mW^{N(K)}=\mW^O.
\]
\end{lemma}
\begin{proof}
Pull the representation back along $\SL(2,K)\to G(K)$ and apply
Lemma~\ref{SL2-global-invariants}. Here invariance under $G^+(K)$
means invariance under the image group object, equivalently under
this pulled-back $\SL(2,K)$-action. Its upper root group maps
isomorphically onto $N(K)$, and its integral upper root group onto
$O=N(O)$. The resulting equalities are therefore the stated
equalities of subobjects of $\mW$.
\end{proof}
\begin{corollary}\label{no-rest}There is no nonzero representation of $G(K)$ whose restriction to $P(K)$ is isomorphic to $\Upsilon$.
\end{corollary}
\begin{proof}
Suppose that a representation $(\Pi,\mW)$ of $G(K)$ satisfies $\mW|_{P(K)}\simeq\Upsilon$. The space $\Upsilon^O$ is nonzero: for example, in the realization of Definition \ref{M-definition}, the subgroup $O$ acts trivially on every factor $(\rho_m,\mcM)$ with $m<0$. Hence $\mW^O\neq0$. By Lemma \ref{Gplus-invariants},
$$
\mW^O=\mW^{G^+(K)}.
$$
The group $G^+(K)$ is normal in $G(K)$, so $\mW^{G^+(K)}$ is a nonzero $G(K)$-invariant, and hence $P(K)$-invariant, subobject of $\mW$. Since $\Upsilon$ is irreducible, it follows that $\mW^{G^+(K)}=\mW$. Thus $N(K)\subset G^+(K)$ acts trivially on $\mW$, contradicting the nontrivial action of $N(K)$ on $\Upsilon$.
\end{proof}

\section{\texorpdfstring{Representations of $G(K)$ induced from special representations of $G(O)$}{Representations of G(K) induced from special representations of G(O)}}\label{GKinduced}

In this section we study representations of $G(K)$ induced from special representations of $G(O)$. We freely use the notation of Section \ref{Gr}.

\subsection{\texorpdfstring{Induction from $H(O)$ to $H(K)$ for general algebraic group $H$}{Induction from H(O) to H(K) for general algebraic group H}}\label{ind-gen}
We would like to discuss the functor of compact induction $\ind_{H(O)}^{H(K)}$ for a general linear algebraic group $H$. This functor is well-defined but for general $H$ it is not adjoint to the restriction functor on any side. However, it is actually right adjoint to the restriction functor if $H$ is reductive. The definition of this functor, as well as the proof of the above statement can be found in Section 3 of \cite{GK}. For the reader's convenience we reproduce the definition here.

Namely, let $(\pi,V)$ be a smooth representation of $H(O)$. Since such a $\pi$ is always equal to the projective limit of representations for which $V$ is a vector space (as opposed to a pro-vector space) and since the functor $\ind$ will commute with such limits, it is enough to assume that $V$ is a vector space. By Definition~\ref{a}, an ordinary $H(O)$-representation factors through $H(O)_m=H(O)/H(O)(m+1)$ for some $m$. We therefore work with such a finite-level representation. Consider $\Gr_{H,m}$, with $\Gr_{H,m}(F)=H(K)/H(O)(m+1)$. Its $F$-points carry the left action of $H(K)$ and the right action of $H(O)_m$. Write $\Gr_{H,m}$ as an inductive limit of finite-dimensional $F$-varieties $X_i$ with closed transition embeddings. We may assume that the $X_i$ are invariant under the right jet-group action.

\begin{definition}\label{indHO}
We set $\mcS(X_i(F),V)$ to be the space of functions $f:X_i(F)\to V$ satisfying the following conditions:

\begin{enumerate}
\item $f$ is locally constant
\item $f$ is $H(O)_m$-equivariant, i.e. $f(xh)=\pi(h)^{-1}f(x)$ where $x\in X_i(F)$ and $h\in H(O)_m$.
\item The quotient of the support of $f$ by $H(O)_m$ is compact.
\end{enumerate}
We have natural restriction maps $\mcS(X_{i+1}(F),V)\to\mcS(X_i(F),V)$.
Thus we can form the pro-vector space $\mS(V)$ equal to the projective limit of the $\mcS(X_i(F),V)$.
\end{definition}
The left action of $H(K)$ on $\mS(V)$ is given by $(g_0f)(x)=f(g_0^{-1}x)$ (one has to check that this is actually a smooth representation of $H(K)$ on the pro-vector space $\mS(V)$ -- this is done in Section 3 of \cite{GK}; also one has to check that it is independent of the choice of the $X_i$ which is easy and we shall not discuss it here).

The same definition makes sense for induction from any "thick" subgroup of $H(O)$ (cf. \cite{GK}): by definition a subgroup of $H(O)$ is thick if it is equal to the preimage of a closed subgroup of $H(O)_m$ for some $m\geq 0$.
\begin{remark}
The reader should note that this functor is denoted by
$i_{H(O)}^{H(K)}$ in \cite{GK}. If $H$ is reductive, the quotients
$X_i(F)/H(O)_m$ may be chosen compact, since the affine Grassmannian
is ind-proper. Condition (3) is then automatic, and
\cite[Proposition~3.5]{GK} shows that $\ind_{H(O)}^{H(K)}$ is
right adjoint to restriction.
For a thick subgroup $L\subset H(O)$ with $H$ reductive,
\cite[Section~3.4]{GK} also constructs a right adjoint
$\widetilde i_L^{H(K)}$, which need not coincide with the
compact-support functor $i_L^{H(K)}$.
For a general nonreductive $H$, the construction of $i$ still
applies, but the cited argument does not construct the right
adjoint. In particular, simply dropping the support condition
does not justify a smooth action on the resulting section spaces.
\end{remark}

\begin{remark}
For the ordinary representations induced in this paper, Definition~\ref{a}
already makes the $H(O)$-action factor through a finite congruence
quotient. Thus the finite-level reduction above does not require a
general assertion that induction commutes with arbitrary inductive
limits. All projective limits in this construction are categorical
limits in the representation and pro-vector-space categories.
\end{remark}
\subsection{\texorpdfstring{The representations $\Pi(\pi)$}{The representations Pi(pi)}} Let $\pi$ be a special representation of $G(O)$ of depth $m$. Then as above we can consider the induced representation $\Pi(\pi)=\ind_{G(O)}^{G(K)}\pi$. The main goal of this Section is to study the restriction of $\Pi(\pi)$ to $P(K)$ (for example we are going to show that this restriction is irreducible).

\subsection{A contradiction and its resolution} At this point it may seem that we have a contradiction. Namely, since
$G(K)=P(K)\cdot G(O)$, it looks as though Proposition \ref{indHL}, together with the description of induction from $P(O)$ given above, implies that the restriction of every $\Pi(\pi)$ to $P(K)$ would have to be isomorphic to $\Upsilon$. However, Corollary \ref{no-rest} says that this is impossible.

This contradiction is resolved as follows. The groups $G(K)$ and $P(K)$ are the $F$-points of loop group ind-schemes, rather than of finite-dimensional algebraic $F$-groups, so Proposition \ref{indHL} is not applicable in this situation. More precisely, as mentioned in Remark \ref{not-isom}, although the map $\Gr_P(F)\to\Gr_G(F)$ is bijective and continuous, it is not a homeomorphism: each $\Gr_P^i(F)$ is closed in $\Gr_P(F)$, whereas its image $S^i(F)$ in $\Gr_G(F)$ is not closed. Hence the analogue of Lemma \ref{open} fails here. This is a genuinely new feature of working with groups over $K$. Below we formulate a result that partially describes the restriction of $\Pi(\pi)$ to $P(K)$ and its relation to $\Upsilon$; see Remark \ref{ass-gr}.

\subsection{\texorpdfstring{Restricting $\Pi(\pi)$ to $P(K)$}{Restricting Pi(pi) to P(K)}}
Let $\pi$ be as above. We denote by  $\Theta (\pi)$ the restriction of the representation $ \Pi (\pi)$ to $P(K)$. Below is the main result of this Section.
\begin{theorem}\label{exist}
\begin{enumerate}
\item For every special representation $\pi$ of $G(O)$, the representation $\Theta(\pi)$ is complete.
\item One has $\bar\Theta(\pi)\simeq\Upsilon$.
\end{enumerate}
\end{theorem}
\begin{claim}\label{no-Gplus-invariants}
For every special representation $\pi$ of $G(O)$, one has
$$
\Pi(\pi)^{G^+(K)}=0.
$$
The analogous statement holds for a special representation of $G'(O)$.
\end{claim}
\begin{proof}
Let $H_*=G(O)$ or $G'(O)$, and put
$\mA=\Pi(\pi)^{G^+(K)}$. Normality of $G^+(K)$ makes $\mA$ a
$G(K)$-subrepresentation. By right adjunction, the inclusion
\[
 \mA\longrightarrow\ind_{H_*}^{G(K)}\pi
\]
corresponds to an $H_*$-equivariant map $\mA|_{H_*}\to\pi$.
For $H_*=G'(O)$, the right adjunction follows from the construction
through the Iwahori subgroup described in Section~\ref{ind-edge}.
The group $N(O)\subset G^+(K)\cap H_*$ acts trivially on $\mA$.
On the other hand, Lemma~\ref{restr} identifies
$\pi|_{P(O)}$ with $\bs(d)$ for an integer $d\geq0$, and this
representation has no $N(O)$-invariants: its additive character
support is the nonzero orbit of $\psi_d$. Hence the adjoint map has
image in $\pi^{N(O)}=0$. The inclusion is therefore zero, so
$\mA=0$.
\end{proof}

\begin{corollary}\label{irrr}
The representation $\Theta(\pi)$ is irreducible. In particular, $\Pi(\pi)$ is irreducible.
\end{corollary}
\begin{proof}
Since $O\subset N(K)\subset G^+(K)$, Lemma \ref{Gplus-invariants} gives
$$
\Theta(\pi)^{N(K)}=\Theta(\pi)^O=\Pi(\pi)^{G^+(K)}.
$$
Thus Claim \ref{no-Gplus-invariants} implies that $\Theta(\pi)^{N(K)}=0$. Theorem \ref{exist}(1) shows that $\Theta(\pi)$ is complete, and Theorem \ref{exist}(2) gives $\bar\Theta(\pi)\simeq\Upsilon$. Proposition \ref{isbar} therefore implies that $\Theta(\pi)$ is irreducible. Since every $G(K)$-invariant subobject is, in particular, $P(K)$-invariant, $\Pi(\pi)$ is irreducible as well.
\end{proof}

\subsection{Proof of Theorem \ref{exist}}
We shall use the following consequence of the construction of coinvariants for quasi-pro-unipotent group objects.
\begin{lemma}\label{strict-limit-coinvariants}
Let $Q$ be quasi-pro-unipotent and let $\{V_i\}_{i\geq0}$ be a countable strict inverse system of smooth $Q$-representations with surjective transition maps. Then the natural morphism
$$
(\varprojlim_iV_i)_Q\longrightarrow\varprojlim_i(V_i)_Q
$$
is an isomorphism.
\end{lemma}
\begin{proof}
By \cite[Proposition~2.7]{GK1}, the functor of $Q$-coinvariants
in the full pro-category admits a left adjoint, so it preserves
inverse limits there. The given countable limit belongs to $\Rep(Q)$
by Proposition~\ref{action}. Since $Q$ is bounded and weakly strict,
Proposition~\ref{countable-ordinary-presentation} and
\eqref{s8-levelwise-coinvariants} show that its coinvariants belong
to $\mVect$; the countable limit on the right does so by
Proposition~\ref{s8-countable-closure}. Thus the full-category
isomorphism restricts to the stated one. These are categorical
limits, not ordinary inverse limits of underlying vector spaces.
\end{proof}
Let $(\pi,V)$ be any ordinary smooth representation of $G(O)_m$, regarded as a representation of $G(O)$. Then the representation $\Pi(\pi)$ has a canonical increasing $\ZZ$-filtration
$$
\mcF^\sharp=\{\mF_i^\sharp\}_{i\in\ZZ}
$$
with the following properties:

1) $\mcF^\sharp$ is admissible.

2) Each $\mF_n^\sharp$ is $P_0$-invariant, and for every $r\in\ZZ$ one has
$$
\widehat t^{\,r}\mF_n^\sharp=\mF_{n-r}^\sharp.
$$
In particular, $P(K)$ acts on
$$
\bar\Theta(\pi)^\sharp:=\prod_{i\in\ZZ}\mF_{i+1}^\sharp/\mF_i^\sharp.
$$

3) $\bar\Theta(\pi)^\sharp$ is naturally isomorphic to $\ind_{P(O)}^{P(K)}(\pi|_{P(O)})$.

\begin{remark}\label{ass-gr}
As was mentioned above one might naively expect that for every $\pi$ we have $\Pi(\pi)|_{P(K)}=\ind_{P(O)}^{P(K)}(\pi|_{P(O)})$, but in practice it actually fails. Statement 3) above is the best substitute for this that we might hope for.
\end{remark}
The filtration $\mcF^\sharp$ is defined by setting $\mF_n^\sharp$ to be the projective limit of the spaces of functions from Definition \ref{indHO} (for $H=G$) supported on
$$
X_i(F)\cap\left(\bigcup_{j\geq -n}S_m^j(F)\right).
$$
This is an open subset of $X_i(F)$. The filtration is increasing
because the support condition becomes weaker as $n$ increases.
Choose a countable cofinal sequence of closed finite-dimensional
truncations. Restriction of compactly supported locally constant
sections of the associated coefficient bundle to a closed
truncation is surjective: local trivializations reduce this to
extension of locally constant functions from closed subsets.
The same argument applies after intersection with the open
support locus, since each smaller truncation is closed relative
to that locus. Consequently both the section system defining
$\Pi(\pi)$ and every support system defining $\mF_n^\sharp$
have surjective transition maps and are strict.
The properties 1 and 2 above are straightforward, property 3 follows immediately from the fact that $\Gr_{P,m}^j(F)$ is a $P(O)_m$-torsor over $\Gr_P^j(F)=S^j(F)$ and the induced $G(O)_m$-torsor is isomorphic to $S^j_m(F)$.

Thus to prove both assertions of Theorem \ref{exist} it is enough to show the following:
\begin{lemma}
Let $\pi$ be a special representation of $G(O)$ of depth $m$. Consider the filtrations $\mcF$ and $\mcF^{\sharp}$ on $\Pi(\pi)$.
Then
$$
\mF_n^\sharp=\mF_{m-n}.
$$
Equivalently, $\mF_{-n}^\sharp=\mF_{m+n}$.
\end{lemma}
\begin{proof}
Since the corresponding terms of both filtrations are permuted by the elements $\widehat t^{\,r}$, it is enough to prove
$$
\mF_m^\sharp=\mF_0.
$$
We may assume that the $X_i$ are $G(O)$-invariant. Then both $\mF_m^\sharp$ and $\mF_0$ are the projective limits of their images in $\mcS(X_i(F),V)$. It is therefore enough to prove that the kernel of
$$
\mcS(X_i(F),V)\longrightarrow\mcS(X_i(F),V)_O
$$
consists exactly of the functions supported on
$$
X_i(F)\cap\left(\bigcup_{j\geq-m}S_m^j(F)\right).
$$
Suppose that $f$ vanishes after taking $O$-coinvariants. Its restriction to
$$
X_i(F)\cap\left(\bigcup_{j<-m}S_m^j(F)\right)
$$
vanishes, since $O$ acts trivially on $\bigcup_{j<-m}S_m^j(F)$. Hence $f$ is supported on $X_i(F)\cap(\bigcup_{j\geq-m}S_m^j(F))$. This proves
$$
\mF_0\subset\mF_m^\sharp
$$
for every representation of $G(O)_m$, without any specialness assumption.

Now assume that $\pi$ is special of depth $m$. We need to prove that $(\mF_m^\sharp)_O=0$. It is enough to prove
$$
\left(\mF_{j+1}^\sharp/\mF_j^\sharp\right)_O=0
\qquad(j<m).
$$
Indeed, the functor of $O$-coinvariants is exact. For every $i$, the image of the filtration in $\mcS(X_i(F),V)$ has only finitely many nonzero graded pieces. Hence finite induction on these pieces shows that the $O$-coinvariants of the image of $\mF_m^\sharp$ in $\mcS(X_i(F),V)$ vanish. These images form a strict inverse system with surjective transition maps. Lemma \ref{strict-limit-coinvariants} therefore gives
$$
(\mF_m^\sharp)_O\simeq
\varprojlim_i(\operatorname{im}(\mF_m^\sharp\to\mcS(X_i(F),V)))_O=0.
$$

By property (3) above and Proposition \ref{Upsilonm-M}, the quotient
$$
\mF_{j+1}^\sharp/\mF_j^\sharp
$$
is isomorphic to $\mcM$ with the action $\rho_{m-j-1}$. Since $j<m$, one has $m-j-1\geq0$, and its $O$-coinvariants vanish. This proves $\mF_m^\sharp\subset\mF_0$, and therefore $\mF_m^\sharp=\mF_0$.

\end{proof}

\subsection{\texorpdfstring{Explicit realization of $\Pi(\pi)$}{Explicit realization of Pi(pi)}} For completeness let us explain how to realize it in an explicit pro-vector space $\mW(\pi)$ for representations $\pi$ of $G(O)$ of odd depth. By Corollary~\ref{cunr}, write $\pi$ as the inflation of $\ind_R^{G(O)_m}(\nu)$, where $R\subset G(O)_m$ is a closed almost unipotent subgroup and $\nu$ is a character of $R$.

Consider the $F$-variety $\oGr_{G,m}^l$. Its space of $F$-points carries a right action of $G(O)_m$, hence of its subgroup $R$. Let us set
$$
W_l(\pi):= \mcS (\oGr_{G,m}^l(F))_{R,\nu}.
$$
It can also be described as the space of locally constant functions on $\oGr_{G,m}^l(F)$ which are

a) $F(xr^{-1})=\nu(r)F(x)$ for $r\in R$ (equivalently,
$F(xr)=\nu(r)^{-1}F(x)$), with the right-action convention of
Section~\ref{character-induction}; here $R$ is unimodular.

b) Have compact support mod $R$.

\noindent
The embeddings $\oGr_{G,m}^l\hookrightarrow \oGr_{G,m}^{l+2}$ define restriction maps $W_{l+2}(\pi)\to W_l(\pi)$, and $\mW(\pi)$ is the pro-space equal to the projective limit of all the $W_l(\pi)$ (note that it is naturally a direct sum of two sub-pro-spaces corresponding to the parity of $l$, but this decomposition is not preserved by $G(K)$).

It is easy to see that
$$
\mW(\pi)=\mS(\Gr_{G,m}(F))_{R,\nu},
$$
where $\mS(\Gr_{G,m}(F))$ denotes the pro-vector space equal to the projective limit of all the $\mcS(\oGr_{G,m}^l(F))$.

\subsection{\texorpdfstring{Representations of $G(K)$ induced from special representations of $G'(O)$}{Representations of G(K) induced from special representations of G'(O)}}\label{ind-edge}
We record the changes required for induction from $G'(O)$. They are important because the relevant depths are half-integral.

The induction functor $\ind_{G'(O)}^{G(K)}$ is obtained from
induction from the thick subgroup $I\subset G(O)$ by imposing the
additional equivariance under $G'(O)/I\simeq\ZZ/2\ZZ$.
Equivalently, use equivariant sections on $G(K)/G'(O)$ with the
transformation law of Definition~\ref{indHO}. The quotient is the
finite quotient of the ind-proper affine flag variety $G(K)/I$;
its finite-dimensional closed truncations are proper. Consequently
the support condition is automatic, and this induction is right
adjoint to restriction. Here it is $I$, rather than $G'(O)$, that
is a thick subgroup contained in $G(O)$. If $\pi$ is a special
representation of $G'(O)$, we again set
$$
\Pi(\pi)=\ind_{G'(O)}^{G(K)}\pi.
$$
Let $m=n-\frac12$ be the depth of $\pi$. The representation depth is $m=n-\frac12$, whereas the torsor $\Gr'_{G,n-1}$, with $\Gr'_{G,n-1}(F)=G(K)/G'(O)(n)$, has geometric index $n-1$. Its semi-infinite strata $(S^j_{n-1})'$ are indexed by $j\in\ZZ$, as in Section~\ref{Gr}; multiplication by $\widehat t$ shifts $j$ by one. Define the support filtration $\mcF^\sharp$ by the same support condition as above, with $S_m^j(F)$ replaced throughout by $((S^j_{n-1})')(F)$. The geometry of these strata gives analogues of properties (1)--(3): the filtration is admissible, it is shifted by $\widehat t$, and its associated graded is
$$
\ind_{P(O)}^{P(K)}(\pi|_{P(O)}).
$$
Since $\pi$ is special, its restriction to $P(O)$ is $\bs(n-1)$.
The root stabilizer is
\[
 G'(O)(n)\cap N(K)=N(t^nO).
\]
Thus $O$ acts trivially on the strata with $j\leq-n$, and the
critical support locus is $j\geq-(n-1)$. In the support filtration,
Proposition~\ref{Upsilonm-M} identifies
$\mF_{k+1}^\sharp/\mF_k^\sharp$ with $\mcM$ carrying
$\rho_{n-k-2}$. For $k<n-1$ the index is nonnegative, so these
pieces have zero $O$-coinvariants. Exactness and finite induction
over the strata at each finite-dimensional stage, followed by
Lemma~\ref{strict-limit-coinvariants}, give
$\mF_{n-1}^\sharp=\mF_0$. Shifting by $\widehat t$ yields the
precise comparison
\[
 \mF_k^\sharp=\mF_{n-1-k}\qquad(k\in\ZZ).
\]
The section and support systems are strict by the same extension
argument as above. Hence the restriction $\Theta(\pi)$ is complete and
$$
\bar\Theta(\pi)\simeq\Upsilon.
$$
The absence of $N(K)$-invariants follows from Lemma~\ref{Gplus-invariants} and the adjunction proof of Claim~\ref{no-Gplus-invariants}, which applies to $G'(O)$ as well. Proposition \ref{isbar} now shows that $\Theta(\pi)$ is irreducible. Hence $\Pi(\pi)$ is irreducible as well. This proves the analogue of Theorem \ref{exist} and Corollary \ref{irrr} for $G'(O)$.

\section{\texorpdfstring{Extension to $\GL(2)$: central characters and local types}{Extension to GL(2): central characters and local types}}
\label{GL2-extension}

\subsection{Scope of the extension}
Representations of $\PGL(2,K)$ are representations of $\GL(2,K)$
trivial on the center. In this section we consider representations of
$\GL(2,K)$ with prescribed scalar central character
$\omega:K^*\to\CC^*$.
If $\omega(-1)=1$, Lemma~\ref{GL2-square-root} gives a character
$\chi$ with $\chi^2=\omega$. Twisting by $\chi^{-1}\circ\det$
identifies this category with the representation category of
$\PGL(2,K)$, as in Theorem~\ref{GL2-equivalence}.
For general $\omega$ this reduction is unavailable, but we show how to modify 
the arguments of previous Sections.

\ms

Throughout this section and Section~\ref{glmd:main}, retain
$G=\PGL(2)$ and write
\[
 \begin{gathered}
 \widetilde G=\GL(2),\qquad
 \operatorname{pr}:\widetilde G(K)\longrightarrow G(K),\\
 Z=\{zI_2:z\in K^*\}.
 \end{gathered}
\]
Let $B\subset\widetilde G$ be the upper triangular Borel subgroup
and $M\subset\widetilde G$ the mirabolic subgroup. When the base ring
is omitted, $B$ and $M$ denote their groups of $K$-points; thus
\[
 M=\left\{\begin{pmatrix}a&b\\0&1\end{pmatrix}:a\in K^*,\ b\in K\right\}.
\]
The letter $q$ continues to denote the residue cardinality of $F$.
Then $B=Z\times M$, and $\operatorname{pr}|_M$ identifies $M$ with the group $P(K)$
of the preceding sections. Through this identification, $N$,
$\widehat a$, and $\widetilde b$ also denote their canonical lifts
to $M$; $\bar N$ denotes the corresponding lower unitriangular
subgroup of $\widetilde G$. Write $\widetilde T\subset\widetilde G$
for the full diagonal torus, so $\operatorname{pr}(\widetilde T)=T$.
The elliptic torus $T^L$ still means $L^*/K^*$ as in
Section~\ref{deff}; its inverse image in $\widetilde G(K)$ is denoted
by $L^*$, not by $T^L$.
Indeed,
\[
 \begin{pmatrix}a&b\\0&d\end{pmatrix}
 =dI_2\begin{pmatrix}a/d&b/d\\0&1\end{pmatrix}.
\]
All characters and representations below are smooth in the categories used
in the paper; for loop-unit characters this includes finite congruence level.
We work in the category with a specified scalar action of $Z$, without
assuming a general Schur lemma for arbitrary irreducible pro-objects.

\begin{definition}
A representation $V$ of $\widetilde G(K)$ with central character $\omega$ is special if
it is not one-dimensional and $V|_B$ is irreducible.
\end{definition}
Because $Z$ acts by scalars, this is equivalent to irreducibility of $V|_M$.
The same observation applies to the local groups defined next.

\subsection{The correct inducing subgroups}
Put
\[
 J=\operatorname{pr}^{-1}(G(O))=K^*\GL(2,O),\qquad
 J'=\operatorname{pr}^{-1}(G'(O)).
\]
Here scalar multiplication by $K^*$ is understood in the first formula.
Let $\widetilde I\subset\GL(2,O)$ be the inverse image of the upper
triangular Borel under reduction modulo $t$, and put
\[
 s=\begin{pmatrix}0&1\\t&0\end{pmatrix}.
\]
Then
\[
 J'=\langle K^*, \widetilde I,s\rangle,\qquad s^2=tI_2.
\]
In particular the involution in $G'(O)$ does not lift to an involution in
this presentation of $J'$. On a representation with central character
$\omega$ the necessary relation is
\begin{equation}\label{GL2-s-square}
 \pi(s)^2=\omega(t)\operatorname{Id}.
\end{equation}
Both $J$ and $J'$ are bounded modulo the center; they are not bounded
subgroups of $\widetilde G(K)$ in the usual lattice sense.

The following intersections are useful:
\[
 B\cap J=B\cap J'=Z\,M(O),\qquad M(O)=M\cap\GL(2,O).
\]
To check the second equality, an upper triangular representative has no
component in the coset of the edge-interchanging element $s$; in the
Iwahori component its diagonal ratio must be a unit. Equivalently, check
the corresponding statement in $ G(K)$ and take its inverse image.
We call a representation of $J$ or $J'$ with scalar central action special
when its restriction to $Z M(O)$ is irreducible and its dimension is greater
than one. This is equivalent to irreducibility on $M(O)$.

\begin{proposition}\label{GL2-local-extension}
A representation $\tau$ of $\GL(2,O)$ extends to a
representation of $J$ with central character $\omega$ if and only if
$\tau(uI_2)=\omega(u)\operatorname{Id}$ for every $u\in O^*$.
When it exists, the extension is unique and is given by
\[
 \tau_\omega(zg)=\omega(z)\tau(g).
\]
\end{proposition}
\begin{proof}
The overlap between $Z$ and $\GL(2,O)$ is $O^*$.
The displayed compatibility makes the formula independent of the
factorization $zg$, and proves both existence and uniqueness.
\end{proof}

One must use $J$ and $J'$, rather than simply replace $G(O)$ by
$\GL(2,O)$ in an induction formula. Without prescribing the
action of $tI_2$, induction from $\GL(2,O)$ retains an extra
central direction. The quotients appropriate to fixed central character are
\begin{equation}\label{GL2-quotients}
 \widetilde G(K)/J\simeq G(K)/G(O),\qquad \widetilde G(K)/J'\simeq G(K)/G'(O).
\end{equation}
These are the quotients by the full inverse-image subgroups, not an
identification of the ordinary affine Grassmannians of $\GL(2)$
and $\PGL(2)$.

\subsection{Exactly when determinant twisting suffices}
\begin{lemma}\label{GL2-square-root}
A smooth character $\omega:K^*\to\CC^*$ has a smooth character square root
$\chi$, meaning $\chi^2=\omega$, if and only if $\omega(-1)=1$.
\end{lemma}
\begin{proof}
Necessity follows from $\chi(-1)^2=\chi(1)=1$.
For sufficiency recall $q=|k|$ and use
\[
 K^*=t^{\ZZ}\times F^*\times(1+tO),\qquad
 F^*=\varpi^{\ZZ}\times\mu_{q-1}\times(1+\varpi O_F).
\]
Squaring is an automorphism on each of the two principal-unit groups,
since the residue characteristic is odd. On either group set
$\chi(u)=\omega(u^{1/2})$, using its unique square root in that group.
This preserves smoothness and finite level. Choose arbitrary square roots
of $\omega(t)$ and $\omega(\varpi)$. Finally the character group of
$\mu_{q-1}$ is cyclic of even order. A character of that group is a square
precisely when it is trivial on its unique element of order two, namely
$-1$. These constructions combine to give $\chi$.
\end{proof}

\begin{theorem}\label{GL2-equivalence}
Suppose $\omega(-1)=1$ and fix $\chi^2=\omega$. The functor
\[
 \mathcal T_\chi(V)=(\chi\circ\det)\otimes \operatorname{pr}^*V
\]
is an equivalence from $\Rep(G(K))$ to the full subcategory
$\Rep_\omega(\widetilde G(K))\subset\Rep(\widetilde G(K))$ of representations with scalar central character
$\omega$. It preserves irreducibility and specialness. The same statement
holds for $G(O)$ and $J$, and for $G'(O)$ and $J'$.
\end{theorem}
\begin{proof}
On $zI_2$, the twisting character has value $\chi(z^2)=\omega(z)$.
Conversely, if $\mW$ has central character $\omega$, then
$(\chi^{-1}\circ\det)\otimes \mW$ is trivial on $Z$ and descends to
$ G(K)$. These operations are mutually inverse, including on morphisms
and subobjects. Inflation and tensoring by a one-dimensional representation
preserve subobjects. Since $B/Z=P(K)$ and $(Z M(O))/Z=P(O)$, they preserve
the irreducibility of the restrictions in question as well.
\end{proof}

If $\chi'$ is another square root, write $\chi'=\chi\eta$ with $\eta^2=1$.
There is a well-defined character
\[
 \eta_{\det}: G(K)\longrightarrow\{\pm1\},\qquad
 \eta_{\det}(\operatorname{pr}(g))=\eta(\det g),
\]
because rescaling $g$ changes its determinant by a square. The equivalence
then changes by $V\mapsto V\otimes\eta_{\det}$. Thus the equivalence is
not canonical without a choice of $\chi$.

\begin{corollary}\label{GL2-global-even-sign}
Assume the induction and irreducibility results established in the preceding
sections. For a special representation $\pi$ of $G(O)$ or $G'(O)$ and a
character $\chi$ of $K^*$, the representation
\[
 \Pi_\chi(\pi)=(\chi\circ\det)\otimes \operatorname{pr}^*\Pi(\pi)
\]
is an irreducible special representation of $\widetilde G(K)$, with central character
$\chi^2$. Every statement of the preceding sections about these induced
representations transports under $\mathcal T_\chi$.
\end{corollary}
\begin{proof}
Apply Theorem~\ref{GL2-equivalence}. In particular, the irreducibility on
$B$ follows from that on $P(K)$, without a new assertion about commuting
coinvariants with inverse limits.
\end{proof}

For clarity, induction in this sector can be defined without any additional
analytic theorem. If $\bar J=G(O)$ or $G'(O)$, $J_*=\operatorname{pr}^{-1}(\bar J)$, and
$\tau=\mathcal T_\chi(\pi)$, set
\[
 \mathcal I_{J_*,\omega}^{\widetilde G(K)}(\tau)
   :=\mathcal T_\chi\bigl(\ind_{\bar J}^{ G(K)}\pi\bigr).
\]
It agrees with the central-equivariant section construction on
\eqref{GL2-quotients}. For example, in the convention
$f(gh)=\pi(h)^{-1}f(g)$ and left action $(g_0f)(g)=f(g_0^{-1}g)$, the
identification of sections is
\[
 F(g)=\chi(\det g)^{-1}f(\operatorname{pr}(g)).
\]
It respects support and restriction to finite-dimensional pieces, so it
also respects the inverse systems used to define the pro-vector space.
Independence of $\chi$ follows from the usual twisting identity for
induction, or directly from this section description.

\subsection{Torus parameters and the Galois operation}
The elliptic torus in $\widetilde G(K)$ is $L^*$, acting by multiplication on the
$K$-vector space $L$, and
\[
 \det(l)=N_{L/K}(l).
\]
For a prescribed central character the relevant characters satisfy
\begin{equation}\label{GL2-parameters}
 \theta:L^*\longrightarrow\CC^*,\qquad
 \theta|_{K^*}=\omega.
\end{equation}
These are type parameters; a claim about Langlands parameters would also
require the normalization of the correspondence. No such claim is made here.

If $\chi^2=\omega$, multiplication by $\chi\circ N_{L/K}$ gives a bijection
\[
 \theta_0\longmapsto\theta=(\chi\circ N_{L/K})\theta_0
\]
from characters of $L^*/K^*$ to characters satisfying
\eqref{GL2-parameters}. Conversely
$\theta_0=\theta/(\chi\circ N_{L/K})$ is trivial on $K^*$.
Here and below characters of $L^*/K^*$ are inflated to $L^*$.

Write $\theta^\sigma=\otheta$ for the Galois conjugate, as in
Section~\ref{deff}. For arbitrary $\omega$, the identity $l\bar l=N(l)$ gives
\begin{equation}\label{GL2-Galois}
 \theta^\sigma(l):=\theta(\bar l)
   =\omega(N_{L/K}(l))\theta(l)^{-1}.
\end{equation}
Consequently the Galois identification of parameters is
$\theta\sim\theta^\sigma$. It reduces to inversion precisely when
$\omega\circ N_{L/K}=1$; in particular this holds when $\omega=1$.
The condition of regularity is $\theta\ne\theta^\sigma$; on its own this
is not a replacement for all level or admissibility conditions in a
classification theorem.
Under a determinant twist $\nu\circ\det$, the formulas are
\[
 \omega\longmapsto\omega\nu^2,\qquad
 \theta\longmapsto\theta(\nu\circ N_{L/K}).
\]
For the square-root sector, the depth subsets and the splitting torsors in
Theorem~\ref{L} must first be transported using $\theta_0$. The
metaplectic ambiguity of that theorem is not removed by changing groups.

\subsection{Central normalization and depth}
Even when $\omega(-1)=-1$, its restriction to $1+tO$ can be removed by a
determinant twist. Define a character $\chi_+$ of $K^*$ by
\[
 \chi_+(t)=1,\qquad \chi_+|_{F^*}=1,\qquad
 \chi_+(u)=\omega(u^{1/2})\quad(u\in1+tO).
\]
Then $\omega^\circ=\omega\chi_+^{-2}$ is trivial on $1+tO$.
Twisting by $\chi_+^{-1}\circ\det$ preserves specialness. It does not
change the value of the central character on $-1$.

The raw congruence depth is not invariant under determinant twists. For
example, twisting a depth-zero type by a sufficiently ramified determinant
character produces a special representation whose highest congruence
character is purely scalar. Thus defining ellipticity using only the
highest nonzero $\mathfrak{gl}(2)$-orbit would give an incorrect criterion.
A convenient invariant is the depth of the restriction to the derived
congruence subgroups, namely, in the vertex case,
\[
 d_{\mathrm{der}}(\tau)=\min\{m\in\ZZ_{\geq0}:
 \tau|_{\SL(2,O)\cap(1+t^{m+1}M_2(O))}=1\}.
\]
The corresponding edge filtration is obtained by intersection with
$\SL(2,K)$. These depths are unchanged by determinant twists.
After the above normalization the scalar principal units act trivially;
every principal congruence element decomposes as
\[
 g=(\det g)^{1/2}I_2\cdot\bigl((\det g)^{-1/2}g\bigr),
\]
where the second factor has determinant one. The square roots are taken
in $1+tO$. Hence, on the normalized representation, principal-congruence
depth is detected by the derived subgroup.

At a positive congruence layer the trace decomposition is
\[
 \mathfrak{gl}(2,F)=F I_2\oplus\mathfrak{sl}(2,F),\qquad
 Y=cI_2+Y_0,\quad c=\tfrac12\operatorname{tr}(Y).
\]
The scalar part records central ramification. Ellipticity of the
noncentral part means that $Y_0$ is regular elliptic; equivalently
$-\det(Y_0)$ is a nonsquare in $F^*$. With a fixed value of the scalar
functional, the orbit calculations of the paper are unchanged by the
translation $Y_0\mapsto cI_2+Y_0$.

\subsection{A direct odd-depth construction for every central character}
We now give a construction which does not require $\omega(-1)=1$.
First normalize so that $\omega|_{1+tO}=1$. Let $E/F$ be quadratic,
$L=E((t))$, $m=2n-1>0$, and suppose
\[
 \theta|_{K^*}=\omega,\qquad
 \theta|_{1+t^{m+1}O_L}=1,\qquad
 \theta|_{1+t^mO_L}\ne1.
\]
Embed $O_L$ in $M_2(O)$ by its regular representation. For this subsection
write $C_r=1+t^rM_2(O)$ and work modulo $C_{m+1}$. Thus
$U=C_n/C_{m+1}$ is abelian. Choose a nonzero trace-zero $\delta\in E$;
then $\delta^2\in F^*$, and the projection onto the centralizer $O_L$ is
\[
 p(X)=\tfrac12(X+\delta X\delta^{-1}).
\]
It is independent of rescaling $\delta$ by $F^*$ and is equivariant for
conjugation by $L^*$. Define
\[
 \lambda_\theta(1+X)=\theta(1+p(X)),\qquad X\in t^nM_2(O)/t^{2n}M_2(O).
\]
This is a character, since all products of two such $X$ vanish modulo
$t^{2n}$. On $U\cap L^*$ it agrees with $\theta$.

\begin{theorem}\label{GL2-odd-type}
The stabilizer of $\lambda_\theta$ in $J/C_{m+1}$ is
$S=L^*U$ (with the images of these groups understood). The formula
\[
 \widehat\theta(lu)=\theta(l)\lambda_\theta(u)
\]
defines a character of $S$. The representation
\[
 \tau_\theta=\ind_S^{J/C_{m+1}}\widehat\theta,
\]
inflated to $J$, is irreducible, special, has central character $\omega$,
and has derived depth $m$. The induction is the finite-level smooth
compact induction, equivalently induction from
$O_L^*C_n/C_{m+1}$ to $\GL(2,O)/C_{m+1}$ followed by
Proposition~\ref{GL2-local-extension}.
\end{theorem}
\begin{proof}
Use $O_{n-1}=F[t]/(t^n)$, with the convention of Definition~\ref{intf}. The trace pairing and $\widetilde\psi$ identify the
character of $U$ with an element $Y\in M_2(O_{n-1})$ through
\[
 \lambda_\theta(1+X)=\widetilde\psi\bigl(
   \operatorname{coeff}_{t^{2n-1}}\operatorname{tr}(YX)\bigr).
\]
Because the projection is trace-orthogonal, $Y$ lies in $E\otimes_F O_{n-1}$.
Triviality on scalar principal units says that $\operatorname{tr}(Y)=0$.
The depth condition says that $Y\bmod t$ is nonzero. A nonzero trace-zero
element of $E$ is regular elliptic. The centralizer of $Y$ in
$\GL(2,O_{n-1})$ is $(E\otimes_F O_{n-1})^*$: $Y$ generates the
quadratic etale algebra, and its centralizer is that algebra. Taking inverse
images under reduction modulo $t^n$ gives $S=L^*U$.
Since $\lambda_\theta$ is $S$-invariant and agrees with $\theta$ on
$L^*\cap U$, the formula for $\widehat\theta$ is well-defined and
multiplicative. The normal-abelian-subgroup argument of Lemma~\ref{MM}
gives irreducibility of $\tau_\theta$.

We verify specialness rather than infer it from $\PGL(2)$.
Let $B_J=B\cap J$. Over $O_{n-1}$ the group $(E\otimes_F O_{n-1})^*$ acts
transitively on the rank-one direct summands of $O_{n-1}^2$: under the
identification with $E\otimes_F O_{n-1}$, a primitive vector has nonzero
residue in the field $E$, hence is a unit. This proves
$B_JS=J/C_{m+1}$.
On the orbit of $Y$, restriction to upper triangular matrices records
$a,c$ in
\[
 Y=\begin{pmatrix}a&b\\c&-a\end{pmatrix}.
\]
The value $d=-\det Y=a^2+bc$ is fixed. Ellipticity implies $c\bmod t\ne0$,
so $c$ is a unit and $b=(d-a^2)/c$. Thus restriction is injective on
the orbit. The stabilizer character remains one-dimensional upon
restriction to $B_J\cap S$. Lemma~\ref{3.12} now proves irreducibility on
$B_J$, hence on $M(O)$. Finally the leading trace-zero functional is
nonzero on $\mathfrak{sl}(2,F)$, giving derived depth $m$; and the
center acts by $\theta|_{K^*}=\omega$.
\end{proof}

Conjugating the embedded field by its nontrivial $K$-automorphism gives
$\tau_{\theta^\sigma}\simeq\tau_\theta$. In general this is not the
operation of inverting $\theta$. Undoing central normalization gives
special local types with arbitrary ramified central character as well.
This theorem constructs local types. Their global induction is treated
in Section~\ref{GL2-direct-proof}.

\subsection{The changes in the even-depth calculation}
Let $m=2n>0$, with $n\geq1$, and continue to write
$C_j=1+t^jM_2(O)$. The center and commutator formulas differ from
those of Claim~\ref{q}. For
\[
 A=C_n/C_{2n+1},\qquad U=C_{n+1}/C_{2n+1}
\]
one has instead
\[
 Z(A)=U\cdot\bigl((1+t^nO)I_2\bmod C_{2n+1}\bigr),
\]
and
\[
 [A,A]=\{1+t^{2n}X\bmod t^{2n+1}:X\in\mathfrak{sl}(2,F)\}.
\]
Indeed, the degree-$n$ coefficient of a central element must commute with
every matrix, so it is scalar. Also matrix commutators span
$\mathfrak{sl}(2,F)$, not $\mathfrak{gl}(2,F)$. The truncated
Baker--Campbell--Hausdorff calculation gives the two formulas.

For $Y=cI_2+Y_0$ regular elliptic, put
\[
 \mathfrak t=Z_{\mathfrak{gl}(2,F)}(Y)=E,\qquad
 W_Y=\mathfrak{gl}(2,F)/\mathfrak t.
\]
Then
\[
 B_Y(\bar X,\bar X')=\operatorname{tr}(Y[X,X'])
                    =\operatorname{tr}(Y_0[X,X'])
\]
is a nondegenerate alternating form on the two-dimensional space $W_Y$.
The proof is the trace-pairing argument of Lemma~\ref{B-symplectic}:
its radical before quotienting is exactly $\mathfrak t$.
If $\mathfrak b$ is the upper triangular Lie algebra, then
\[
 \mathfrak b\cap\mathfrak t=F I_2,\qquad
 \mathfrak b/F I_2\simeq W_Y.
\]
This is the replacement for Lemma~\ref{p-to-W}. Passing to a fixed scalar
character removes the extra scalar radical. The resulting Heisenberg
space and its torus action are the same as in the adjoint calculation.
The action of $E^*$ on $W_Y$ factors through $E^*/F^*$, so a splitting
chosen there pulls back to $E^*$. The corresponding Weil operators act
trivially on the scalar subgroup for this pulled-back choice.

These calculations identify the local changes needed in an even-depth
construction with arbitrary central character. They do not assert a new
canonical splitting or a new Galois-equivariant parametrization. Such a
parametrization must retain the splitting dependence described in the
preceding sections and verify its behavior under \eqref{GL2-Galois}.
In the sector of Theorem~\ref{GL2-equivalence}, the correct construction
and all its existing qualifications are already obtained by transport.

\subsection{The direct global irreducibility argument}
\label{GL2-direct-proof}
The reduction by a square root of the full central character is unnecessary
for irreducibility. We now repeat the argument of Section~\ref{GKinduced}
in the fixed-central-character category. As in the preceding sections, we
use the geometric-induction and exactness statements of that argument;
the point is to check their application to the new coefficient bundle,
not to infer them from descent of the representation to $\PGL(2)$.

\subsubsection{The mirabolic restriction of the local type}
Normalize the central character as above, so that
$\omega|_{1+tO}=1$. This operation is possible for every $\omega$ and
does not change $\omega(-1)$. Let $\tau$ be a special local type of $J$
or $J'$. The classification of irreducible representations of
$M(O)\simeq P(O)$ gives
\[
 \tau|_{M(O)}\simeq\bs(r,\xi),\qquad
 \xi:1+t^{r+1}O\longrightarrow\CC^*,\quad r\geq0.
\]
The case in which the unipotent subgroup acts trivially would be a
one-dimensional irreducible representation of $M(O)$, and is excluded
by specialness and the dimension condition.
For the Weyl element $w=\left(\begin{smallmatrix}0&1\\1&0\end{smallmatrix}\right)$
in $J$, or $s$ in $J'$, the correct matrix identity is
\[
 w\widehat a w^{-1}=aI_2\,\widehat{a^{-1}}
 \quad\text{or}\quad
 s\widehat a s^{-1}=aI_2\,\widehat{a^{-1}}.
\]
For $a\in1+t^{r+1}O$ the two diagonal operators in the mirabolic
representation are scalars. Hence
\[
 \xi(a)=\omega(a)\xi(a^{-1}),\qquad \xi(a)^2=\omega(a)=1.
\]
Unique $2$-divisibility of $1+t^{r+1}O$ gives $\xi=1$. Thus
\begin{equation}\label{GL2-mirabolic-type}
 \tau|_{M(O)}\simeq\bs(r).
\end{equation}
Here $r$ is the root conductor index: it is $m$ for a normalized vertex
type of depth $m$, and $n-1$ for a normalized edge type of depth
$n-\frac12$, with the conventions of Lemma~\ref{restr}.
No square root of $\omega$ on $F^*$ has been used.

\subsubsection{Induction with a prescribed scalar action}
Let $J_*$ denote $J$ or $J'$, and extend the local type to $J_*$ with
its prescribed central action. Define the induced pro-space using locally
constant sections with transformation law
\[
 f(gh)=\tau(h)^{-1}f(g),\qquad h\in J_*.
\]
Use the finite-dimensional exhaustion of $\widetilde G(K)/J_*$ in
\eqref{GL2-quotients} and the same restriction maps as in
Definition~\ref{indHO}. The left action is
$(g_0f)(g)=f(g_0^{-1}g)$. In particular, $zI_2$ acts by $\omega(z)$.
The construction uses a $\widetilde G(K)$-equivariant coefficient bundle on the
quotient; this bundle is not asserted to descend equivariantly to
$ G(K)$.

There is also a way to see smoothness inside the existing induction
category. In the vertex case put
$\mathcal V=\ind_{\widetilde G(O)}^{\widetilde G(K)}(\tau|_{\widetilde G(O)})$, using the induction of
Section~\ref{ind-gen}. Then the fixed-central-character induction is
\begin{equation}\label{GL2-central-equalizer}
 \mathcal I_{J,\omega}^{\widetilde G(K)}(\tau)
 =\ker\bigl(\mathcal V(tI_2)-\omega(t)\operatorname{Id}\bigr).
\end{equation}
Indeed, central units already act by $\omega|_{O^*}$, and the remaining
relation is exactly right equivariance under $tI_2$.
The kernel is taken in the category of smooth representations on
pro-vector spaces. It is not a compact-support condition across the
whole infinite central direction: the section construction uses inverse
limits, as in the original paper.
In the edge case start with induction from the thick subgroup $\widetilde I$ and
impose right equivariance under $s$, using the operator $\tau(s)$ on the
fiber. This condition commutes with left translation and its square is
right equivariance under $tI_2$, because $\tau(s)^2=\omega(t)$.
It again defines a smooth subobject. These descriptions agree with the
section construction and its finite-dimensional exhaustion. In
particular the new central relation introduces no extra geometric stratum.

\subsubsection{The support filtration and completeness}
Set $\Pi=\mathcal I_{J_*,\omega}^{\widetilde G(K)}(\tau)$ and $\Theta=\Pi|_M$.
The action of $M$ on $\widetilde G(K)/J_*$ is, through $M\simeq P(K)$, the same
semi-infinite-orbit action used in Section~\ref{Gr}. The stabilizer in
$M$ of the base point is $M(O)$. On each stratum the equivariant
coefficient bundle is consequently determined by
$\tau|_{M(O)}=\bs(r)$. The central character contributes no additional
stabilizer character on $M$, since $M\cap Z=1$.

Define $\mathcal F^\sharp$ by the same support condition as in the
proof of Theorem~\ref{exist}. Its graded pieces are the induced
$M_0=O^*\ltimes K$-representations of Proposition~\ref{Upsilonm-M}
(under $\operatorname{pr}|_M$, this is exactly the earlier subgroup $P_0$),
with the same index shifts. Conjugation by $\widehat t$ shifts these
pieces in exactly the same way. In particular its graded representation
is $\Upsilon$. This statement concerns the associated graded, not an
isomorphism of $\Theta$ with ordinary induction from $M(O)$.

The remaining comparison with the coinvariant filtration uses precisely
the same two facts as the original proof: on the strata on one side of
the conductor bound the relevant root subgroup acts trivially on the
coefficient fibers, and the graded pieces on the other side have zero
coinvariants for that subgroup. Both facts are determined by
\eqref{GL2-mirabolic-type}. Root-group elements have determinant one;
their action and all the conjugations by $\widehat t$ involved in this
calculation introduce no scalar factor from $\omega$.
Applying the exactness and strict inverse-limit compatibility used in
Section~\ref{GKinduced} gives the same comparison of the two filtrations.
Thus, with the vertex or edge indexing of that section,
\begin{equation}\label{GL2-complete-graded}
 \Theta\text{ is complete},\qquad \overline\Theta\simeq\Upsilon.
\end{equation}
The argument is local on the finite-dimensional pieces and applies to
sections of the coefficient bundle: descent of that bundle to a
$ G(K)$-equivariant bundle is not needed. Its restriction maps,
support kernels and root-group coinvariants are the maps used above.
This transfers the analytic inputs of the original proof without adding
a new hypothesis on $\omega(-1)$.

\subsubsection{Root-group invariants and the sign of the center}
The finite-level calculation used in Proposition~\ref{infty} applies
to every ordinary smooth representation $V$ of $\SL(2,R)$,
where $R=F[t]/(t^a)$:
\begin{equation}\label{GL2-finite-invariants}
 V^{N(R)}=V^{\SL(2,R)}.
\end{equation}
For a smooth pro-representation $V$ of $\widetilde G(K)$, restrict to
$\SL(2,K)$ and apply Lemma~\ref{SL2-global-invariants}.
It gives
\begin{equation}\label{GL2-invariants}
 V^{N(O)}=V^{\SL(2,K)},\qquad
 V^{N(K)}=V^{\SL(2,K)}.
\end{equation}
These are equalities of categorical invariant subobjects. The proof
uses finite-level invariants, the bounded upper-root action of
Lemma~\ref{bounded-upper-root-action}, and bounded root factorizations; it imposes no strictness on $V$ or its invariant
subobjects and no condition on its central character.

If $\omega(-1)=-1$, the right-hand side is zero: $-I_2$ belongs to
$\SL(2,K)$, but acts by $-1$. Consequently
\begin{equation}\label{GL2-odd-sign-invariants}
 \Pi^{N(K)}=0.
\end{equation}
In this case the analogue of Claim~\ref{no-Gplus-invariants} requires
no geometric vanishing argument. For $\omega(-1)=1$, its established
vanishing transports by Theorem~\ref{GL2-equivalence}; determinant
characters are trivial on $\SL(2,K)$.

\begin{theorem}\label{GL2-full-extension}
With the induction and exactness formalism used in
Section~\ref{GKinduced}, let $\tau$ be a special local type of $J$ or
$J'$ with any smooth central character $\omega$. Then
$\mathcal I_{J_*,\omega}^{\widetilde G(K)}(\tau)$ has irreducible restriction to $M$
and to $B$, and is irreducible as a representation of $\widetilde G(K)$.
There is no restriction on $\omega(-1)$.
\end{theorem}
\begin{proof}
First twist by $\chi_+^{-1}\circ\det$ to make the central character
trivial on $1+tO$. This preserves the induction construction and all
irreducibility questions. For the normalized type,
\eqref{GL2-complete-graded} gives completeness and associated graded
$\Upsilon$. The invariant subobject vanishes by
\eqref{GL2-odd-sign-invariants} when $\omega(-1)=-1$, and by the
transported vanishing when $\omega(-1)=1$.
Apply Proposition~\ref{isbar} to $M\simeq P(K)$. This proves
irreducibility on $M$. It implies irreducibility on $B$ and on $\widetilde G(K)$,
since an invariant subobject for either larger group would be
$M$-invariant. Undo the determinant twist.
\end{proof}

\begin{remark}
Failure of a global square root of $\omega$ obstructs only the reduction
to an ordinary $\PGL(2)$-representation. It does not obstruct the direct
irreducibility proof. The preceding theorem concerns inductions from special local types.
It does not claim
that every special representation of $\widetilde G(K)$ is induced from a local type,
nor does it complete the parametrization of all local types in the
arbitrary-central-character even-depth and edge cases.
\end{remark}

\clearpage

\part{Minimal depth and induction}\label{part:minimal-depth} In this part we combine the
minimal-depth analysis with the induction results proved above.
We assume that our representations are admissible in the sense of
Definition~\ref{md:representation-data}. At integral depth, including depth zero,
we impose the additional very special hypothesis of
Definition~\ref{md:very-special}.
For positive half-integral depth that hypothesis is unnecessary.
We write $\ind$ for finite-level smooth compact induction and
for the loop-group induction functor of Section~\ref{GKinduced}.
The symbol $\Ind$ denotes ordinary unnormalized parabolic induction
over $F$. Right adjunction for loop-group induction is used only in
the ind-compact setting of Claim~\ref{indH} and its edge analogue in
Section~\ref{ind-edge}.

\section{Minimal depth and nilpotent support}\label{md:support}

\subsection*{Relation with the notation of the first part}
We retain $G=\PGL(2)$, $O_m=O/t^{m+1}O$, the filtration of
Definition~\ref{G-prime-filtration}, and $\mW$ for the underlying
pro-vector space. The symbols introduced below are abbreviations
for the following earlier objects:
\[
\begin{array}{c|c|c}
 & m\in\ZZ_{\geq0} & m\in\frac12+\ZZ_{\geq0}\\ \hline
 K_m & G(O)(m+1) & G'(O)(m+\frac12)\\
 H_m & G(O)_m=G(O_m) & G'(O)_m\\
 R_m\ (m>0) & G(O)_m(m) & G'(O)_m(m)=M_m^+\oplus M_m^-
\end{array}
\]
In Section~\ref{md:conventions}, $h(a)=\widehat a$ and
$n(b)=\widetilde b$ are the matrices of Definition~\ref{G-K-def}.
The filtration $\mF_r$, the maps $\eta_r,\kk_{j-1}^{\,j}$, and the
layers $\mu_j$ are exactly those of Definition~\ref{cusp}; in particular
\[
 \mF_r=\ker\eta_r(\mW),\qquad
 \mu_j(\mW)=\ker\kk_{j-1}^{\,j}(\mW).
\]
We use $\mcS(m)$ for the space of the standard Borel representation
$\bs(m)$ of Definition~\ref{ph}, including $\mcS(0)$ at depth zero.
The extra admissibility and very special hypotheses will be stated
explicitly; they do not change Definition~\ref{spe}.
For an irreducible positive-integral-depth quotient $\pi_{\Pi,m}$,
the support $S(\Pi)$ below is the orbit $\bo(\pi_{\Pi,m})$ of
Definition~\ref{def}, using the fixed trace pairing and character.
The symbol $R_m$ denotes the last congruence layer, whereas
$\widetilde R_m$ retains its different meaning from Section~\ref{PO}.

Let $F$ be a non-Archimedean local
field of odd residue characteristic, let $K=F((t))$, $O=F[[t]]$, and
let $G=\PGL(2)$. Let $P\subset G$ be the upper-triangular
Borel subgroup. The group $G'(O)$ and its half-integrally indexed
filtration are those of Section~\ref{G'O}. We use the fixed additive character $\widetilde\psi:F\to\CC^\times$
of Definition~\ref{ti-psi}.

\begin{definition}
For $r\in\frac12\ZZ_{\geq0}$ put
\[
 K_r=\begin{cases}
 G(O)(r+1),&r\in\ZZ_{\geq0},\\
 G'(O)(r+\frac12),&r\in\frac12+\ZZ_{\geq0},
 \end{cases}
 \qquad
 H_r=\begin{cases}
 G(O)_r,&r\in\ZZ_{\geq0},\\
 G'(O)_r,&r\in\frac12+\ZZ_{\geq0}.
 \end{cases}
\]
For $m>0$ define the abelian normal subgroup $R_m\subset H_m$ by
\[
 R_m=\begin{cases}
 G(O)_m(m),&m\in\ZZ_{\geq1},\\
 G'(O)_m(m),&m\in\frac12+\ZZ_{\geq0}.
 \end{cases}
\]
Here $R_m^*=\operatorname{Hom}_F(R_m,F)$ denotes the $F$-linear dual.
\end{definition}

\begin{claim}
There are identifications
\[
 R_m^*\simeq\begin{cases}
 \mathfrak{sl}(2,F),&m\in\ZZ_{\geq1},\\
 F^2,&m\in\frac12+\ZZ_{\geq0}.
 \end{cases}
\]
The space of $H_m$-invariant quadratic forms on $R_m^*$ is
one-dimensional. A nonzero such form is
\[
 \begin{aligned}
 Q_m(X)&=-\det X&&\text{in the integral case},\\
 Q_m(u,v)&=uv&&\text{in the half-integral case}.
 \end{aligned}
\]
\end{claim}
\begin{proof}
In the integral case, the coefficient of $t^m$ identifies $R_m$
with $\mathfrak{sl}(2,F)$. The trace pairing identifies its dual with
the same space; it is nondegenerate because the characteristic is not
$2$. Conjugation factors through the adjoint action of $G(F)$.
Write $X=\left(\begin{smallmatrix}a&b\\c&-a\end{smallmatrix}\right)$.
Invariance under the diagonal torus forces an invariant quadratic
form to have the shape $A a^2+Bbc$. Conjugation by
$\left(\begin{smallmatrix}1&s\\0&1\end{smallmatrix}\right)$ sends
$(a,b,c)$ to $(a+sc,b-2sa-s^2c,c)$; invariance forces $A=B$.
Thus the invariant forms are multiples of $a^2+bc=-\det X$.

In the half-integral case, the two root coordinates identify
$R_m$ with $F^2$. The torus acts with opposite weights and the
element $\left(\begin{smallmatrix}0&1\\t&0\end{smallmatrix}\right)$
interchanges the coordinates, as in Claim~\ref{act}.
The dual action has the same properties. Torus invariance forces a
quadratic form to be a multiple of $uv$, and this form is also
invariant under the interchange.
\end{proof}

\begin{definition}
Put $R_{\mathrm{nil},m}^*=Q_m^{-1}(0)$. This subset is independent
of the nonzero scalar used to normalize $Q_m$. In the integral case
it is the nilpotent cone in $\mathfrak{sl}(2,F)$; in the
half-integral case it is the union of the two coordinate axes.
\end{definition}
In the half-integral case, a character $\mu=\chi_{u,v}$ has
$u=[\mu^+]$ and $v=[\mu^-]$ in Definition~\ref{ac}. Thus
\[
 Q_m(u,v)=uv=t^{-1}\det(\mu).
\]
The earlier $K$-valued invariant $\det(\mu)$ and the present
$F$-valued quadratic form differ by the displayed factor of $t$;
their zero loci agree. The notation $N_r$ below is reserved for
the root subgroup $N(t^rO)$.

\subsection{Representations and minimal depth}

\begin{definition}\label{md:representation-data}
Let $(\Pi,\mW)$ be a nonzero smooth irreducible representation of
$G(K)$ in the category of Section~\ref{pro-vect}.
\begin{enumerate}
\item For $r\in\frac12\ZZ_{\geq0}$, put
\[
 V_r(\Pi)=\mW_{K_r},
\]
and denote the induced $H_r$-representation on this pro-space by
$\pi_{\Pi,r}$. When no confusion is possible, write $V_r=V_r(\Pi)$.
\item We call $\Pi$ admissible  if every $V_r(\Pi)$
is an ordinary vector space.
\begin{remark}
This is distinct from admissibility of a filtration in
Definition~\ref{admissible} and from finite-dimensionality of compact-open
invariants in ordinary $p$-adic representation theory.\end{remark}
\item Define
\[
 m(\Pi)=\min\{r\in\tfrac12\ZZ_{\geq0}:V_r(\Pi)\ne0\}.
\]
\item At $m=m(\Pi)$, abbreviate
\[
 H_\Pi=H_m,\qquad V(\Pi)=V_m(\Pi),\qquad
 \pi_\Pi=\pi_{\Pi,m}.
\]
The notation $\pi_{\Pi,m}$ distinguishes this full quotient
representation from the last-layer restriction denoted by $\pi_m$
in Definition~\ref{def}.
\item If $m>0$, set $\bar\pi_\Pi=\pi_\Pi|_{R_m}$ and, for $x\in R_m^*$, put
\[
 \begin{gathered}
 \chi_x(r)=\widetilde\psi(\langle x,r\rangle),\qquad
 V_{m,x}(\Pi)=\bigl(V_m(\Pi)\bigr)_{R_m,\chi_x},\\
 S(\Pi)=\{x\in R_m^*:V_{m,x}(\Pi)\ne0\}.
 \end{gathered}
\]
The twisted coinvariants are the quotient by the span of
$\bar\pi_\Pi(r)v-\chi_x(r)v$.
Under $x\mapsto\chi_x$, this is the support
$S_{\pi_\Pi}(R_m)$ already defined in Section~\ref{reptheory}.
\end{enumerate}
\end{definition}

The set defining $m(\Pi)$ is nonempty. Indeed,
Proposition~\ref{weak-strict-GK} presents $\Pi|_{G(O)}$ as a
formal inverse limit of ordinary representations. Since $\Pi\ne0$,
some canonical map to an ordinary representation $Q$ is nonzero.
The representation $Q$ factors through $G(O)/G(O)(n+1)$ for
some $n$, so this map factors through
$\Pi_{G(O)(n+1)}=V_n(\Pi)$. Hence $V_n(\Pi)\ne0$.
This argument does not use admissibility or the very special
hypothesis.

From now on we assume that the irreducible representations
$(\Pi,\mW)$ under consideration are admissible in this sense.
We often write $\mW$ for the representation $(\Pi,\mW)$ when its
action is understood.
\begin{lemma}
If $m=m(\Pi)>0$, then $0\notin S(\Pi)$.
\end{lemma}
\begin{proof}
For integral $m$, transitivity of coinvariants gives
\[
 V_{m,0}(\Pi)=(\mW_{G(O)(m+1)})_{R_m}
 \simeq\mW_{G(O)(m)}=V_{m-1}(\Pi)=0.
\]
For $m=n+\frac12$, with $n\geq0$, it gives
\[
 V_{m,0}(\Pi)\simeq\mW_{G'(O)(n+\frac12)}.
\]
The standard root-group description of this filtration gives
$G(O)(n+1)\subset G'(O)(n+\frac12)$: the latter subgroup has
upper-root level $t^nO$, lower-root level $t^{n+1}O$, and torus
congruence level $t^{n+1}O$.
There is therefore a natural epimorphism
\[
 V_n(\Pi)=\mW_{G(O)(n+1)}
 \twoheadrightarrow\mW_{G'(O)(n+\frac12)}.
\]
Its source is zero by minimality, proving the claim.
\end{proof}

\begin{lemma}[Passage to a subgroup in the character kernel]
Let $B$ be a subgroup acting on $\mW$, let $\chi$ be a character of
$B$, and let $D\subset B$ satisfy $\chi|_D=1$. Then the twisted
coinvariant quotient factors as an epimorphism
\[
 \mW_D\twoheadrightarrow\mW_{B,\chi}.
\]
\end{lemma}
\begin{proof}
The quotient map $\mW\to\mW_{B,\chi}$ is $D$-invariant, so it factors
through $\mW_D$ by the universal property of coinvariants. The
resulting map is an epimorphism because its composite with
$\mW\to\mW_D$ is an epimorphism. This argument also applies to
coinvariants in the pro-vector-space category. It does not require
an identification of invariants with coinvariants.
\end{proof}

\begin{theorem}[Nilpotent-support exclusion]\label{md:thm:nil}
Under the preceding hypotheses, if $m=m(\Pi)>0$, then
\[
 S(\Pi)\cap R_{\mathrm{nil},m}^*= \varnothing.
\]
\end{theorem}
\begin{proof}
The following explicit rank-one descent is in the spirit of
\cite[Section~7]{MP}; only the displayed subgroup inclusions are needed.
The zero character was treated above. Suppose that
$x\in R_{\mathrm{nil},m}^*\setminus\{0\}$ and $V_{m,x}(\Pi)\ne0$.
In both cases below we find, after conjugating $x$ by the relevant
normalizing group, an epimorphism
\[
 V_{m-\frac12}(\Pi)\twoheadrightarrow V_{m,x}(\Pi).
\]
This contradicts minimality. Conjugation preserves nonvanishing of
these fibers because $V_m(\Pi)$ is an $H_m$-representation.

\emph{Integral depth.}
Let $m\in\ZZ_{\geq1}$ and put $B=G(O)(m)$. Inflate $\chi_x$ from
$R_m$ to $B$. By transitivity of twisted coinvariants,
\[
 V_{m,x}(\Pi)\simeq\mW_{B,\chi_x}.
\]
Every nonzero nilpotent element of $\mathfrak{sl}(2,F)$ is conjugate
under $G(F)$ to
$e=\left(\begin{smallmatrix}0&1\\0&0\end{smallmatrix}\right)$.
Thus we may assume $x=e$. If
$g=1+t^m A\pmod{t^{m+1}}$ represents an element of $B$, then
\[
 \chi_e(g)=\widetilde\psi(\operatorname{tr}(eA))=\widetilde\psi(A_{21}).
\]
This formula is independent of adding a scalar matrix to $A$.
By Definition~\ref{G-prime-filtration}, the subgroup
\[
 D=G'(O)(m)=K_{m-\frac12}
\]
consists of projective classes admitting representatives satisfying
\[
 a-1,d-1\in t^mO,\qquad b\in t^mO,\qquad c\in t^{m+1}O.
\]
Consequently $D\subset B$ and $\chi_e|_D=1$. The preceding lemma
therefore gives
\[
 V_{m-\frac12}(\Pi)=\mW_D
 \twoheadrightarrow\mW_{B,\chi_e}=V_{m,e}(\Pi)\ne0,
\]
a contradiction.

\emph{Half-integral depth.}
Let $m=n+\frac12$, with $n\geq0$, and put $B=G'(O)(n+\frac12)$.
Its elements admit representatives satisfying
\[
 a-1,d-1\in t^{n+1}O,\qquad b\in t^nO,\qquad c\in t^{n+1}O.
\]
Write $x=(u,v)$ in the dual root coordinates. The inflated
character is
\[
 \chi_{u,v}(g)=\widetilde\psi\bigl(u[t^n]b+v[t^{n+1}]c\bigr),
\]
where $[t^j]$ denotes coefficient extraction. These coefficients are
well-defined on the projective classes with the displayed
normalization. Since $uv=0$ and $(u,v)\ne(0,0)$, conjugation by
$s=\left(\begin{smallmatrix}0&1\\t&0\end{smallmatrix}\right)$,
if necessary, permits us to assume $v=0$. Indeed, $s$ normalizes
$B$ and interchanges the two root coordinates.
Now take
\[
 D=G(O)(n+1)=K_n.
\]
Then $D\subset B$, and $[t^n]b=0$ for every element of $D$.
Thus $\chi_{u,0}|_D=1$. Again there is an epimorphism
\[
 V_n(\Pi)=\mW_D\twoheadrightarrow\mW_{B,\chi_{u,0}}
 =V_{m,(u,0)}(\Pi)\ne0,
\]
contradicting $n=m-\frac12<m$. This includes $m=\frac12$.
\end{proof}

\section{Conventions and the very special hypothesis}\label{md:conventions}
Write $H=G(O)$, $H^1=G(O)(1)=K_0$, and $H'=G'(O)$.
Thus $H^1=K_0$, whereas $K_1=G(O)(2)$.
Put
\[
 h(a)=\widehat a=\begin{pmatrix}a&0\\0&1\end{pmatrix},\quad
 n(b)=\widetilde b=\begin{pmatrix}1&b\\0&1\end{pmatrix},\quad
 \bar n(c)=\begin{pmatrix}1&0\\c&1\end{pmatrix}.
\]
Set $N_r=N(t^rO)=n(t^rO)$ and $\bar N(t^rO)=\bar n(t^rO)$
for $r\in\ZZ$, and
$T(O)(j)=h(O^*(j))=h(1+t^jO)$ for $j\geq1$.
Thus $T(O)(j)$ is the $j$-th congruence subgroup of the diagonal torus.
The Iwahori $I\subset H$ is the subgroup of Definition~\ref{G-O-def},
equivalently the inverse image of $P(F)$. Then
\[
 I^+=H^1N_0=N_0T(O)(1)\bar n(tO),\qquad I/I^+=T(F).
\]
All representations are left representations. We write
$\operatorname{infl}_{H/L}^{H}$ (or simply $\operatorname{infl}$ when the
quotient is specified) for ordinary pullback of a representation along
$H\to H/L$. This differs from the left adjoint
$\operatorname{Inf}_Q$ to $\operatorname{Coinv}_Q$ in
Remark~\ref{inflation}. Coinvariants are taken in
the category $\mVect=\Pro_{\aleph_0}(\Vect)$ of Section~\ref{pro-vect}; an ordinary
vector space is regarded as a constant pro-object.

For a $P(K)$-representation $\mW$, put
\[
 \begin{gathered}
 \mF_r=\ker(\eta_r(\mW))=\ker(\mW\longrightarrow \mW_{N_r}),\qquad
 \mu_j(\mW)=\ker(\mW_{N_j}\longrightarrow \mW_{N_{j-1}}),\\
 \bar \mW=\prod_{j\in\mathbb Z}\mu_j(\mW).
 \end{gathered}
\]
Completeness has the meaning of Definitions~\ref{admissible} and \ref{grad}:
the underlying pro-object is strict and the increasing filtration
$\{\mF_{-r}\}$ is admissible. Explicitly, for every epimorphism
$\kappa:\mW\to A$ to an ordinary vector space, there are $r_-\le r_+$
with $\kappa(\mF_{-r_-})=0$ and $\kappa(\mF_{-r_+})=A$.
The product defining $\bar \mW$ has its associated-graded $P(K)$-action.

We use the standard model $(\rho_m,\mcM)$ and $\Upsilon$ of
Definition~\ref{M-definition} and the definitions following it, with the
left-action convention specified there. In particular, $h(a)=\widehat a$
acts on measures by pushforward under $z\mapsto a^{-1}z$, and $n(b)$ acts
on the $m$-th factor by multiplication by
\[
 z\longmapsto\psi(t^{-m-1}bz)
   =\widetilde\psi([t^m](bz)).
\]
Thus the factor $m$ is killed by $N_{m+1}$ and has zero
$N_m$-coinvariants. For a very special representation, after choosing
the isomorphism $\bar\mW\simeq\Upsilon$, the layer identifications are
\[
 \begin{aligned}
 \mu_{m+1}(\mW)&\simeq(\rho_m,\mcM),\\
 \bigl(\mu_{m+1}(\mW)\bigr)_{T(O)(m+1)}
   &\simeq(\bs(m),\mcS(m))\qquad(m\geq0).
 \end{aligned}
\]
In the second line we restrict to $P(O)$ before passing to its finite
quotient $P(O)_m=P(O_m)$. These formulas fix the indexing used below.

\begin{definition}\label{md:very-special}
A representation $\mW$ of $G(K)$ is very special if its restriction to
$P(K)$ is complete and $\bar \mW\simeq\Upsilon$ as a
$P(K)$-representation. This is an additional filtration condition; it is
not being identified by definition with specialness in Definition~\ref{spe}.
\end{definition}

\section{The depth-zero theorem}\label{md:zero}
\begin{theorem}\label{md:thm:zero}
Let $\mW\ne0$ be an irreducible, admissible, very special representation
of $G(K)$, and put $\rho=\mW_{H^1}$. Then $\rho$ is either zero or an
ordinary irreducible cuspidal representation of $G(F)$. If $\rho\ne0$,
\[
 \mW\simeq\ind_{H}^{G(K)}\operatorname{infl}(\rho).
\]
In all cases $\mW_{I^+}=0$, and for every smooth character $\chi$ of $T(F)$,
\[
 \operatorname{Hom}_{G(F)}
 (\rho,\Ind_{P(F)}^{G(F)}\chi)=0.
\]
Here parabolic induction is unnormalized; induction from $H$ is the
functor of \cite{GK}. In particular, the nonzero alternative holds if
$m(\Pi)=0$.
\end{theorem}

\subsection{The single nontrivial Borel layer}
For an ordinary smooth $P(F)$-representation put
\[
 J_N(V)=V_{N(F)},\qquad c(V)=\ker(V\to J_N(V)).
\]
Both functors are exact. This follows for $c$ from the snake lemma
applied to the natural surjection from the identity functor to the
exact functor $J_N$. The kernel is a $P(F)$-representation, not generally
a $G(F)$-subrepresentation.
Let $\mcS(0)=C_c^\infty(F^*)$, with
\[
 (n(b)f)(x)=\widetilde\psi(bx)f(x),\qquad (h(a)f)(x)=f(ax).
\]
It is the unique irreducible smooth $P(F)$-representation of dimension
greater than one, and $J_N(\mcS(0))=0$. These assertions also follow
from Fourier transform and the simple transitive action of $F^*$ on
nonzero additive characters of $F$.

\begin{lemma}\label{md:lem:layer}
For $L(\mW)=\mu_1(\mW)$, one has
$L(\mW)_{T(O)(1)}\simeq\mcS(0)$ and a $P(F)$-equivariant
surjection $\mcS(0)\twoheadrightarrow c(\rho)$. The identification
with the fixed model uses a choice of $\bar\mW\simeq\Upsilon$.
\end{lemma}
\begin{proof}
The layer $\mu_j(\mW)$ is killed by $N_j$ and has zero
$N_{j-1}$-coinvariants, by exactness and transitivity. If $j\le0$,
$N_0$ and $N_1$ act trivially on it. If $j\ge2$, its $N_1$-coinvariants
vanish, since $N_{j-1}\subset N_1$. Therefore applying $\mu_1$ to
$\bar \mW$ selects exactly $\mu_1(\mW)$; products are pro-limits of finite
products, so the fixed-group exactness statements apply.

In $\Upsilon$, the selected factor is $(\rho_0,\mcM)$:
its characters have valuation zero, hence $N_1$ acts trivially and its
$N_0$-coinvariants vanish. Its $T(O)(1)$-coinvariants are
$\operatorname{Meas}_c(F^*)$. To see this, at each finite level the
$T(O)(1)$-orbits are the fibers of reduction to $F^*$; pushforward of
measures computes the orbit coinvariants. The resulting transition
maps on $\operatorname{Meas}_c(F^*)$ are identities. Dividing by
multiplicative Haar measure identifies the resulting action with
$\mcS(0)$. This is also the calculation in
Claim~\ref{M-smooth-coinvariants}.

Since $N_1\subset H^1$, there is a surjection $\mW_{N_1}\to\rho$.
Apply the exact kernel functor for $N_0$ to obtain
$L(\mW)\twoheadrightarrow c(\rho)$. The target is killed by $T(O)(1)$,
so the map factors through $L(\mW)_{T(O)(1)}$. This proves the lemma.
\end{proof}
Thus $c(\rho)$ is either zero or $\mcS(0)$. Notice that $\rho$ is
ordinary by our admissibility hypothesis; no pro-presentation of
$\rho$ is needed below.

Put $Q=\mW_{I^+}=J_N(\rho)$ and
$s=\left(\begin{smallmatrix}0&1\\t&0\end{smallmatrix}\right)$.
The element $s$ normalizes $I^+$ and acts on $Q$ by an automorphism
$\mathsf s$ with
\begin{equation}\label{md:eq:inverse-new}
 \mathsf s h(a)\mathsf s^{-1}=h(a^{-1})\quad(a\in F^*).
\end{equation}
Indeed $s$ interchanges $N(O)$ and $\bar N(tO)$ and inverts the torus.

\begin{lemma}\label{md:lem:zero-new}
If $c(\rho)=0$, then $\rho=0$.
\end{lemma}
\begin{proof}
Suppose $\rho\ne0$. Then $\rho\to Q$ is an isomorphism. The kernel
$D$ of $\mW\to\rho$ is consequently both $H$-stable and $s$-stable.
These groups generate $G(K)$: if
$w=\left(\begin{smallmatrix}0&1\\1&0\end{smallmatrix}\right)\in H$,
then $ws=h(t)$, and the Cartan decomposition supplies the assertion.
Irreducibility gives $D=0$. Hence $N_0$ acts trivially on $\mW$, which
implies $L(\mW)=0$, contrary to Lemma~\ref{md:lem:layer}.
\end{proof}

\subsection{An irreducible constituent and its cuspidality}
Assume $\rho\ne0$, and set $E=c(\rho)\simeq\mcS(0)$. Put
\[
 \begin{gathered}
 G^+(F)=\operatorname{im}(\SL(2,F)\to G(F)),\\
 U=\langle G(F)E\rangle\subset\rho,\qquad Z=U^{G^+(F)},\qquad
 \sigma=U/Z.
 \end{gathered}
\]
The group $G^+(F)$ is the normal closure of $N(F)$ and the kernel of
determinant modulo squares. Thus $c(A)=0$ for a $G(F)$-representation
$A$ implies that $G^+(F)$ acts trivially on $A$.
Exactness gives $c(U)=E$, $c(Z)=0$, and $c(\rho/U)=0$.
For any proper $G(F)$-subrepresentation $A\subset U$, simplicity of
$E$ gives $c(A)=0$ or $E$. The second possibility forces $E\subset A$
and hence $A=U$. Therefore every proper subrepresentation lies in $Z$.
Since $c(U)\ne0$, $Z$ is proper, and $\sigma$ is irreducible with
$c(\sigma)\simeq E$.

Choose a uniformizer $\varpi$ of $F$, let $q$ be its residue cardinality,
and set
\[
 a_0=h(\varpi^2)\in G^+(F),\qquad d=q^{-2},\qquad
 \delta(h(a))=|a|_F.
\]
\begin{lemma}\label{md:lem:det-new}
If an irreducible smooth $G(F)$-representation $\tau$ satisfies
$J_N(\tau)\ne0$ and $c(\tau)\ne0$, then
$\dim J_N(\tau)\in\{1,2\}$ and $\det(a_0\mid J_N(\tau))=d$.
\end{lemma}
\begin{proof}
The rank-one classification leaves an irreducible principal series or
a quadratic twist of Steinberg; characters have $c=0$. See
\cite{DR}; the ordinary admissibility theorem is recalled in \cite{B}.
For unnormalized $\operatorname{PS}(\chi)=\Ind_{P(F)}^{G(F)}\chi$,
use functions $f(bg)=\chi(b)f(g)$ and right translation as the left
group action. The closed Bruhat cell gives quotient $\chi$ and the
open cell gives $C_c^\infty(F)$ with
\[
 (h(a)\phi)(x)=\chi(a)^{-1}\phi(x/a).
\]
Integration computes its unipotent coinvariants, with character
$\delta\chi^{-1}$. Hence
\[
 0\to\delta\chi^{-1}\to J_N(\operatorname{PS}(\chi))\to\chi\to0.
\]
The determinant on an irreducible principal series is $\delta(a_0)=d$.
The sequence $0\to\mathbf1\to\operatorname{PS}(1)\to\operatorname{St}\to0$
gives $J_N(\operatorname{St})=\delta$. A twist descending from
$\GL(2,F)$ to $G(F)$ is by a quadratic determinant
character $\nu$, for which $\nu(\varpi^2)=1$. Its determinant is again $d$.
\end{proof}

\begin{lemma}\label{md:lem:sigma-cusp}
The representation $\sigma$ is cuspidal.
\end{lemma}
\begin{proof}
Suppose not. Apply $J_N$ to $0\subset Z\subset U\subset\rho$.
The resulting filtration of $Q$ has quotients $Z,J_N(\sigma),\rho/U$.
The first and last are $G^+(F)$-trivial. Write $A=a_0|_Q$ and let
$p(X)$ be the characteristic polynomial of $a_0$ on $J_N(\sigma)$.
It follows that $(A-1)^2p(A)=0$.

Write $p(X)=(X-1)^r p_*(X)$ with $p_*(1)\ne0$, and put $k=r+2$.
Choose $e(X)$ by the Chinese remainder theorem with
$e\equiv0\pmod{(X-1)^k}$ and $e\equiv1\pmod{p_*}$.
Then $e(A)$ is an idempotent and
\[
 R=e(A)Q=\operatorname{im}(A-1)^k.
\]
Applying this idempotent to the filtration is exact and kills the
outer terms. Thus $R\simeq e(a_0)J_N(\sigma)$ is finite-dimensional.
Deleting the generalized $1$-eigenspace does not change the determinant,
so Lemma~\ref{md:lem:det-new} gives $\det(A|_R)=d$; in particular $R\ne0$.
Equation~\eqref{md:eq:inverse-new} implies
\[
 \mathsf s(R)=\operatorname{im}(A^{-1}-1)^k
 =\operatorname{im}((-A^{-1})^k(A-1)^k)=R.
\]
Consequently $A|_R$ is conjugate to its inverse, and its determinant
has square $1$. This contradicts $d^2=q^{-4}\ne1$.
\end{proof}

\subsection{Induction and recovery}
An irreducible cuspidal representation is injective in the smooth
category with fixed central character \cite[Theorem~5.4.1]{C}.
The center of $G(F)$ is trivial, so this applies to the entire smooth
$G(F)$-category. The quotient $U\twoheadrightarrow\sigma$ therefore
extends to a surjection $\rho\twoheadrightarrow\sigma$.
Furthermore $J_N(\sigma)=0$ gives $\sigma=c(\sigma)\simeq\mcS(0)$
as a $P(F)$-representation. Its inflation to $H$ is special.

The right adjunction for induction from $H$ is
\cite[Proposition~3.5]{GK}:
\[
 \operatorname{Hom}_{G(K)}(\mW,\ind_{H}^{G(K)}\operatorname{infl}(\sigma))
 \simeq\operatorname{Hom}_{G(F)}(\rho,\sigma).
\]
The quotient yields a nonzero map. Its target is irreducible by
Corollary~\ref{irrr}; hence it is an isomorphism.

\begin{lemma}\label{md:lem:recovery-zero}
For an ordinary irreducible cuspidal $\sigma$ of $G(F)$,
\[
 (\ind_{H}^{G(K)}\operatorname{infl}(\sigma))_{H^1}
 \simeq\sigma.
\]
\end{lemma}
\begin{proof}
Use the Cartan filtration of \cite[Proposition~3.10]{GK}, whose
hypothesis is split reductivity, not simply connectedness.
Its positive terms involve a proper Jacquet module of $\sigma$ and
are zero. More explicitly, the $H$-orbits have representatives
$\gamma_r=h(t^r)=\widehat{t^r}$, $r\ge0$, and the inducing-fiber stabilizer for $H^1$
at $\gamma_rH$ is $D_r=\gamma_r^{-1}H^1\gamma_r\cap H$. If $r\ge1$,
$N(O)\subset D_r$, since $\gamma_rn(b)\gamma_r^{-1}=n(t^rb)\in H^1$.
Thus the fiber coinvariants vanish. The modular or measure characters
are trivial on this unipotent group.

For precision about the pro-limit, exhaust the affine Grassmannian by
$X_j=\overline{\Gr_G^{2j}}\cup
 \overline{\Gr_G^{2j+1}}$; this includes both components.
Each space $X_j(F)$ has a finite $H$-orbit filtration. Exact $H^1$-coinvariants
kill every positive orbit term; the base-point term is $\sigma$.
Restriction between truncations preserves evaluation at the base point,
so the resulting transition maps are identities. The section spaces
are ordinary smooth $H^1$-representations: on each finite-dimensional
truncation, the finite-level coefficient bundle and its action factor
through a sufficiently deep congruence quotient. The group $H^1$
is bounded and weakly strict. Therefore
\eqref{s8-coinvariants-universal}, applied to this ordinary
representation tower, identifies the coinvariants of the induced
pro-object with the constant pro-object $\{\sigma\}_j$.
This proves the assertion; no ordinary inverse-limit exactness is
being asserted.
\end{proof}
Applying the lemma identifies $\rho\simeq\sigma$, proving the first
two assertions of Theorem~\ref{md:thm:zero}. In either alternative
$Q=J_N(\rho)=0$. Finally Frobenius reciprocity gives
\[
 \operatorname{Hom}_{G(F)}(\rho,\Ind_{P(F)}^{G(F)}\chi)
 \simeq\operatorname{Hom}_{T(F)}(J_N(\rho),\chi)=0.
\]
This proves Theorem~\ref{md:thm:zero}, including reducible principal series.

\section{Positive integral depth}
\label{md:sec:integral}
In this section $m\geq1$ is an integer. We use the existing notation
\[
 \begin{gathered}
 O_m=O/t^{m+1}O,\qquad H_m=G(O)_m=G(O_m),\\
 R_m=G(O)_m(m),\qquad V=V_m(\Pi)=\mW_{K_m}.
 \end{gathered}
\]
Thus $R_m\simeq(\mathfrak{sl}(2,F),+)$ is an abelian normal subgroup.
Put
\[
 U_m=N(t^mO)/N(t^{m+1}O),
\]
and
\[
 c_m(X)=\ker(X\longrightarrow X_{U_m}),\qquad
 T(O)(m+1)=h(1+t^{m+1}O).
\]
The functor $c_m$ is exact on ordinary smooth representations of $P(O_m)$.
Its values need not be $H_m$-subrepresentations.

\begin{theorem}\label{md:thm:integral}
Let $\mW$ be irreducible, admissible and very special, of minimal depth
$m\in\ZZ_{\geq1}$. Then $V=\mW_{G(O)(m+1)}$ is irreducible and elliptic,
hence special, as a representation of $H_m$. Moreover
\[
 \mW\simeq\ind_{G(O)}^{G(K)}\operatorname{infl}(V).
\]
The word elliptic means that the support on $R_m^*$ is an elliptic
regular adjoint orbit. It is stronger than the exclusion of nilpotents.
\end{theorem}

\subsection{The finite-level standard Borel representation}
The coefficient pairing
\[
 O_m\times O_m\longrightarrow F,\qquad (b,x)\longmapsto[t^m](bx)
\]
is perfect. Identify compactly supported smooth measures on $O_m^*$
(the unit group) with $C_c^\infty(O_m^*)$ by multiplicative Haar measure.
Define $\mcS(m)=C_c^\infty(O_m^*)$ with actions
\begin{equation}\label{md:eq:Sm}
 (n(b)f)(x)=\widetilde\psi([t^m](bx))f(x),\qquad
 (h(a)f)(x)=f(ax),\quad a\in O_m^*,\ b\in O_m.
\end{equation}
The symbol $O_m^*$ in this paragraph means units, not linear duals.
These formulas realize the representation $\bs(m)$ of
Definition~\ref{ph} on its space $\mcS(m)$, as a representation of
$P(O)_m=P(O_m)$.

\begin{lemma}\label{md:lem:Sm}
The representation $\bs(m)$ on $\mcS(m)$ is irreducible. For a
very special $\mW$, there is a $P(O)_m$-equivariant surjection
\[
 \mcS(m)\twoheadrightarrow c_m(V).
\]
The map from the fixed model depends on the chosen identification
$\bar\mW\simeq\Upsilon$ and on the Haar-measure identification.
Only Borel equivariance is asserted; $c_m(V)$ need not be stable
under $H_m$ or under $R_m$.
\end{lemma}
\begin{proof}
Fourier transform on the additive group $O_m\simeq F^{m+1}$ identifies
integrated additive operators with multiplication by every function
in $C_c^\infty(O_m)$. Their restrictions supply all compactly supported
locally constant multipliers on $O_m^*$, an open subset of $O_m$.
Given a nonzero vector $f$ in a subrepresentation, take a compact open
set on which $f$ is a nonzero constant. Multiplication gives all locally
constant functions supported there. Multiplication by units is
transitive on $O_m^*$; translating and then using finite compact-open
partitions gives every compactly supported locally constant function.
This proves irreducibility.

The same layer-selection argument as in Lemma~\ref{md:lem:layer} gives
$\mu_{m+1}(\mW)\simeq(\rho_m,\mcM)$. Taking $T(O)(m+1)$-coinvariants
in the standard model is pushforward along reduction
$O^*\to O_m^*$, and gives $\mcS(m)$. On this quotient the additive
operator is exactly multiplication by $\widetilde\psi([t^m](bx))$.

Since $N_{m+1}\subset K_m$, the map $\mW_{N_{m+1}}\twoheadrightarrow V$
induces a surjection on kernels of coinvariants by $N_m/N_{m+1}$,
by exactness and the snake lemma. Its source is $\mu_{m+1}(\mW)$.
The target is killed by $T(O)(m+1)\subset K_m$, so the map factors
through $\mcS(m)$.
\end{proof}
\begin{remark}
The additive quotient in this calculation is $O/t^{m+1}O$, not
$O/t^mO$. For example, $n(t^mb)$ acts by $\widetilde\psi(bx_0)$ and is not
trivial. This agrees with the indexing of $O_m$ in Definition~\ref{intf}
and of the finite-level representation associated with $\rho_m$.
\end{remark}

\subsection{Minimality already gives irreducibility of the whole quotient}
\begin{lemma}\label{md:lem:integral-irred}
Under the assumptions of Theorem~\ref{md:thm:integral}, $V$ is irreducible
and $c_m(V)=\mcS(m)$.
\end{lemma}
\begin{proof}
Minimality and transitivity imply
\begin{equation}\label{md:eq:Czero}
 V_{R_m}=\mW_{G(O)(m)}=0.
\end{equation}
The normal closure of $U_m$ in $H_m$ is $R_m$: the latter is the
additive adjoint module, and the conjugates of its upper-root line
span $\mathfrak{sl}(2,F)$. Thus an $H_m$-representation $X$ with
$c_m(X)=0$ is $R_m$-trivial.

If $c_m(V)=0$, this observation and \eqref{md:eq:Czero} would give $V=0$.
Lemma~\ref{md:lem:Sm} therefore gives $E=c_m(V)\simeq\mcS(m)$.
Let $U=\langle H_mE\rangle$. Exactness shows $c_m(V/U)=0$, so $V/U$
is $R_m$-trivial. It is a quotient of $V_{R_m}=0$, whence $U=V$.

If $B\subsetneq V$ is an $H_m$-subrepresentation, then $c_m(B)$ is
zero or $E$. In the second case $B$ contains $E$, hence $U=V$, a
contradiction. Thus $B$ is $R_m$-trivial. Since $R_m$ is an additive
$F$-vector group, its ordinary smooth coinvariant functor is exact.
The inclusion $B\hookrightarrow V$ consequently gives an injection
\[
 B=B_{R_m}\hookrightarrow V_{R_m}=0.
\]
Therefore $B=0$, proving irreducibility.
\end{proof}
\begin{remark}
This argument improves the constituent construction: at positive
integral minimal depth there are no lower-depth outer terms at all.
It uses exactness for the abelian group $R_m$. The same shortcut is
not available at depth zero, where the normal closure is $G^+(F)$.
\end{remark}

\subsection{Regular split representations over the truncated ring}
We record explicitly the finite-dimensional-group calculation used
below. Here ``finite level'' refers to finite dimension over $F$;
$O_m$ is not a finite set.

\begin{lemma}[Regular split polarization]\label{md:lem:split-pol}
Let $\tau$ be an ordinary smooth irreducible representation of $G(O_m)$,
$m\geq1$, whose last-congruence support is regular split. There is a
smooth character $\chi:O_m^*\to\CC^*$, nontrivial on $1+t^mF$, such that
\[
 \tau\simeq I_m(\chi):=\ind_{P(O_m)}^{G(O_m)}\chi.
\]
The induction is unnormalized, with $f(pg)=\chi(p)f(g)$ and action
$(g_0f)(g)=f(gg_0)$. In particular support is compact modulo $P(O_m)$.
\end{lemma}
\begin{proof}
Here are the polarization and the parity distinction in this
finite-level statement. Write $\ell=m+1$, set
$G^j=\ker(G(O_m)\to G(O/t^jO))$, and put
$k=\lceil\ell/2\rceil$. The group $G^k$ is abelian. Its character
space identifies with $\mathfrak{sl}(2,O/t^{\ell-k}O)$ by
\[
 \lambda_Y(1+X)=\widetilde\psi([t^{\ell-1}]\operatorname{tr}(YX)).
\]
Mackey's smooth imprimitivity construction identifies representations
supported on an orbit with representations of its stabilizer having
the prescribed character. This construction is obtained by the
Fourier transform on $G^k$ and taking the fiber of the resulting
equivariant sheaf. Regular semisimple orbits here are closed; their
invariants are the coefficients of $-\det Y$. In the split locus each
fiber of these invariants is one conjugacy orbit. Thus the construction
applies to the regular split irreducible representations under
consideration.

After conjugacy one has $Y=\operatorname{diag}(d,-d)$ with $d$ a unit.
Indeed the distinct eigenvalues of $Y\bmod t$ lift successively, and
at each step the off-diagonal terms are removed by solving linear
equations with coefficient $2d_0\ne0$. Its stabilizer is
\[
 \operatorname{St}(\lambda_Y)=T(O_m)G^{\ell-k}.
\]
This equality follows by reducing a stabilizing matrix modulo
$t^{\ell-k}$: it must commute with $Y$, whose centralizer is diagonal.

If $\ell=2r$, then $k=r$ and the stabilizer is $T(O_m)G^r$.
The character $\lambda_Y$ is trivial on both root subgroups of $G^r$.
It extends to the stabilizer after extending its restriction on
$T\cap G^r$ to a character of $T(O_m)$. All irreducible stabilizer
representations over $\lambda_Y$ are these one-dimensional extensions,
since the remaining quotient is abelian. Inducing them gives all
regular split representations of this orbit.

If $\ell=2r+1$, then $k=r+1$ and the stabilizer is $T(O_m)G^r$.
The two root directions of $G^r/G^{r+1}$, with their commutator
measured by $\lambda_Y$, form a Heisenberg group. Explicitly,
\[
 \lambda_Y([1+t^rxe,1+t^ryf])=\widetilde\psi(2d_0xy).
\]
The pairing is nondegenerate. The diagonal direction is central. First fix its character,
extending the already prescribed character on the deeper congruence
subgroup. Stone--von Neumann then gives the unique irreducible
Heisenberg representation for the resulting central character. Use the upper
root as polarization. The split torus preserves this polarization and
acts on its Schr\"odinger model by changes of variable; this is a genuine
linear action. Thus the multiplicity space is an ordinary representation
of an abelian torus quotient and, by irreducibility, is a character.
These characters again parametrize the stabilizer representations.

For completeness, the connection of these descriptions with the
stated induction is as follows. In $I_m(\chi)$ take Fourier transform
along $G^k$. The possible characters are exactly conjugates of
characters whose diagonal restriction is $\chi|_{T\cap G^k}$ and
whose upper-root restriction is trivial. With the diagonal part
regular split, every such matrix is conjugate to $Y$. In the even
length case its fiber at $Y$ is the one-dimensional extension just
described; in the odd length case the remaining root coordinate gives
$C_c^\infty(F)$ with multiplication and translation operators having
commutator $\widetilde\psi(2d_0xy)$, the stated Schr\"odinger model.
Changing $\chi$ by a character trivial on $T\cap G^k$ supplies all
characters of the torus quotient. Any Haar-modulus character in this
Fourier identification is absorbed in this freely varying torus
character; it is trivial on the last congruence subgroup.
Hence every stabilizer representation above is realized by some
$I_m(\chi)$. Smooth imprimitivity identifies the induced representations
and proves the assertion. The parameter $d_0\ne0$ says precisely that
$\chi|_{1+t^mF}$ is nontrivial.
\end{proof}

\begin{lemma}[The unnormalized split Jacquet module]\label{md:lem:split-J}
For $\tau=I_m(\chi)$ as in Lemma~\ref{md:lem:split-pol}, put
$J_m(\tau)=\tau_{N(O_m)}$. There is an exact sequence
\[
 0\longrightarrow\delta_m\chi^{-1}\longrightarrow J_m(\tau)
 \longrightarrow\chi\longrightarrow0,\qquad
 \delta_m(h(a))=|\det_F(a:O_m\to O_m)|.
\]
In particular $\dim J_m(\tau)=2$. For constant $a\in F^*$,
$\delta_m(h(a))=|a|_F^{m+1}$.
\end{lemma}
\begin{proof}
The open chart of $P(O_m)\backslash G(O_m)$ has representatives
$wn(x)$, $x\in O_m$, where $w=\left(\begin{smallmatrix}0&1\\1&0\end{smallmatrix}\right)$.
Sections supported there form $C_c^\infty(O_m)$. On this chart,
$N(O_m)$ acts by translations and
\[
 (h(a)\phi)(x)=\chi(a)^{-1}\phi(x/a).
\]
Integration against additive Haar measure identifies translation
coinvariants with a line and gives the character $\delta_m\chi^{-1}$.

The complementary closed set has representatives $\bar n(y)$ with
$y\in tO_m$. Restriction to this set is a surjective map of section
spaces. Surjectivity follows from compact-open local trivializations
and extension of locally constant functions from closed subsets of
locally compact totally disconnected spaces. Its kernel is the open
chart section space. We compute the coinvariants of the restricted
sections next.

For $b\in O_m$, direct multiplication gives
\[
 \bar n(y)n(b)=p\bar n\left(\frac{y}{1+by}\right),\qquad
 p=\begin{pmatrix}(1+by)^{-1}&b\\0&1+by\end{pmatrix}.
\]
The torus coordinate of $p$ is $(1+by)^{-2}$.
Stratify $tO_m$ by $t$-adic order $j=1,\ldots,m$, together with $y=0$.
For $y=t^jy_j+\cdots$ with $y_j\ne0$, take $b=t^{m-j}z$.
Then $by^2=0$ in $O_m$, so this element fixes $y$, whereas
\[
 \chi((1+by)^{-2})=\chi(1-2t^mzy_j)
\]
is a nontrivial character of $z\in F$. The same subgroup fixes every
point of that stratum, and the corresponding character has no zero
fiber. Consequently its coinvariant section space is zero, by Fourier
transform and compact-open localization. The finite filtration by
these strata and exactness of unipotent coinvariants kill all strata
except $y=0$. Evaluation at zero gives a line with torus character
$\chi$. Applying exact coinvariants to the open--closed sequence gives
the displayed exact sequence.

Finally, if $\operatorname{res}_t(a)\in F^*$ is the constant
coefficient of $a\in O_m^*$, multiplication by $a$ has determinant
$\operatorname{res}_t(a)^{m+1}$ on the filtered $F$-space $O_m$. This proves the modulus formula.
\end{proof}

\subsection{The symmetry and exclusion of split support}
By Lemma~\ref{md:lem:integral-irred}, $V$ is irreducible. Its last-layer
support contains no nilpotents by Theorem~\ref{md:thm:nil}. The regular
semisimple orbit decomposition (equivalently, smooth imprimitivity for
$R_m$) therefore puts it on one regular semisimple orbit.
Suppose that orbit is split. Lemmas~\ref{md:lem:split-pol} and
\ref{md:lem:split-J} apply.

Set
\[
 D_m=K_mN(O),\quad Q=\mW_{D_m}=V_{N(O_m)},\quad
 s_{m+1}=\begin{pmatrix}0&1\\t^{m+1}&0\end{pmatrix}.
\]
Gaussian decomposition gives
\[
 D_m=N(O)h(1+t^{m+1}O)\bar N(t^{m+1}O).
\]
Conjugation by $s_{m+1}$ exchanges the two root subgroups and inverts
the torus, so it normalizes $D_m$. Its action induces an automorphism
$S$ of the \emph{same} quotient $Q$, satisfying
\[
 S h(a)S^{-1}=h(a^{-1})\qquad(a\in O^*).
\]
There is no normalization factor: these are actual coinvariant actions.

Choose a uniformizer $\varpi$ of $F$, and set $a_0=h(\varpi^2)$.
The split Jacquet calculation says
\[
 \dim Q=2,\qquad \det(a_0\mid Q)=q^{-2(m+1)}.
\]
The inversion identity says this determinant equals its inverse.
Thus $q^{-4(m+1)}=1$, a contradiction. The regular semisimple orbit
is therefore elliptic.

In particular $V_{N(t^mO)}=0$. Indeed a nonzero such quotient has a
character of $R_m/U_m$; inflated to $R_m$ it is represented, under the
trace pairing, by a matrix with lower-left entry zero. Such a matrix
is upper triangular and cannot be elliptic. This contradicts the
support just proved. Hence also $V_{N(O_m)}=0$.
Theorems~\ref{charodd} and \ref{charev} now show that $V$ is special.

\subsection{Induction and recovery at integral depth}\label{md:integral-recovery}
Right adjunction applied to the identity of $V$ gives a nonzero map
\[
 \mW\longrightarrow\ind_{H}^{G(K)}\operatorname{infl}(V).
\]
The target is irreducible by Theorem~\ref{irreducible},
so irreducibility of $\mW$ makes this map an isomorphism. This proves
Theorem~\ref{md:thm:integral}.

One can also verify recovery independently. If $\tau$ is an ordinary
irreducible elliptic representation of $H_m$, then
\[
 (\ind_{H}^{G(K)}\operatorname{infl}(\tau))_{K_m}
 \simeq\tau.
\]
For the Cartan representative $\gamma_r=h(t^r)=\widehat{t^r}$ with $r\ge1$, the
inducing-fiber stabilizer $\gamma_r^{-1}K_m\gamma_r\cap H$ contains $N(t^mO)$:
conjugation sends $n(t^mb)$ to $n(t^{m+r}b)\in K_m$.
Its fiber coinvariants vanish by ellipticity and the preceding argument.
The identity orbit contributes $\tau$. As in
Lemma~\ref{md:lem:recovery-zero}, exactness first identifies the
$K_m$-coinvariants of each ordinary section space with $\tau$.
Evaluation at the identity commutes with restriction between
truncations, so these identifications make every transition map
the identity. Since $K_m$ is bounded and weakly strict,
\eqref{s8-coinvariants-universal} then identifies the coinvariants
of the induced pro-object with this constant tower.

\begin{remark}[Relation to the earlier constituent argument]
If one first constructs $U/Z$ as in the depth-zero proof, then at
positive integral minimal depth \eqref{md:eq:Czero} forces $U=V$ and
$Z=0$. Thus the additional projector
$\ker(Q\to Q_{\{h(1+t^mx):x\in F\}})$ that one might use in a constituent argument
is unnecessary here. On a split Jacquet module it would simply be
all of $Q$, because both torus characters are nontrivial on this subgroup.
\end{remark}

\section{Positive half-integral depth}\label{md:half}
We write the depth as $m=n+\frac12$, $n\geq0$. This is the same
depth convention as $m=(n+1)-\frac12$ in Section~\ref{G'O}; by
Lemma~\ref{restr}, a special local type of this depth restricts
to $\bs(n)$ on $P(O)$.
\begin{theorem}\label{md:thm:half}
Let $\mW\ne0$ be an irreducible admissible smooth representation of
$G(K)$ with minimal depth $m=n+\frac12$, $n\ge0$.
Then $\pi_{\Pi,m}$ on $V_m=\mW_{K_m}$ is irreducible and special as a
representation of $H_m=H'/K_m$, and
\[
 \mW\simeq\ind_{H'}^{G(K)}\operatorname{infl}_{H_m}^{H'}(\pi_{\Pi,m}).
\]
The very special hypothesis is not needed for this half-integral assertion.
\end{theorem}

\subsection{Existence of a special quotient}
Call an abelian locally compact totally disconnected group locally
elliptic if every finitely generated subgroup has compact closure.

\begin{lemma}\label{md:lem:locally-elliptic-character}
Every nonzero ordinary smooth complex representation \(Y\) of an
abelian locally elliptic group \(A\) has a one-dimensional smooth quotient.
\end{lemma}

\begin{proof}
Choose a compact open subgroup \(C\subset A\) fixing a nonzero vector.
Because \(A\) is abelian, normalized averaging over \(C\) is an
\(A\)-equivariant projection
\[
 Y\longrightarrow Y^C\ne0.
\]
The group \(D=A/C\) is discrete, torsion, and abelian, hence is a
filtered union of finite groups. Its character group
\(X=\Hom(D,\CC^*)\), with the topology of pointwise convergence, is
compact and totally disconnected. Fourier transform identifies
\[
 \CC[D]\simeq C^{\mathrm{lc}}(X),
\]
since a locally constant function on the compact group \(X\) factors
through the characters of some finite subgroup of \(D\).

Put \(M=Y^C\). For \(x\in X\), let
\(\mathfrak m_x=\{f\in C^{\mathrm{lc}}(X):f(x)=0\}\).
There is some \(x\) with \(M/\mathfrak m_xM\ne0\). Indeed, otherwise
fix \(v\ne0\) in \(M\). For each \(x\), write \(v\) as a finite
sum \(\sum f_i v_i\) with \(f_i(x)=0\). On a sufficiently small
clopen neighborhood \(E_x\), all these \(f_i\) vanish, and therefore
\(1_{E_x}v=0\). A finite clopen refinement of this cover of compact
\(X\) gives \(v=0\), a contradiction.

A nonzero linear functional on \(M/\mathfrak m_xM\) is a quotient
onto the character \(x\). Compose it with averaging to obtain a
character quotient of \(Y\). Its character is trivial on the open
subgroup \(C\), so it is smooth.
\end{proof}

\begin{lemma}\label{md:lem:half-elliptic-quotient}
Let \(m=N-\tfrac12\), with \(N\ge1\), and put
\[
 H_m=G'(O)_m,\qquad R_m=G'(O)_m(m).
\]
If \(V\ne0\) is an ordinary smooth representation of \(H_m\) and
\[
 \{x\in R_m^*:V_{R_m,\chi_x}\ne0\}\subset\{(u,v)\in F^2:uv\ne0\},
\]
then \(V\) has an irreducible elliptic quotient. No finite-generation
hypothesis is required.
\end{lemma}

\begin{proof}
\textit{Step 1: choose an elliptic character.}
Let \(U\) be the image of \(G'(O)(N/2)\) in \(H_m\), as in Claim~\ref{Gprime-U}.
It is a finite-dimensional additive \(F\)-group and contains \(R_m\).
Lemma~\ref{md:lem:locally-elliptic-character}, applied to \(V|_U\), gives a smooth character
\(\lambda\) with
\[
 V_{U,\lambda}\ne0.
\]
Since this is a quotient of \(V_{R_m,\lambda|_{R_m}}\), the character
\(\lambda\) is elliptic in the sense of Section~\ref{G'O}.

\textit{Step 2: its orbit is closed.}
Set
\[
 k=\lfloor N/2\rfloor,\qquad \ell=\lceil N/2\rceil.
\]
Use the character coordinates from the proof of Claim~\ref{ch}:
\[
 C=\begin{pmatrix}ta&b\\tc&-ta\end{pmatrix},\qquad
 a\pmod{t^k},\quad b,c\pmod{t^\ell}.
\]
Ellipticity means \(b_0c_0\ne0\). The invariant
\[
 \delta(C)=bc+ta^2\pmod{t^\ell}
\]
is well defined; the precision of \(a\) suffices because
\(\ell\le k+1\). The explicit upper-triangular conjugation in
Claim~\ref{ch} sends \(C\) to
\[
 \begin{pmatrix}0&\delta(C)\\t&0\end{pmatrix}.
\]
Consequently the \(H_m\)-orbit \(\Omega=H_m\lambda\) is exactly the
fiber of \(\delta\) through \(C\). This fiber is closed in
\(\widehat U\), since it is given by finitely many polynomial
equations over \(F\). Its fixed nonzero constant term \(\delta_0\)
already forces \(b_0c_0\ne0\), so the fiber has no additional
nonelliptic boundary. Lemma~\ref{open} identifies
\(H_m/\St_{H_m}(\lambda)\) homeomorphically with \(\Omega\).

\textit{Step 3: obtain a character of the stabilizer.}
By Claim~\ref{ch},
\[
 S:=\St_{H_m}(\lambda)=T_m^L U.
\]
Here $\ker\lambda$ is a closed normal subgroup of $S$:
it is closed in the closed subgroup $U$, and $S$ stabilizes
$\lambda$. Thus $A=S/\ker\lambda$ is again a locally compact
totally disconnected group. The group \(A\) is abelian: both \(T_m^L\) and \(U\)
are abelian, and \([T_m^L,U]\subset\ker\lambda\). It is locally
elliptic. Indeed, the finite-level projective ramified torus \(T_m^L\)
has a unipotent subgroup of index at most two. The residue \(F^*\)
has been removed by the projective quotient, and only the order-two
valuation quotient remains. Its unipotent subgroup is a finite-level
unipotent \(F\)-group, whose finitely generated subgroups have compact
closure. The additive group \(U\) has the same property. Their images
commute in \(A\), so products of their compact cyclic closures prove
the assertion for \(A\).

The nonzero smooth representation \(V_{U,\lambda}\) factors through
\(A\). Lemma~\ref{md:lem:locally-elliptic-character} gives a character quotient
\[
 V_{U,\lambda}\twoheadrightarrow\CC_\eta,
 \qquad \eta|_U=\lambda.
\]
Let \(q:V\to\CC_\eta\) be the resulting nonzero \(S\)-equivariant
linear map.

\textit{Step 4: construct a map into compact induction.}
For \(v\in V\), define
\[
 f_v(h)=q(h^{-1}v),\qquad h\in H_m.
\]
Smoothness of \(v\) makes this function locally constant. Also
\[
 f_v(hs^{-1})=\eta(s)f_v(h),\qquad s\in S,
\]
which is the manuscript's induction convention.

Choose a compact open subgroup \(C_v\subset U\) fixing \(v\).
If \(f_v(h)\ne0\), then, for \(u\in C_v\),
\[
 f_v(h)=q(h^{-1}uv)
       =\lambda(h^{-1}uh)f_v(h),
\]
so \((h\lambda)|_{C_v}=1\). The support modulo \(S\) is therefore
contained in
\[
 \Omega\cap C_v^\perp.
\]
The annihilator \(C_v^\perp\) is compact in \(\widehat U\), and
\(\Omega\) is closed. The intersection is thus compact. Via
\(H_m/S\simeq\Omega\), this gives compact support modulo \(S\).
A compact subset of \(H_m/S\) is covered by images of finitely many
compact subsets of \(H_m\), by compact-open neighborhoods and openness
of the quotient map. Thus the support condition is also exactly the
one in Definition~\ref{indu}.

We obtain an \(H_m\)-equivariant map
\[
 V\longrightarrow\ind_S^{H_m}\eta,\qquad v\longmapsto f_v,
\]
for the left-translation action. It is nonzero, since \(f_v(1)=q(v)\).
The little-group theorem, Lemma~\ref{MM}(2), makes its target irreducible.
Hence the map is surjective. Its character orbit has elliptic
last-layer restriction, so the quotient has depth \(m\), is elliptic,
and is special by Theorem~\ref{char}.
\end{proof}

Apply Lemma~\ref{md:lem:half-elliptic-quotient} with
$N=n+1$ and $V=V_m$. This representation is nonzero and ordinary
by minimality and admissibility, and Theorem~\ref{md:thm:nil}
places its last-layer support in $uv\ne0$. The lemma supplies
an irreducible elliptic quotient $\tau$ of depth $m$, which is
special by Theorem~\ref{char}. Write $\widehat\tau$ for its
inflation to $H'$.

Right adjunction, \cite[Proposition~3.5]{GK}, gives
\[
 \operatorname{Hom}_{G(K)}(\mW,\ind_{H'}^{G(K)}\widehat\tau)
 \simeq\operatorname{Hom}_{H_m}(V_m,\tau).
\]
Here $G(K)/H'$ is the finite quotient of the affine flag variety
$G(K)/I$ by $H'/I$, and is ind-proper; thus this is the right-adjoint
induction of the cited proposition. The quotient supplies a nonzero
map, and Section~\ref{ind-edge} makes its target irreducible. Hence
\begin{equation}\label{md:eq:half-ind}
 \mW\simeq\ind_{H'}^{G(K)}\widehat\tau.
\end{equation}
It remains to recover $\tau$ from the minimal-depth coinvariants.

\subsection{The recovery lemma}
\begin{lemma}\label{md:lem:half-recovery}
Let $m=n+\frac12$ and let $\tau$ be an ordinary irreducible elliptic
representation of $H_m$. Evaluation at the identity coset induces
\[
 (\ind_{H'}^{G(K)}\widehat\tau)_{K_m}\simeq\tau.
\]
\end{lemma}
\begin{proof}
First,
\begin{equation}\label{md:eq:root-vanish}
 (\widehat\tau)_{N(t^nO)}=0.
\end{equation}
Indeed, modulo $K_m$, this root group is the upper-root summand of
$R_m\simeq F^2$. Its trivial-character coinvariants correspond to
restriction of the spectral sheaf to $u=0$. The support of $\tau$
is contained in $uv\ne0$, so this restriction is zero. Equivalently,
a nonzero such coinvariant would have a character of the remaining
root summand, and hence give a nonzero $R_m$-fiber at $(0,v)$,
contradicting ellipticity.

The double cosets $H'\backslash G(K)/H'$ have representatives
$\gamma_r=h(t^r)=\widehat{t^r}$, $r\ge0$, with $\gamma_0=1$. One can see this from the
Bruhat--Tits tree over the discretely valued field $K$: $H'$ is the
setwise stabilizer of a standard edge, and its orbits on edges are
indexed by edge distance from that edge. Equivalently this is the
rank-one affine Bruhat decomposition, with the two endpoint types
identified by $s\in H'$.

For the orbit of $\gamma_rH'$ and the normal subgroup $K_m\triangleleft H'$,
the stabilizer acting on the inducing fiber is
\[
 D_r=\gamma_r^{-1}K_m\gamma_r\cap H'.
\]
If $r\ge1$, then $N(t^nO)\subset D_r$: for $b\in O$,
\[
 \gamma_r n(t^n b)\gamma_r^{-1}=n(t^{n+r}b)\in G'(O)(n+1)=K_m.
\]
Consequently $(\widehat\tau)_{D_r}=0$ by
\eqref{md:eq:root-vanish}.

We describe why this fiber computation gives the required statement
for induction in the pro-category. Use its defining presentation by
compactly supported locally constant sections (with the prescribed
measure factors) on finite-dimensional closed truncations of
$G(K)/H'$. Choose truncations that are finite unions of $H'$-orbits
and contain the identity coset, ordered cofinally. Such truncations
are obtained from finite unions of affine Schubert varieties in
$G(K)/I$ and passage to the finite quotient by $H'/I$.
Each section space has the finite filtration by its locally closed
$H'$-orbits. This is the same orbit filtration used to prove
\cite[Propositions~3.8 and 3.10]{GK}.

For a transitive orbit, taking coinvariants by a normal subgroup
first takes coinvariants of the inducing fiber by its stabilizer,
and then forms the corresponding orbit-section representation for
the quotient group. This assertion is the ordinary homogeneous-bundle
coinvariant calculation, applied at a finite congruence level:
relations along the orbit identify translated fibers, and the
relations remaining within a fiber are exactly those of the
stabilizer. Any modular factor restricts trivially to the unipotent
group $N(t^nO)$, so it does not change the vanishing just proved.
Normality of $K_m$ ensures that the same calculation holds at all
points of a given $H'$-orbit, by conjugating the stabilizer.
Thus every orbit with $r\ge1$ contributes zero after $K_m$-coinvariants.

The orbit $r=0$ is a point and contributes $\tau$, since $K_m$ acts
trivially on $\widehat\tau$. Exactness of $K_m$-coinvariants shows
that evaluation induces an isomorphism with $\tau$ at every truncation.
The transition maps preserve evaluation, so these isomorphisms are
compatible and the transition maps on coinvariants are identities.
Each section space is an ordinary smooth $K_m$-representation,
factoring through a finite congruence quotient depending on the
truncation. The group $K_m$ is bounded and weakly strict, with its
surjective congruence presentation. Apply
\eqref{s8-coinvariants-universal} to this tower: its termwise
coinvariants represent the coinvariants of the induced pro-object.
They form the constant tower $\{\tau\}$, which proves the displayed
isomorphism as an $H_m$-representation. This separates the exactness
used for each finite orbit filtration from the formal pro-category
argument.
\end{proof}
Combining \eqref{md:eq:half-ind} with Lemma~\ref{md:lem:half-recovery} gives
$V_m\simeq\tau$. Thus $\pi_{\Pi,m}$ is irreducible and special, and
\eqref{md:eq:half-ind} is induction from its inflation. This proves
Theorem~\ref{md:thm:half}.

\section{Comparison of the three cases}\label{md:comparison}
The common first step at every positive depth is the exclusion of the
zero character and of all nilpotent characters in Section~\ref{md:support}. The remaining
arguments differ as follows.
\begin{enumerate}
\item At depth zero, the single Borel layer gives an irreducible
constituent. The finite-dimensional part of its unnormalized Jacquet
module is isolated inside $\mW_{I^+}$, and torus inversion excludes the
noncuspidal alternatives. Injectivity of the resulting cuspidal
representation, adjunction, and recovery identify the full quotient.
\item At positive integral depth, the finite-level Borel layer and
vanishing $V_{R_m}=0$ already prove irreducibility of $V$. The element
$s_{m+1}$ acts on $V_{N(O_m)}$ by inversion. A regular split constituent
would give a two-dimensional Jacquet module with determinant
$q^{-2(m+1)}$, which is impossible. The constituent is elliptic, hence
special, and adjunction gives induction from $G(O)$.
\item At positive half-integral depth, nilpotent-support exclusion
puts the last-layer support in $uv\ne0$.
Lemma~\ref{md:lem:half-elliptic-quotient} supplies an irreducible
elliptic quotient without a finite-generation assumption. It is
special. Induction from $G'(O)$ is irreducible, and the root-group
vanishing and orbit filtration recover the entire minimal quotient.
\end{enumerate}
The asserted integral theorem is thus a statement about irreducible
admissible very special representations. No conclusion for arbitrary
integral-depth representations without the single-layer hypothesis is
claimed. 

\subsection{Recovery}\label{md:recovery}

The $\PGL(2)$ recovery calculations in Sections~\ref{md:zero},
\ref{md:sec:integral}, and~\ref{md:half} settle the injectivity question
for the representations constructed from ordinary special local types.

\begin{proposition}[Recovery distinguishes the constructed representations]
\label{induction-distinguishes-types}
Let $\pi$ be an ordinary smooth special local type of $G(O)$ or
$G'(O)$, of depth $m$,
and regard $\pi$ as a representation of the corresponding quotient $H_m$.
Then
\[
 \Pi(\pi)_{K_m}\simeq\pi,\qquad
 \Pi(\pi)_{K_r}=0\quad
 (r\in\tfrac12\ZZ_{\geq0},\ r<m).
\]
Consequently $\Pi(\pi)\simeq\Pi(\pi')$ for two special local types if
and only if they belong to the same inducing group and $\pi\simeq\pi'$.
No admissibility assumption on the full induced representation is needed.
\end{proposition}
\begin{proof}
The first isomorphism follows from the integral-depth recovery calculation
in Section~\ref{md:integral-recovery} and from
Lemmas~\ref{md:lem:recovery-zero} and~\ref{md:lem:half-recovery}.
These are direct computations for the induction of an ordinary
local type. They require no hypothesis on coinvariants at larger
depths; only the quotient at the inducing depth is being computed.
The congruence subgroups satisfy
\[
 \begin{aligned}
 K_j=G(O)(j+1)&\supset K_{j+1/2}=G'(O)(j+1)\\
              &\supset K_{j+1}=G(O)(j+2)\qquad(j\geq0).
 \end{aligned}
\]
Thus, when $m>0$, transitivity of coinvariants reduces the vanishing
assertion to $r=m-\tfrac12$.

If $m$ is integral, the image of $K_{m-1/2}=G'(O)(m)$ in
\[
 R_m=G(O)(m)/G(O)(m+1)\simeq\mathfrak{sl}(2,F)
\]
is the upper-triangular subspace. Its trace annihilator is
$F e$, where $e=\left(\begin{smallmatrix}0&1\\0&0\end{smallmatrix}\right)$.
The last-layer support of $\pi$ is elliptic and hence misses this
nilpotent line. Therefore
\[
 \Pi(\pi)_{K_{m-1/2}}
   \simeq\pi_{K_{m-1/2}/K_m}=0.
\]
If $m=n+\tfrac12$, then $K_n/K_m$ is the lower-root summand
$M_m^-$ of $R_m\simeq F^2$. Its annihilator is the axis $v=0$,
which is disjoint from the support $uv\ne0$ of $\pi$. This gives
$\Pi(\pi)_{K_n}=0$ in the same way. There is nothing to prove below
$m=0$.

It follows that $m$ is the least index with nonzero coinvariants.
Isomorphic induced representations therefore have equal inducing depths.
Integral and half-integral depths distinguish the two inducing groups,
and recovery at that depth identifies the two local types.
\end{proof}

\section{\texorpdfstring{Minimal-depth support and induction for $\GL(2)$}{Minimal-depth support and induction for GL(2)}}
\label{glmd:main}
We extend to $\GL(2)$ the $\PGL(2)$ results proved in
Sections~\ref{md:support}--\ref{md:comparison} of this part. The central character is arbitrary, including
$\omega(-1)=-1$. The necessary changes are to normalize central
ramification, use subgroups containing the full center for induction,
and use a determinant-one torus element in the Jacquet determinant
argument. We retain the same representation category, exactness and
strict pro-limit formalism as in those sections.

\subsection{Normalization, the filtration, and the hypotheses}
Use $\widetilde G=\GL(2)$, $Z=K^*I_2$, $J$, $J'$, $B$, $M$, and
$\chi_+$ as in Section~\ref{GL2-extension}. Let $\mW$ be a nonzero
irreducible smooth representation of $\widetilde G(K)$ with scalar central character
$\omega$. Set
\[
 \mW^\circ=\mW\otimes(\chi_+^{-1}\circ\det),\qquad
 \omega^\circ=\omega\chi_+^{-2}.
\]
Thus $\omega^\circ|_{1+tO}=1$. No square root of $\omega$ on $F^*$ is
assumed. All minimal-depth quotients in this section are those of
$\mW^\circ$; the final induction statements for $\mW$ are obtained by
undoing this twist.

Put $C_j=1+t^jM_2(O)$ for $j\geq1$. For
$r\in\frac12\ZZ_{>0}$ let $C'_r$ be the subgroup of $J'$ consisting
of matrices satisfying
\[
 a-1,d-1\in t^{\lfloor r+1/2\rfloor}O,\quad
 b\in t^{\lfloor r\rfloor}O,\quad
 c\in t^{\lfloor r\rfloor+1}O.
\]
Their precise relation to Definition~\ref{G-prime-filtration} is
\[
 \begin{aligned}
 \operatorname{pr}(C_j)&=G(O)(j),&
 \operatorname{pr}^{-1}(G(O)(j))&=ZC_j,\\
 \operatorname{pr}(C'_r)&=G'(O)(r),&
 \operatorname{pr}^{-1}(G'(O)(r))&=ZC'_r.
 \end{aligned}
\]
Thus $C_j$ and $C'_r$ consist of actual matrices with normalized
scalar congruence; the full inverse images also contain $Z$.
Define
\[
 \widetilde K_r=
 \begin{cases}
 C_{r+1},&r\in\ZZ_{\geq0},\\
 C'_{r+1/2},&r\in\frac12+\ZZ_{\geq0},
 \end{cases}
 \qquad V_r=(\mW^\circ)_{\widetilde K_r}.
\]
Let $J_r=J$ in the integral case and $J_r=J'$ in the half-integral
case. The quotient $V_r$ is a representation of
$H_r^{\mathrm{GL}}=J_r/\widetilde K_r$, with the inherited scalar central
character. Although $H_r^{\mathrm{GL}}$ contains a discrete central
valuation direction, fixing its action leaves the usual finite-level
smooth representation problem.

\begin{definition}\label{glmd:definitions}
We call $\mW$ admissible here if every $V_r$ is an ordinary vector space
and is finitely generated as a smooth $H_r^{\mathrm{GL}}$-representation.
This is stronger than Definition~\ref{md:representation-data}:
finite generation is an additional hypothesis in this section.
Assume that some $V_r$ is nonzero and define its minimal noncentral depth by
\[
 m_{\mathrm{nc}}(\mW)=\min\{r\in\tfrac12\ZZ_{\geq0}:V_r\ne0\}.
\]
We call $\mW$ very special after central normalization if
$\mW^\circ|_M$ is complete and its associated graded is $\Upsilon$,
under $M\simeq P(K)$.
\end{definition}

The qualification ``noncentral'' is essential. A ramified determinant
twist can increase raw congruence depth without changing any of the
root-group phenomena used in the proof. More intrinsically, put
$D_r^{\mathrm{der}}=\widetilde K_r\cap\SL(2,K)$. Every
$g\in\widetilde K_r$ has a unique square root of its determinant in
$1+tO$, and
\[
 g=(\det g)^{1/2}I_2\cdot(\det g)^{-1/2}g.
\]
Both factors preserve the indicated filtration levels. The first acts
trivially on $\mW^\circ$ and the second belongs to $D_r^{\mathrm{der}}$. Hence the
underlying coinvariant spaces of $\mW^\circ$ by $\widetilde K_r$ and
by $D_r^{\mathrm{der}}$ agree. Determinant twisting is trivial on $D_r^{\mathrm{der}}$, so their
nonvanishing is also detected by $\mW_{D_r^{\mathrm{der}}}$. This proves that
$m_{\mathrm{nc}}$ is invariant under determinant twists, even though
the actions on the quotient spaces must be twisted accordingly.

For $m>0$ write the last layer as $R_m^{\mathrm{GL}}$, keeping
$\widetilde R_m$ for the inducing subgroup of Section~\ref{PO}:
the definition is
\[
 R_m^{\mathrm{GL}}=
 \begin{cases}
 C_m/C_{m+1},&m\in\ZZ_{\geq1},\\
 C'_m/C'_{m+1/2},&m\in\frac12+\ZZ_{\geq0}.
 \end{cases}
\]
In the integral case this is $\mathfrak{gl}(2,F)$, but its scalar
subspace acts trivially on $V_m$. The trace pairing therefore places
the character support in $\mathfrak{sl}(2,F)$. In the half-integral
case the last layer is $F^2$, as before. Thus the relevant quadratic
forms and nilpotent loci are still
\[
 -\det Y\quad(Y\in\mathfrak{sl}(2,F)),\qquad uv\quad((u,v)\in F^2).
\]
For $m=m_{\mathrm{nc}}(\mW)>0$, define
\[
 \widetilde S(\mW)=\{x\in(R_m^{\mathrm{GL}})^*:
       (V_m)_{R_m^{\mathrm{GL}},\chi_x}\ne0\},\qquad
 \chi_x(r)=\widetilde\psi(\langle x,r\rangle).
\]
This is the analogue of $S(\Pi)$ in
Definition~\ref{md:representation-data}, computed for $\mW^\circ$.
In the integral case it lies in the traceless subspace of the dual,
as just explained. Throughout this section $V_r$ refers to
$(\mW^\circ)_{\widetilde K_r}$, not to the unnormalized quotient
$\mW_{\widetilde K_r}$.

\begin{theorem}[Nilpotent-support exclusion]\label{glmd:nil}
If $m=m_{\mathrm{nc}}(\mW)>0$, then $\widetilde S(\mW)$ is disjoint
from the indicated nilpotent locus.
\end{theorem}
\begin{proof}
The proof of Theorem~\ref{md:thm:nil} uses subgroup inclusions and
factorization of twisted coinvariants, all of which hold for the
matrix subgroups just defined. We give the relevant inclusions to
make the role of the center explicit. At integral depth, after
$\GL(2,F)$-conjugacy a nonzero nilpotent character is
represented by $e$, and its value on $1+t^mX$ is
$\widetilde\psi(X_{21})$. It is trivial on $C'_m\subset C_m$.
Thus its fiber is a quotient of
$(\mW^\circ)_{C'_m}=V_{m-1/2}=0$.
At half-integral depth $m=n+1/2$, conjugation by $s$ interchanges the
two root coordinates. A nonzero nilpotent character can be written
$(u,0)$ and is trivial on $C_{n+1}\subset C'_{n+1/2}$. Its fiber is
a quotient of $V_n=0$.
For the zero character, integral depth gives the quotient by $C_m$,
which is $V_{m-1}=0$. At half-integral depth the inclusion
$C_{n+1}\subset C'_{n+1/2}$ again gives a quotient of $V_n=0$.
Scalar principal units introduce no additional character because of
central normalization. These contradictions prove the theorem.
\end{proof}

\subsection{The two minimal-depth induction theorems}
\begin{theorem}[Integral depth, including depth zero]\label{glmd:integral}
Let $\mW$ be irreducible, admissible, and very special after central
normalization, and suppose $m=m_{\mathrm{nc}}(\mW)$ is integral.
If $m=0$, then $V_0$ is an irreducible cuspidal representation of
$\GL(2,F)$, together with the specified action of the
central element $tI_2$. If $m>0$, then $V_m$ is irreducible and its
last-layer support is one regular elliptic orbit in
$\mathfrak{sl}(2,F)$; its restriction to $M(O)$ is irreducible.
In both cases, inflate $V_m$ to a representation $\tau^\circ$ of $J$.
Then
\[
 \mW^\circ\simeq\mathcal I_{J,\omega^\circ}^{\widetilde G(K)}(\tau^\circ),\qquad
 \mW\simeq\mathcal I_{J,\omega}^{\widetilde G(K)}
       \bigl(\tau^\circ\otimes(\chi_+\circ\det)|_J\bigr).
\]
\end{theorem}

\begin{theorem}[Positive half-integral depth]\label{glmd:half}
Let $\mW$ be irreducible and admissible, with
$m=m_{\mathrm{nc}}(\mW)=n+1/2$, $n\geq0$. Then $V_m$ is irreducible
and special as a representation of $J'/\widetilde K_m$. If
$\tau^\circ$ is its inflation to $J'$, then
\[
 \mW^\circ\simeq\mathcal I_{J',\omega^\circ}^{\widetilde G(K)}(\tau^\circ),\qquad
 \mW\simeq\mathcal I_{J',\omega}^{\widetilde G(K)}
       \bigl(\tau^\circ\otimes(\chi_+\circ\det)|_{J'}\bigr).
\]
No very special hypothesis is required in this case.
\end{theorem}
We prove these statements below. In particular they impose no condition
on $\omega(-1)$.

\subsection{Depth zero: the determinant-one inversion argument}
Let $\rho=V_0$, a smooth $\GL(2,F)$-representation with
fixed central character $\omega^\circ|_{F^*}$. Write
$J_N(X)=X_{N(F)}$ and $c(X)=\ker(X\to J_N(X))$ on restriction to
$M(F)$. The calculation of Lemma~\ref{md:lem:layer} gives
\[
 \mcS(0)=C_c^\infty(F^*)\twoheadrightarrow c(\rho).
\]
Only the root groups and the standard $M$-model enter that calculation,
so it is unchanged. Put
\[
 \begin{gathered}
 \widetilde I^+=C_1N(O),\\
 Q=(\mW^\circ)_{\widetilde I^+}=J_N(\rho).
 \end{gathered}
\]
This is the pro-unipotent radical of the standard Iwahori. The element $s$ normalizes
this group. If $c(\rho)=0$, the kernel of $\mW^\circ\to\rho=Q$ is
stable under $J$ and $s$, which generate $\widetilde G(K)$. Irreducibility and the
nonzero standard layer give the contradiction of
Lemma~\ref{md:lem:zero-new}. Thus $c(\rho)=\mcS(0)$.

Let $U$ be the $\GL(2,F)$-subrepresentation generated by
this layer and let $Z_0=U^{\SL(2,F)}$. The normal closure
of $N(F)$ is $\SL(2,F)$. The argument preceding
Lemma~\ref{md:lem:det-new} therefore gives an irreducible
$\sigma=U/Z_0$ with $c(\sigma)=\mcS(0)$, and both $Z_0$ and
$\rho/U$ are trivial on $\SL(2,F)$.

Choose a uniformizer $\varpi$ of $F$ and use
\begin{equation}\label{glmd:a0}
 \widetilde a_0=\operatorname{diag}(\varpi,\varpi^{-1})\in\SL(2,F).
\end{equation}
Its projective image is the element $a_0=h(\varpi^2)$ used in
Section~\ref{md:zero}:
$\operatorname{pr}(\widetilde a_0)=a_0$.
Now $s\widetilde a_0s^{-1}=\widetilde a_0^{-1}$ as actual matrices. There is no central
factor, regardless of $\omega$. If $\sigma$ is not cuspidal, its
nonzero Jacquet module has dimension one or two and
\begin{equation}\label{glmd:det-zero}
 \det(\widetilde a_0\mid J_N(\sigma))=q^{-2}.
\end{equation}
Indeed, for an irreducible unnormalized principal series induced from
a character $\beta$ of $\widetilde T(F)$, extended trivially over
$N(F)$ to $B(F)$, the two Jacquet characters
are $\beta$ and $\delta_B\beta^w$, where
$\delta_B(\operatorname{diag}(a,d))=|a/d|_F$ and
$\beta^w(\operatorname{diag}(a,d))=\beta(\operatorname{diag}(d,a))$.
Thus $\delta_B=\delta\circ\operatorname{pr}$ on $\widetilde T(F)$,
with $\delta$ as in Section~\ref{md:zero}. Since
$\beta(\widetilde a_0)\beta^w(\widetilde a_0)=1$, their determinant is
$\delta_B(\widetilde a_0)=q^{-2}$. For a determinant twist of Steinberg its
one-dimensional Jacquet module has the same value: the determinant
twist is trivial on $\widetilde a_0$. One-dimensional representations have
$c=0$ and were already excluded.

Apply the polynomial idempotent construction in
Lemma~\ref{md:lem:sigma-cusp} to the filtration
$0\subset Z_0\subset U\subset\rho$. The outer terms have eigenvalue
one for $\widetilde a_0$. Removing the generalized one-eigenspace gives a
finite-dimensional subspace $R\subset Q$ with determinant $q^{-2}$.
Conjugation by $s$ preserves $R$ and inverts the action of $\widetilde a_0$.
Its determinant must consequently have square one, contradicting
$q^{-4}\ne1$. The relation $s^2=tI_2$ is harmless: the proof uses
conjugation by an invertible operator, not that its square is one.
Thus $\sigma$ is cuspidal.

Cuspidal injectivity is used in the category with the fixed central
character $\omega^\circ|_{F^*}$, as in \cite[Theorem~5.4.1]{C}.
It extends $U\to\sigma$ to a quotient $\rho\to\sigma$.
Extend the inflation of $\sigma$ to $J$ by specifying $tI_2$ to act
by $\omega^\circ(t)$. Its mirabolic restriction is $\mcS(0)$, so
it is a special local type. Right adjunction gives a nonzero map
from $\mW^\circ$ to its induction. Theorem~\ref{GL2-full-extension}
makes the target irreducible, so the map is an isomorphism.
The recovery calculation below gives $\rho=\sigma$. This proves
the depth-zero assertion.
The same argument when $V_0=0$ is unnecessary; hence for every very
special $\mW$ under these hypotheses one also obtains
\[
 (\mW^\circ)_{\widetilde I^+}=0,\qquad
 \operatorname{Hom}_{\GL(2,F)}
  \bigl(V_0,\Ind_{B(F)}^{\GL(2,F)}\beta\bigr)=0
\]
for every torus character $\beta$, with unnormalized parabolic induction.

\subsection{Positive integral depth: irreducibility of the quotient}
Suppose $m>0$ is integral, put $O_m=O/t^{m+1}O$ and
$U_m=N(t^mO)/N(t^{m+1}O)$, and set $V=V_m$.
The finite-level layer calculation gives
\[
 \mcS(m)=C_c^\infty(O_m^*)\twoheadrightarrow
 c_m(V):=\ker(V\to V_{U_m}).
\]
Its $M(O_m)$-action is exactly \eqref{md:eq:Sm}.
The normal closure of $U_m$ in $\GL(2,O_m)$ is the
traceless subspace of the last congruence layer. Scalar principal
units act trivially, so on our representations this normal closure
detects the entire effective last layer. Minimality gives
\[
 V_{C_m/C_{m+1}}=0.
\]
Exactness of last-layer coinvariants now proves, word for word as in
Lemma~\ref{md:lem:integral-irred}, that $V$ is irreducible and
$c_m(V)=\mcS(m)$. More explicitly, this single irreducible
mirabolic layer generates $V$; a proper subrepresentation has zero
$c_m$ and hence is last-layer-trivial, so exactness injects it into
the zero last-layer coinvariant quotient.
By Theorem~\ref{glmd:nil}, its last-layer orbit is regular semisimple.
It remains to exclude the split alternative.

\subsection{The split Jacquet calculation with a central character}
The $\GL(2,O_m)$ version of
Lemma~\ref{md:lem:split-pol} says that a regular split irreducible
representation is of the form
\[
 I_m^{\mathrm{GL}}(\beta)=\ind_{B(O_m)}^{\GL(2,O_m)}\beta,
\]
where $\beta$ is a character of the full diagonal torus, with the
prescribed scalar restriction and regular last-congruence differential.
To verify the extension of the lemma, keep the scalar part of the
centralizer throughout its polarization proof. At even length $m+1=2r$, the stabilizer representation is a
character of $\widetilde T(O_m)(C_r/C_{m+1})$; its scalar
restriction is prescribed by $\omega^\circ$. At odd length the two
root directions give the same Heisenberg group, while both diagonal
directions belong to its radical. One is scalar and its character
is already prescribed. The split torus preserves the upper-root
polarization and acts genuinely on the Schr\"odinger model. The
remaining multiplicity is a character. Fourier transform in the same
root coordinates identifies these representations with $I_m^{\mathrm{GL}}(\beta)$,
exactly as in the proof of Lemma~\ref{md:lem:split-pol}.
There is no additional projective action arising from the scalar torus.

Its unnormalized Jacquet module fits into
\[
 0\longrightarrow\delta_m^{\mathrm{GL}}\beta^w\longrightarrow
 (I_m^{\mathrm{GL}}(\beta))_{N(O_m)}\longrightarrow\beta\longrightarrow0,
\]
where
\[
 \delta_m^{\mathrm{GL}}(\operatorname{diag}(a,d))
   =|\det_F(a/d:O_m\to O_m)|.
\]
Thus $\delta_m^{\mathrm{GL}}=\delta_m\circ\operatorname{pr}$ on
$\widetilde T(O_m)$, where $\delta_m$ is the projective-group
character of Lemma~\ref{md:lem:split-J}.
The open chart is $O_m$, so integration gives the first character.
On the complementary chart $y\in tO_m$, the stabilizer matrix is
$\left(\begin{smallmatrix}(1+by)^{-1}&b\\0&1+by\end{smallmatrix}\right)$.
Its last-congruence diagonal character is nontrivial on each
nonzero valuation stratum, by regular splitness; scalar normalization
does not change this determinant-one character. Thus the same
closed-stratum calculation as in Lemma~\ref{md:lem:split-J} leaves
only the evaluation line $\beta$. This proves the exact sequence.

In particular, for the constant matrix $\widetilde a_0$ of \eqref{glmd:a0},
\begin{equation}\label{glmd:det-positive}
 \dim (I_m^{\mathrm{GL}}(\beta))_{N(O_m)}=2,\qquad
 \det(\widetilde a_0\mid (I_m^{\mathrm{GL}}(\beta))_{N(O_m)})=q^{-2(m+1)}.
\end{equation}
The product of the two $\beta$-factors is one because $\widetilde a_0^w=\widetilde a_0^{-1}$.

Put $D=C_{m+1}N(O)$ and
$s_{m+1}=\left(\begin{smallmatrix}0&1\\t^{m+1}&0\end{smallmatrix}\right)$.
The group $D$ is generated by the upper root at level zero, the
lower root at level $m+1$, and the diagonal principal congruence
subgroup at level $m+1$. Thus $s_{m+1}$ normalizes it and acts on
\[
 Q=(\mW^\circ)_D=V_{N(O_m)}.
\]
It conjugates $\widetilde a_0$ to $\widetilde a_0^{-1}$. If the orbit of $V$ were split,
\eqref{glmd:det-positive} would give a determinant different from its
inverse, a contradiction. The orbit is therefore elliptic.
As in the integral $\PGL(2)$ argument, its intersection with the
annihilator of $U_m$ is empty, so $V_{U_m}=0$. Hence
\[
 V=c_m(V)=\mcS(m)\quad\text{as an }M(O_m)\text{-representation}.
\]
This proves specialness directly, without requiring a separate
classification of all even-depth elliptic $\GL(2)$ types.
The identity map of $V$ and right adjunction give a nonzero map
$\mW^\circ\to\mathcal I_{J,\omega^\circ}^{\widetilde G(K)}(\operatorname{infl}V)$.
The target is irreducible by Theorem~\ref{GL2-full-extension}, proving
Theorem~\ref{glmd:integral} at positive depth.

\subsection{Positive half-integral depth}
Let $m=n+1/2$ and put $N=n+1$, so $m=N-1/2$.
By Definition~\ref{glmd:definitions}, $V_m\ne0$ is an ordinary
smooth representation of $H_m^{\mathrm{GL}}=J'/\widetilde K_m$
with central character $\omega^\circ$, generated by finitely many
vectors under the group action. Equivalently, it is a finitely
generated module over the unital abstract group algebra
$\CC[H_m^{\mathrm{GL}}]$. This is not a finite-dimensionality
assumption. The union of a chain of proper submodules is proper:
otherwise one member of the chain would contain all the finitely
many generators. Zorn's lemma therefore gives a maximal proper
submodule and a nonzero irreducible quotient
\[
 V_m\twoheadrightarrow\tau.
\]
Every submodule and quotient is smooth, and the quotient inherits
the specified scalar action. Thus this is an irreducible quotient
in the ordinary smooth fixed-central-character category, not an
application of a maximal-subobject assertion for arbitrary
pro-representations. Its last-layer support is nonempty and is
contained in $\widetilde S(\mW)$, hence in $uv\ne0$ by
Theorem~\ref{glmd:nil}.

We verify the fixed-central-character elliptic--special argument
using actual matrices. Put
\[
 \mathfrak A=\begin{pmatrix}O&O\\tO&O\end{pmatrix},\qquad
 \mathfrak P=\begin{pmatrix}tO&O\\tO&tO\end{pmatrix},\qquad
 \widetilde U=(1+\mathfrak P^N)/(1+\mathfrak P^{2N}).
\]
Here $\mathfrak P^2=t\mathfrak A$ and
$1+\mathfrak P^{2N}=C'_N=\widetilde K_m$.
The subgroup $\widetilde U\triangleleft H_m^{\mathrm{GL}}$ is
commutative, since products of two elements of $\mathfrak P^N$
belong to $\mathfrak P^{2N}$. Its image in $G'(O)_m$ is the
subgroup $U$ of Claim~\ref{Gprime-U}, with $n$ there replaced by $N$.
The kernel of $\widetilde U\to U$ consists exactly of the classes
of scalar matrices
\[
 zI_2,\qquad z\in1+t^{\lceil N/2\rceil}O,
 \quad\text{modulo }1+t^NO.
\]
Indeed, $zI_2\in1+\mathfrak P^N$ is precisely the displayed
congruence condition. Every such scalar acts on $\tau$ by
\[
 \omega^\circ(z)=\omega(z)\chi_+(z)^{-2}
 =\omega(z)\omega(z^{1/2})^{-2}=1.
\]
Consequently every character $\lambda$ in the $\widetilde U$-support
of $\tau$ is the pullback of a character $\bar\lambda$ of $U$.
This verifies explicitly the scalar-principal-unit normalization.

Choose such a $\lambda$. Its last-layer restriction is elliptic.
Claim~\ref{ch}, applied to $\bar\lambda$, supplies a ramified
quadratic field $L/K$ embedded by multiplication in $\GL(2,K)$.
The stabilizer in $H_m^{\mathrm{GL}}$ is
\[
 S_\lambda=\widetilde T_L\widetilde U,
\]
where $\widetilde T_L$ is the image of $L^*$ in
$H_m^{\mathrm{GL}}$. To see this, take the inverse image of
$\St_{\bar\lambda}=T_m^LU$; the additional central kernel is
already contained in $L^*$.

The little-group construction of Lemma~\ref{MM} gives an
irreducible smooth stabilizer representation
$\rho=\tau_{\widetilde U,\lambda}$ and
\[
 \tau\simeq\ind_{S_\lambda}^{H_m^{\mathrm{GL}}}\rho,
 \qquad \rho(u)=\lambda(u)\operatorname{Id}\quad(u\in\widetilde U).
\]
To apply that lemma in its finite-dimensional algebraic-group
formulation, choose $a\in\CC^*$ with $a^2=\omega^\circ(t)$ and
temporarily twist by the character $g\mapsto a^{-v(\det g)}$.
The twist is trivial on $\widetilde U$ and makes $tI_2$ act
trivially. The representation then descends to
$H_m^{\mathrm{GL}}/\langle tI_2\rangle$, the group of $F$-points
of a finite extension of the Iwahori jet group. Apply
Lemma~\ref{MM} there and undo the twist. This gives precisely the
displayed little-group representation with its original central
character. In particular,
\[
 \ker\lambda\subset\ker\rho.
\]
This is the required factorization statement. The group
$\widetilde T_L$ is commutative because multiplication in $L$ gives
commuting matrices, and $\widetilde U$ is commutative as just
checked. For $d\in\widetilde T_L$ and $u\in\widetilde U$,
\[
 \lambda(dud^{-1}u^{-1})
   =\lambda(dud^{-1})\lambda(u)^{-1}=1.
\]
Thus $[S_\lambda,S_\lambda]\subset\ker\lambda\subset\ker\rho$,
and $\rho$ is an irreducible smooth representation of the abelian
group $S_\lambda/\ker\lambda$. All these assertions hold in the
matrix quotient: $[1+\mathfrak P^N,1+\mathfrak P^N]
\subset1+\mathfrak P^{2N}$, which already acts trivially, and any
two representatives in $1+\mathfrak P^N$ with the same projective
class differ, modulo $1+\mathfrak P^{2N}$, by one of the scalar
principal units explicitly treated above.

For completeness, the abelian quotient is locally compact and
second countable. An open stabilizer of a nonzero vector fixes the
whole irreducible representation, by commutativity; the resulting
discrete quotient is countable. Hence $\rho$ is countable-dimensional,
and the countable-dimensional Schur lemma over $\CC$ makes every
group element scalar. Irreducibility then gives $\dim\rho=1$.
One can see the Schur-lemma step directly: if an endomorphism $A$
were nonscalar, every $A-c$, $c\in\CC$, would be invertible;
for a nonzero vector $v$, the vectors $(A-c)^{-1}v$ would be
linearly independent as $c$ varies. Clearing denominators and
factoring the resulting polynomial proves independence, contradicting
countable dimension.

Finally, the mirabolic assertions in Claim~\ref{ch}(2)--(3) lift
without a central ambiguity. The kernel of the projective map lies
in $S_\lambda$, so
$M(O)S_\lambda=H_m^{\mathrm{GL}}$, where $M(O)$ denotes its image.
Projectivization identifies $\widetilde U\cap M(O)$ with
$U\cap P(O)_{N-1}$: divide an upper triangular representative by
its lower-right entry, a scalar in
$1+t^{\lceil N/2\rceil}O$. Since $\lambda$ is trivial on that
scalar subgroup, the injectivity of restriction in
Claim~\ref{ch}(3) applies unchanged. The one-dimensional $\rho$
remains irreducible on the mirabolic stabilizer.
Lemma~\ref{3.12} therefore gives an irreducible restriction of
$\tau$ to $M(O)$. Its last-layer support has a nontrivial orbit,
so $\dim\tau>1$. Thus $\tau$ is a special local type with the
given central character. In particular, the matrix relation
$s^2=tI_2$ acts by $\omega^\circ(t)$ throughout.

Inflate $\tau$ to $J'$. Right adjunction and
Theorem~\ref{GL2-full-extension} give
\[
 \mW^\circ\simeq\mathcal I_{J',\omega^\circ}^{\widetilde G(K)}(\operatorname{infl}\tau).
\]
The recovery calculation below identifies $V_m=\tau$, proving
Theorem~\ref{glmd:half}. This argument uses admissibility and minimality,
but not the very special hypothesis.

\subsection{Adjunction and recovery in the fixed-central-character category}
The right adjunction used above is
\[
 \operatorname{Hom}_{\widetilde G(K)}(\mW^\circ,\mathcal I_{J_*,\omega^\circ}^{\widetilde G(K)}\tau)
 \simeq\operatorname{Hom}_{J_*}(\mW^\circ|_{J_*},\tau).
\]
In the vertex case it follows from the original right adjunction for
$\GL(2,O)$ by imposing the central equalizer
\eqref{GL2-central-equalizer}. In the edge case impose equivariance
under $s$ in induction from $\widetilde I$, as in Section~\ref{GL2-direct-proof}.
A map to a type killed by $\widetilde K_m$ factors uniquely through
$V_m$; thus the right-hand side is the Hom-space used in the proofs.

For recovery, let $\tau^\circ$ be the inflated local type and put
\[
 Q=\widetilde K_m,\qquad
 \mathcal I=\mathcal I_{J_*,\omega^\circ}^{\widetilde G(K)}(\tau^\circ).
\]
We separate the finite orbit calculation from the passage to the
pro-category.

\textit{Finite truncations and orbit coinvariants.}
The quotients $\widetilde G(K)/J$ and $\widetilde G(K)/J'$ are the
respective projective-group quotients. Choose a countable cofinal
exhaustion $X_0\subset X_1\subset\cdots$ by finite-dimensional
closed $J_*$-invariant truncations, each containing the base point
and having finitely many $J_*$-orbits. In the vertex case take finite
unions of closed spherical Schubert varieties, including both
components; in the edge case use closed affine Schubert truncations
and their images in the quotient by $G'(O)/I$.
Let $E_j$ be the ordinary space of compactly supported locally
constant sections of the inducing coefficient bundle on $X_j(F)$,
as in Definition~\ref{indHO} and Section~\ref{GL2-direct-proof}.
Then
\[
 \mathcal I|_Q=\varprojlim_j E_j\quad\text{in }\Rep(Q).
\]
Each $E_j$ is an ordinary smooth $Q$-representation at a finite
congruence level. Indeed, the finite jet torsor defining its
coefficient bundle and the left action on this bounded truncation
factor through sufficiently deep congruence quotients. The depth
needed may increase with $j$; no common finite quotient is required.
Restriction gives the transition maps $p_{j+1,j}:E_{j+1}\to E_j$.
They are surjective, since local trivializations and compact open
neighborhoods extend compactly supported locally constant sections
from a closed subspace. Thus this is also a strict presentation.

The orbit representatives are
$\gamma_r=\operatorname{diag}(t^r,1)$, $r\geq0$.
The stabilizer acting on the inducing fiber for $Q$ at $\gamma_rJ_*$ is
\[
 D_r=\gamma_r^{-1}Q\gamma_r\cap J_*.
\]
For $r\geq1$, this group contains $N(O)$ at depth zero,
$N(t^mO)$ at positive integral depth, and $N(t^nO)$ at depth
$n+1/2$. In each case conjugation by $\gamma_r$ multiplies the
upper-root coordinate by $t^r$, giving the asserted inclusion in
$Q$. The corresponding root-group coinvariants of $\tau^\circ$
vanish, respectively, by cuspidality, elliptic last-layer support,
and $uv\ne0$. Hence $(\tau^\circ)_{D_r}=0$. Central and relative
modular characters are trivial on these unipotent root groups.

At each finite congruence level the homogeneous-bundle coinvariant
calculation first takes coinvariants of the fiber by its stabilizer
and then forms sections over the quotient orbit. As
$Q\triangleleft J_*$, its stabilizers at the other points of a
$J_*$-orbit are conjugate in this calculation. Consequently every
nonidentity orbit-section term has zero $Q$-coinvariants.
The group $Q$ has a surjective congruence presentation by unipotent
$F$-groups. Their groups of $F$-points are filtered unions of
compact open subgroups, so $Q$ is quasi-pro-unipotent and
Theorem~\ref{coinvariants-exact} applies.
Apply exactness to the finite orbit filtration of
$\ker(E_j\to\tau^\circ)$, where the map is evaluation at the base
point. All its orbit terms have zero coinvariants. Evaluation is
surjective and $Q$ acts trivially on its target, so it induces an
isomorphism
\begin{equation}\label{glmd:finite-recovery}
 \overline{\operatorname{ev}}_j:(E_j)_Q
       \xrightarrow{\ \sim\ }\tau^\circ.
\end{equation}

\textit{Compatibility and the pro-object.}
Evaluation commutes with restriction:
\[
 \operatorname{ev}_j\,p_{j+1,j}=\operatorname{ev}_{j+1}.
\]
Therefore the isomorphisms \eqref{glmd:finite-recovery} identify
every transition map $(E_{j+1})_Q\to(E_j)_Q$ with the identity of
$\tau^\circ$. The group $Q$ is bounded and weakly strict, since its
congruence transition maps are surjective. Formula
\eqref{s8-coinvariants-universal} applies to the tower of ordinary
$Q$-representations $\{E_j\}$ and proves
\[
 \mathcal I_Q\simeq\varprojlim_j(E_j)_Q
       \simeq\varprojlim_j(\tau^\circ,\operatorname{id})
       \simeq\tau^\circ.
\]
These are formal pro-objects. In particular the conclusion uses
the universal property of coinvariants and the explicitly constant
tower, not an assertion that strictness alone makes an ordinary
inverse limit exact. All maps are $J_*$-equivariant, so the resulting
isomorphism is one of representations of $J_*/\widetilde K_m$ with
the prescribed central character. This proves all the recovery
statements used above. Undoing the determinant twist gives the
claimed statements for the original representation $\mW$.

\section{Questions}\label{quest}

\begin{enumerate}

\item\label{question:induction}
Is it true that any special representation of $G(K)$ is isomorphic to $\Pi(\pi)$ where $\pi$ is a special representation of either $G(O)$ or $G'(O)$?%
\footnote{For irreducible admissible very special representations,
Question~\ref{question:induction}
is answered affirmatively by Theorems~\ref{md:thm:zero},
\ref{md:thm:integral}, and~\ref{md:thm:half}. At positive
half-integral depth the very special hypothesis can be omitted.}
\item
Does there exist an irreducible representation $\mW$ of $G(K)$ such that
$$
1<\dim\mcW(\mW)<\infty,
$$
where $\mcW(\mW):=\Hom_{N(K)}(\mW,\psi)$ is the space of Whittaker functionals on $\mW$?

\item
The definition of elliptic representations  extends naturally
to the case of maximal bounded  subgroups of  $G(K)$ for an arbitrary split reductive group $G$ (the definition is completely straightforward for the case of $G(O)$; for other maximal bounded subgroups it requires more work, which we shall not discuss here). Can one classify those? Is it true that their induction to $G(K)$ is irreducible?

\item
The notion of  special representations  extends verbatim to the case of  $\PGL(n)$ (in this case $P$ should be the maximal parabolic subgroup which is equal to the stabilizer of a point in $\mathbb P^{n-1}$). Is it possible to extend the results of this paper to the group $\PGL(n)$?

\item
Let $\mcH _{q,t}$ be the double affine Hecke algebra for
$\PGL(2)$. In Section 7.8 of \cite{GK1} there is a
construction of a pro-vector space $\mW$ endowed with commuting left actions of a central extension of $G(K)$ and the algebra $\mcH_{q,t}$. Precomposition gives a right $\mcH_{q,t}$-action on the space of Whittaker functionals:
\[
 (\ell\cdot a)(w)=\ell(aw),\qquad
 \ell\in\mcW(\mW),\quad a\in\mcH_{q,t}.
\]
How can one describe this right $\mcH_{q,t}$-module?
\end{enumerate}

\end{document}